\documentclass[letterpaper]{amsart}

\usepackage[letterpaper]{geometry}
\usepackage{amsmath, amsthm, amsfonts, amsbsy, thmtools, amssymb}

\usepackage{thm-restate}
\usepackage{hyperref} 
\usepackage[]{ytableau}	
\usepackage[mathscr]{euscript}
\usepackage{stmaryrd}
\usepackage[all, arc, curve, frame]{xy} 
\usepackage[T1]{fontenc}
\usepackage{tikz}
\usetikzlibrary{arrows}

\numberwithin{equation}{section}

\SetSymbolFont{stmry}{bold}{U}{stmry}{m}{n}

\DeclareMathOperator{\Sign}{Sign}

\DeclareMathOperator{\Id}{Id} 
\DeclareMathOperator{\b|}{\boldsymbol{|}}
\DeclareMathOperator{\Cov}{Cov}
\DeclareMathOperator{\Ai}{Ai}

\DeclareMathOperator{\Var}{Var}

\DeclareMathOperator{\BR}{BR}

\title{Current Fluctuations of the Stationary ASEP and Six-Vertex Model} 
\author{Amol Aggarwal}

\begin{document}

\newtheorem{thm}{Theorem}[section]
\newtheorem{prop}[thm]{Proposition}
\newtheorem{lem}[thm]{Lemma}
\newtheorem{cor}[thm]{Corollary}
\newtheorem{conj}[thm]{Conjecture}
\newtheorem{que}[thm]{Question}
\theoremstyle{remark}
\newtheorem{rem}[thm]{Remark}
\theoremstyle{definition}
\newtheorem{definition}[thm]{Definition}
\newtheorem{exa}[thm]{Example}

\begin{abstract}

Our results in this paper are two-fold. First, we consider current fluctuations of the stationary asymmetric simple exclusion process (ASEP), run for some long time $T$, and show that they are of order $T^{1 / 3}$ along a characteristic line. Upon scaling by $T^{1 / 3}$, we establish that these fluctuations converge to the long-time height fluctuations of the stationary KPZ equation, that is, to the Baik-Rains distribution. This result has long been predicted under the context of KPZ universality and in particular extends upon a number of results in the field, including the work of Ferrari and Spohn in 2005 (who established the same result for the TASEP), and the work of Bal\'{a}zs and Sepp\"{a}l\"{a}inen in 2010 (who established the $T^{1 / 3}$ scaling for the general ASEP). 

Second, we introduce a class of translation-invariant Gibbs measures that characterizes a one-parameter family of slopes for an arbitrary ferroelectric, symmetric six-vertex model. This family of slopes corresponds to what is known as the \emph{conical singularity} (or \emph{tricritical point}) in the free energy profile for the ferroelectric six-vertex model. We consider fluctuations of the height function of this model on a large grid of size $T$ and show that they too are of order $T^{1 / 3}$ along a certain characteristic line; this confirms a prediction of Bukman and Shore from 1995 stating that the ferroelectric six-vertex model should exhibit KPZ growth at the conical singularity. 

Upon scaling the height fluctuations by $T^{1 / 3}$, we again recover the Baik-Rains distribution in the large $T$ limit. Recasting this statement in terms of the (asymmetric) stochastic six-vertex model confirms a prediction of Gwa and Spohn from 1992.

\end{abstract}

\maketitle

\tableofcontents

\section{Introduction}

\label{Introduction}

Over the past several decades, significant effort has been devoted towards the study of statistical mechanical models at steady (stationary) state. In this paper we address questions of this type for two models, namely, the asymmetric simple exclusion process and the (both symmetric and stochastic) six-vertex model. 

We begin in Section \ref{ProcessModel} by defining these two models and their associated observables. In Section \ref{EquationBoundary}, we provide some context for our results, which will be more carefully stated in Section \ref{StationaryAsymptotics}.

\subsection{The ASEP and Stochastic Six-Vertex Model}

\label{ProcessModel}

Here, we define the asymmetric simple exclusion process and stochastic six-vertex model. Although our results also apply to the standard (symmetric) six-vertex model, we do not define it in this section; its detailed description, as well as a way of mapping it to the stochastic six-vertex model, is given in Appendix \ref{SixVertex} below.

\subsubsection{The Asymmetric Simple Exclusion Process}

\label{AsymmetricExclusions}

Introduced to the mathematics community by Spitzer \cite{IMP} in 1970 (and also appearing two years earlier in the biology work of MacDonald, Gibbs, and Pipkin \cite{KBNAT}), the \emph{asymmetric simple exclusion process} (ASEP) is a continuous time Markov process that can be described as follows. Particles are initially (at time $0$) placed on $\mathbb{Z}$ such that at most one particle occupies any site. Associated with each particle are two exponential clocks, one of rate $L$ and one of rate $R$; we assume that $R > L \ge 0$ and that all clocks are mutually independent. When some particle's left clock rings, the particle attempts to jump one space to the left; similarly, when its right clock rings, it attempts to jump one space to the right. If the destination of the jump is unoccupied, the jump is performed; otherwise it is not. This is sometimes referred to as the \emph{exclusion restriction}. 

Associated with the ASEP is an observable called the \emph{current}. To define this quantity, we \emph{tag} the particles of the ASEP, meaning that we track their evolution over time by indexing them based on initial position. Specifically, let the initial positions of the particles be $\cdots < X_{-1} (0) < X_0 (0) < X_1 (0) < \cdots$, where $X_{-1} (0) \le 0 < X_0 (0)$. The particle initially at site $X_k (0)$ will be referred to as \emph{particle $k$}. For each $k \in \mathbb{Z}$ and $t > 0$, let $X_k (t)$ denote the position of particle $k$ at time $t$. Since all jumps are nearest-neighbor, we have that $\cdots < X_{-1} (t) < X_0 (t) < X_1 (t) < \cdots $ for all $t \ge 0$.

Now, consider the ASEP after running for some time $t$. For any $x \in \mathbb{R}$, define the \emph{current} $J_t (x)$ to be the almost surely finite sum 
\begin{flalign} 
\label{jtx}
J_t (x) = \displaystyle\sum_{i = - \infty}^{\infty} \big( \textbf{1}_{X_i (0) \le 0} \textbf{1}_{X_i (t) > x} - \textbf{1}_{X_i (0) > 0} \textbf{1}_{X_i (t) \le x} \big). 
\end{flalign}

Observe in particular that at most one of the two summands on the right side of \eqref{jtx} is nonzero. Further observe that $J_t (x)$ has the following combinatorial interpretation. Color all particles initially to the right of $0$ red, and all particles initially at or to the left of $0$ blue. Then, $J_t (x)$ denotes the number of red particles at or to the left of $x$ at time $t$ subtracted from the number of blue particles to the right of $x$ at time $t$.  

One of the purposes of this paper is to analyze the long-time fluctuations for the current of the ASEP under a certain type of \emph{double-sided $(b_1, b_2)$-Bernoulli initial data}, for fixed $b_1, b_2 \in (0, 1)$. This means that one initially places a particle at each site $i \in \mathbb{Z}_{\le 0}$ with probability $b_1$ and at each site $i \in \mathbb{Z}_{> 0}$ with probability $b_2$; all placements are independent. 

We will in fact be interested in the \emph{stationary} case of this initial data, when $b_1 = b_2$, but we postpone further discussion about this to Section \ref{EquationBoundary} and Section \ref{StationaryProcess}.

\begin{figure}[t]

\begin{center}

\begin{tikzpicture}[
      >=stealth,
			]

			\draw[->, black	] (0, 0) -- (0, 4.5);
			\draw[->, black] (0, 0) -- (4.5, 0);
			\draw[->,black, thick] (0, .5) -- (.45, .5);
			\draw[->,black, thick] (0, 1) -- (.45, 1);
			\draw[->,black, thick] (0, 1.5) -- (.45, 1.5);
			\draw[->,black, thick] (0, 2) -- (.45, 2);
			\draw[->,black, thick] (0, 2.5) -- (.45, 2.5);
			\draw[->,black, thick] (0, 3) -- (.45, 3);
			\draw[->,black, thick] (0, 3.5) -- (.45, 3.5);

			\draw[->,black, thick] (.55, .5) -- (.95, .5);
			\draw[->,black, thick] (.55, 1) -- (.95, 1);
			\draw[->,black, thick] (.55, 1.5) -- (.95, 1.5);
			\draw[->,black, thick] (.55, 2.5) -- (.95, 2.5);
			\draw[->,black, thick] (.55, 3) -- (.95, 3);
			\draw[->,black, thick] (.55, 3.5) -- (.95, 3.5);
			\draw[->,black, thick] (.5, 2.05) -- (.5, 2.45);
			\draw[->,black, thick] (.5, 2.55) -- (.5, 2.95);
			\draw[->,black, thick] (.5, 3.05) -- (.5, 3.45);
			\draw[->,black, thick] (.5, 3.55) -- (.5, 3.95);
			
			\draw[->,black, thick] (1.05, .5) -- (1.45, .5);
			\draw[->,black, thick] (1.05, 1) -- (1.45, 1);
			\draw[->,black, thick] (1.05, 2) -- (1.45, 2);
			\draw[->,black, thick] (1.05, 2.5) -- (1.45, 2.5);
			\draw[->,black, thick] (1.05, 3.5) -- (1.45, 3.5);
			\draw[->,black, thick] (1, 1.55) -- (1, 1.95);
			\draw[->,black, thick] (1, 3.05) -- (1, 3.45);
			\draw[->,black, thick] (1, 3.55) -- (1, 3.95);
			
			\draw[->,black, thick] (1.5, .55) -- (1.5, .95);
			\draw[->,black, thick] (1.5, 1.05) -- (1.5, 1.45);
			\draw[->,black, thick] (1.5, 1.55) -- (1.5, 1.95);
			\draw[->,black, thick] (1.5, 2.05) -- (1.5, 2.45);
			\draw[->,black, thick] (1.5, 2.55) -- (1.5, 2.95);
			\draw[->,black, thick] (1.5, 3.55) -- (1.5, 3.95);
			\draw[->,black, thick] (1.55, 1) -- (1.95, 1);
			\draw[->,black, thick] (1.55, 2) -- (1.95, 2);
			\draw[->,black, thick] (1.55, 2.5) -- (1.95, 2.5);
			\draw[->,black, thick] (1.55, 3) -- (1.95, 3);

			\draw[->,black, thick] (2, 1.05) -- (2, 1.45);
			\draw[->,black, thick] (2, 3.05) -- (2, 3.45);
			\draw[->,black, thick] (2, 3.55) -- (2, 3.95);
			\draw[->,black, thick] (2.05, 1.5) -- (2.45, 1.5);
			\draw[->,black, thick] (2.05, 2) -- (2.45, 2);
			\draw[->,black, thick] (2.05, 2.5) -- (2.45, 2.5);

			\draw[->,black, thick] (2.5, 1.55) -- (2.5, 1.95);
			\draw[->,black, thick] (2.5, 2.05) -- (2.5, 2.45);
			\draw[->,black, thick] (2.5, 2.55) -- (2.5, 2.95);
			\draw[->,black, thick] (2.5, 3.05) -- (2.5, 3.45);
			\draw[->,black, thick] (2.5, 3.55) -- (2.5, 3.95);
			\draw[->,black, thick] (2.55, 2) -- (2.95, 2);
			\draw[->,black, thick] (2.55, 2.5) -- (2.95, 2.5);

			\draw[->,black, thick] (3, 2.55) -- (3, 2.95);
			\draw[->,black, thick] (3.05, 2) -- (3.45, 2);
			\draw[->,black, thick] (3.05, 3) -- (3.45, 3);

			\draw[->,black, thick] (3.5, 2.05) -- (3.5, 2.45);
			\draw[->,black, thick] (3.5, 2.55) -- (3.5, 2.95);
			\draw[->,black, thick] (3.5, 3.05) -- (3.5, 3.45);
			\draw[->,black, thick] (3.5, 3.55) -- (3.5, 3.95);
			\draw[->,black, thick] (3.55, 3) -- (3.95, 3);
		
			\draw[->,black, thick] (4, 3.05) -- (4, 3.45);
			\draw[->,black, thick] (4, 3.55) -- (4, 3.95);

			\filldraw[fill=gray!50!white, draw=black] (.5, .5) circle [radius=.05];
			\filldraw[fill=gray!50!white, draw=black] (.5, 1) circle [radius=.05];
			\filldraw[fill=gray!50!white, draw=black] (.5, 1.5) circle [radius=.05];
			\filldraw[fill=gray!50!white, draw=black] (.5, 2) circle [radius=.05];
			\filldraw[fill=gray!50!white, draw=black] (.5, 2.5) circle [radius=.05];
			\filldraw[fill=gray!50!white, draw=black] (.5, 3) circle [radius=.05];
			\filldraw[fill=gray!50!white, draw=black] (.5, 3.5) circle [radius=.05];

			\filldraw[fill=gray!50!white, draw=black] (1, .5) circle [radius=.05];
			\filldraw[fill=gray!50!white, draw=black] (1, 1) circle [radius=.05];
			\filldraw[fill=gray!50!white, draw=black] (1, 1.5) circle [radius=.05];
			\filldraw[fill=gray!50!white, draw=black] (1, 2) circle [radius=.05];
			\filldraw[fill=gray!50!white, draw=black] (1, 2.5) circle [radius=.05];
			\filldraw[fill=gray!50!white, draw=black] (1, 3) circle [radius=.05];
			\filldraw[fill=gray!50!white, draw=black] (1, 3.5) circle [radius=.05];
			
			\filldraw[fill=gray!50!white, draw=black] (1.5, .5) circle [radius=.05];
			\filldraw[fill=gray!50!white, draw=black] (1.5, 1) circle [radius=.05];
			\filldraw[fill=gray!50!white, draw=black] (1.5, 1.5) circle [radius=.05];
			\filldraw[fill=gray!50!white, draw=black] (1.5, 2) circle [radius=.05];
			\filldraw[fill=gray!50!white, draw=black] (1.5, 2.5) circle [radius=.05];
			\filldraw[fill=gray!50!white, draw=black] (1.5, 3) circle [radius=.05];
			\filldraw[fill=gray!50!white, draw=black] (1.5, 3.5) circle [radius=.05];
			
			\filldraw[fill=gray!50!white, draw=black] (2, .5) circle [radius=.05];
			\filldraw[fill=gray!50!white, draw=black] (2, 1) circle [radius=.05];
			\filldraw[fill=gray!50!white, draw=black] (2, 1.5) circle [radius=.05];
			\filldraw[fill=gray!50!white, draw=black] (2, 2) circle [radius=.05];
			\filldraw[fill=gray!50!white, draw=black] (2, 2.5) circle [radius=.05];
			\filldraw[fill=gray!50!white, draw=black] (2, 3) circle [radius=.05];
			\filldraw[fill=gray!50!white, draw=black] (2, 3.5) circle [radius=.05];
			
			\filldraw[fill=gray!50!white, draw=black] (2.5, .5) circle [radius=.05];
			\filldraw[fill=gray!50!white, draw=black] (2.5, 1) circle [radius=.05];
			\filldraw[fill=gray!50!white, draw=black] (2.5, 1.5) circle [radius=.05];
			\filldraw[fill=gray!50!white, draw=black] (2.5, 2) circle [radius=.05];
			\filldraw[fill=gray!50!white, draw=black] (2.5, 2.5) circle [radius=.05];
			\filldraw[fill=gray!50!white, draw=black] (2.5, 3) circle [radius=.05];
			\filldraw[fill=gray!50!white, draw=black] (2.5, 3.5) circle [radius=.05];
			
			\filldraw[fill=gray!50!white, draw=black] (3, .5) circle [radius=.05];
			\filldraw[fill=gray!50!white, draw=black] (3, 1) circle [radius=.05];
			\filldraw[fill=gray!50!white, draw=black] (3, 1.5) circle [radius=.05];
			\filldraw[fill=gray!50!white, draw=black] (3, 2) circle [radius=.05];
			\filldraw[fill=gray!50!white, draw=black] (3, 2.5) circle [radius=.05];
			\filldraw[fill=gray!50!white, draw=black] (3, 3) circle [radius=.05];
			\filldraw[fill=gray!50!white, draw=black] (3, 3.5) circle [radius=.05];
			
			\filldraw[fill=gray!50!white, draw=black] (3.5, .5) circle [radius=.05];
			\filldraw[fill=gray!50!white, draw=black] (3.5, 1) circle [radius=.05];
			\filldraw[fill=gray!50!white, draw=black] (3.5, 1.5) circle [radius=.05];
			\filldraw[fill=gray!50!white, draw=black] (3.5, 2) circle [radius=.05];
			\filldraw[fill=gray!50!white, draw=black] (3.5, 2.5) circle [radius=.05];
			\filldraw[fill=gray!50!white, draw=black] (3.5, 3) circle [radius=.05];
			\filldraw[fill=gray!50!white, draw=black] (3.5, 3.5) circle [radius=.05];
			
			\filldraw[fill=gray!50!white, draw=black] (4, .5) circle [radius=.05];
			\filldraw[fill=gray!50!white, draw=black] (4, 1) circle [radius=.05];
			\filldraw[fill=gray!50!white, draw=black] (4, 1.5) circle [radius=.05];
			\filldraw[fill=gray!50!white, draw=black] (4, 2) circle [radius=.05];
			\filldraw[fill=gray!50!white, draw=black] (4, 2.5) circle [radius=.05];
			\filldraw[fill=gray!50!white, draw=black] (4, 3) circle [radius=.05];
			\filldraw[fill=gray!50!white, draw=black] (4, 3.5) circle [radius=.05];

\end{tikzpicture}

\end{center}	

\caption{\label{figurevertexwedge} A sample of the stochastic six-vertex model with step initial data is depicted above.} 
\end{figure}
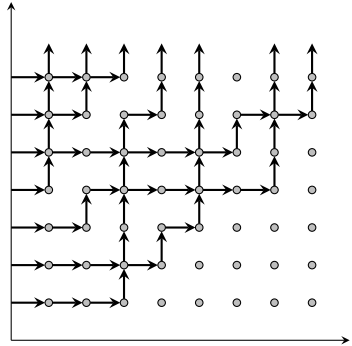

\subsubsection{The Stochastic Six-Vertex Model}

\label{StochasticVertex}

The stochastic six-vertex model was first introduced to the mathematical physics community by Gwa and Spohn \cite{SVMRSASH} in 1992 as a stochastic version of the older \cite{TEIOCSRAA, TT} six-vertex model studied by Lieb \cite{RESI}, Baxter \cite{ESMSM}, and Sutherland-Yang-Yang \cite{ESMTDFAEEF}; it was also studied more recently in \cite{PTAEPSSVM, SSVM, HSVMRSF, IPSVMSF, SVMCP, LSSSVM}. This model can be defined in several equivalent ways, including as a Gibbs measure for a ferroelectric, symmetric six-vertex model (see Section 2.1 of \cite{SSVM}); as an interacting particle system with push dynamics and an exclusion restriction (see \cite{DSENE, SVMRSASH} or Section 2.2 of \cite{SSVM}); or as a probability measure on directed-path ensembles (see Section 2 of \cite{SSVM} or Section 1 of \cite{HSVMRSF}). In this section we will define the model through path ensembles, although we will also explain the interpretation as a Gibbs measure for the six-vertex model in Appendix \ref{SixVertex}.

A \emph{six-vertex directed-path ensemble} is a family of up-right directed paths in the non-negative quadrant $\mathbb{Z}_{> 0}^2$, such that each path emanates from either the $x$-axis or $y$-axis, and such that no two paths share an edge (although they may share vertices); see Figure \ref{figurevertexwedge}. In particular, each vertex has six possible \emph{arrow configurations}, which are listed in the top row of Figure \ref{sixvertexfigure}. We view the second arrow configuration in Figure \ref{sixvertexfigure} as two paths reflecting off of one another as opposed to crossing each other.

\emph{Initial data}, or \emph{boundary data}, for such an ensemble is prescribed by dictating which vertices on the positive $x$-axis and positive $y$-axis are entrance sites for a directed path. One example of initial data is \emph{step initial data}, in which paths only enter through the $y$-axis, and all vertices on the $y$-axis are entrance sites for paths; see Figure \ref{figurevertexwedge}. A more general example is \emph{double-sided $(b_1, b_2)$-Bernoulli initial data}, in which sites on the $y$-axis are independently entrance sites with probability $b_1$, and sites on the $x$-axis are independently entrance sites with probability $b_2$. 

Now, fix parameters $\delta_1, \delta_2 \in [0, 1]$ and some initial data. The \emph{stochastic six-vertex model} $\mathcal{P} = \mathcal{P} (\delta_1, \delta_2)$ will be the infinite-volume limit of a family of probability measures $\mathcal{P}_n = \mathcal{P}_n (\delta_1, \delta_2)$ defined on the set of six-vertex directed-path ensembles whose vertices are all contained in triangles of the form $\mathbb{T}_n = \{ (x, y) \in \mathbb{Z}_{\ge 0}^2: x + y \le n \}$. The first such probability measure $\mathcal{P}_1$ is supported by the empty ensemble.

For each positive integer $n$, we define $\mathcal{P}_{n + 1}$ from $\mathcal{P}_n$ through the following Markovian update rules. Use $\mathcal{P}_n$ to sample a directed-path ensemble $\mathcal{E}_n$ on $\mathbb{T}_n$. This gives arrow configurations (of the type shown in Figure \ref{sixvertexfigure}) to all vertices in the positive quadrant strictly below the diagonal $\mathbb{D}_n = \{ (x, y) \in \mathbb{Z}_{> 0}^2: x + y = n \}$. Each vertex on $\mathbb{D}_n$ is also given ``half'' of an arrow configuration, in the sense that it is given the directions of all entering paths but no direction of any exiting path. 

To extend $\mathcal{E}_n$ to a path ensemble on $\mathbb{T}_{n + 1}$, we must ``complete'' the configurations (specify the exiting paths) at all vertices $(x, y) \in \mathbb{D}_n$. Any half-configuration can be completed in at most two ways; selecting between these completions is done randomly, according to the probabilities given in the second row of Figure \ref{sixvertexfigure}. All choices are mutually independent. 

In this way, we obtain a random ensemble $\mathcal{E}_{n + 1}$ on $\mathbb{T}_{n + 1}$; the resulting probability measure on path ensembles with vertices in $\mathbb{T}_{n + 1}$ is denoted $\mathcal{P}_{n + 1}$. Now, set $\mathcal{P} = \lim_{n \rightarrow \infty} \mathcal{P}_n$. 

As in the ASEP, there exists an observable of interest for stochastic six-vertex model called the \emph{height function} (although we will sometimes also use the term \emph{current}). To define it, we color a path red if it emanates from the $x$-axis, and we color a path blue if it emanates from the $y$-axis. Let $(X, Y) \in \mathbb{R}_{> 0}^2$. The \emph{current} (or \emph{height function}) $\mathfrak{H} (X, Y)$ of the stochastic six-vertex model at $(X, Y)$ is the number of red paths that intersect the line $y = Y$ at or to the left of $(X, Y)$ subtracted from the number of blue paths that intersect the line $y = Y$ to the right of $(X, Y)$; at most one of these numbers is nonzero, as on the right side of \eqref{jtx}. 

We will again be interested in understanding the current fluctuations for the stochastic six-vertex model under a certain type of \emph{translation-invariant} double-sided Bernoulli initial data; we will discuss this further in the next two sections.

\begin{figure}[t]

\begin{center}

\begin{tikzpicture}[
      >=stealth,
			scale = .7
			]

			\draw[-, black] (-7.5, -.8) -- (7.5, -.8);
			\draw[-, black] (-7.5, 0) -- (7.5, 0);
			\draw[-, black] (-7.5, 2) -- (7.5, 2);
			\draw[-, black] (-7.5, -.8) -- (-7.5, 2);
			\draw[-, black] (7.5, -.8) -- (7.5, 2);
			\draw[-, black] (-5, -.8) -- (-5, 2);
			\draw[-, black] (5, -.8) -- (5, 2);
			\draw[-, black] (-2.5, -.8) -- (-2.5, 2);
			\draw[-, black] (2.5, -.8) -- (2.5, 2);
			\draw[-, black] (0, -.8) -- (0, 2);

			\draw[->, black,  thick] (3.85, 1) -- (4.65, 1);
			\draw[->, black,  thick] (3.75, .1) -- (3.75, .9);

			\draw[->, black,  thick] (-1.25, .1) -- (-1.25, .9);
			\draw[->, black,  thick] (-1.25, 1.1) -- (-1.25, 1.9);

			\draw[->, black,  thick] (1.35, 1) -- (2.15, 1);
			\draw[->, black,  thick] (.35, 1) -- (1.15, 1);
			
			\draw[->, black,  thick] (6.25, 1.1) -- (6.25, 1.9);
			\draw[->, black,  thick] (5.35, 1) -- (6.15, 1);
			
			\draw[->, black,  thick] (-3.75, 1.1) -- (-3.75, 1.9);
			\draw[->, black,  thick] (-3.75, .1) -- (-3.75, .9);
			\draw[->, black,  thick] (-3.65, 1) -- (-2.85, 1);
			\draw[->, black,  thick] (-4.65, 1) -- (-3.85, 1);
				
			\filldraw[fill=gray!50!white, draw=black] (-6.25, 1) circle [radius=.1] node [black,below=21] {$1$};
			\filldraw[fill=gray!50!white, draw=black] (-3.75, 1) circle [radius=.1] node [black,below=21] {$1$};
			\filldraw[fill=gray!50!white, draw=black] (-1.25, 1) circle [radius=.1] node [black,below=21] {$\delta_1$};
			\filldraw[fill=gray!50!white, draw=black] (1.25, 1) circle [radius=.1] node [black,below=21] {$\delta_2$};
			\filldraw[fill=gray!50!white, draw=black] (3.75, 1) circle [radius=.1] node [black,below=21] {$1 - \delta_1$};
			\filldraw[fill=gray!50!white, draw=black] (6.25, 1) circle [radius=.1] node [black,below=21] {$1 - \delta_2$};

\end{tikzpicture}

\end{center}

\caption{\label{sixvertexfigure} The top row in the chart shows the six possible arrow configurations at vertices in the stochastic six-vertex model; the bottom row shows the corresponding probabilities. }
\end{figure}

\subsection{Context and Background}

\label{EquationBoundary}

The phenomenon that guides our results is commonly termed \emph{KPZ universality}. About thirty years ago in their seminal paper \cite{DSGI}, Kardar, Parisi, and Zhang considered a family of random growth models that exhibit ostensibly unusual (although now known to be quite ubiquitous) scaling behavior. As part of this work, they predicted the scaling exponents for all one-dimensional models in this family; specifically, after running such a model for some large time $T$, they predicted fluctuations of order $T^{1 / 3}$ and non-trivial spacial correlation on scales $T^{2 / 3}$. This family of random growth models is now called the \emph{Kardar-Parisi-Zhang (KPZ) universality class}, and consists of many more models (including the ASEP and stochastic six-vertex model) than those originally considered in \cite{DSGI}; we refer to the surveys \cite{EUC, PQISUC, IUE} for more information.

In addition to predicting these exponents, Kardar, Parisi, and Zhang proposed a stochastic differential equation that in a sense embodies all of the models in their class; this equation, now known as the \emph{KPZ equation}, is 
\begin{flalign}
\label{stochasticequation} 
\partial_t \mathscr{H} = \partial_x^2 \mathscr{H} + \displaystyle\frac{1}{2} \left( \partial_x \mathscr{H} \right)^2 + \dot{\mathscr{W}},
\end{flalign}

\noindent where $\dot{\mathscr{W}}$ refers to space-time white noise. 

Granting the well-posedness \cite{SEPS, ATR, SE} of \eqref{stochasticequation} (which by no means immediate, due to the non-linearity $\frac{1}{2} \big( \partial_x \mathscr{H} \big)^2$), it is widely believed that the long-time statistics of a stochastic model in the KPZ universality class should coincide with the long-time statistics of \eqref{stochasticequation}, whose initial data should be somehow chosen to ``match'' the initial data of the model, in a suitable way. 

From a probabilistic standpoint, perhaps the most interesting type of initial data on which to understand this conjecture is \emph{stationary initial data}. For the KPZ equation, 	this is equivalent to what is known as \emph{Brownian initial data}, in which $\mathscr{H}$ is a two-sided Brownian motion. This was studied by Bal\'{a}zs-Quastel-Sepp\"{a}l\"{a}inen \cite{SES}, Corwin-Quastel \cite{CDERF}, Imamura-Sasamoto \cite{SCE}, and Borodin-Corwin-Ferrari-Vet\H{o} \cite{HFSE}. In particular, it was shown in \cite{HFSE} that the height fluctuations of the stationary KPZ equation, after running for some large time $T$, are of order $T^{1 / 3}$ and scale to the \emph{Baik-Rains distribution} (see Definition \ref{stationarydistribution}); this is a distribution that was introduced in 2000 by Baik and Rains \cite{LDPGMES} in the context of a polynuclear growth (PNG) model with critical boundary conditions. In view of the KPZ universality conjecture, one expects to see similar statistics in the stationary ASEP and translation-invariant stochastic six-vertex model. 

Understanding this conjecture for the \emph{stationary ASEP} has been the topic of intense study since the 1980s \cite{LPSCL, CTPS, FEP, OCVDA, ENDDS, SEPS, CFA, SLSCST, SAEPD, DFAEP, SFRAEP, LSDIP}. This refers to the ASEP with double-sided $(b, b)$-Bernoulli (also called \emph{$b$-stationary}) initial data, for some fixed $b \in (0, 1)$; recall that this means that each site is initially occupied with probability $b \in (0, 1)$, and all occupations are independent. 

The first work to predict KPZ growth in the stationary ASEP was due to van Beijern, Kutner, and Spohn \cite{ENDDS} in 1985, in fact one year before the article \cite{DSGI} of Kardar-Parisi-Zhang. In that paper \cite{ENDDS}, the three authors consider the \emph{two-point function} $S_t (x)$ of the stationary ASEP, defined by $S_t (x) = \Cov \big( \eta_0 (0), \eta_t (x) \big) = \mathbb{E} \big[ \eta_0 (0) \eta_t (x) \big] - b^2$, for any integer $x$ and non-negative real number $t$; here, $\eta_t (x)$ denotes the indicator that site $x$ is occupied at time $t$. 

In \cite{ENDDS}, van Beijern, Kutner, and Spohn expected $S_T ( vT )$ to decay exponentially in $T$ for all fixed $v \in \mathbb{R}$, except for one value of $v = (1 - 2b) (R - L)$. At this $v$ they predicted that $S_T (vT)$ should decay as $T^{-2 / 3}$ and also that the same polynomial decay holds for $S_T (vT + c T^{2 / 3})$, for any fixed $c > 0$. The value $v = (1 - 2b) (R - L)$ is often called the \emph{characteristic velocity} of the stationary ASEP, and the line $x = v t$ is often called the \emph{characteristic line}. 

The two-point function $S_t$ and current $J_t$ (recall its definition from Section \ref{AsymmetricExclusions}) of the stationary ASEP are in fact related by the identity 
\begin{flalign}
\label{variancefunction}
2 S_t (x) = \Var J_t (x + 1) - 2\Var J_t (x) + \Var J_t (x - 1);
\end{flalign}

\noindent we refer to Proposition 4.1 of \cite{CFTEP} or Proposition 2.1 of \cite{FEP} for a proof of \eqref{variancefunction}. Thus, the current is in a sense a more ``general quantity'' than the two point function; in order to obtain the scaling limit of the latter, it suffices\footnote{This is up to a certain tightness condition needed to access the underlying cancellations on the right side of \eqref{variancefunction}. For the totally asymmetric ($L = 0$) case of the ASEP, this was obtained in \cite{CTPS}; it is plausible that, using the methods introduced in this paper, the same tightness condition might now be accessible for the general $L \ne 0$ ASEP.} to understand the scaling limit of the former. The prediction of van Beijern, Kutner, and Spohn states that the right side of \eqref{variancefunction} is of order $T^{-2 / 3}$ for all $x$ in a $T^{2 / 3}$-neighborhood of $vT$ and decays exponentially in $T$ elsewhere. Thus, since the right side of \eqref{variancefunction} can be viewed as a discrete Laplacian of $\Var J_t (x)$, this translates to the statement that $\Var J_T (vT)$ should be of order $T^{2 / 3}$, meaning that the fluctuations of the current $J_T (vT)$ should be of order $T^{1 / 3}$. 

This is consistent with the KPZ universality conjecture. In fact, the KPZ universality conjecture states more; it suggests that the fluctuations of $J_T (vT)$ (after rescaling by $T^{1 / 3}$) should converge to the long-time statistics of the stationary KPZ equation or, equivalently, to the Baik-Rains distribution. 

This prediction has attracted the attention of many probabilists and mathematical physicists over the past thirty years. The first mathematical work in this direction was by Ferrari and Fontes in 1994 \cite{CFA}, who showed that $\Var J_T (vT) = o(T)$, thereby verifying that the fluctuations of $J_T (vT)$ are of lower order than $T^{1 / 2}$. 

In 2006, Ferrari and Spohn \cite{SLSCST} considered the stationary TASEP ($L = 0$ case of the ASEP) and showed that the fluctuations of $J_T (vT)$ are of order $T^{1 / 3}$ and converge to the Baik-Rains distribution after $T^{1 / 3}$ scaling, thereby establishing the KPZ universality conjecture for the stationary TASEP. The proof of this result strongly relied on the \emph{free-fermionicity} (complete determinantal structure) possessed by the TASEP. This property is not believed to hold for the more general ASEP, which had until now posed trouble for extending Ferrari-Spohn's result to the ASEP with $L \ne 0$. 

In fact, even showing that the current fluctuations are of order $T^{1 / 3}$ remained open for several years. This was resolved 2010, by Bal\'{a}zs and Sepp\"{a}l\"{a}inen \cite{OCVDA}, who showed that $\Var J_T (vT)$ is of order $T^{2 / 3}$ but were not able to identify the Baik-Rains distribution as the scaling limit. 

One of the purposes of this paper will be to resolve the KPZ universality conjecture for the stationary ASEP on the level of exact statistics by showing that, upon $T^{1 / 3}$ scaling, the fluctuations of $J_T \big( v T \big)$ converge to the Baik-Rains distribution as $T$ tends to $\infty$; this improves upon all of the results mentioned above. We will state this more precisely in Section \ref{StationaryProcess} but, before doing so, let us explain some of the predictions for the (stochastic and symmetric) six-vertex model. 

Due to its two-dimensional nature and also its more complex Markov update rule, the stochastic six-vertex model has been less amenable to analysis than the ASEP. Few predictions and results have been made or established about this model. One, however, was proposed by Gwa and Spohn in \cite{DSENE}. In that work, the two authors predicted that the stochastic six-vertex model $\mathcal{P} (\delta_1, \delta_2)$ with double-sided $(b_1, b_2)$-Bernoulli initial data (see Section \ref{StochasticVertex} for the definitions) should be translation-invariant if $b_1$ and $b_2$ satisfy the relation $\beta_1 = \kappa \beta_2$, where $\kappa = (1 - \delta_1) / (1 - \delta_2)$ and $\beta_i = b_i / (1 - b_i)$ for each $i \in \{1 , 2 \}$. 

Under this initial data, Gwa and Spohn \cite{DSENE} predicted that the two-point function for the stochastic six-vertex model should share the same scaling limit as the two-point function of the stationary ASEP. Using \eqref{variancefunction} to translate this into a statement about currents, this prediction becomes that the fluctuations of the current $\mathfrak{H} (xT, yT)$ of the stochastic six-vertex model should be of order $T^{1 / 3}$ and scale to the Baik-Rains distribution when $x / y$ is of a certain characteristic value. 

This conjecture is closely related with a prediction of Bukman and Shore \cite{TCPFSVM} on the standard, symmetric six-vertex model; we refer to Section \ref{SixVertex} for definitions and notation on the latter model. In particular, these two authors considered the free energy profile of an arbitrary ferroelectric symmetric six-vertex model and made the following observation that was apparently missed in the original analysis of the six-vertex model by Sutherland-Yang-Yang \cite{ESMTDFAEEF}. Denoting the free energy of the ferroelectric six-vertex model with weights $(a, a, b, b, c, c)$ and slope $(h, v)$ by $F(h, v) = F_{a, b, c} (h, v)$, the free energy profile $F(h, v)$ exhibits a second-order singularity along a one-parameter family of slopes $(h, v) = (b_1, b_2)$ satisfying $\beta_1 = \kappa \beta_2$; here, $\beta_i = b_i / (1 - b_i)$ for each $i \in \{ 1, 2 \}$ and $\kappa = (1 - \delta_1) / (1 - \delta_2)$, where $\delta_1$ and $\delta_2$ are related to $a$, $b$, and $c$ by the identity \eqref{transformparameters} below. 

This family of slopes is typically referred to as the \emph{conical singularity} or \emph{tricritical point} of the ferroelectric six-vertex model and has been studied extensively in the physics literature, particularly within the study of facet corners \cite{TECS, CSFODDGC, ASVM, CPSVMCSFM} in the body-centered solid-on-solid (BCSOS) model for crystal growth. Upon further analysis of this singularity, they found that the second derivative of the free energy diverged near this family of slopes as $\gamma^{-1 / 3}$, where $\gamma$ denotes the distance from $(h, v)$ to the curve of $(b_1, b_2)$ satisfying $\beta_1 = \kappa \beta_2$. From this, they predicted that one should expect some form of KPZ growth in Gibbs measures for the ferroelectric six-vertex model with slopes $(b_1, b_2)$ satisfying $\beta_1 = \kappa \beta_2$ . 

The second purpose of this paper will be to establish both this prediction and the previously mentioned prediction of Gwa and Spohn. This will be stated more carefully in Section \ref{StationaryModel}.

\subsection{Results}

\label{StationaryAsymptotics}

In this section, we state our results, which concern the stationary ASEP and six-vertex model. Our results on the ASEP are discussed in Section \ref{StationaryProcess}, and our results on the six-vertex model are discussed in Section \ref{StationaryModel}.

\subsubsection{Asymptotics for the ASEP}

\label{StationaryProcess}

To state our results, we must define the Baik-Rains distribution; this is given by Definition \ref{stationarydistribution}. However, we first require some preliminary functions given by Definition \ref{stationarykernel1} and Definition \ref{stationarykernel}. 

\begin{definition}
\label{stationarykernel1}

Denoting by $\Ai (x)$ the Airy function, the \emph{Airy kernel} $K_{\Ai} (x, y)$ is defined by 
\begin{flalign*}
K_{\Ai} (x, y) & = \displaystyle\int_0^{\infty} \Ai (x + \lambda) \Ai (y + \lambda) d \lambda = \displaystyle\frac{1}{(2 \pi \textbf{i})^2} \displaystyle\oint \displaystyle\oint \exp \left( \displaystyle\frac{w^3}{3} - \displaystyle\frac{v^3}{3} - xv + yw \right) \displaystyle\frac{dw dv}{w - v}, 
\end{flalign*}

\noindent where on the right side the $w$-contour is piecewise linear from $\infty e^{- \pi \textbf{i} / 3}$ to $\infty e^{\pi \textbf{i} / 3}$, the $v$-contour is piecewise linear from $\infty e^{- 2 \pi \textbf{i} / 3}$ to $\infty e^{2 \pi \textbf{i} / 3}$, and the $w$-contour is to the right of the $v$-contour everywhere; see the right side of Figure \ref{l1l2l3} in Section \ref{DeterminantNew} below. 

Also define the \emph{shifted Airy kernel} $K_{\Ai; a} (x, y) = K_{\Ai} (x + a, y + a)$, for any $a \in \mathbb{R}$. 
	
\end{definition}
	
\begin{definition}
\label{stationarykernel}

For each $x, y, c, s \in \mathbb{R}$, define the functions $\mathcal{R} = \mathcal{R}_{c, s}$, $\Phi (x) = \Phi_{c, s} (x)$, and $\Psi (y) = \Psi_{c, s} (y)$ through 
\begin{flalign*}
\mathcal{R} & = s + e^{-2 c^3 / 3} \displaystyle\int_s^{\infty} \displaystyle\int_0^{\infty} e^{-c (x + y)} \Ai (x + y + c^2) dy dx; \\
\Phi (x) & = e^{\frac{-2 c^3}{3}} \displaystyle\int_0^{\infty} \displaystyle\int_s^{\infty} e^{-cy} \Ai (x + c^2 + \lambda) \Ai (y + c^2 + \lambda) dy d \lambda - \displaystyle\int_0^{\infty} e^{cy} \Ai (x + y + c^2) dy; \\
\Psi (y) & = e^{2 c^3 / 3 + cy} - \displaystyle\int_0^{\infty} e^{-cx} \Ai (x + y + c^2) dx. 
\end{flalign*}

\end{definition}

\begin{definition}
\label{stationarydistribution} 

Let $c, s \in \mathbb{R}$, and define the operator $\textbf{P}_s = \textbf{1}_{x \ge s}$ on $L^2 (\mathbb{R})$; here, $\textbf{1}_E$ denotes the indicator of an event $E$. Set 
\begin{flalign}
\label{gstationary} 
g(c, s) = \mathcal{R} - \big\langle \big( \Id - K_{\Ai, c^2 + s} \big)^{-1} \textbf{P}_s \Phi, \textbf{P}_s \Psi \big\rangle, 
\end{flalign} 

\noindent where we take the $L^2 (\mathbb{R})$ inner product on the right side of \eqref{gstationary}. 

Define the \emph{Baik-Rains distribution} $F_{\BR; c} (s)$ by 
\begin{flalign*}
F_{\BR; c} (s) = \displaystyle\frac{\partial}{\partial s} \left( g(c, s) \det \big( \Id - K_{\Ai, c^2 + s} \big)_{L^2 (\mathbb{R}_{> 0})} \right). 
\end{flalign*}

\end{definition}

Having defined the Baik-Rains distribution,\footnote{It is not immediate from Definition \ref{stationarydistribution} that $F_{\BR; c}$ is indeed a probability distribution. In fact, the simplest derivation of this fact we are aware of proceeds by showing that it is a limit of probability distributions arising from some family of discrete models (such as the PNG model \cite{LDPGMES} or TASEP \cite{SLSCST}).} we can now state the following theorem, which confirms the KPZ universality conjecture for the stationary ASEP. 

\begin{thm}
\label{currentfluctuations}

Consider the stationary ASEP with left jump rate $L$, right jump rate $R$, and Bernoulli parameter $b \in (0, 1)$. Assume that $R > L$, and set $\delta = R - L$ and $\chi = b (1 - b)$. Then, for any real numbers $c, s \in \mathbb{R}$, we have that 
\begin{flalign}
\label{probabilitycurrentfluctuations} 
\displaystyle\lim_{T \rightarrow \infty} \mathbb{P} \left[ J_{\delta^{-1} T} \big( (1 - 2b) T  + 2 c \chi^{1 / 3} T^{2 / 3} \big) \ge b^2 T - 2bc \chi^{1 / 3} T^{2 / 3} - \chi^{2 / 3} s T^{1 / 3} \right] = F_{\BR; c} (s). 
\end{flalign}

\end{thm}

\begin{rem}

If one directly tracks the error in all estimates involved in the proof of Theorem \ref{currentfluctuations}, one will find that the error in \eqref{probabilitycurrentfluctuations} is $\mathcal{O} \big( T^{- 1 / 3} \big)$. One may use this to deduce that $\Var J_{T} \big( (1 - 2b) (R - L) T \big)$ is of order $T^{2 / 3}$, thereby yielding an alternative proof of Corollary 2.3 of \cite{OCVDA}. 

\end{rem}

\subsubsection{Asymptotics for the Six-Vertex Model}

\label{StationaryModel}

In this section we state our results for the stochastic six-vertex model with a certain type of \emph{translation-invariant} initial data, and outline the consequences of these results for the standard (symmetric) six-vertex model; a more detailed explanation on the link between the stochastic six-vertex model and the symmetric six-vertex model will be provided in Appendix \ref{SixVertex}. 

Fix $b_1, b_2 \in (0, 1)$ and $0 < \delta_1 < \delta_2 < 1$, and consider the stochastic six-vertex model $\mathcal{P} (\delta_1, \delta_2)$ with double-sided $(b_1, b_2)$-Bernoulli initial data (as defined in Section \ref{StochasticVertex}); we denote the resulting six-vertex measure by $\mathcal{P} (\delta_1, \delta_2; b_1, b_2)$. We call this initial data (and also the measure $\mathcal{P} (\delta_1, \delta_2; b_1, b_2)$) \emph{translation-invariant} if $\beta_1 = \kappa \beta_2$, where we have denoted $\kappa = (1 - \delta_1) / (1 - \delta_2)$ and $\beta_i = b_i / (1 - b_i)$ for each $i \in \{ 1, 2 \}$.

This notation is justified more precisely in Lemma \ref{stationarystochastic} below, which states that any translation-invariant measure $\mathcal{P} (\delta_1, \delta_2; b_1, b_2)$ possesses the following property. For each $(x, y) \in \mathbb{Z}_{> 0}^2$, let $\ell_{(x, y)}^{(v)}$ denote the vertical, up-pointing ray emanating from $(x, y)$, and let $\ell_{(x, y)}^{(h)}$ denote the horizontal, right-pointing ray emanating from $(x, y)$. Then, each vertex on $\ell_{(x, y)}^{(v)}$ is an entrance site for a path with probability $b_1$, and each vertex on $\ell_{(x, y)}^{(h)}$ is an entrance site for a path with probability $b_2$; all of these events are independent. As to be explained in more detail in Section \ref{Translation}, this translation-invariance property allows one to extend the measure $\mathcal{P} (\delta_1, \delta_2; b_1, b_2)$ to the entire lattice $\mathbb{Z}^2$. Thus, the stochastic six-vertex model described above can be viewed as the restriction to the positive quadrant of a translation-invariant stochastic six-vertex model that resides on $\mathbb{Z}^2$. 

In view of this translation-invariance, the KPZ universality conjecture (and also Gwa-Spohn \cite{DSENE}) hypothesizes that, under this translation-invariant initial data, there exists a characteristic line along which the current fluctuations of the stochastic six-vertex model are of order $T^{1 / 3}$ and scale to the Baik-Rains distribution. 

The following theorem confirms this hypothesis.

\begin{thm}
\label{processfluctuations}

Let $\delta_1, \delta_2, b_1, b_2 \in (0, 1)$ be positive real numbers. Assume that $\delta_1 < \delta_2$, and denote $\kappa = (1 - \delta_1) / (1 - \delta_2)$, $\beta_1 = b_1 / (1 - b_1)$, and $\beta_2 = b_2 / (1 - b_2)$. Assume that $\beta_1 = \kappa \beta_2$. 

Consider the translation-invariant stochastic six-vertex model $\mathcal{P} (\delta_1, \delta_2; b_1, b_2)$. Define 
\begin{flalign*}
\Lambda_1 = b_1 + \kappa (1 - b_1); \qquad \Lambda_2 = b_2 + \kappa^{-1} (1 - b_2); \qquad \chi_i = b_i (1 - b_i),
\end{flalign*}

\noindent for each index $i \in \{ 1, 2 \}$. Also denote $x = \Lambda_1 (1 - \delta_2)$, $y = \Lambda_2 (1 - \delta_1)$, and 
\begin{flalign*}
\varsigma = \displaystyle\frac{2 (\delta_2 - \delta_1)^{2 / 3} \chi_1^{1 / 6} \chi_2^{1 / 6}}{(1 - \delta_1)^{1 / 2} (1 - \delta_2)^{1 / 2}} ; \qquad \mathcal{F} = (\delta_1 - \delta_2)^{1 / 3} \chi_1^{1 / 3} \chi_2^{1 / 3}. 
\end{flalign*}

\noindent Then, for any real numbers $c, s \in \mathbb{R}$, we have that 
\begin{flalign*}
\displaystyle\lim_{T \rightarrow \infty} \mathbb{P} \Big[ \mathfrak{H} \big( x (T + \varsigma c T^{2 / 3}), y T \big) \ge b_1 b_2 (\delta_2 - \delta_1) T - b_1 (1 - \delta_2) \varsigma c T^{2 / 3} - \mathcal{F} s T^{1 / 3} \Big] = F_{\BR; c} (s). 
\end{flalign*}
\end{thm}

Let us discuss the relationship between the stochastic six-vertex model and the symmetric six-vertex model, as well as the consequences of Theorem \ref{processfluctuations} for the latter model. We refer to Appendix \ref{SixVertex} for a more detailed explanation of all terms we use about the six-vertex model. 

To that end, consider a symmetric, ferroelectric six-vertex model with weights $(a, a, b, b, c, c)$, and assume without loss of generality that $a > b$. Then, as to be explained in Appendix \ref{PropertyMapping}, there exist $\delta_1 = \delta_1 (a, b, c)$ and $\delta_2 = \delta_2 (a, b, c)$ (explicitly given by \eqref{transformparameters}) such that the stochastic six-vertex model $\mathcal{P} (\delta_1, \delta_2)$ prescribes a Gibbs measure for this ferroelectric six-vertex model. The translation-invariance of the stochastic six-vertex model $\mathcal{P} (\delta_1, \delta_2; b_1, b_2)$ implies that the corresponding Gibbs measure is translation-invariant. As discussed in Section \ref{Translation}, the slope corresponding to the latter measure can be quickly seen to be $(b_1, b_2)$. Therefore, $\mathcal{P} (\delta_1, \delta_2; b_1, b_2)$ becomes an unusually explicit class of translation-invariant Gibbs measures that characterizes a one-parameter family of slopes for this symmetric six-vertex model; in fact, the free energy per site of this model can be evaluated exactly and is given by Proposition \ref{energymeasure} in Appendix \ref{Translation}. 

Thus, Theorem \ref{processfluctuations} above also holds for the symmetric six-vertex model with weights $(a, a, b, b, c, c)$, where $(a, b, c)$ are defined in terms of $\delta_1$ and $\delta_2$ through \eqref{transformparameters}. This provides a very precise sense in which the height function of the ferroelectric six-vertex model (defined in the same way as for the stochastic six-vertex model) exhibits KPZ growth at the conical singularity, thereby confirming the prediction of Bukman and Shore from \cite{TCPFSVM}. 

Let us make two additional comments. First, Theorem \ref{processfluctuations} (when recast in terms of the ferroelectric, symmetric six-vertex model) states that KPZ growth holds for the height function of the six-vertex model only along a single characteristic direction. Although this was predicted for the stochastic six-vertex model by Gwa and Spohn \cite{DSENE}, this reinterpretation in terms of the symmetric six-vertex model was not predicted by Bukman-Shore \cite{TCPFSVM} and appears to be new. In particular, it suggests that the two-point function for this six-vertex model should decay exponentially, except along a certain characteristic direction, along which it decays as a power law in the distance (with exponent $2 / 3$). This is a highly unusual phenomenon in both critical (in which the two-point function should decay as a power law in every direction) and non-critical (in which the two-point function should decay exponentially in every direction) spin systems, and we are not aware of any recording of it in the physics or mathematics literature. 

Second, Theorem \ref{processfluctuations} provides exact limiting statistics for translation-invariant, generic ferroelectric six-vertex models along or near the characteristic line. Such detailed limiting statistics for the height fluctuations of translation-invariant six-vertex models are very rare. To the best of our knowledge, they have only been proven for dimer-type models \cite{IPRTP, OD, DA, CPLGR}. These can typically be mapped \cite{DTSVMFFP, SVLTMIC} to the free-fermionic point of the six-vertex model, which can then be accessed through determinantal methods (such as through analysis of the Kasteleyn matrix \cite{GTCP} and determinantal point processes \cite{DPP, IP}). 

For the generic ferroelectric six-vertex model, such free-fermionic methods are not available. Our result is the first we know of to mathematically establish the exact limiting behavior of a translation-invariant generic, ferroelectric six-vertex model.

\subsection{Outline}

\label{Sections}

What enables the proofs of Theorem \ref{currentfluctuations} and Theorem \ref{processfluctuations} are Fredholm determinant identites for the ASEP and stochastic six-vertex model with double-sided Bernoulli initial data (given by Theorem \ref{determinant} and Theorem \ref{determinantw}, respectively). 

The integrability of the ASEP with double-sided initial data is not entirely new. Indeed, in \cite{ATBIC}, Tracy and Widom implemented a coordinate Bethe ansatz to provide a Fredholm determinant identity for the ASEP with this initial data. However, these identities involve certain infinite sums that make them unsuitable for asymptotic analysis, except in the case of step-Bernoulli initial data \cite{ASBIC}. For this reason, a KPZ-type current fluctuation result for the stationary ASEP had, until now, remained unproven, even after several years of Tracy-Widom's work. 

It seems that one of the obstructions towards producing tractable Fredholm determinant identities for the ASEP with double-sided Bernoulli initial data is that most known proofs of such identities either rely on the coordinate Bethe ansatz \cite{WAEP, FDRA, IASEP} or duality \cite{DDQA, WAEP}. Both of these methods often require some degree of guesswork, which can pose issues, especially if the complexity of the underlying identity makes it troublesome to guess. 

However, a new proof of these identities in the case of step initial data was very recently proposed by Borodin and Petrov \cite{HSVMRSF}. In their work, these authors introduce a family of rational symmetric functions that can be viewed as partition functions of higher spin vertex models. Analysis of these symmetric functions, based on the Yang-Baxter relation, and then degenerating to the six-vertex or ASEP setting yields Fredholm determinant identities that are suitable for asymptotic analysis.

In this work we extend on their framework to produce Fredholm determinant identities (that are also amenable to asymptotic analysis) for the ASEP and stochastic six-vertex model under (a slightly restricted class of) double-sided Bernoulli initial data. This will proceed as follows. 

After providing a brief exposition of the work \cite{HSVMRSF} of Borodin and Petrov in Section \ref{GeneralVertexModels}, we explain how to directly degenerate the stochastic higher spin vertex model to a stochastic six-vertex model with a certain type of double-sided initial data in Section \ref{SpecializationM}. Analogous degenerations have been performed \cite{PTAEPSSVM} in the case of single-sided initial data (in which paths enter only through the $y$-axis); this involved deforming spins on the columns of the inhomogeneous vertex model. 

However, the development that produces a double-sided model (in which paths simultaneously enter through both the $x$-axis and $y$-axis) appears to be new. It is based on the idea of \emph{row fusion} in vertex models that dates back to Kulish-Reshetikhin-Sklyanin \cite{ERT} in a representation theoretic context, but was more recently developed in a probabilistic context by Corwin-Petrov \cite{SHSVML} and Borodin-Petrov \cite{HSVMRSF, IPSVMSF}. The double-sided initial data produced by this row fusion will in fact not coincide with double-sided Bernoulli initial data; in particular, only a finite number $J$ of particles will enter through the $x$-axis. However, it will resemble double-sided Bernoulli initial data in a way to be clarified at the end of Section \ref{PrincipalDoubleSided}. 

Next, we degenerate the results of \cite{HSVMRSF} to produce contour integral identities for certain observables ($q$-moments) of the stochastic six-vertex model with this type of double-sided initial data; this will be the topic of Section \ref{ContoursCurrent}. It is known \cite{MP, SSVM, DDQA} how to use such identities to obtain Fredholm determinant identities, which we will do in Section \ref{DeterminantsDoubleSided}. Remarkably, it will be possible to analytically continue these identities in the variable $J$ (which denotes the number of paths entering through the $x$-axis); this analytic continuation will give rise to Fredholm determinant identities for the stochastic six-vertex model with genuine double-sided $(b_1, b_2)$-Bernoulli initial data. Using the result of \cite{CSSVMEP}, which establishes a limit degeneration from the stochastic six-vertex model to the ASEP, we can degenerate these identities for the stochastic six-vertex model to produce analogous identities for the ASEP with double-sided $(b_1, b_2)$-Bernoulli initial data. Interestingly, we do not know of any other way to access these determinant identities for the double-sided ASEP other than to first establish them for the double-sided stochastic six-vertex model and then degenerate. 

The set of Bernoulli parameters $b_1$ and $b_2$ for which the Fredholm determinant identities Theorem \ref{determinantw} and Theorem \ref{determinant} apply is restricted; for instance, in the ASEP case, Theorem \ref{determinant} applies only when $b_1 > b_2$. This restriction is believed to correspond to the existence of a \emph{rarefaction fan} in the scaling limit, which is a cone (or interval) in which current fluctuations are of order less than $T^{1 / 2}$; outside this cone, fluctuations are expected to either have exponent $1 / 2$ or be exponentially small. We beleive that it should be possible to asymptotically analyze the Fredholm determinant identities Theorem \ref{determinantw} and Theorem \ref{determinant}, in order to verify the existence of rarefaction fans and produce central limit theorems for the current; this would provide a proof of part of what is known as the \emph{Pr\"{a}hofer-Spohn conjecture} for the ASEP, which was established for the TASEP in \cite{CFTP} through other (more probabilistic) methods. 

However, we do not pursue this in this paper. Instead, we analyze the stationary or translation-invariant cases, corresponding to $b_1 = b_2$ (for the ASEP) and $\beta_1 = \kappa \beta_2$ (for the stochastic six-vertex model). In Section \ref{StationaryReduction}, we will begin the asymptotic analysis of these Fredholm determinant identities and explain how they can be used to establish Theorem \ref{currentfluctuations} and Theorem \ref{processfluctuations}. Asymptotics for determinants of this type have been evaluated in \cite{LPSCL, HFSE, SLSCST}, but some new issues arise in our setting that complicate the analysis; these will be resolved in the Section \ref{DeterminantNew}.  We will complete the asymptotic analysis (and thus the proofs of Theorem \ref{currentfluctuations} and Theorem \ref{processfluctuations}) in Section \ref{RightStationary}, which will mainly follow Section 6, Section 7, and Section 8 of \cite{HFSE}.

\subsection*{Acknowledgements}

The author heartily thanks Alexei Borodin and Herbert Spohn for many valuable conversations. The author is also grateful to Ivan Corwin for some helpful remarks on the Baik-Rains distribution, to Leonid Petrov for some useful advice on $q$-hypergeometric series, and to the anonymous referees for numerous detailed comments and suggestions. This work was funded by the Eric Cooper and Naomi Siegel Graduate Student Fellowship I and the NSF Graduate Research Fellowship under grant number DGE1144152.

\section{Stochastic Higher Spin Vertex Models}

\label{GeneralVertexModels}

Both the ASEP and stochastic six-vertex model are degenerations of a larger class of vertex models called the \emph{inhomogeneous stochastic higher spin vertex models}, which were recently introduced by Borodin and Petrov in \cite{HSVMRSF}. These models are in a sense the original source of integrability for the ASEP, the stochastic six-vertex model, and in fact most models proven to be in the Kardar-Parisi-Zhang (KPZ) universality class. 

For that reason, we begin our discussion by defining these vertex models. Similar to the stochastic six-vertex model, these models take place on \emph{directed-path ensembles}. We first carefully define what we mean by a directed-path ensemble in Section \ref{PathEnsembles}, and then we define the stochastic higher spin vertex model in Section \ref{ProbabilityMeasures}. In Section \ref{InteractingParticles}, we re-interpret these models as interacting particle systems. This will be useful in Section \ref{RationalFunctions}, where we define the family of inhomogeneous symmetric functions introduced by Borodin and Petrov \cite{HSVMRSF} that lead to the integrability of those models.

\subsection{Directed Path Ensembles}

\label{PathEnsembles}

For the purpose of this paper, a \emph{directed path} is a collection of \emph{vertices}, which are lattice points in the non-negative quadrant $\mathbb{Z}_{\ge 0}^2$, connected by \emph{directed edges} (which we may also refer to as \emph{arrows}). A directed edge can connect a vertex $(i, j)$ to either $(i + 1, j)$ or $(i, j + 1)$, if $(i, j) \in \mathbb{Z}_{> 0}^2$; we also allow directed edges to connect $(k, 0)$ to $(k, 1)$ or $(0, k)$ to $(1, k)$, for any positive integer $k$. Thus, directed edges connect adjacent points, always point either up or to the right, and do not lie on the $x$-axis or $y$-axis. 

A \emph{directed-path ensemble} is a collection of paths satisfying the following two properties.

\begin{itemize}
\item{ Each path must contain an edge connecting $(0, k)$ to $(1, k)$ or $(k, 0)$ to $(k, 1)$ for some $k > 0$; stated alternatively, every path ``emanates'' from either the $x$-axis or the $y$-axis.}

\item{ No two distinct paths can share a horizontal edge; however, in contrast with the six-vertex case, they may share vertical edges.}

\end{itemize}

\noindent  An example of a directed-path ensemble was previously shown in Figure \ref{figurevertexwedge}. See also Figure \ref{shift2generalizedvertex} for more examples. 

Associated with each $(x, y) \in \mathbb{Z}_{> 0 }^2$ in a path ensemble is an \emph{arrow configuration}, which is a quadruple $(i_1, j_1; i_2, j_2) = (i_1, j_1; i_2, j_2)_{(x, y)}$ of non-negative integers. Here, $i_1$ denotes the number of directed edges from $(x, y - 1)$ to $(x, y)$; equivalently, $i_1$ denotes the number of vertical incoming arrows at $(x, y)$. Similarly, $j_1$ denotes the number of horizontal incoming arrows; $i_2$ denotes the number of vertical outgoing arrows; and $j_2$ denotes the number of horizontal outgoing arrows. Thus $j_1, j_2 \in \{ 0, 1 \}$ at every vertex in a path ensemble, since no two paths share a horizontal edge. An example of an arrow configuration (which cannot be a vertex in a directed-path ensemble, since $j_1, j_2 \notin \{ 0, 1 \}$) is depicted in Figure \ref{arrows}. 

Assigning values $j_1$ to vertices on the line $(1, y)$ and values $i_1$ to vertices on the line $(x, 1)$ can be viewed as imposing boundary conditions on the vertex model. If $j_1 = 1$ at $(1, k)$ and $i_1 = 0$ at $(k, 1)$ for each $k > 0$, then all paths enter through the $y$-axis, and every vertex on the positive $y$-axis is an entrance site for some path. This particular assignment is called \emph{step initial data}; it was depicted in Figure \ref{figurevertexwedge}, and it is also depicted on the left side of Figure \ref{shift2generalizedvertex}. In general, we will refer to any assignment of $i_1$ to $\mathbb{Z}_{> 0} \times \{ 1 \}$ and $j_1$ to $\{ 1 \} \times \mathbb{Z}_{> 0}$ as \emph{initial data}, which can be deterministic (like step) or random. 
	
Observe that, at any vertex in the positive quadrant, the total number of incoming arrows is equal to the total number of outgoing arrows; that is, $i_1 + j_1 = i_2 + j_2$. This is sometimes referred to as \emph{arrow conservation} (or \emph{spin conservation}). Any arrow configuration to all vertices of $\mathbb{Z}_{> 0}^2$ that satisfies arrow conservation corresponds to a unique directed-path ensemble, where paths can share both vertical and horizontal edges. 

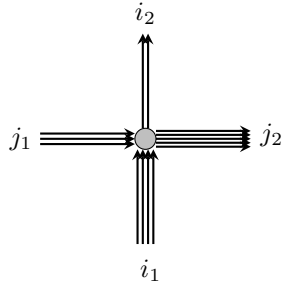
\begin{figure}

\begin{center} 

\begin{tikzpicture}[
      >=stealth,
      scale = .7
			]
			
			\draw[->,black, thick] (-.05, .2) -- (-.05, 2) node[right = 1, above]{$i_2$};
			\draw[->,black, thick] (.05, .2) -- (.05, 2);
			\draw[->,black, thick] (-.15,-2) -- (-.15, -.2);
			\draw[->,black, thick] (-.05,-2) -- (-.05, -.2) node[right = 3, below = 38]{$i_1$};
			\draw[->,black, thick] (.05,-2) -- (.05, -.2);
			\draw[->,black, thick] (.15,-2) -- (.15, -.2);
			\draw[->,black, thick] (-2, -.1) -- (-.2, -.1) node[above = 2, left = 35]{$j_1$};
			\draw[->,black, thick] (-2, 0) -- (-.2, 0);
			\draw[->,black, thick] (-2,.1) -- (-.2, .1);
			\draw[->,black, thick] (.2, -.15) -- (2, -.15); 
			\draw[->,black, thick] (.2, -.075) -- (2, -.075); 
			\draw[->,black, thick] (.2, 0) -- (2, 0); 
			\draw[->,black, thick] (.2, .075) -- (2, .075) node[right]{$j_2$}; 
			\draw[->,black, thick] (.2, .15) -- (2, .15);

			\filldraw[fill=gray!50!white, draw=black] (0, 0) circle [radius=.2];

\end{tikzpicture}

\end{center}

\caption{\label{arrows} Above is a vertex at which $(i_1, j_1; i_2, j_2) = (4, 3; 2, 5)$.} 
\end{figure}

\subsection{Probability Measures on Path Ensembles}

\label{ProbabilityMeasures}

The definition of the stochastic higher spin vertex models will closely resemble the definition of the stochastic six-vertex model given in Section \ref{StochasticVertex}. Specifically, we will first define probability measures $\mathbb{P}_n$ on the set of directed-path ensembles whose vertices are all contained in triangles of the form $\mathbb{T}_n = \{ (x, y) \in \mathbb{Z}_{\ge 0}^2: x + y \le n \}$, and then we will take a limit as $n$ tends to infinity to obtain the vertex models in infinite volume. The first two measures $\mathbb{P}_0$ and $\mathbb{P}_1$ are both supported by the empty ensembles. 

For each positive integer $n \ge 1$, we will define $\mathbb{P}_{n + 1}$ from $\mathbb{P}_n$ through the following Markovian update rules. Use $\mathbb{P}_n$ to sample a directed-path ensemble $\mathcal{E}_n$ on $\mathbb{T}_n$. This yields arrow configurations for all vertices in the triangle $\mathbb{T}_{n - 1}$. To extend this to a path ensemble on $\mathbb{T}_{n + 1}$, we must prescribe arrow configurations to all vertices $(x, y)$ on the complement $\mathbb{T}_n \setminus\mathbb{T}_{n - 1}$, which is the diagonal $\mathbb{D}_n = \{ (x, y) \in \mathbb{Z}_{> 0}^2: x + y = n \}$. Since any incoming arrow to $\mathbb{D}_n$ is an outgoing arrow from $\mathbb{D}_{n - 1}$, $\mathcal{E}_n$ and the initial data prescribe the first two coordinates, $i_1$ and $j_1$, of the arrow configuration to each $(x, y) \in \mathbb{D}_n$. Thus, it remains to explain how to assign the second two coordinates ($i_2$ and $j_2$) to any vertex on $\mathbb{D}_n$, given the first two coordinates. 

This is done by producing $(i_2, j_2)_{(x, y)}$ from $(i_1, j_1)_{(x, y)}$ according to the transition probabilities 
\begin{flalign}
\begin{aligned}
\label{configurationprobabilities} 
& \mathbb{P}_n \big[ (i_2, j_2) = (k, 0) \big| (i_1, j_1) = (k, 0) \big] = \displaystyle\frac{1 - q^k s_x \xi_x u_y}{1 - s_x \xi_x u_y}, \\
& \mathbb{P}_n \big[ (i_2, j_2) = (k - 1, 1) \big| (i_1, j_1) = (k, 0) \big] = \displaystyle\frac{(q^k - 1) s_x \xi_x u_y}{1 - s_x \xi_x u_y}, \\
& \mathbb{P}_n \big[ (i_2, j_2) = (k + 1, 0) \big| (i_1, j_1) = (k, 1) \big] = \displaystyle\frac{1 - q^k s_x^2}{1 - s_x \xi_x u_y},  \\
& \mathbb{P}_n \big[ (i_2, j_2) = (k, 1) \big| (i_1, j_1) = (k, 1) \big] = \displaystyle\frac{q^k s_x^2 - s_x \xi_x u_y}{1 - s_x \xi_x u_y},
\end{aligned}   
\end{flalign}

\noindent for any non-negative integer $k$. We also set $\mathbb{P}_n [(i_2, j_2) | (i_1, j_1)] = 0$ for all $(i_1, j_1; i_2, j_2)$ not of the above form. In the above, $q \in \mathbb{C}$ is a complex number and $U = (u_1, u_2, \ldots ) \subset \mathbb{C}$, $\Xi = (\xi_1, \xi_2, \ldots ) \subset \mathbb{C}$, and $S = (s_1, s_2, \ldots ) \subset \mathbb{C}$ are infinite sets of complex numbers, chosen to ensure that all of the above probabilities are non-negative. This can be arranged for instance when $q \in (0, 1)$, $U \subset (-\infty, 0]$, and $S, \Xi \subset [0, 1]$, although there are other suitable choices as well. 

Choosing $(i_2, j_2)$ according to the above transition probabilities yields a random directed-path ensemble $\mathcal{E}_{n + 1}$, now defined on $\mathbb{T}_{n + 1}$; the probability distribution of $\mathcal{E}_{n + 1}$ is then denoted by $\mathbb{P}_{n + 1}$. We define $\mathbb{P} = \lim_{n \rightarrow \infty} \mathbb{P}_n$. Then, $\mathbb{P}$ is a probability measure on the set of directed-path ensembles that is dependent on the complex parameters $q$, $U$, $\Xi$, and $S$. The variables $U = (u_1, u_2, \ldots )$ are occasionally referred to as \emph{spectral parameters}, the variables $\Xi = (\xi_1, \xi_2, \ldots )$ as \emph{spacial inhomogeneity parameters}, and the variables $S = (s_1, s_2, \ldots )$ as \emph{spin parameters}.

\subsection{Vertex Models and Interacting Particle Systems}

\label{InteractingParticles}

Let $\mathbb{P}^{(T)}$ denote the restriction of the random path ensemble with step initial data (given by the measure $\mathbb{P}$ from the previous section) to the strip $\mathbb{Z}_{> 0} \times [0, T]$. Assume that all $T$ paths in this restriction almost surely exit the strip $\mathbb{Z}_{> 0} \times [0, T]$ through its top boundary. This will be the case, for instance, if $\mathbb{P}_n \big[ (i_2, j_2) = (0, k) \big| (i_1, j_1) = (0, k) \big] < 1$ uniformly in $n$ and $k$. 
	
We will use the probability measure $\mathbb{P}^{(T)}$ to produce a discrete-time interacting particle system on $\mathbb{Z}_{> 0}$, defined up to time $T - 1$, as follows. Sample a line ensemble $\mathcal{E}$ randomly under $\mathbb{P}^{(T)}$, and consider the arrow configuration it associates with some vertex $(p, t) \in \mathbb{Z}_{> 0} \times [1, T]$. We will place $k$ particles at position $p$ and time $t - 1$ if and only if $i_1 = k $ at the vertex $(p, t)$. Therefore, the paths in the path ensemble $\mathcal{E}$ correspond to space-time trajectories of the particles. 

Let us introduce notation for particle positions. A \emph{non-negative signature} $\lambda = (\lambda_1, \lambda_2, \ldots , \lambda_n)$ of \emph{length} $n$ (also denoted $\ell (\lambda) = n$) is a non-increasing sequence of $n$ integers $\lambda_1 \ge \lambda_2 \ge \cdots \ge \lambda_n \ge 0$. We denote the set of non-negative signatures of length $n$ by $\Sign_n^+$, and the set of all non-negative signatures by $\Sign^+ = \bigcup_{N = 0}^{\infty} \Sign_n^+$. For any signature $\lambda$ and integer $j$, let $m_j (\lambda)$ denote the number of indices $i$ for which $\lambda_i = j$; that is, $m_j (\lambda)$ is the multiplicity of $j$ in $\lambda$. 

We can associate a configuration of $n$ particles in $\mathbb{Z}_{\ge 0}$ with a signature of length $n$ as follows. A signature $\lambda = (\lambda_1, \lambda_2, \ldots , \lambda_n)$ is associated with the particle configuration that has $m_j (\lambda)$ particles at position $j$, for each non-negative integer $j$. Stated alternatively, $\lambda$ is the ordered set of positions in the particle configuration. If $\mathcal{E}$ is a directed line ensemble on $\mathbb{Z}_{\ge 0} \times [0, n]$, let $p_n (\mathcal{E}) \in \Sign_n^+$ denote the signature associated with the particle configuration produced from $\mathcal{E}$ at time $n$. 

Then, $\mathbb{P}^{(n)}$ induces a probability measure on $\Sign_n^+$,  defined as follows. 

\begin{definition}

\label{measuresignaturesvertexmodel} 

For any positive integer $n$, let $\textbf{M}_n$ denote the measure on $\Sign_n^+$ (dependent on the parameters $q$, $U$, $S$, and $\Xi$) defined by setting $\textbf{M}_n (\lambda) = \mathbb{P}^{(n)} [p_n (\mathcal{E}) = \lambda]$, for each $\lambda \in \Sign_n^+$. 
\end{definition}

\subsection{The Inhomogeneous Symmetric Functions}

\label{RationalFunctions}

In this section we define the inhomogeneous rational symmetric functions introduced by Borodin and Petrov in \cite{HSVMRSF} and explain their connection to the higher spin vertex models. 

The functions we introduce, which will be denoted by $F_{\lambda / \mu}$ and $G_{\lambda / \mu}$ (for signatures $\lambda$ and $\mu$), will be partition functions of path ensembles that are weighted in a specific way. Each path ensemble consists of a collection of vertices and arrow configurations assigned to each vertex. To each vertex will be associated a \emph{vertex weight}, which will depend on the position of the vertex and its arrow configuration. The weight of the directed-path ensemble will then be the product of the weights of all vertices in the ensemble. 

Let us explain this in more detail. For this section only, we ``shift'' all directed-path ensembles to the left by one in order to make our presentation consistent with that of Borodin and Petrov in Section 4 of \cite{HSVMRSF}. That is, any directed path in a path ensemble either contains an edge from $(-1, k + 1)$ to $(0, k + 1)$ or from $(k, 0)$ to $(k, 1)$, for some non-negative integer $k$. Arrow configurations are now defined on $\mathbb{Z}_{\ge 0} \times \mathbb{Z}_{> 0}$ instead of on $\mathbb{Z}_{> 0}^2$; see Figure \ref{fgpaths}. 

Fix $q \in \mathbb{C}$; $U = (u_1, u_2, \ldots ) \subset \mathbb{C}$; $\Xi = (\xi_0, \xi_1, \xi_2, \ldots) \subset \mathbb{C}$, and $S = (s_0, s_1, s_2, \ldots ) \subset \mathbb{C}$. To each $(x, y) \in \mathbb{Z}_{\ge 0} \times \mathbb{Z}_{> 0}$, define the vertex weight $w_{(x, y)} = w_{u_y; \xi_x, s_x}$ by 
\begin{flalign*}
& w_{(x, y)} (k, 0; k, 0) = \displaystyle\frac{1 - q^k s_x \xi_x u_y}{1 - s_x \xi_x u_y}; \qquad \quad w_{(x, y)} (k + 1, 0; k, 1) = \displaystyle\frac{(1 - q^k s_x^2) \xi_x u_y}{1 - s_x \xi_x u_y}, \\
& w_{(x, y)} (k, 1; k + 1, 0) = \displaystyle\frac{1 - q^{k + 1}}{1 - s_x \xi_x u_y}; \qquad w_{(x, y)} (k, 1; k, 1) = \displaystyle\frac{\xi_x u_y - q^k s_x}{1 - s_x \xi_x u_y},
\end{flalign*} 

\noindent for any non-negative integer $k$, and $w_{(x, y)} (i_1, j_1; i_2, j_2) = 0$ for any $(i_1, j_1; i_2, j_2)$ not of the above form. These functions assign weight $w_{(x, y)} (i_1, j_1; i_2, j_2) $ to vertex $(x, y)$ if it has arrow configuration $(i_1, j_1; i_2, j_2)$. 

Let $\mathcal{E}$ be a directed-path ensemble, all of whose vertices are contained in some region $[-1, \infty) \times [0, N + 1]$, for some positive integer $N$. Define the \emph{weight} of $\mathcal{E}$ to be the product of the weights of the \emph{interior vertices} of $\mathcal{E}$, that is, the vertices of $\mathcal{E}$ contained in the strip $[0, \infty) \times [1, N] \subset \mathbb{Z}_{\ge 0} \times \mathbb{Z}_{> 0}$; these are the vertices of $\mathcal{E}$ satisfying arrow conservation. Assuming that all paths of $\mathcal{E}$ are finite, this product is well defined since all but finitely many vertices in $[0, \infty) \times [0, N]$ have vertex type $(0, 0; 0, 0)$, and $w_{(x, y)} (0, 0; 0, 0) = 1$ for any $(x, y) \in \mathbb{Z}_{\ge 0} \times \mathbb{Z}_{> 0}$. 

Now, using these weights we can define the rational symmetric functions $F_{\lambda / \mu}$ and $G_{\lambda / \mu}$. These were originally given by Borodin and Petrov in \cite{HSVMRSF} as Definition 4.4 and Definition 4.3, respectively, as generalizations of the homogeneous functions introduced by Povolotsky \cite{IZRCMFSS} and Borodin \cite{FRSF}. Let us also mentioned that some of their precursors were known to mathematical physicists far earlier. For instance, see the original work of Bethe \cite{TMLAC} from 1931 and also the works of Babbitt-Thomas \cite{GSRIF}, Kirillov-Reshetikhin \cite{ASG}, Babbitt-Gutkin \cite{ISC}, Jing \cite{VOSF}, Felder-Tarasov-Varchenko \cite{EEA}, and van Deijen \cite{CRBCW}; in fact, Proposition 4 of the work \cite{EEA} introduced elliptic generalizations of the $F_{\lambda / \mu}$ functions defined below.

\begin{definition}[{\cite[Definition 4.4]{HSVMRSF}}]
\label{definitionf}
Let $M, N \ge 0$ be integers and $\lambda \in \Sign_{M + N}^+$ and $\mu \in \Sign_M^+$ be signatures. Then, the rational function $F_{\lambda / \mu} \big( u_1, u_2, \ldots , u_N \b| \Xi, S \big)$ denotes the sum of the weights of all directed-path ensembles, consisting of $M + N$ paths whose interior vertices are contained in the rectangle $[0, \lambda_1] \times [1, N]$, satisfying the following two properties. 

\begin{itemize}
\item{Every path contains one edge that either connects $(-1, k)$ to $(0, k)$ for some $k \in [1, N]$ or connects $(\mu_j, 0)$ to $(\mu_j, 1)$ for some $j \in [1, M]$; the latter part of this statement holds with multiplicity, meaning that $m_i (\mu)$ paths connect $(i, 0)$ to $(i, 1)$ for each $i \in \mu$.}
\item{ Every path contains an edge connecting $(\lambda_k, N)$ to $(\lambda_k, N + 1)$, for some $k \in [1, M + N]$; again, this holds with multiplicity.}

\end{itemize} 
\end{definition}

\begin{definition}[{\cite[Definition 4.3]{HSVMRSF}}] 
\label{definitiong} 
Let $M, N \ge 0$ be integers and $\lambda, \mu \in \Sign_M^+$ be signatures. Then, the rational function $G_{\lambda / \mu} \big( u_1, u_2, \ldots , u_N \b| \Xi, S \big)$ denotes the sum of the weights of all directed-path ensembles, consisting of $M$ paths whose interior vertices are contained in the rectangle $[0, \lambda_1] \times [1, N]$, satisfying the following two properties (both with multiplicity). 

\begin{itemize}
\item{Each path contains an edge connecting $(\mu_k, 0)$ to $(\mu_k, 1)$ for some $k \in [1, M]$.}

\item{Each path contains an edge connecting $(\lambda_k, N)$ to $(\lambda_k, N + 1)$ for some $k \in [1, M]$.}
\end{itemize}

\end{definition}

Examples of the types of paths counted by $F_{\lambda / \mu}$ and $G_{\lambda / \mu}$ are depicted in Figure \ref{fgpaths}. 

We define $F_{\lambda}$ to be the function $F_{\lambda / \mu}$ when $\mu$ is empty; similarly, we define $G_{\lambda} = G_{\lambda / \mu}$, where $\mu = (0, 0, \ldots , 0)$ is the signature with $\ell (\lambda)$ zeroes. 

\begin{figure}

\begin{center}

\begin{tikzpicture}[
      >=stealth,
			scale = 1.2
			]

			\draw[->, black, dashed	] (.5, 0) -- (.5, 3);
			\draw[->, black, dashed] (.5, 0) -- (3, 0);
			\draw[->,black, thick] (0, .5) -- (.45, .5);
			\draw[->,black, thick] (0, 1) -- (.45, 1);
			\draw[->,black, thick] (0, 1.5) -- (.45, 1.5);
			\draw[->,black, thick] (0, 2) -- (.45, 2);
			
			\draw[->, black, thick	] (.5, 0) -- (.5, .45);
			\draw[->,black, thick] (.53, .55) -- (.53, .95);
			\draw[->,black, thick] (.47, .55) -- (.47, .95);
			\draw[->,black, thick] (.55, 1) -- (.95, 1);
			\draw[->,black, thick] (.55, 1.5) -- (.95, 1.5);
			\draw[->,black, thick] (.53, 1.05) -- (.53, 1.45);
			\draw[->,black, thick] (.47, 1.05) -- (.47, 1.45);
			\draw[->,black, thick] (.53, 1.55) -- (.53, 1.95);
			\draw[->,black, thick] (.47, 1.55) -- (.47, 1.95);
			\draw[->,black, thick] (.55, 2.05) -- (.55, 2.45);
			\draw[->,black, thick] (.5, 2.05) -- (.5, 2.45);
			\draw[->,black, thick] (.45, 2.05) -- (.45, 2.45);

			\draw[->,black, thick] (1.05, 1.5) -- (1.45, 1.5);
			\draw[->,black, thick] (1, 1.55) -- (1, 1.95);
			\draw[->,black, thick] (1, 1.05) -- (1, 1.45);
			\draw[->,black, thick] (1, 2.05) -- (1, 2.45);
			\draw[->, black, thick	] (1, 0) -- (1, .45);
			\draw[->,black, thick] (1.05, .5) -- (1.45, .5);

			\draw[->,black, thick] (1.5, .55) -- (1.5, .95);
			\draw[->,black, thick] (1.5, 1.05) -- (1.5, 1.45);
			\draw[->,black, thick] (1.47, 1.55) -- (1.47, 1.95);
			\draw[->,black, thick] (1.53, 1.55) -- (1.53, 1.95);
			\draw[->,black, thick] (1.5, 2.05) -- (1.5, 2.45);
			\draw[->,black, thick] (1.55, 2) -- (1.95, 2);

			\draw[->,black, thick] (2.05, .5) -- (2.45, .5);
			\draw[->,black, thick] (2.05, 2) -- (2.45, 2);
			\draw[->, black, thick] (2, 0) -- (2, .45);
			
			\draw[->, black, thick] (2.5, .55) -- (2.5, .95);
			\draw[->, black, thick] (2.5, 1.05) -- (2.5, 1.45);
			\draw[->, black, thick] (2.5, 1.55) -- (2.5, 1.95);
			\draw[->, black, thick] (2.47, 2.05) -- (2.47, 2.5);
			\draw[->, black, thick] (2.53, 2.05) -- (2.53, 2.5);

			\filldraw[fill=gray!50!white, draw=black] (.5, .5) circle [radius=.05];
			\filldraw[fill=gray!50!white, draw=black] (.5, 1) circle [radius=.05];
			\filldraw[fill=gray!50!white, draw=black] (.5, 1.5) circle [radius=.05];
			\filldraw[fill=gray!50!white, draw=black] (.5, 2) circle [radius=.05];

			\filldraw[fill=gray!50!white, draw=black] (1, .5) circle [radius=.05];
			\filldraw[fill=gray!50!white, draw=black] (1, 1) circle [radius=.05];
			\filldraw[fill=gray!50!white, draw=black] (1, 1.5) circle [radius=.05];
			\filldraw[fill=gray!50!white, draw=black] (1, 2) circle [radius=.05];
			
			\filldraw[fill=gray!50!white, draw=black] (1.5, .5) circle [radius=.05];
			\filldraw[fill=gray!50!white, draw=black] (1.5, 1) circle [radius=.05];
			\filldraw[fill=gray!50!white, draw=black] (1.5, 1.5) circle [radius=.05];
			\filldraw[fill=gray!50!white, draw=black] (1.5, 2) circle [radius=.05];

			\filldraw[fill=gray!50!white, draw=black] (2, .5) circle [radius=.05];
			\filldraw[fill=gray!50!white, draw=black] (2, 1) circle [radius=.05];
			\filldraw[fill=gray!50!white, draw=black] (2, 1.5) circle [radius=.05];
			\filldraw[fill=gray!50!white, draw=black] (2, 2) circle [radius=.05];
			
			\filldraw[fill=gray!50!white, draw=black] (2.5, .5) circle [radius=.05];
			\filldraw[fill=gray!50!white, draw=black] (2.5, 1) circle [radius=.05];
			\filldraw[fill=gray!50!white, draw=black] (2.5, 1.5) circle [radius=.05];
			\filldraw[fill=gray!50!white, draw=black] (2.5, 2) circle [radius=.05];

			\draw[->, black, dashed] (7.5, 0) -- (7.5, 3);
			\draw[->, black, dashed] (7.5, 0) -- (10, 0);
			
			\draw[->, black, thick	] (7.5, 0) -- (7.5, .45);
			\draw[->,black, thick] (7.5, .55) -- (7.5, .95);
			\draw[->,black, thick] (7.55, 1) -- (7.95, 1);

			\draw[->,black, thick] (8.05, 0) -- (8.05, .45);
			\draw[->,black, thick] (8, 0) -- (8, .45);
			\draw[->,black, thick] (7.95, 0) -- (7.95, .45);
			\draw[->,black, thick] (8.03, .55) -- (8.03, .95);
			\draw[->,black, thick] (7.97, .55) -- (7.97, .95);
			\draw[->,black, thick] (8, 1.55) -- (8, 1.95);
			
			\draw[->,black, thick] (8.05, 1.5) -- (8.45, 1.5);
			\draw[->,black, thick] (8.03, 1.05) -- (8.03, 1.45);
			\draw[->,black, thick] (7.97, 1.05) -- (7.97, 1.45);
			\draw[->,black, thick] (8.05, 1) -- (8.45, 1);
			\draw[->,black, thick] (8.05, .5) -- (8.45, .5);
			\draw[->,black, thick] (8.05, 2) -- (8.45, 2);

			\draw[->,black, thick] (8.5, .55) -- (8.5, .95);
			\draw[->,black, thick] (8.5, 1.05) -- (8.5, 1.45);
			\draw[->,black, thick] (8.47, 1.55) -- (8.47, 1.95);
			\draw[->,black, thick] (8.53, 1.55) -- (8.53, 1.95);
			\draw[->,black, thick] (8.47, 2.05) -- (8.47, 2.45);
			\draw[->,black, thick] (8.53, 2.05) -- (8.53, 2.45);
			\draw[->,black, thick] (8.55, 2) -- (8.95, 2);
			\draw[->,black, thick] (8.55, 1) -- (8.95, 1);

			\draw[->,black, thick] (9.05, .5) -- (9.45, .5);
			\draw[->,black, thick] (9.05, 2) -- (9.45, 2);
			\draw[->, black, thick] (9, 0) -- (9, .45);
			\draw[->, black, thick] (9, 1) -- (9, 1.45);
			\draw[->, black, thick] (9, 1.55) -- (9, 1.95);
			\draw[->, black, thick] (9, 2.05) -- (9, 2.45);
			
			\draw[->, black, thick] (9.5, .55) -- (9.5, .95);
			\draw[->, black, thick] (9.5, 1.05) -- (9.5, 1.45);
			\draw[->, black, thick] (9.5, 1.55) -- (9.5, 1.95);
			\draw[->, black, thick] (9.47, 2.05) -- (9.47, 2.5);
			\draw[->, black, thick] (9.53, 2.05) -- (9.53, 2.5);

			\filldraw[fill=gray!50!white, draw=black] (7.5, .5) circle [radius=.05];
			\filldraw[fill=gray!50!white, draw=black] (7.5, 1) circle [radius=.05];
			\filldraw[fill=gray!50!white, draw=black] (7.5, 1.5) circle [radius=.05];
			\filldraw[fill=gray!50!white, draw=black] (7.5, 2) circle [radius=.05];
			
			\filldraw[fill=gray!50!white, draw=black] (8, .5) circle [radius=.05];
			\filldraw[fill=gray!50!white, draw=black] (8, 1) circle [radius=.05];
			\filldraw[fill=gray!50!white, draw=black] (8, 1.5) circle [radius=.05];
			\filldraw[fill=gray!50!white, draw=black] (8, 2) circle [radius=.05];

			\filldraw[fill=gray!50!white, draw=black] (8.5, .5) circle [radius=.05];
			\filldraw[fill=gray!50!white, draw=black] (8.5, 1) circle [radius=.05];
			\filldraw[fill=gray!50!white, draw=black] (8.5, 1.5) circle [radius=.05];
			\filldraw[fill=gray!50!white, draw=black] (8.5, 2) circle [radius=.05];
			
			\filldraw[fill=gray!50!white, draw=black] (9, .5) circle [radius=.05];
			\filldraw[fill=gray!50!white, draw=black] (9, 1) circle [radius=.05];
			\filldraw[fill=gray!50!white, draw=black] (9, 1.5) circle [radius=.05];
			\filldraw[fill=gray!50!white, draw=black] (9, 2) circle [radius=.05];

			\filldraw[fill=gray!50!white, draw=black] (9.5, .5) circle [radius=.05];
			\filldraw[fill=gray!50!white, draw=black] (9.5, 1) circle [radius=.05];
			\filldraw[fill=gray!50!white, draw=black] (9.5, 1.5) circle [radius=.05];
			\filldraw[fill=gray!50!white, draw=black] (9.5, 2) circle [radius=.05];

\end{tikzpicture}

\end{center}

\caption{\label{fgpaths} To the left is a path that would be counted by $F_{\lambda / \mu}$, with $\lambda = (4, 4, 2, 1, 0, 0, 0)$ and $\mu = (3, 1, 0)$; to the right is a path that would be counted by $G_{\lambda / \mu}$, with $\lambda = (4, 4, 3, 2, 2)$ and $\mu = (3, 1, 1, 1, 0)$.  } 

\end{figure}
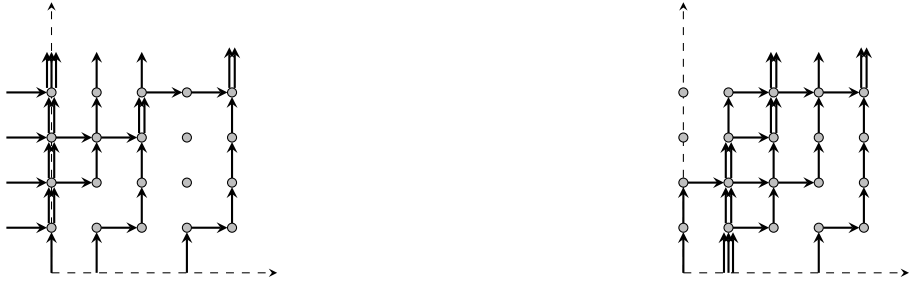

Before explaining the relationship between these symmetric functions and the stochastic higher spin vertex models, we define the \emph{conjugation} $G^c$ of $G$ through
\begin{flalign}
\label{gc}
G_{\lambda / \mu}^c \big( U \b| \Xi, S \big) = \left( \displaystyle\frac{c_S (\lambda) }{c_S (\mu)} \right) G_{\lambda / \mu} \big( U \b| \Xi, S \big), \qquad \text{where} \quad c_S (\nu) = \displaystyle\prod_{i = 1}^{\infty} \displaystyle\frac{(s_i^2; q)_{m_i (\nu)}}{(q; q)_{m_i (\nu)}},
\end{flalign}

\noindent for any $\nu \in \Sign^+$. For any $v \in \mathbb{C}$, also define 
\begin{flalign}
\label{grhoq}
G_{\lambda} \big( \rho (v, q^J) \b| \Xi, S \big) = G_{\lambda} \big( v, qv, \ldots , q^{J - 1} v \b| \Xi, S \big). 	
\end{flalign}

\noindent As shown in the proof of Proposition 6.7 of \cite{HSVMRSF}, this is a rational function in $q^J$; hence, by analytic continuation, $G_{\lambda} \big( \rho (v, w) \b| \Xi, S \big)$ is well-defined for any $w \in \mathbb{C}$. In particular, we can denote
\begin{flalign}
\label{grho}
G_{\lambda}^c \big( \rho \b| \Xi, S \big) = \displaystyle\lim_{v \rightarrow 0} G_{\lambda}^c \big( \rho (v, \xi_0^{-1} s_0^{-1} v^{-1}) \b| \Xi, S \big). 
\end{flalign}

\noindent The following lemma, which is Proposition 6.7 of \cite{HSVMRSF}, ensures that the limit $G_{\lambda} \big( \rho \b| \Xi, S \big)$ exists and also evaluates it explicitly. In what follows, $e_i (\lambda)$ denotes the number of indices $j$ for which $\lambda_j > i$. 

\begin{prop}[{\cite[Proposition 6.7]{HSVMRSF}}]
\label{productg}

Let $n \ge 0$ be an integer, and let $\lambda \in \Sign_n^+$. If $\lambda_n = 0$, then $G_{\lambda} (\rho \b| \Xi^{-1}, S) = 0$. Otherwise, 
\begin{flalign*} 
G_{\lambda} (\rho \b| \Xi^{-1}, S) = \displaystyle\frac{(s_0^2; q)_n}{(- s_0)^n} \displaystyle\prod_{i = 1}^n (-s_i)^{e_i (\lambda)}. 
\end{flalign*} 
\end{prop}

As the next proposition (which is Corollary 4.16 of \cite{HSVMRSF}) indicates, it is also possible to explicitly evaluate $F_{\lambda}$ at a principal specialization.

\begin{prop}[{\cite[Corollary 4.16]{HSVMRSF}}]
\label{productf}

Let $J$ be a positive integer. For any complex number $u$ and signature $\lambda \in \Sign_J^+$, we have that 
\begin{flalign*}
F_{\lambda} \big( u, q u, \ldots , q^{J - 1} u \b| \Xi, S \big) = (q; q)_J \displaystyle\prod_{k = 1}^J \displaystyle\frac{1}{1 - s_{\lambda_k} \xi_{\lambda_k} q^{k - 1} u} \displaystyle\prod_{h = 0}^{\lambda_k - 1} \displaystyle\frac{\xi_h q^{k - 1} u - s_h}{1 - s_h \xi_h q^{k - 1} u}. 
\end{flalign*}
\end{prop}

Now let us explain how the $F$ and $G$ symmetric functions are related to vertex models. The following definition is a special case of Definition 6.1 of \cite{HSVMRSF}; its relevance (and its relationship to the stochastic higher spin vertex models) will be explained by Proposition \ref{measuresignatures}. 

\begin{definition}[{\cite[Definition 6.1]{HSVMRSF}}]

\label{signaturemeasures}

For any $\lambda \in \Sign_n^+$, define 
\begin{flalign*}
\mathscr{M} (\lambda) = \mathscr{M}_{U; \Xi, S} (\lambda) = Z_{\rho, U}^{-1} F_{\lambda} \big( u_1, u_2, \ldots , u_n \b| \Xi, S \big) G_{\lambda}^c \big( \rho \b| \Xi^{-1}, S \big), 
\end{flalign*} 

\noindent where $Z_{\rho, U} = s_0^{-N} (q; q)_N \prod_{i = 1}^N (s_0 - \xi_0 u_i) (1 - s_0 \xi_0 u_i)^{-1}$ is a normalization constant.
\end{definition}

\begin{prop}[{\cite[Section 6]{HSVMRSF}}]

\label{measuresignatures} 

For any $n \ge 1$, signature $\lambda \in \Sign_n^+$, and sets of complex variables $U$, $S$, and $\Xi$ such that the probabilities \eqref{configurationprobabilities} are positive, we have that $\mathscr{M} (\lambda) = \textbf{\emph{M}}_n (\lambda)$ (given by Definition \ref{measuresignaturesvertexmodel}). 
\end{prop}

\section{Observables for Models With Double-Sided Initial Data}

\label{ObservablesDoubleSided}

In this section we define and analyze a specialization of the measure $\mathscr{M}$ (from Definition \ref{signaturemeasures}), which we call $\mathscr{M}^{\mathfrak{S}}$. In Section \ref{SpecializationM}, we explain how $\mathscr{M}^{\mathfrak{S}}$ prescribes a stochastic six-vertex model with a certain type of double-sided initial condition, in which paths enter through each vertex on the $y$-axis independently and with probability $b_1$, but only finitely many paths enter through the $x$-axis. Using the results of \cite{HSVMRSF}, we produce observables for this model in Section \ref{ContoursCurrent}.

\subsection{A Double-Sided Specialization of \texorpdfstring{$\mathscr{M}$}{}} 

\label{SpecializationM}

\subsubsection{The Double-Sided Specialization of \texorpdfstring{$\mathscr{M}$}{}} 

\label{DoubleSidedSpecialization}

We begin with the following definition, which introduces the measure $\mathscr{M}^{\mathfrak{S}}$ mentioned above.	

\begin{definition}

\label{ms} 

Fix $n, J \in \mathbb{Z}_{> 0}$; $\delta_1, \delta_2, b_1 \in (0, 1)$; $b_2 \in \mathbb{C}$; and $a \in (-1, 0)$. For each $i \in \{ 1, 2 \}$, define  
\begin{flalign}
\label{stochasticparameters}
\beta_i = \displaystyle\frac{b_i}{1 - b_i}; \quad q = \displaystyle\frac{\delta_1}{\delta_2} < 1; \quad \kappa = \displaystyle\frac{1 - \delta_1}{1 - \delta_2} > 1; \quad s = q^{-1 / 2}; \quad u = \kappa s; \quad v = - \displaystyle\frac{1}{\beta_2 s}; \quad \xi = - \displaystyle\frac{\beta_1}{a u}. 
\end{flalign}

Define the sets $\varpi = (v, qv, \ldots , q^{J - 1} v)$; $\overline{U} = (u, u, \ldots , u)$, where $u$ appears with multiplicity $n$; and $U = \varpi \cup \overline{U}$. Moreover, define the parameter sets $\Xi = (\xi_0, \xi_1, \ldots )$ and $S = (s_0, s_1, \ldots )$ by setting $\xi_i = 1$ if $i \ne 1$; $\xi_1 = \xi$; $s_i = s$ if $i \ne 1$; and $s_1 = a$. 

Define the formal measure $\mathscr{M}^{\mathfrak{S}} = \mathscr{M}_{U; \Xi, S}$ on $\Sign_{n + J}^+$, where $\mathscr{M}_{U; \Xi, S}$ is from Definition \ref{signaturemeasures}. 
\end{definition}

If the probabilities \eqref{configurationprobabilities} are all positive under the above specialization, then Proposition \ref{measuresignatures} states that $\mathscr{M}^{\mathfrak{S}}$ prescribes a stochastic higher spin vertex model (with step initial data) whose first $J$ spectral parameters are $v, qv, \ldots , q^{J - 1} v$; whose other spectral parameters all equal $u$; whose spin parameter in the first column is $a$; and whose other spins all equal $s$. See Figure \ref{shift2generalizedvertex} for an example.

Let us analyze how the transition probabilities \eqref{configurationprobabilities} degenerate under $\mathscr{M}^{\mathfrak{S}}$ when $y > J$ (that is, when $u_y = u$). There are two cases to consider, when $x = 1$ and when $x > 1$. 

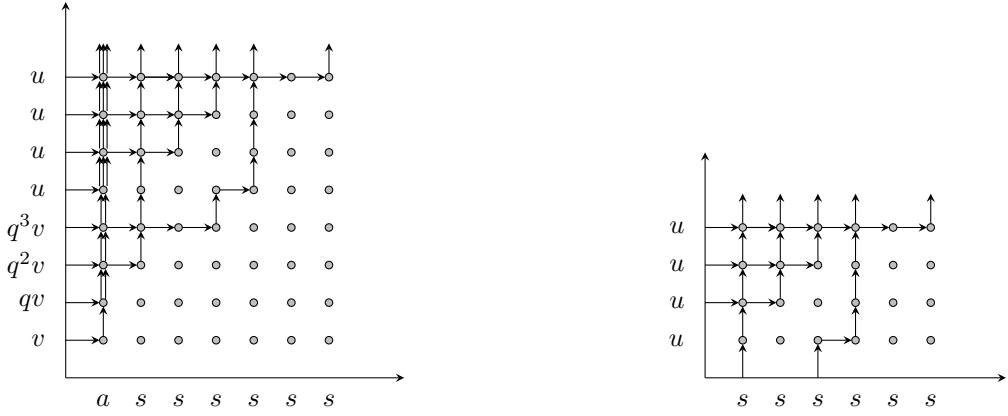
\begin{figure} 

\begin{center}

\begin{tikzpicture}[
      >=stealth,
			scale = 1
			]

			\draw[->, black	] (0, 0) -- (4.5, 0);
			\draw[->, black] (0, 0) -- (0, 5);
			
			\draw[->, black] (0, .5) -- (.45, .5) node[black, left=17] {$v$};
			\draw[->, black] (0, 1) -- (.45, 1) node[black, left=17] {$qv$};
			\draw[->, black] (0, 1.5) -- (.45, 1.5) node[black, left=17] {$q^2 v$};
			\draw[->, black] (0, 2) -- (.45, 2) node[black, left=17] {$q^3 v$};
			\draw[->, black] (0, 2.5) -- (.45, 2.5) node[black, left=17] {$u$};
			\draw[->, black] (0, 3) -- (.45, 3) node[black, left=17] {$u$};
			\draw[->, black] (0, 3.5) -- (.45, 3.5) node[black, left=17] {$u$};
			\draw[->, black] (0, 4) -- (.45, 4) node[black, left=17] {$u$};

			\draw[->, black] (.55, 1.5) -- (.95, 1.5);
			\draw[->, black] (.55, 2) -- (.95, 2);
			\draw[->, black] (.55, 3) -- (.95, 3);
			\draw[->, black] (.55, 3.5) -- (.95, 3.5);
			\draw[->, black] (.55, 4) -- (.95, 4);

			\draw[->, black] (1.05, 2) -- (1.45, 2);
			\draw[->, black] (1.05, 3) -- (1.45, 3);
			\draw[->, black] (1.05, 4) -- (1.45, 4);
			
			\draw[->, black] (1.55, 2) -- (1.95, 2);
			\draw[->, black] (1.55, 3.5) -- (1.95, 3.5);
			
			\draw[->, black] (2.05, 2.5) -- (2.45, 2.5);
			
			\draw[->, black] (2.55, 4) -- (2.95, 4);

			\draw[->, black] (.5, .55) -- (.5, .95);
			\draw[->, black] (.47, 1.05) -- (.47, 1.45);
			\draw[->, black] (.53, 1.05) -- (.53, 1.45);
			\draw[->, black] (.47, 1.55) -- (.47, 1.95);
			\draw[->, black] (.53, 1.55) -- (.53, 1.95);
			\draw[->, black] (.47, 2.05) -- (.47, 2.45);
			\draw[->, black] (.53, 2.05) -- (.53, 2.45);
			\draw[->, black] (.55, 2.55) -- (.55, 2.95);
			\draw[->, black] (.45, 2.55) -- (.45, 2.95);
			\draw[->, black] (.5, 2.55) -- (.5, 2.95);
			\draw[->, black] (.55, 3.05 ) -- (.55, 3.45);
			\draw[->, black] (.45, 3.05) -- (.45, 3.45);
			\draw[->, black] (.5, 3.05) -- (.5, 3.45);
			\draw[->, black] (.55, 3.55) -- (.55, 3.95);
			\draw[->, black] (.45, 3.55) -- (.45, 3.95);
			\draw[->, black] (.5, 3.55) -- (.5, 3.95);
			\draw[->, black] (.55, 4.05 ) -- (.55, 4.45);
			\draw[->, black] (.45, 4.05) -- (.45, 4.45);
			\draw[->, black] (.5, 4.05) -- (.5, 4.45);

			\draw[->, black] (1, 1.55) -- (1, 1.95);
			\draw[->, black] (1, 2.05) -- (1, 2.45);
			\draw[->, black] (1, 2.55) -- (1, 2.95);
			\draw[->, black] (1, 3.05 ) -- (1, 3.45);
			\draw[->, black] (1, 3.55) -- (1, 3.95);
			\draw[->, black] (1, 4.05) -- (1, 4.45);
			\draw[->, black] (1.05, 4) -- (1.45, 4); 
			\draw[->, black] (1.05, 3.5) -- (1.45, 3.5); 
			
			\draw[->, black] (1.5, 3.05 ) -- (1.5, 3.45);
			\draw[->, black] (1.5, 4.05) -- (1.5, 4.45);  
			\draw[->, black] (1.5, 3.55) -- (1.5, 3.95);
			\draw[->, black] (1.55, 4) -- (1.95, 4); 
			
			\draw[->, black] (2, 2.05 ) -- (2, 2.45);
			\draw[->, black] (2, 3.55) -- (2, 3.95);
			\draw[->, black] (2, 4.05) -- (2, 4.45); 
			\draw[->, black] (2.05, 4) -- (2.45, 4);

			\draw[->, black] (2.5, 2.55 ) -- (2.5, 2.95);
			\draw[->, black] (2.5, 3.05) -- (2.5, 3.45);
			\draw[->, black] (2.5, 3.55) -- (2.5, 3.95);
	
			\draw[->, black] (3.05, 4) -- (3.45, 4);
			
			\draw[->, black] (2.5, 4.05) -- (2.5, 4.45);
			\draw[->, black] (3.5, 4.05) -- (3.5, 4.45);

			\filldraw[fill=gray!50!white, draw=black] (.5, .5) circle [radius=.05] node[black, below = 17] {$a$};
			\filldraw[fill=gray!50!white, draw=black] (.5, 1) circle [radius=.05];
			\filldraw[fill=gray!50!white, draw=black] (.5, 1.5) circle [radius=.05];
			\filldraw[fill=gray!50!white, draw=black] (.5, 2) circle [radius=.05];
			\filldraw[fill=gray!50!white, draw=black] (.5, 2.5) circle [radius=.05];
			\filldraw[fill=gray!50!white, draw=black] (.5, 3) circle [radius=.05];
			\filldraw[fill=gray!50!white, draw=black] (.5, 3.5) circle [radius=.05];
			\filldraw[fill=gray!50!white, draw=black] (.5, 4) circle [radius=.05];

			\filldraw[fill=gray!50!white, draw=black] (1, .5) circle [radius=.05] node[black, below = 17] {$s$};
			\filldraw[fill=gray!50!white, draw=black] (1, 1) circle [radius=.05];
			\filldraw[fill=gray!50!white, draw=black] (1, 1.5) circle [radius=.05];
			\filldraw[fill=gray!50!white, draw=black] (1, 2) circle [radius=.05];
			\filldraw[fill=gray!50!white, draw=black] (1, 2.5) circle [radius=.05];
			\filldraw[fill=gray!50!white, draw=black] (1, 3) circle [radius=.05];
			\filldraw[fill=gray!50!white, draw=black] (1, 3.5) circle [radius=.05];
			\filldraw[fill=gray!50!white, draw=black] (1, 4) circle [radius=.05];
			
			\filldraw[fill=gray!50!white, draw=black] (1.5, .5) circle [radius=.05] node[black, below = 17] {$s$};
			\filldraw[fill=gray!50!white, draw=black] (1.5, 1) circle [radius=.05];
			\filldraw[fill=gray!50!white, draw=black] (1.5, 1.5) circle [radius=.05];
			\filldraw[fill=gray!50!white, draw=black] (1.5, 2) circle [radius=.05];
			\filldraw[fill=gray!50!white, draw=black] (1.5, 2.5) circle [radius=.05];
			\filldraw[fill=gray!50!white, draw=black] (1.5, 3) circle [radius=.05];
			\filldraw[fill=gray!50!white, draw=black] (1.5, 3.5) circle [radius=.05];
			\filldraw[fill=gray!50!white, draw=black] (1.5, 4) circle [radius=.05];

			\filldraw[fill=gray!50!white, draw=black] (2, .5) circle [radius=.05] node[black, below = 17] {$s$};
			\filldraw[fill=gray!50!white, draw=black] (2, 1) circle [radius=.05];
			\filldraw[fill=gray!50!white, draw=black] (2, 1.5) circle [radius=.05];
			\filldraw[fill=gray!50!white, draw=black] (2, 2) circle [radius=.05];
			\filldraw[fill=gray!50!white, draw=black] (2, 2.5) circle [radius=.05];
			\filldraw[fill=gray!50!white, draw=black] (2, 3) circle [radius=.05];
			\filldraw[fill=gray!50!white, draw=black] (2, 3.5) circle [radius=.05];
			\filldraw[fill=gray!50!white, draw=black] (2, 4) circle [radius=.05];
			
			\filldraw[fill=gray!50!white, draw=black] (2.5, .5) circle [radius=.05] node[black, below = 17] {$s$};
			\filldraw[fill=gray!50!white, draw=black] (2.5, 1) circle [radius=.05];
			\filldraw[fill=gray!50!white, draw=black] (2.5, 1.5) circle [radius=.05];
			\filldraw[fill=gray!50!white, draw=black] (2.5, 2) circle [radius=.05];
			\filldraw[fill=gray!50!white, draw=black] (2.5, 2.5) circle [radius=.05];
			\filldraw[fill=gray!50!white, draw=black] (2.5, 3) circle [radius=.05];
			\filldraw[fill=gray!50!white, draw=black] (2.5, 3.5) circle [radius=.05];
			\filldraw[fill=gray!50!white, draw=black] (2.5, 4) circle [radius=.05];

			\filldraw[fill=gray!50!white, draw=black] (3, .5) circle [radius=.05] node[black, below = 17] {$s$};
			\filldraw[fill=gray!50!white, draw=black] (3, 1) circle [radius=.05];
			\filldraw[fill=gray!50!white, draw=black] (3, 1.5) circle [radius=.05];
			\filldraw[fill=gray!50!white, draw=black] (3, 2) circle [radius=.05];
			\filldraw[fill=gray!50!white, draw=black] (3, 2.5) circle [radius=.05];
			\filldraw[fill=gray!50!white, draw=black] (3, 3) circle [radius=.05];
			\filldraw[fill=gray!50!white, draw=black] (3, 3.5) circle [radius=.05];
			\filldraw[fill=gray!50!white, draw=black] (3, 4) circle [radius=.05];
			
			\filldraw[fill=gray!50!white, draw=black] (3.5, .5) circle [radius=.05] node[black, below = 17] {$s$};
			\filldraw[fill=gray!50!white, draw=black] (3.5, 1) circle [radius=.05];
			\filldraw[fill=gray!50!white, draw=black] (3.5, 1.5) circle [radius=.05];
			\filldraw[fill=gray!50!white, draw=black] (3.5, 2) circle [radius=.05];
			\filldraw[fill=gray!50!white, draw=black] (3.5, 2.5) circle [radius=.05];
			\filldraw[fill=gray!50!white, draw=black] (3.5, 3) circle [radius=.05];
			\filldraw[fill=gray!50!white, draw=black] (3.5, 3.5) circle [radius=.05];
			\filldraw[fill=gray!50!white, draw=black] (3.5, 4) circle [radius=.05];

	\draw[->, black	] (8.5, 0) -- (12.5, 0);
			\draw[->, black] (8.5, 0) -- (8.5, 3);

			\draw[->, black] (8.5, 1) -- (8.95, 1) node[black, left=17] {$u$};
			\draw[->, black] (8.5, 1.5) -- (8.95, 1.5) node[black, left=17] {$u$};
			\draw[->, black] (8.5, 2) -- (8.95, 2) node[black, left=17] {$u$};

			\draw[->, black] (9.05, 1) -- (9.45, 1);
			\draw[->, black] (9.05, 2) -- (9.45, 2);
			
			\draw[->, black] (9.55, 1.5) -- (9.95, 1.5);
			
			\draw[->, black] (10.05, .5) -- (10.45, .5) node[black, left=60] {$u$};
			
			\draw[->, black] (10.55, 2) -- (10.95, 2);
		
			\draw[->, black] (9, 0) -- (9, .45);
			\draw[->, black] (9, .55) -- (9, .95);
			\draw[->, black] (9, 1.05) -- (9, 1.45);
			\draw[->, black] (9, 1.55) -- (9, 1.95);
			\draw[->, black] (9, 2.05) -- (9, 2.45);
			\draw[->, black] (9.05, 2) -- (9.45, 2); 
			\draw[->, black] (9.05, 1.5) -- (9.45, 1.5); 
			
			\draw[->, black] (9.5, 1.05) -- (9.5, 1.45);
			\draw[->, black] (9.5, 2.05) -- (9.5, 2.45);  
			\draw[->, black] (9.5, 1.55) -- (9.5, 1.95);
			\draw[->, black] (9.55, 2) -- (9.95, 2); 
			
			\draw[->, black] (10, 0) -- (10, .45);
			\draw[->, black] (10, 1.55) -- (10, 1.95);
			\draw[->, black] (10, 2.05) -- (10, 2.45); 
			\draw[->, black] (10.05, 2) -- (10.45, 2);

			\draw[->, black] (10.5, .55 ) -- (10.5, .95);
			\draw[->, black] (10.5, 1.05) -- (10.5, 1.45);
			\draw[->, black] (10.5, 1.55) -- (10.5, 1.95);
	
			\draw[->, black] (11.05, 2) -- (11.45, 2);
			
			\draw[->, black] (10.5, 2.05) -- (10.5, 2.45);
			\draw[->, black] (11.5, 2.05) -- (11.5, 2.45);

			\filldraw[fill=gray!50!white, draw=black] (9, .5) circle [radius=.05] node[black, below = 17] {$s$};
			\filldraw[fill=gray!50!white, draw=black] (9, 1) circle [radius=.05];
			\filldraw[fill=gray!50!white, draw=black] (9, 1.5) circle [radius=.05];
			\filldraw[fill=gray!50!white, draw=black] (9, 2) circle [radius=.05];
			
			\filldraw[fill=gray!50!white, draw=black] (9.5, .5) circle [radius=.05] node[black, below = 17] {$s$};
			\filldraw[fill=gray!50!white, draw=black] (9.5, 1) circle [radius=.05];
			\filldraw[fill=gray!50!white, draw=black] (9.5, 1.5) circle [radius=.05];
			\filldraw[fill=gray!50!white, draw=black] (9.5, 2) circle [radius=.05];
			
			\filldraw[fill=gray!50!white, draw=black] (10, .5) circle [radius=.05] node[black, below = 17] {$s$};
			\filldraw[fill=gray!50!white, draw=black] (10, 1) circle [radius=.05];
			\filldraw[fill=gray!50!white, draw=black] (10, 1.5) circle [radius=.05];
			\filldraw[fill=gray!50!white, draw=black] (10, 2) circle [radius=.05];
			
			\filldraw[fill=gray!50!white, draw=black] (10.5, .5) circle [radius=.05] node[black, below = 17] {$s$};
			\filldraw[fill=gray!50!white, draw=black] (10.5, 1) circle [radius=.05];
			\filldraw[fill=gray!50!white, draw=black] (10.5, 1.5) circle [radius=.05];
			\filldraw[fill=gray!50!white, draw=black] (10.5, 2) circle [radius=.05];
			
			\filldraw[fill=gray!50!white, draw=black] (11, .5) circle [radius=.05] node[black, below = 17] {$s$};
			\filldraw[fill=gray!50!white, draw=black] (11, 1) circle [radius=.05];
			\filldraw[fill=gray!50!white, draw=black] (11, 1.5) circle [radius=.05];
			\filldraw[fill=gray!50!white, draw=black] (11, 2) circle [radius=.05];
			
			\filldraw[fill=gray!50!white, draw=black] (11.5, .5) circle [radius=.05] node[black, below = 17] {$s$};
			\filldraw[fill=gray!50!white, draw=black] (11.5, 1) circle [radius=.05];
			\filldraw[fill=gray!50!white, draw=black] (11.5, 1.5) circle [radius=.05];
			\filldraw[fill=gray!50!white, draw=black] (11.5, 2) circle [radius=.05];

\end{tikzpicture}

\end{center}

\caption{\label{shift2generalizedvertex} To the left is the stochastic higher spin vertex model specialized as in Definition \ref{ms}, with $n = 4 = J$. To the right is this model shifted to the left by $1$ and down by $J = 4$. }

\end{figure}

First assume that $x > 1$. We will show later (see Proposition \ref{fg}) that $i_2 \in \{ 0, 1 \}$ at $(x, J)$, almost surely for each $x > 1$. Since $s_x = s = q^{-1 / 2}$, we have that $\mathbb{P}_n [(i_2, j_2) = (2, 0) | (i_1, j_1) = (1, 1)] = (1 - q s^2) / (1 - su) = 0$, meaning that $i_1, i_2 \in \{ 0, 1 \}$ at each $(x, y)$ with $x > 1$ and $y > J$. 

Substituting our specialization into the other probabilities in \eqref{configurationprobabilities} yields	
\begin{flalign}
\begin{aligned} 
\label{sixvertexprobabilities} 
& \mathbb{P}_n \big[ (i_2, j_2) = (0, 0) \big| (i_1, j_1) = (0, 0) \big] = 1 =  \mathbb{P}_n \big[ (i_2, j_2) = (1, 1) \big| (i_1, j_1) = (1, 1) \big],  \\
& \mathbb{P}_n \big[ (i_2, j_2) = (1, 0) \big| (i_1, j_1) = (1, 0) \big] = \delta_1; \quad \mathbb{P}_n \big[ (i_2, j_2) = (0, 1) \big| (i_1, j_1) = (1, 0) \big] = 1 - \delta_1,  \\
& \mathbb{P}_n \big[ (i_2, j_2) = (0, 1) \big| (i_1, j_1) = (0, 1) \big] = \delta_2; \quad \mathbb{P}_n \big[ (i_2, j_2) = (1, 0) \big| (i_1, j_1) = (0, 1) \big] = 1 - \delta_2. 
\end{aligned}
\end{flalign}

Since the probabilities \eqref{sixvertexprobabilities} coincide with the probabilities depicted in Figure \ref{sixvertexfigure}, it follows that the stochastic higher spin vertex model parametrized by $\mathscr{M}^{\mathfrak{S}}$ evolves as the stochastic six-vertex model $\mathcal{P} (\delta_1, \delta_2)$ defined in Section \ref{StochasticVertex}, in the sector $x > 1$, $y > J$. 

When $x = 1$, we instead have that $\xi_x = - \beta_1 / au$ and $s_x = a$. Inserting these parameters into \eqref{configurationprobabilities}, we obtain
\begin{flalign}
\label{verticalconfigurationprobabilities3}
\begin{aligned} 
& \mathbb{P}_n \big[ (i_2, j_2) = (k, 0) \big| (i_1, j_1) = (k, 0) \big] = 1 - b_1 + q^k b_1,  \\
& \mathbb{P}_n \big[ (i_2, j_2) = (k - 1, 1) \big| (i_1, j_1) = (k, 0) \big] = (1 - q^k) b_1,  \\
& \mathbb{P}_n \big[ (i_2, j_2) = (k + 1, 0) \big| (i_1, j_1) = (k, 1) \big] = (1 - q^k a^2)(1 - b_1),  \\
& \mathbb{P}_n \big[ (i_2, j_2) = (k, 1) \big| (i_1, j_1) = (k, 1) \big] = q^k a^2 (1 - b_1) + b_1.  
\end{aligned}
\end{flalign}

Due to the step initial data of our model, there is one incoming horizontal arrow at each vertex $(1, y)$, so we must have $j_1 = 1$; thus, we can ignore the first two probabilities in \eqref{verticalconfigurationprobabilities3}. Letting $a$ tend to $0$, the third and fourth probabilities in \eqref{verticalconfigurationprobabilities3} become 
\begin{flalign}
\label{verticalconfigurationprobabilities2} 
& \mathbb{P}_n \big[ (i_2, j_2) = (k + 1, 0) \big| (i_1, j_1) = (k, 1) \big] = 1 - b_1; \quad \mathbb{P}_n \big[ (i_2, j_2) = (k, 1) \big| (i_1, j_1) = (k, 1) \big] = b_1.  
\end{flalign}

Hence, if we ``shift'' the model to the left one coordinate and down $J$ coordinates (and ignore all arrows outside of the new positive quadrant; see Figure \ref{shift2generalizedvertex}), we obtain a stochastic six-vertex model in which paths enter through each vertex on the positive $y$-axis independently and with probability $b_1$, and finitely many (at most $J$) paths enter through the $x$-axis. Our next goal is to understand the entrance law of these $J$ paths.

\subsubsection{The Value of \texorpdfstring{$\mathscr{M}^{\mathfrak{S}}$}{} on Partitions of Length \texorpdfstring{$J$}{J}} 

\label{LengthJ}

In view of the discussion in the previous section, the measure $\mathscr{M}^{\mathfrak{S}}$ can be viewed as follows. First one takes $J$ steps of a stochastic higher spin vertex model with spectral parameters $v, qv, q^2 v, \ldots , q^{J - 1} v$. After that, one evolves according to the stochastic six-vertex model, but with a deformed spin in the first column; this deformation gives rise to a $b_1$-Bernoulli initial condition through the vertical line $x = 1$. The first $J$ steps of the stochastic higher spin vertex model also produces some random configuration of $J$ particles on the horizontal line $y = J$. 

The goal of this section is to explicitly determine the probability distribution of this configuration, which we do by evaluating $\mathscr{M}^{\mathfrak{S}}$ on signatures $\lambda$ of length $J$; this corresponds to the $n = 0$ (or equivalently $\overline{U}$ empty) case of Definition \ref{ms}. When $\lambda \in \Sign_J^+$, we have that $\mathscr{M}^{\mathfrak{S}} (\lambda) = \mathscr{M}_{\varpi; S, \Xi} (\lambda)$, where $\varpi$, $S$, and $\Xi$ are given by Definition \ref{ms}. Thus, we would like to evaluate $\mathscr{M}_{\varpi; S, \Xi} (\lambda) = Z_{\rho, \varpi}^{-1} F_{\lambda} \big( \varpi\b| \Xi, S \big) G_{\lambda}^c \big( \rho \b| \Xi^{-1}, S \big) $ (recall Definition \ref{signaturemeasures}) for each $\lambda \in \Sign_J^+$. To that end, we have the following proposition. 

\begin{prop}
\label{fg} 

Recall the notation from Definition \ref{ms}, and set $\vartheta = q^{-J}$. Let $\lambda \in \Sign_J^+$. If $\lambda_J = 0$ or $m_i (\lambda) > 1$ for some $i > 1$, then $F_{\lambda} \big( \varpi \b| \Xi, S \big) G_{\lambda}^c \big( \rho \b| \Xi^{-1}, S \big) = 0$. Otherwise, 
\begin{flalign*}
F_{\lambda} \big( \varpi\b| \Xi, S \big) G_{\lambda}^c \big( \rho \b| \Xi^{-1}, S \big) & = C \left( \displaystyle\frac{\kappa \beta_2}{\beta_1} \right)^m \displaystyle\prod_{j = 1}^m \displaystyle\frac{(1 - q^{j - 1} \vartheta) (1 - a^2 q^{j - 1}) }{(1 - q^j) (1 - a^2 \kappa \beta_2 \beta_1^{-1} q^{j - 1} \vartheta) } \displaystyle\prod_{k = 0}^{J - m - 1} (b_2 - q^{k + m} \vartheta b_2) \\
& \quad \times  \displaystyle\prod_{k = 0}^{J - m - 1} \big( 1 - b_2 + q^{k + m} \vartheta b_2 \big)^{\lambda_{J - m - k} - \max \{ \lambda_{J - m - k + 1}, 1 \} - 1}, 
\end{flalign*}

\noindent where $m = m_1 (\lambda)$, and $C = C (q, v, J, \Xi, S)$ is some constant independent of $\lambda$. 
\end{prop}

To facilitate the proof of Proposition \ref{fg}, we use the following lemma. 

\begin{lem}

\label{1deformationg} 

Let $\lambda \in \Sign_J^+$, and specialize $\Xi$ and $S$ as in Definition \ref{ms}. If $\lambda_J = 0$ or $m_i (\lambda) > 1$ for any integer $i \ge 2$, then $G_{\lambda}^c (\rho \b| \Xi^{-1}, S) = 0$. Otherwise, 
\begin{flalign*} 
G_{\lambda}^c (\rho \b| \Xi^{-1}, S) = (q; q)_J (-s)^{|\lambda| - m_1 (\lambda) - J} a^{J - m_1 (\lambda)} \displaystyle\prod_{k = 1}^{m_1 (\lambda)} \displaystyle\frac{1 - a^2 q^{i - 1}}{1 - q^i}. 
\end{flalign*} 
\end{lem}

\begin{proof}
If $\lambda_J = 0$, then the result follows from Proposition \ref{productg}. Otherwise, Proposition \ref{productg} yields 
\begin{flalign*}
G_{\lambda} (\rho \b| \Xi^{-1}, S) & = (s^2; q)_J (-s)^{|\lambda| - 2 J - e_1 (\lambda)} (-a)^{e_1 (\lambda)}, 
\end{flalign*}

\noindent so 
\begin{flalign}
\label{grhoa}
\begin{aligned}
G_{\lambda}^c (\rho \b| \Xi^{-1}, S) &= \displaystyle\frac{(q; q)_J}{(s^2; q)_J} c_{S} (\lambda) G_{\lambda} (\rho \b| \Xi^{-1}, S)  \\
&= (q; q)_J (-s)^{|\lambda| - 2 J - e_1 (\lambda)} (-a)^{e_1 (\lambda)} \displaystyle\frac{(a^2; q)_{m_1 (\lambda)}}{(q; q)_{m_1 (\lambda)}} \displaystyle\prod_{i = 2}^{\infty} \displaystyle\frac{(s^2; q)_{m_i (\lambda)}}{(q; q)_{m_i (\lambda)}},
\end{aligned}
\end{flalign}

\noindent where we used \eqref{gc} and the fact that $c(0^J) = (s^2; q)_J / (q; q)_J$ in the first equality. Now, if $m_i (\lambda) > 1$ for some integer $i > 1$, then $(s^2; q)_{m_i (\lambda)} = 0$ since $s^2 = q^{-1}$; this gives $G_{\lambda}^c \big( \rho \b| \Xi^{-1}, S \big) = 0$. 

Thus, assume that $m_i (\lambda) \le 1$ for each $i > 1$. Then, $(s^2; q)_{m_i (\lambda)} (q; q)_{m_i (\lambda)}^{-1} = (-s^2)^{m_i (\lambda)}$ for each $i > 1$. The lemma follows from inserting this into \eqref{grhoa} and using the facts that $e_1 (\lambda) + m_1 (\lambda) = J$ and $\sum_{i = 2}^{\infty} m_i (\lambda) = e_1 (\lambda)$. 
\end{proof}

\begin{proof}[Proof of Proposition \ref{fg}]

If $\lambda_J = 0$ or $m_i (\lambda) > 1$ for some $i > 1$, then the result follows from Proposition \ref{1deformationg}. 

Thus, assume that $\lambda_J > 0$. Then, Proposition \ref{productf} implies that $F_{\lambda} \big( \varpi \b| \Xi, S \big)$ is equal to
\begin{flalign*}
 (q; q)_J & \displaystyle\prod_{k = 1}^{J - m} \displaystyle\frac{\xi q^{k - 1} v - a}{(1 - s q^{k - 1} v) (1 - a \xi q^{k - 1} v)} \left( \displaystyle\frac{q^{k - 1} v - s}{1 - s q^{k - 1} v} \right)^{\lambda_k - 1} \displaystyle\prod_{h = J - m + 1}^J \displaystyle\frac{q^{h - 1} v - s}{(1 - a \xi q^{h - 1} v) (1 - s q^{h - 1} v) }.
\end{flalign*}

\noindent Applying Corollary \ref{1deformationg} therefore yields 
\begin{flalign}
\label{product1}
\begin{aligned}
F_{\lambda} \big( \varpi \b| \Xi, S \big) G_{\lambda}^c \big( \rho \b| \Xi, S \big) &= s^{- J} (q; q)_J^2 \displaystyle\prod_{i = 1}^J \displaystyle\frac{a^2 s - as \xi q^{i - 1} v}{(1 - s q^{i - 1} v) (1 - a \xi q^{i - 1} v)} \displaystyle\prod_{j = 1}^m \displaystyle\frac{(1 - a^2 q^{j - 1}) (q^{J - j} v - s)}{(1 - q^j) (a^2 s - a s \xi q^{J - j	} v) } \\
& \qquad \times \displaystyle\prod_{k = 1}^{J - m} \left( \displaystyle\frac{s^2 - s q^{k - 1} v}{1 - s q^{k - 1} v} \right)^{\lambda_k  - 1},  
\end{aligned}
\end{flalign}

\noindent where we have used the fact that $s^{|\lambda| - J} = s^{\sum_{j = 1}^J (\lambda_j - 1)}$. 

Now, using the facts that $q = s^{-2}$ and $\lambda_j = 1$ for all $j > J - m$, one can quickly verify that the last product $P = \prod_{k = 1}^{J - m} \big( (s^2 - s q^{k - 1} v) / (1 - s q^{k - 1} v) \big)^{\lambda_k - 1}$ in \eqref{product1} is equal to 
\begin{flalign}
\label{product2}
\begin{aligned}
P & = q^J \left( \displaystyle\frac{1 - s v}{1 - q^J sv} \right) \displaystyle\prod_{i = 1}^J \left( \displaystyle\frac{q^{-i} - sv}{1 - sv} \right)^{\lambda_i - \lambda_{i + 1}} \\
& = \displaystyle\prod_{i = 1}^J \displaystyle\frac{q^{-i} - sv}{1 - q^{-i}} \displaystyle\prod_{h = J - m + 1}^J  \displaystyle\frac{1 - q^{-h}}{q^{-h} - sv} \displaystyle\prod_{k = 1}^{J - m} \left( \displaystyle\frac{q^{-k} - sv}{1 - sv} \right)^{\lambda_k - \max \{ \lambda_{k + 1}, 1 \} - 1} \displaystyle\frac{1 - q^{-k}}{1 - sv}, 
\end{aligned}
\end{flalign} 

\noindent where we have set $\lambda_{J + 1} = 0$. 

Changing variables $j = J - h + 1$ and $g = J - m - k$ in \eqref{product2} and inserting into \eqref{product1} yields that $F_{\lambda} \big( \varpi \b| \Xi, S \big) G_{\lambda}^c \big( \rho \b| \Xi, S \big)$ is equal to 
\begin{flalign*}
C  \displaystyle\prod_{g = 0}^{J - m - 1} \left( \displaystyle\frac{q^{m + g - J} - sv}{1 - sv} \right)^{\lambda_{J - m - g} - \max \{ \lambda_{J - m - g + 1}, 1 \} - 1} \displaystyle\frac{1 - q^{m + g - J}}{1 - sv} \displaystyle\prod_{j = 1}^m \displaystyle\frac{(1 - a^2 q^{j - 1}) (1 - q^{j - J - 1})}{(1 - q^j) (a s^2 \xi v - a^2 q^{j - J - 1})}, 
\end{flalign*}

\noindent where $C = C(q, v, J, \Xi, S) = (q; q)_J (a \xi v; q)_J^{-1} \prod_{i = 1}^J (a \xi q^{i - 1} v - a^2) $ is independent of $\lambda$. 

Now, the proposition follows from the parametrization \eqref{stochasticparameters} and the fact that $\vartheta = q^{-J}$.  
\end{proof}

\subsubsection{Interpretation of Proposition \ref{fg} Through Vertex Weights} 

\label{PrincipalDoubleSided}

The goal of this section is to interpret Proposition \ref{fg} as a horizontal initial condition for the stochastic six-vertex model. As in Section \ref{DoubleSidedSpecialization}, we take the limit as $a$ tends to $0$ in that proposition. 

Then, for any $\lambda \in \Sign_J^+$, $\mathscr{M}^{\mathfrak{S}} (\lambda) = 0$ if either $\lambda_J = 0$ or $m_i (\lambda) > 1$ for some $i > 1$. Otherwise, denoting $m = m_1(\lambda)$, Proposition \ref{fg} implies that $\mathscr{M}^{\mathfrak{S}} (\lambda)$ is proportional to 
\begin{flalign}
\label{msa0}
\left( \displaystyle\frac{\kappa \beta_2}{\beta_1} \right)^m \displaystyle\frac{(\vartheta; q)_m}{(q; q)_m} \displaystyle\prod_{k = 0}^{J - m - 1} \big( b_2 - q^{k + m} \vartheta b_2 \big) \big( 1 - b_2 + q^{k + m} \vartheta b_2\big)^{\lambda_{J - m - k} - \max \{ \lambda_{J - m - k + 1}, 1 \} - 1}. 
\end{flalign}

As indicated at the end of Section \ref{DoubleSidedSpecialization}, \eqref{msa0} can be viewed as the probability distribution of a stochastic higher spin vertex model with $J$ rows and spectral parameters equal to $v, qv, \ldots , q^{J - 1} v$. However, let us explain how \eqref{msa0} also has a different interpretation as a single-row vertex model, in which $J$ directed paths enter through the $y$-axis from the point $(0, 1)$ and then turn right and up; see Figure \ref{singlerowdouble}. This is an example of what is known as \emph{row fusion} in vertex models; see Section 5 of \cite{HSVMRSF} for more information. 

As in Section \ref{PathEnsembles}, each vertex $(x, 1)$ on the line $x = 1$ is associated with four integers $i_1$, $i_2$, $j_1$, and $j_2$, which count the number of incoming and outgoing vertical and horizontal paths through $(x, 1)$. In the model we consider, no paths will emanate from the $x$-axis, so $i_1$ is always $0$. Thus, we abbreviate $i = i_2$; we also denote $j = J - j_1$. Arrow conservation implies that $j_2 = j_1 - i = J - j - i$, so the value of $j_2$ (and hence the entire arrow configuration) at any vertex $(x, 1)$ is determined by the values of $i$ and $j$ at $(x, 1)$. Furthermore, we must always have $j = 0$ at the vertex $(1, 1)$; thus, the arrow configuration at $(1, 1)$ is determined by only the integer $i$, which we relabel $m$. 

Now, to each vertex $(x, 1)$, with $x > 1$, we will associate a vertex weight $W (j; i)$. The case $x = 1$ will be distinguished, so at $(1, 1)$ we will associate a different weight $\widehat{W} (m)$. These weights are given explicitly by the following definition. 

\begin{definition}

\label{xvertexweights}

Fix $\vartheta, q, b_1, b_2 \in \mathbb{C}$ such that $b_1, b_2 \notin \{ 0, 1 \} $. For each $j \in \mathbb{Z}_{\ge 0}$, set  $W_{\vartheta, q, b_2} (j; i) = 0$ if $i \notin \{ 0, 1 \}$. Otherwise, define 
\begin{flalign*}
W_{\vartheta, q, b_2} (j; 0) = 1 - b_2 + q^j \vartheta b_2; \qquad W_{\vartheta, q, b_2} (j; 1) = b_2 - q^j \vartheta b_2.
\end{flalign*}

\noindent Denoting $\kappa$, $\beta_1$, and $\beta_2$ as in \eqref{stochasticparameters}, also define
\begin{flalign*}
\widehat{W}_{q, \vartheta, \kappa, b_1, b_2} (m) = \displaystyle\frac{(\kappa \beta_2 \beta_1^{-1}; q)_{\infty} }{(\vartheta \kappa \beta_2 \beta_1^{-1}; q)_{\infty}}\left( \displaystyle\frac{\kappa \beta_2}{\beta_1} \right)^m \displaystyle\frac{(\vartheta; q)_m}{(q; q)_m}. 
\end{flalign*}
\end{definition}

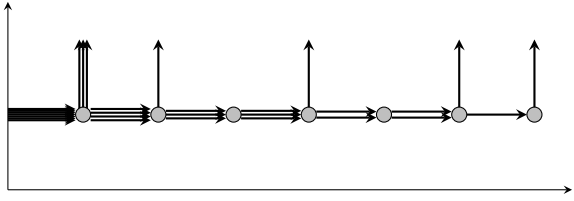
\begin{figure}

\begin{center}

\begin{tikzpicture}[
      >=stealth,
        scale=.5
			]
			
			\draw[->,black] (0, 0) --(15, 0);
			\draw[->,black] (0, 0) --(0, 5);
			
			\draw[->,black, thick] (0, 1.85) -- (1.8, 1.85);
			\draw[->,black, thick] (0, 1.9) -- (1.8, 1.9);
			\draw[->,black, thick] (0, 1.95) -- (1.8, 1.95);
			\draw[->,black, thick] (0, 2) -- (1.8, 2);
			\draw[->,black, thick] (0, 2.05) -- (1.8, 2.05);
			\draw[->,black, thick] (0, 2.1) -- (1.8, 2.1);
			\draw[->,black, thick] (0, 2.15) -- (1.8, 2.15);
			
			\draw[->,black, thick] (2.2, 1.85) -- (3.8, 1.85);
			\draw[->,black, thick] (2.2, 1.95) -- (3.8, 1.95);
			\draw[->,black, thick] (2.2, 2.05) -- (3.8, 2.05);
			\draw[->,black, thick] (2.2, 2.15) -- (3.8, 2.15);
			\draw[->,black, thick] (1.9, 2) -- (1.9, 4);
			\draw[->,black, thick] (2, 2) -- (2, 4);
			\draw[->,black, thick] (2.1, 2) -- (2.1, 4);
			
			\draw[->,black, thick] (4.2, 1.9) -- (5.8, 1.9);
			\draw[->,black, thick] (4.2, 2) -- (5.8, 2);
			\draw[->,black, thick] (4.2, 2.1) -- (5.8, 2.1);
			\draw[->,black, thick] (4, 2) -- (4, 4);
			
			\draw[->,black, thick] (6.2, 1.9) -- (7.8, 1.9);
			\draw[->,black, thick] (6.2, 2) -- (7.8, 2);
			\draw[->,black, thick] (6.2, 2.1) -- (7.8, 2.1);
			
			\draw[->,black, thick] (8.2, 1.92) -- (9.8, 1.92);
			\draw[->,black, thick] (8.2, 2.08) -- (9.8, 2.08);
			\draw[->,black, thick] (8, 2) -- (8, 4);
			
			\draw[->,black, thick] (10.2, 1.92) -- (11.8, 1.92);
			\draw[->,black, thick] (10.2, 2.08) -- (11.8, 2.08);
			\draw[->,black, thick] (12, 2) -- (12, 4);
			
			\draw[->,black, thick] (12.2, 2) -- (13.8, 2);
			\draw[->,black, thick] (14, 2) -- (14, 4);

			\filldraw[fill=gray!50!white, draw=black] (2, 2) circle [radius=.2];
			\filldraw[fill=gray!50!white, draw=black] (4, 2) circle [radius=.2];
			\filldraw[fill=gray!50!white, draw=black] (6, 2) circle [radius=.2];
			\filldraw[fill=gray!50!white, draw=black] (8, 2) circle [radius=.2];
			\filldraw[fill=gray!50!white, draw=black] (10, 2) circle [radius=.2];
			\filldraw[fill=gray!50!white, draw=black] (12, 2) circle [radius=.2];
			\filldraw[fill=gray!50!white, draw=black] (14, 2) circle [radius=.2];

\end{tikzpicture}

\end{center}

\caption{\label{singlerowdouble} Depicted above is a path ensemble to which $\mathcal{W}$ might give non-zero weight. The signature corresponding to the above path ensemble is $(7, 6, 4, 2, 1, 1, 1)$} 
\end{figure}

\begin{rem}

\label{w1} 

We have that $\sum_{i = 0}^{\infty} W_{\vartheta, q, b_2} (j; i) = 1$. Furthermore, the $q$-binomial identity implies that $\sum_{m = 0}^{\infty} \widehat{W}_{q, \vartheta, \kappa, b_1, b_2} (m) = 1$ if $|\kappa \beta_2| < |\beta_1|$. 
\end{rem}

\begin{rem}

\label{wj}

If $\vartheta = q^{-J}$, then $W_{\vartheta, q, b_2} (J; 1) = 0$. 

\end{rem}

Since our model is only on one row, there is a correspondence between path ensembles and signatures; see Figure \ref{singlerowdouble}. Specifically, a path ensemble corresponds to a signature $\lambda \in \Sign^+$ if and only if the number of outgoing vertical arrows through $(x, 1)$ is equal to $m_x (\lambda)$, for each $x \ge 1$. Thus, we can assign weights to signatures by taking an appropriate product of the vertex weights $W$ and $\widehat{W}$; this yields $\mathcal{W}$ defined below. It will also be useful for later to define the weight $\overline{\mathcal{W}}^{(M)} (\lambda)$ that ignores the weight $\widehat{W}$ at $(1, 1)$ and increments all values of $j$ by $M$. 

\begin{definition}

\label{wlambda}

Fix $x \in \mathbb{Z}_{> 0}$ and $\vartheta, q, \kappa, b_1, b_2 \in \mathbb{C}$. Let $\lambda \in \Sign^+$. Define $\mathcal{W} (\lambda) = \mathcal{W}_x (\lambda) = \mathcal{W}_{x; \vartheta, q, \kappa, b_1, b_2} (\lambda)$ by 
\begin{flalign}
\label{w} 
\mathcal{W} (\lambda) = \textbf{1}_{0 \notin \lambda} \textbf{1}_{\lambda_1 \le x} \widehat{W}_{q, \vartheta, \kappa, b_1, b_2} \big( m_1 (\lambda) \big) \displaystyle\prod_{j = 2}^x W_{\vartheta, q, b_2} \left( \displaystyle\sum_{i = 1}^{j - 1} m_i (\lambda); m_j (\lambda) \right). 
\end{flalign}

\noindent Similarly, for each positive integer $M \ge 0$, define $\overline{\mathcal{W}}^{(M)} (\lambda) = \overline{\mathcal{W}_x}^{(M)} (\lambda) = \overline{\mathcal{W}_{x; \vartheta, q, \kappa, b_1, b_2}}^{(M)} (\lambda)$ by
\begin{flalign}
\label{wm}
\overline{\mathcal{W}}^{(M)} (\lambda) = \textbf{1}_{0 \notin \lambda} \textbf{1}_{\lambda_1 \le x} \displaystyle\prod_{j = 2}^x W_{\vartheta, q, b_2} \left( M + \displaystyle\sum_{i = 2}^{j - 1} m_i (\lambda); m_j (\lambda) \right). 
\end{flalign}
\end{definition}

\begin{rem}

\label{ww}

\noindent Observe in particular that $\mathcal{W} (\lambda) = \widehat{W}_{q, \vartheta, \kappa, b_1, b_2} \big( m_1 (\lambda) \big) \overline{\mathcal{W}}^{(m_1 (\lambda))} (\lambda)$. 

\end{rem}

\begin{rem}

\label{w12}

If $| \kappa \beta_2 | < | \beta_1 |$, then Remark \ref{w1} implies that $\sum_{\lambda \in \Sign^+} \mathcal{W}_x (\lambda) = 1$. 

\end{rem}

The following proposition indicates that the weight $\mathcal{W}$ is related to the formal measure $\mathscr{M}^{\mathfrak{S}}$. 

\begin{prop}

\label{wm2}

Fix $\delta_1, \delta_2, b_1 \in (0, 1)$ and $b_2 \in \mathbb{C} \setminus \{ 0, 1 \}$. Define $q$, $u$, $\kappa$, $s$, $\beta_1$, and $\beta_2$ as in \eqref{stochasticparameters}. Let $J \in \mathbb{Z}_{> 0}$, denote $\vartheta = q^{-J}$, and let $\lambda \in \bigcup_{k = J}^{\infty} \Sign_k^+$ be a signature with $\ell (\lambda) \ge J$. Let $x \in \mathbb{Z}$ satisfy $x > \lambda_1$. 

If $\ell (\lambda) > J$, then $\mathcal{W} (\lambda) = 0$. Otherwise, $\mathcal{W}_x (\lambda) = \mathscr{M}^{\mathfrak{S}} (\lambda)$, where $\mathscr{M}^{\mathfrak{S}}$ is given by Definition \ref{ms}, with the parameter $a$ set to $0$. 

\end{prop}

\begin{proof}
If $\ell (\lambda) > J$, then at least one of several possibilities must occur. Either $0 \in \lambda$; $m_1 (\lambda) > J$; there exists some integer $j > 1$ such that $m_j (\lambda) > 1$; or the weight $W_{\vartheta, q, b_2} (J; 1)$ appears in the product on the right side of \eqref{w}. In each case, we find that $\mathcal{W} (\lambda) = 0$; this establishes the first part of the proposition.

Now suppose $\lambda \in \Sign_J^+$; we would like to show that $\mathscr{M}^{\mathfrak{S}} (\lambda) = \mathcal{W} (\lambda)$. The result holds if $0 \in \lambda$ since then $\mathscr{M}^{\mathfrak{S}} (\lambda) = 0 = \mathcal{W} (\lambda)$. 

So, assume $0 \notin \lambda$. Then $\mathscr{M}^{\mathfrak{S}} (\lambda)$ is proportional to \eqref{msa0}. The product $\big( \kappa \beta_2 \beta_1^{-1} \big)^m (\vartheta; q)_m (q; q)_m^{-1}$ in \eqref{msa0} is proportional to the weight $\widehat{W}_{q, \vartheta, \kappa, b_1, b_2} (m)$. Since we furthermore have that 
\begin{flalign*}
\displaystyle\prod_{j = \lambda_{J - m - k} + 1}^{\lambda_{J - m - k - 1}} W_{\vartheta, q, b_2} & \left( \displaystyle\sum_{i = 1}^j m_i (\lambda); m_j (\lambda) \right) \\
& = (b_2 - q^{k + m} \vartheta b_2) (1 - b_2 + q^{k + m} \vartheta b_2)^{\lambda_{J - m + k - 1} - \max \{ \lambda_{J - m + k}, 1 \} - 1}, 
\end{flalign*}

\noindent it follows that the right side of \eqref{w} is proportional to \eqref{msa0}. Hence, $\mathcal{W}_x (\lambda)$ is proportional to $\mathscr{M}^{\mathfrak{S}} (\lambda)$. We would like to show that they are in fact equal. 

To that end, first assume that $b_2$ is real and sufficiently small so that $-q^{2J} < b_2 < 0$. Then, $b_2 - q^{k + m} \vartheta b_2, 1 + q^{k + m} \vartheta b_2 - b_2 \in (0, 1)$ for any integer $k \in [0, J - m - 1]$. This, the fact that $(\vartheta; q)_k = 0$ for $k > J$, and Proposition \ref{fg} (or the identity \eqref{msa0}) together imply that $\sum_{\nu \in \Sign_J^+} \big| F_{\nu} \big( \varpi \b| \Xi, S \big) G_{\nu}^c \big( \rho \b| \Xi^{-1}, S \big) \big| < \infty$. Hence, we deduce from the Cauchy identity for the $F$ and $G$ symmetric functions (see Corollary 4.13 of \cite{HSVMRSF}) that $\sum_{\nu \in \Sign_J^+} \mathscr{M}^{\mathfrak{S}} (\nu) = 1$. 

Now, observe that $\mathcal{W}_x (\lambda) = \mathcal{W}_y (\lambda)$ for any $y > x$, since $x > \lambda_1$, $J > \ell (\lambda)$, and $W_{\vartheta, q, b_2} (J; 0) = 1$. Thus, the limit $\mathcal{W}_{\infty} (\nu) = \lim_{y \rightarrow \infty} \mathcal{W}_y (\nu)$ exists, for each $\nu \in \Sign_J^+$. Furthermore, since $W_{\vartheta, q, b_2} (a; b) \in [0, 1)$ for any $a < J$ (and is uniformly bounded away from $1$), we find that $\mathcal{W}_{\infty} (\nu) = \lim_{y \rightarrow \infty} \mathcal{W}_y (\nu) = 0$ for any $\nu$ of length less than $J$. 

By Remark \ref{w12}, we have that $\sum_{\nu \in \Sign^+} \mathcal{W}_x (\nu) = 1$. Taking the limit as $x$ tends to $\infty$ (justified by the exponential decay of $\mathcal{W}_x (\nu)$ in $|\nu|$, which holds since $b_2 - q^{k + m} \vartheta b_2, 1 + q^{k + m} \vartheta b_2 - b_2 \in (0, 1)$), we deduce that $\sum_{\nu \in \Sign_J^+} \mathcal{W}_{\infty} (\nu) = \sum_{\nu \in \Sign^+} \mathcal{W}_{\infty} (\nu) = 1$. 

Now, $\mathcal{W}_{\infty} (\nu)$ is proportional to $\mathscr{M}^{\mathfrak{S}} (\nu)$ for each $\nu \in \Sign_J^+$ since $\mathcal{W}_x (\nu)$ is proportional to $\mathscr{M}^{\mathfrak{S}} (\nu)$ if $x \ge \nu_1$. Thus $\sum_{\nu \in \Sign_J^+} \mathscr{M}^{\mathfrak{S}} (\nu) = 1 = \sum_{\nu \in \Sign_J^+} \mathcal{W}_{\infty} (\nu)$ implies that $\mathcal{W}_x (\lambda) = \mathcal{W}_{\infty} (\lambda) = \mathscr{M}^{\mathfrak{S}} (\lambda)$. 

This establishes the proposition when $b_2 \in (-q^{2J}, 0)$; hence, it also holds for any $b_2 \in \mathbb{C}$ by uniqueness of analytic continuation. 
\end{proof} 

In general, the weights $\mathcal{W}_{x; \vartheta, q, \kappa, b_1, b_2} (\lambda)$ are complex. However, for certain choices of the parameters $q$, $\vartheta$, $\kappa$, $b_1$, and $b_2$, they can be made to be real and positive; this occurs, for example, if $b_1, b_2, q, \vartheta \in (0, 1)$, $\kappa > 1$, and $\beta_1 > \kappa \beta_2$. In this case, $\mathcal{W}_x (\lambda)$ is a probability measure on $\Sign_J^+$ that has the following probabilistic interpretation. 

Set $m_1 (\lambda) = m$ with probability $\widehat{W}_{q, \vartheta, \kappa, b_1, b_2} (m)$, for each non-negative integer $m$; by Remark \ref{w1}, these probabilities sum to $1$. Now, suppose that $j \in [2, x]$ is a positive integer. Let $\sum_{i = 2}^{j - 1} m_i (\lambda) = k$, meaning that there are $k + m$ elements of $\lambda$ less than $j$. Append $j$ to $\lambda$ (with multiplicity $1$) with probability $W_{\vartheta, q, b_2} (k + m; 1)$, and do not append $j$ to $\lambda$ with probability $W_{\vartheta, q, b_2} (k + m; 0)$; again, these probabilities sum to $1$. No integer $j > x$ is appended to $\lambda$. 

This procedure samples a random signature $\lambda$. For example, if $\vartheta = 0$, then $W_{\vartheta, q, b_2} (k + m; 1) = b_2$, and $W_{\vartheta, q, b_2} (k + m; 0) = 1 - b_2$, meaning that integers in the interval $[2, x]$ are randomly and independently appended to $\lambda$, each with probability $b_2$. Recall that one can associate a particle configuration with $\lambda$ by placing $m_j (\lambda)$ particles at position $j \ge 0$, for each non-negative integer $j$. In the case $\vartheta = 0$, this configuration randomly and independently places a particle at positions between $2$ and $x$, inclusive, each with probability $b_2$; it also places $m$ particles at position $1$ under some type of $q$-binomial distribution given explicitly by the weight $\widehat{W}$. 

Taking the limit as $x$ tends to $\infty$ and ignoring the particles at position $1$ (as depicted by the shifted model on the right side of Figure \ref{shift2generalizedvertex}) gives rise to horizontal Bernoulli initial data with parameter $b_2$. Thus if we could set $\vartheta = 0$, then the measure $\mathscr{M}^{\mathfrak{S}} (\lambda)$ would define Bernoulli-type initial data on the half line $\mathbb{Z}_{\ge 1}$. 

Unfortunately, the fact that $\vartheta = q^{-J}$ implies that $\vartheta$ is non-zero. Still, at least from a heuristic standpoint, it will be useful for us to view $\vartheta$ as an arbitrary complex number. In fact, we will later (in Section \ref{mucomplex}) justify this heuristic through an analytic continuation.

\subsection{Identities for the Current of the Stochastic Six-Vertex Model} 

\label{ContoursCurrent}

The goal of this section is to establish Corollary \ref{heightgeneral}, which is a multi-fold contour integral identity for the current of the stochastic six-vertex model with $b_1$-Bernoulli initial data on the $y$-axis and whose initial data on the $x$-axis is weighted by $\mathcal{W}$ (from Definition \ref{wlambda}). This will follow from combining the discussion in the previous section with the results of \cite{HSVMRSF}, which provide contour integral identities for $q$-moments of the height function of $\mathscr{M}$.

\subsubsection{Observables for the Stochastic Six-Vertex Model}

\label{ContourMoments}

Our goal in this section is to state Corollary \ref{secondheightgeneral}, which is an identity for $q$-moments of the height function of the stochastic higher spin vertex model in a reasonably general setting. However, we first require some notation. 

In what follows, we will have $t, k \in \mathbb{Z}_{> 0}$; $q \in (0, 1)$; $v \in \mathbb{C}$; $U = (u_1, u_2, \ldots , u_n) \subset \mathbb{C}$; $\Xi = (\xi_1, \xi_2, \ldots ) \subset \mathbb{C}$; and $S = (s_1, s_2, \ldots ) \subset \mathbb{C}$. The following provides a restriction on these parameters so that the contours we consider exist. 

\begin{definition}

\label{spacedparameters}

We call the quadruple $(v, U; \Xi, S)$ \emph{suitably spaced} if the following three conditions are satisfied. 

\begin{itemize}

\item{The complex number $v$ is sufficiently close to $0$ so that $|q^{-k - 1} v| < \min_{i \ge 1} \{ |u_i|^{-1}, |s_i \xi_i| \}$.} 

\item{The elements of $U$ are real and sufficiently close together so that they all have the same sign and $q \max_{1 \le i \le t} |u_i| < \min_{1 \le i \le t} |u_i|$.}

\item{No number of the form $s_i \xi_i$ is contained in the interval $\big( \min_{1 \le i \le t} u_i^{-1}, \max_{1 \le i \le t} u_i^{-1} \big)$.}

\end{itemize}

\end{definition}

The following result is Corollary 9.9 of \cite{HSVMRSF}. Here, for any $\lambda \in \Sign^+$, $\mathfrak{h}_{\lambda} (x) = e_{x - 1} (\lambda)$ denotes the number of indices $i$ for which $\lambda_i \ge x$. Furthermore, for any $r \in \mathbb{R}_{> 0}$ and contour $\mathcal{C} \subset \mathbb{C}$, let $r \mathcal{C}$ denote the image of $\mathcal{C}$ upon multiplication by $r$.

\begin{prop}[{\cite[Corollary 9.9]{HSVMRSF}}]
\label{firstheightgeneral} 

Fix parameters $k, x, t, J \in \mathbb{Z}_{> 0}$; $q \in (0, 1)$; $u_0 > 0$; $\overline{U}  = (u_1, u_2, \ldots , u_t) \subset \mathbb{R}_{> 0}$; $S = (s_1, s_2, \ldots ) \subset (-1, 0)$; and $\Xi = (\xi_0, \xi_1, \ldots ) \subset \mathbb{R}_{> 0}$. Assume that $(0, \{ u_0 \} \cup U; \Xi, S)$ is suitably spaced, and moreover that $\min_{i \ge 1} \xi_i^{-1} |s_i| > q \max_{i \ge 1} \xi_i^{-1} |s_i|$ and $\min_{i \ge 1} \xi_i^{-1} |s_i|^{-1} > q \max_{i \ge 1} \xi_i^{-1} |s_i|$. 

Under the measure $\mathscr{M}$ (from Definition \ref{signaturemeasures}) whose parameter sets are $\Xi = (\xi_1, \xi_2, \ldots )$, $S = (s_1, s_2, \ldots )$, and $U = \{ u_0, q u_0, \ldots , q^{J - 1} u_0 \} \cup \overline{U}$, we have that 
\begin{flalign}
\label{firstheightgeneralequation}
\begin{aligned}
\mathbb{E}_{\mathscr{M}} [q^{k \mathfrak{h}_{\lambda} (x)}] & = \displaystyle\frac{q^{\binom{k}{2}}}{(2 \pi \textbf{\emph{i}})^k} \displaystyle\oint \cdots \displaystyle\oint \displaystyle\prod_{i = 1}^k \displaystyle\frac{1 - q^J u_0 w_i}{1 - u_0 w_i} \left( \displaystyle\prod_{j = 1}^{x - 1} \displaystyle\frac{s_j \xi_j - s_j^2 w_i}{s_j \xi_j - w_i} \displaystyle\prod_{j = 1}^t \displaystyle\frac{1 - q u_j w_i}{1 - u_j w_i} \right)  \\
& \qquad \qquad \qquad  \times \displaystyle\prod_{1 \le i < j \le k} \displaystyle\frac{w_i - w_j}{w_i - q w_j} \displaystyle\prod_{i = 1}^k \displaystyle\frac{d w_i}{w_i}, 
\end{aligned}
\end{flalign}

\noindent where each $w_i$ is integrated along a contour $\Gamma_i (U; \Xi, S) \subset \mathbb{C}$, which is the union of two positively oriented circles, $C_i^{(1)}$ and $C_i^{(2)}$, satisfying the following three properties. First, for each $i \ge 1$, $C_i^{(2)}$ is a positively oriented circle centered at $0$ that contains $q^{-1} C_{i - 1}^{(2)}$ (where we set $C_0^{(2)} = \{ 0 \}$), but leaves outside each $q u_j^{-1}$ and $s_j \xi_j$. Second, for each $i \ge 1$, $C_i^{(1)}$ is a circle that contains $\bigcup_{j = 0}^k \{ u_j^{-1}, q^{-1} u_j^{-1}, \ldots , q^{1 - i} u_j^{-1} \}$, but leaves outside each $s_j \xi_j$ and each $q u_j^{-1}$. Third, no $C_i^{(1)}$ intersects the interior of $q^{-1} C_k^{(2)}$. 

\end{prop}

See the right side of Figure \ref{firstcontoursfigure} for examples of the contours $\Gamma$. 

The conditions on $U$, $\Xi$, and $S$ given in Proposition \ref{firstheightgeneral} do not quite fit our requirements; as we will see later, it will be convenient to take $u_0$ 	very large, while keeping $u_1, u_2, \ldots , u_t$ fixed. It is possible to modify Proposition \ref{firstheightgeneral} to also hold in that setting. However, to that end, we require different contours, given by the following definition; see the left side of Figure \ref{firstcontoursfigure} for examples. 

\begin{definition}

\label{firstcontours}

Suppose that $(v, U; \Xi, S)$ is suitably spaced. We define the set of $k$ positively oriented contours $\gamma_1 (U; \Xi, S), \gamma_2 (U; \Xi, S), \ldots , \gamma_k (U; \Xi, S)$ as follows. Each contour $\gamma_i = \gamma_i (U; \Xi, S)$ will be the disjoint union of two positively oriented circles $c_i^{(1)}$ and $c_i^{(2)}$. 

Here, $c^{(1)} = c_1^{(1)} = c_2^{(1)} = \cdots = c_k^{(1)}$ are all the same circle that contain each of the $u_i^{-1}$ and leave outside $\xi_1 s_1, \xi_2 s_2, \ldots $. We also assume that $c^{(1)}$ is sufficiently small so that its interior is disjoint with the image of its interior under multiplication by $q$. 

The circles $c_1^{(2)}, c_2^{(2)}, \ldots , c_k^{(2)}$ will be centered at $0$ and sufficiently small such that the following two properties hold. 

\begin{itemize}

\item{The circles are \emph{$q$-nested}, meaning that, for each $i \in [1, k - 1]$, the circle $c_{i + 1}^{(2)}$ strictly contains the circle $q^{-1} c_i^{(2)}$ in its interior.} 

\item{All circles contain $v$ but are sufficiently small so that the interior of $c_k^{(2)}$ is disjoint with the interior of $q c^{(1)}$ and so that the interior of $c_k^{(2)}$ does not contain $s_1 \xi_1, s_2 \xi_2, \ldots $. }

\end{itemize}

The fact that $(v, U; \Xi, S)$ is suitably spaced guarantees the existence of these contours. 

\end{definition}

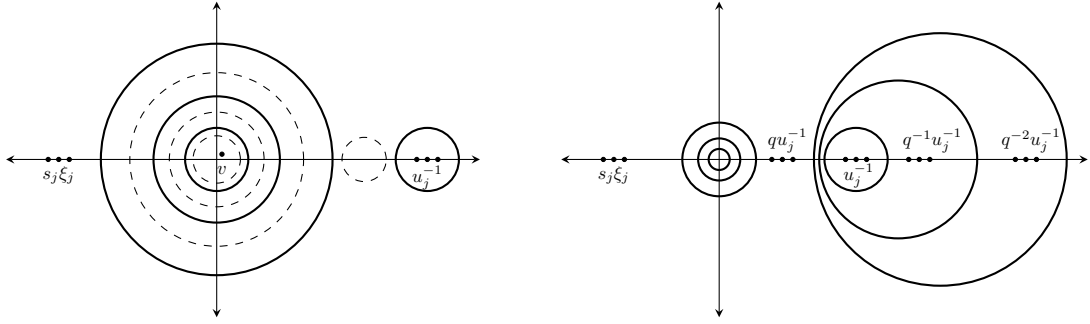
\begin{figure}

\begin{minipage}{0.45\linewidth}

\centering 

\begin{tikzpicture}[
      >=stealth,
			scale = .7
			]

			\draw[<->, black] (-4, 0) -- (5, 0); 
			\draw[<->, black] (0, -3) -- (0, 3);

			\draw[black, fill] (.1, .1) circle [radius=.04] node [black, below=2, scale = .7] {$v$};
			\draw[black, fill] (4, 0) circle [radius=.04] node [black, below=0, scale = .7] {$u_j^{-1}$};
			\draw[black, fill] (4.2, 0) circle [radius=.04];
			\draw[black, fill] (3.8, 0) circle [radius=.04];
			\draw[black, fill] (-3, 0) circle [radius=.04] node [black, below=0, scale = .7] {$s_j \xi_j$};
			\draw[black, fill] (-3.2, 0) circle [radius=.04];
			\draw[black, fill] (-2.8, 0) circle [radius=.04];

			\draw[black, thick] (0, 0) circle [radius=.6]; 
			\draw[black, thick] (0, 0) circle [radius=1.2]; 
			\draw[black, thick] (0, 0) circle [radius=2.2];
			
			\draw[black, dashed] (0, 0) circle [radius=.45]; 
			\draw[black, dashed] (0, 0) circle [radius=.9];
			\draw[black, dashed] (0, 0) circle [radius=1.65]; 
			 
			\draw[black, thick] (4, 0) circle [radius=.6];
			\draw[black, dashed] (2.8, 0) circle [radius=.42];

\end{tikzpicture}

\end{minipage}
\qquad
\begin{minipage}{0.45\linewidth}

\centering

\begin{tikzpicture}[
      >=stealth,
			scale = .7	
			]

			\draw[<->, black] (-3, 0) -- (7	, 0); 
			\draw[<->, black] (0, -3) -- (0, 3);

			\draw[black, fill] (2.6, 0) circle [radius=.04];
			\draw[black, fill] (2.8, 0) circle [radius=.04] node [black, left = 3, below=-1, scale = .7] {$u_j^{-1}$};
			\draw[black, fill] (2.4, 0) circle [radius=.04];
			
			\draw[black, fill] (1, 0) circle [radius=.04];
			\draw[black, fill] (1.2, 0) circle [radius=.04] node [black, right = 2, above = 0, scale = .7] {$q u_j^{-1}$};
			\draw[black, fill] (1.4, 0) circle [radius=.04];
			
			\draw[black, fill] (-2, 0) circle [radius=.04] node [black, below=0, scale = .7] {$s_j \xi_j$};
			\draw[black, fill] (-2.2, 0) circle [radius=.04];
			\draw[black, fill] (-1.8, 0) circle [radius=.04]; 
			\draw[black, fill] (3.8, 0) circle [radius=.04];
			\draw[black, fill] (4, 0) circle [radius=.04];
			\draw[black, fill] (3.6, 0) circle [radius=.04] node [black, right = 9, above=0, scale = .7] {$q^{-1} u_j^{-1}$};
			\draw[black, fill] (6.025, 0) circle [radius=.04] node [black, left = 2, above=0, scale = .7] {$q^{-2} u_j^{-1}$};
			\draw[black, fill] (5.825, 0) circle [radius=.04];
			\draw[black, fill] (5.625, 0) circle [radius=.04];

			\draw[black, thick] (0, 0) circle [radius=.2]; 
			\draw[black, thick] (0, 0) circle [radius=.4]; 
			\draw[black, thick] (0, 0) circle [radius=.7];
			 
			\draw[black, thick] (2.6, 0) circle [radius=.6];
			\draw[black, thick] (3.4, 0) circle [radius=1.5];
			\draw[black, thick] (4.2, 0) circle [radius=2.4];

\end{tikzpicture}

\end{minipage}

\caption{\label{firstcontoursfigure} Shown above and to the left in solid are possible contours $\gamma_1 (v, U; \Xi, S)$, $\gamma_2 (v, U; \Xi, S)$, and $\gamma_3 (v, U; \Xi, S)$; the dashed curves are the images of the solid curves under multiplication by $q$. Shown above and to the right are possible contours $\Gamma_1 (U; \Xi, S)$, $\Gamma_2 (U; \Xi, S)$, and $\Gamma_3 (U; \Xi, S)$.}
\end{figure}

The following corollary is an adaptation of Proposition \ref{firstheightgeneral} that imposes different restrictions on $U$, $\Xi$, and $S$. 	

\begin{cor}
\label{secondheightgeneral} 

Fix $k, x, t, J \in \mathbb{Z}_{> 0}$; $q \in (0, 1)$, $u_0 \in \mathbb{C}$ non-zero; $\overline{U} = (u_1, u_2, \ldots , u_t) \subset \mathbb{R}_{> 0}$, $S = (s_0, s_1, \ldots ) \subset \mathbb{R}$; and $\Xi = (\xi_0, \xi_1, \ldots ) \subset \mathbb{R}$. Suppose that $(u_0^{-1}, U; \Xi, S)$ is suitably spaced and moreover that 
\begin{flalign}
\label{convergenceprobabilities}
\sup_{i, j \ge 1} \left| \displaystyle\frac{s_j \xi_j u_i - s_j^2}{1 - s_j \xi_j u_i} \right| < 1; \qquad \sup_{j \ge 1} \left| \displaystyle\frac{s_j \xi_j u_0 - s_j^2}{1 - s_j \xi_j u_0} \right| < 1. 
\end{flalign}

\noindent Under the measure $\mathscr{M}$ (from Definition \ref{signaturemeasures}) whose parameter sets are $\Xi = (\xi_1, \xi_2, \ldots )$, $S = (s_1, s_2, \ldots )$, and $U = \{ u_0, qu_0, q^2 u_0, \ldots , q^{J - 1} u_0 \} \cup \overline{U}$, we have that $\mathbb{E}_{\mathscr{M}} \big[ q^{k \mathfrak{h}_{\lambda} (x)} \big]$ is equal to the right side of \eqref{firstheightgeneralequation}, where now each $w_i$ is integrated along the contour $\gamma_i (u_0^{-1}, U; \Xi, S)$. 
\end{cor}

\begin{proof}[Proof (Outline)]

This corollary is essentially an analytic continuation of Proposition \ref{firstheightgeneral}, so we only outline its proof. 

Denote the right side of \eqref{firstheightgeneralequation} by $I_1 (u_0) = I_1 (u_0; q, \overline{U}, S, \Xi)$ when all $w_i$ are integrated along $\Gamma_i$, and denote it by $I_2 (u_0) = I_2 (u_0; q, \overline{U}, S, \Xi)$ when all $w_i$ are integrated along $\gamma_i$. By integrating the right side of \eqref{firstheightgeneralequation} over $w_1, w_2, \ldots , w_k$ (in that order), we deduce that both $I_1$ and $I_2$ are rational functions in all parameters $q$, $u_0$, $\overline{U}$, $S$, and $\Xi$, since they can be written as finite sums of residues (each of which is a rational function). 

In particular, $I_2$ can be evaluated as a sum of residues corresponding to the poles $w_i = u_j^{-1}$, $w_i = 0$, or $w_i = q^{-r_i} u_0^{-1}$ for some integers $r_i \in [0, k - 1]$. If $\mathcal{A}$ is an assignment of $w_1, w_2, \ldots , w_k$ to the poles $u_j^{-1}$, $0$, and $u_0^{-1}, q^{-1} u_0^{-1}, \ldots , q^{1 - k} u_0$, then let $R_{\mathcal{A}} (u_0) = R_{\mathcal{A}} (u_0; q, \overline{U}, \Xi, S)$ denote the corresponding residue. Letting $\mathfrak{A}$ denote the set of all possible such assignments, we deduce that $I_2 (u_0) = \sum_{\mathcal{A} \in \mathfrak{A}} R_{\mathcal{A}} (u_0)$. 

Similarly, $I_1$ can be evaluated as a sum of residues corresponding to the poles where $w_i = 0$ or $w_i = q^{-r_i} u_j^{-1}$ for some $j \in [0, t]$ and $r_i \in [0, k - 1]$. If $\mathcal{A}$ is an assignment of $w_1, w_2, \ldots , w_k$ to the poles $0$ and $u_j^{-1}, q^{-1} u_j^{-1}, \ldots , q^{1 - k} u_j^{-1}$, then let $\overline{R}_{\mathcal{A}} (u) = \overline{R}_{\mathcal{A}} (u_0; q, \overline{U}, \Xi, S)$ denote the corresponding residue. Letting $\mathfrak{A}'$ denote the set of all possible such assignments, we deduce that $I_1 (u_0) = \sum_{\mathcal{A} \in \mathfrak{A}'} \overline{R}_{\mathcal{A}} (u_0)$. 

Observe that $\mathfrak{A} \subseteq \mathfrak{A}'$, but that $\mathfrak{A}'$ is larger if $k \ge 2$ (for instance, an element of $\mathfrak{A}'$ can assign $w_2$ to $q^{-1} u_1^{-1}$, which would not be permitted by any element of $\mathfrak{A}$). However, it was shown in the proof of the second part of Theorem 8.13 of \cite{HSVMRSF} that $R_{\mathcal{A}} (u_0) = 0$ if $\mathcal{A} \in \mathfrak{A}' \setminus \mathfrak{A}$. Hence, $I_1 (u_0) = \sum_{\mathcal{A} \in \mathfrak{A}} \overline{R}_{\mathcal{A}} (u_0)$. Since $I_1$ and $I_2$ are integrals of the same integrand, it follows that $R_{\mathcal{A}} (z) = \overline{R}_{\mathcal{A}} (z)$ as rational functions in $z$. Therefore, $\sum_{\mathcal{A} \in \mathfrak{A}} R_{\mathcal{A}} (u_0) = \sum_{\mathcal{A} \in \mathfrak{A}'} \overline{R}_{\mathcal{A}} (u_0)$, for each $u_0 \in \mathbb{C}$. 

Now the corollary can be quickly deduced from combining a fairly standard (see, for instance, Section 10 of \cite{HSVMRSF}) analytic continuation argument with Lemma 9.1 of \cite{HSVMRSF}, which states that, under the assumption \eqref{convergenceprobabilities}, $\mathbb{E}_{\mathscr{M}} \big[ q^{k \mathfrak{h}_{\lambda} (x)} \big]$ is a rational function in the parameters $q$, $U$, $S$, and $\Xi$. We omit further details. 
\end{proof}

\subsubsection{Contour Integral Identities for \texorpdfstring{$q$}{}-Moments of the Current}

\label{jpositive}

The goal of this section is to establish Corollary \ref{heightgeneral}. We will do this by first applying Corollary \ref{secondheightgeneral} to obtain multi-fold contour integral identities for $q$-moments of the measure $\mathscr{M}^{\mathfrak{S}}$. By Proposition \ref{wm2}, this yields contour integral identities for the $q$-moments for the height function of the stochastic six-vertex model whose initial data is weighted according to the formal measure $\mathcal{W}$, with parameter $\vartheta$ of the form $q^{-J}$ for some $J \in \mathbb{Z}_{> 0}$; this is given by Proposition \ref{heightsixvertex}. We will then reinterpret the $q$-moments of the height as $q$-moments of the current, so that Corollary \ref{heightgeneral} will follow from Proposition \ref{heightsixvertex}. 

Throughout this section, we have $k, x, t, J \in \mathbb{Z}_{> 0}$; $\delta_1, \delta_2, b_1 \in (0, 1)$ will satisfy $\delta_1 < \delta_2$; and $b_2 \in \mathbb{C} \setminus \{ 0, 1 \}$. Define $q$, $\kappa$, $s$, $u$, $\beta_1$, and $\beta_2$ as in \eqref{stochasticparameters}, and denote $\vartheta = q^{-J}$. 

The following proposition is a direct degeneration of Corollary \ref{secondheightgeneral} to the case $\mathscr{M} = \mathscr{M}^{\mathfrak{S}}$, under certain analytic constraints on the parameters. 

\begin{prop}

\label{qms}

Assume that $b_2 \in (-1, 0]$; $- \kappa^{-1} q^{k + 2} < \beta_2 s < 0$; and $|\kappa \beta_2| < q^{k + 2} \beta_1$. Then,
\begin{flalign}
\label{qmsequation}
\begin{aligned}
\mathbb{E}_{\mathscr{M}^{\mathfrak{S}}} \big[ q^{k (\mathfrak{h}_{\lambda} (x) - J)} \big] = \displaystyle\frac{q^{\binom{k}{2}}}{(2 \pi \textbf{\emph{i}})^k} \displaystyle\oint \cdots & \displaystyle\oint \displaystyle\prod_{i = 1}^k \left( \displaystyle\frac{1 + y_i}{1 + q^{-1} y_i} \right)^t \left( \displaystyle\frac{1 + q^{-1} \kappa^{-1} y_i}{1 + \kappa^{-1} y_i}  \right)^{x - 1} \displaystyle\prod_{1 \le i < j \le k} \displaystyle\frac{y_i - y_j}{y_i - q y_j}   \\ 
& \times \displaystyle\prod_{i = 1}^k \displaystyle\frac{1}{1 - q^{-1} \beta_1^{-1} y_i} \left( \displaystyle\frac{1 - \vartheta \kappa \beta_2 y_i^{-1}}{1 - \kappa \beta_2 y_i^{-1}} \right) \displaystyle\frac{d y_i}{y_i}, 
\end{aligned}
\end{flalign}

\noindent where the expectation on the left side is with respect to the measure $\mathscr{M}^{\mathfrak{S}}$, with the parameter $a$ set to $0$. In the above, each $y_i$ is integrated along the contour $\gamma_i (\kappa \beta_2, \widehat{U}; \widehat{\Xi}, \widehat{S})$, where the contours $\gamma_i$ are given by Definition \ref{firstcontours}. Here, $\widehat{U} = (q^{-1}, q^{-1}, \ldots , q^{-1})$, $\widehat{\Xi} = (1, 1, 1, \ldots )$, and $\widehat{S} = (q \beta_1, - \kappa, - \kappa, \ldots )$. 
\end{prop} 

\begin{proof}
We apply Corollary \ref{secondheightgeneral} with $v = -\beta_2^{-1} s^{-1}$, $u_1 = u_2 = \cdots = u_t = u$, $S = (s, a, s, s, s, \ldots )$, and $\Xi = (1, \xi, 1, 1, 1, \ldots )$, where $\xi = - \beta_1 a^{-1} u^{-1}$, and $a \in (-1, 0)$. The conditions of the corollary (that $(v^{-1}, U; \Xi, S)$ be suitably spaced and that the estimates \eqref{convergenceprobabilities} hold) can be quickly verified from our constraints $b_2 \in (-1, 0]$, $- \kappa^{-1} q^{k + 2} < \beta_2 s < 0$, and $|\kappa \beta_2| < q^{k + 2} \beta_1$, and from the definitions the parameters given by \eqref{stochasticparameters}. 

Now, we apply Corollary \ref{secondheightgeneral} to obtain that 
\begin{flalign}
\label{msintegral}
\begin{aligned}
\mathbb{E}_{\mathscr{M}^{\mathfrak{S}}} [q^{k \mathfrak{h}_{\lambda} (x)}] = \displaystyle\frac{q^{\binom{k}{2}}}{(2 \pi \textbf{i})^k} \displaystyle\oint \cdots & \displaystyle\oint \displaystyle\prod_{i = 1}^k \left( \displaystyle\frac{1 - q u w_i}{1 - u w_i} \right)^t \left( \displaystyle\frac{s - s^2 w_i}{s - w_i}  \right)^{x - 1} \displaystyle\prod_{1 \le i < j \le k} \displaystyle\frac{w_i - w_j}{w_i - q w_j}	 \\ 
& \times \displaystyle\prod_{i = 1}^k \displaystyle\frac{\beta_1 u^{-1} + a^2 w_i }{\beta_1 u^{-1} + w_i} \left( \displaystyle\frac{1 + \vartheta^{-1}  s^{-1} \beta_2^{-1} w_i}{1 + s^{-1} \beta_2^{-1} w_i} \right) \displaystyle\frac{d w_i}{w_i}, 
\end{aligned}
\end{flalign}

\noindent where each $w_i$ is integrated along the contour $\gamma_i (v^{-1}, U; \Xi, S)$. Thus, the proposition follows from setting $a = 0$ in \eqref{msintegral}, changing variables $y_i = - q u w_i$, and recalling that $\vartheta = q^{-J}$. 
\end{proof}

Using Proposition \ref{wm2} and Proposition \ref{qms}, we can establish the following proposition that gives $q$-moments of the height function for the stochastic six-vertex model whose initial data is weighted by $\mathcal{W}$ (recall Definition \ref{wlambda}). 

In what follows, let $\nu \in \Sign^+$ be a signature such that $m_i (\nu) \le 1$ for each $i > 1$. Let $\mathbb{E}^{(\nu)}$ denote the expectation under the stochastic six-vertex model $\mathcal{P} (\delta_1, \delta_2)$ (see Section \ref{StochasticVertex}) with the following initial data. Each vertex on the positive $y$-axis is an entrance site for a path independently and with probability $b_1$; furthermore, a vertex $(i, 0)$ on the positive $x$-axis is an entrance site for a path if and only if $m_{i + 1} (\nu) = 1$. Observe that the initial data for this model does not depend on $m_1 (\nu)$. 

Furthermore, if $\nu \in \Sign^+$ is a signature, then we define the signature $\nu (x, J) \in \bigcup_{n = J}^{\infty} \Sign_n^+$ as follows. If the length of $\ell (\nu) \ge J$ or if $\nu_1 > x$, then set $\nu (x, J) = \nu$. If $\nu_1 \le x$ and $\ell (\nu) < J$, then set $\nu (x, J) \in \Sign_J^+$ to be the signature that appends $x + 1, x + 2, \ldots , x + J - \ell (\nu)$ to $\nu$. 

In the following proposition, $\mathfrak{h}_t (x) = \mathfrak{h}_{\lambda} (x)$, where $\lambda \in \Sign_{J + t}^+$ denotes the configuration of particles (see Section \ref{InteractingParticles}) of the stochastic six-vertex model on the line $y = t$. 

\begin{prop}
\label{heightsixvertex}

Suppose that $|\kappa \beta_2 | < q^{k + 1} \min \{ \beta_1, 1 \} $. Then, the right side of \eqref{qmsequation} equals
\begin{flalign}
\label{heightsixvertexequation}
\displaystyle\sum_{\nu \in \Sign^+} \mathcal{W}_x (\nu) \mathbb{E}^{(\nu (x, J))} [q^{k (\mathfrak{h}_t (x) - J)} \big].
\end{flalign}

\end{prop}

\begin{proof}

We first prove this proposition under some analytic constraints on the parameters, and then we establish the result in general through analytic continuation. 

To that end, first assume that $b_2 \in (-1, 0]$, $- \kappa^{-1} q^{k + J + 2} < \beta_2 s < 0$, and $|\kappa \beta_2| < q^{k + 2} \beta_1$.  Then, Proposition \ref{qms} applies, so \eqref{qmsequation} holds. Furthermore, one can check under this restriction that $W_{\vartheta, q, b_2} (j; i) \in [0, 1]$ for each $i \in \{ 0, 1 \}$ and $j \in \{ 0, 1, \ldots , J \}$, that $W_{\vartheta, q, b_2} (J; 0) = 1$, and that $W_{\vartheta, q, b_2} (J; 1) = 0$ (recall Definition \ref{xvertexweights}). It follows that the the weights $\mathcal{W} (\lambda)$ are in the interval $[0, 1]$ and decay exponentially in $|\lambda|$; by Proposition \ref{wm2}, the same statement holds for $\mathscr{M}^{\mathfrak{S}} (\lambda)$. 

Now, let us recall the dynamics of the vertex model prescribed by $\mathscr{M}^{\mathfrak{S}}$. In view of the discussion at the end of Section \ref{DoubleSidedSpecialization}, $J$ particles are randomly configured on the line $y = J$, which then evolve according to a stochastic six-vertex model whose spin in the first column is deformed; this deformation gives rise to a vertical $b_1$-Bernoulli initial condition on the line $x = 1$. Furthermore, the probability that the particle configuration on the line $y = J$ is $\nu \in \Sign_J^+$ equals $\mathscr{M}^{\mathfrak{S}} (\nu) \in [0, 1]$. 

Hence,  
\begin{flalign}
\label{qh1}
\mathbb{E}_{\mathscr{M}^{\mathfrak{S}}} \big[ q^{k (\mathfrak{h}_{\lambda} (x) - J)} \big] = \displaystyle\sum_{\nu \in \Sign_J^+} \mathscr{M}^{\mathfrak{S}} (\nu) \mathbb{E}^{(\nu)} \big[ q^{k (\mathfrak{h}_t (x - 1) - J)} \big] = \displaystyle\sum_{\nu \in \Sign_J^+} \mathcal{W}_{\infty} (\nu) \mathbb{E}^{(\nu)} \big[ q^{k (\mathfrak{h}_t (x) - J)} \big], 
\end{flalign}

\noindent where in the last identity we recalled from the proof of Proposition \ref{wm2} that the limit $\mathcal{W}_{\infty} (\nu) = \lim_{y \rightarrow \infty} \mathcal{W}_y (\nu)$ exists, under the assumed restrictions on the parameters, and that $\mathcal{W}_{\infty} (\nu) = \mathcal{W}_y (\nu) = \mathscr{M}^{\mathfrak{S}} (\nu)$ if $y > \nu_1$. 

Observe that 
\begin{flalign}
\label{qh2}
\displaystyle\sum_{\nu \in \Sign_J^+} \mathcal{W}_{\infty} (\nu) \mathbb{E}^{(\nu)} \big[ q^{k (\mathfrak{h}_t (x) - J)} \big] = \displaystyle\sum_{\substack{\overline{\nu} \in \Sign^+ \\ \overline{\nu}_1 \le x}} \displaystyle\sum_{\substack{\nu \in \Sign_J^+ \\ \nu^{(x)} = \overline{\nu}}} \mathcal{W}_{\infty} (\nu) \mathbb{E}^{(\nu)} \big[ q^{k (\mathfrak{h}_t (x) - J)} \big], 
\end{flalign}

\noindent where, for any signature $\nu \in \Sign_J^+$, $\nu^{(x)}$ denotes the signature obtained from $\nu$ by removing all elements greater than $x$. 

Moreover, observe that $\mathbb{E}^{(\nu)} \big[ q^{k (\mathfrak{h}_t (x) - J)} \big] = \mathbb{E}^{(\nu^{(x)} (x, J))} \big[ q^{k (\mathfrak{h}_t (x) - J)} \big]$ for any $\nu \in \Sign_J^+$, since particles of the stochastic six-vertex model initially to the right of $x$ remain to the right of $x$ for all time. Thus, 
\begin{flalign}
\label{qh3}
\begin{aligned}
\displaystyle\sum_{\substack{\overline{\nu} \in \Sign^+ \\ \overline{\nu}_1 \le x}} \displaystyle\sum_{\substack{\nu \in \Sign_J^+ \\ \nu^{(x)} = \overline{\nu}}} \mathcal{W}_{\infty} (\nu) \mathbb{E}^{(\nu)} \big[ q^{k (\mathfrak{h}_t (x) - J)} \big] & = \displaystyle\sum_{\substack{\overline{\nu} \in \Sign^+ \\ \overline{\nu}_1 \le x}} \displaystyle\sum_{\substack{\nu \in \Sign_J^+ \\ \nu^{(x)} = \overline{\nu}}} \mathcal{W}_{\infty} (\nu) \mathbb{E}^{(\overline{\nu} (x, J))} \big[ q^{k (\mathfrak{h}_t (x) - J)} \big] \\
& = \displaystyle\sum_{\overline{\nu} \in \Sign^+}  \mathcal{W}_x (\overline{\nu}) \mathbb{E}^{(\overline{\nu} (x, J))} \big[ q^{k (\mathfrak{h}_t (x) - J)} \big]. 
\end{aligned}
\end{flalign}

\noindent In the second equality above, we used the facts that $\sum_{\nu \in \Sign_J^+, \nu^{(x)} = \overline{\nu}} \mathcal{W}_{\infty} (\nu) = \mathcal{W}_x (\overline{\nu})$ (which follows from Remark \ref{w1} and the exponential decay of $\mathcal{W}_x (\nu)$ in $|\nu|$) and that $\mathcal{W}_x (\overline{\nu}) = 0$ if $\overline{\nu}_1 > x$. Now, inserting \eqref{qh1}, \eqref{qh2}, and \eqref{qh3} into Proposition \ref{qms}, we deduce that the proposition holds under the assumptions $b_2 \in (-1, 0]$, $- \kappa^{-1} q^{k + 2} < \beta_2 s < 0$, and $|\kappa \beta_2| < q^{k + 2} \beta_1$. 

To establish the proposition in general, we use uniqueness of analytic continuation. By expanding as a sum of residues, we find that the right side of \eqref{qmsequation} is analytic in $\beta_2$ on the domain $|\kappa \beta_2 | < q^{k + 1} \min \{ \beta_1, 1 \} $; the latter restriction is required to guarantee the existence of the contours $\gamma_h (\kappa \beta_2, \widehat{U}; \widehat{\Xi}, \widehat{S})$. Since the quantity $\mathbb{E}^{(\nu)} [q^{k (\mathfrak{h}_{\lambda} (x) - J)}]$ is independent of $\beta_2$ for any signature $\nu$, and 	since each $\mathcal{W}_x (\lambda)$ is analytic in $\beta_2$ (when $|\kappa \beta_2| < |\beta_1|$) and only non-zero for finitely many $\nu \in \Sign_J^+$, we deduce that \eqref{heightsixvertexequation} is also analytic in $\beta_2$ on the same domain.

Thus the proposition follows from analytic continuation. 
\end{proof}

The following corollary restates Proposition \ref{heightsixvertex} in terms of the current $\mathfrak{H} (X, Y)$ of the stochastic six-vertex model (recall its definition from Section \ref{StochasticVertex}). 

\begin{cor}
\label{heightgeneral} 

Fix $k, x, t, J \in \mathbb{Z}_{> 0}$; $\delta_1, \delta_2, b_1 \in (0, 1)$ satisfying $\delta_1 < \delta_2$; and $b_2 \in \mathbb{C}$. Define $q$, $\kappa$, $s$, $u$, $\beta_1$, and $\beta_2$ as in \eqref{stochasticparameters}, and let $\vartheta = q^{-J}$. 

If $u > q^{-k - 1} |\beta_2 s|$ and $|\kappa \beta_2| < q^{k + 1} \beta_1$, then the right side of \eqref{qmsequation} equals 
\begin{flalign}
\label{current1}
\displaystyle\frac{(\kappa \beta_2 \beta_1^{-1}; q)_{\infty}}{(\vartheta \kappa \beta_2 \beta_1^{-1}; q)_{\infty}} & \displaystyle\sum_{M = 0}^{\infty} \left( \displaystyle\frac{\kappa \beta_2}{\beta_1} \right)^M \displaystyle\frac{(\vartheta; q)_M}{(q; q)_M}  \displaystyle\sum_{\nu \in \Sign^+} \overline{\mathcal{W}_x}^{(M)} (\nu) \mathbb{E}^{(\nu)} \big[ q^{k (\mathfrak{H} (x - 1, t) - M)} \big]. 
\end{flalign}
 
\end{cor}

\begin{rem}
	\label{b20heightvertex}
	
	Let us examine what happens in Corollary \ref{heightgeneral} when we set $b_2 = 0$. Then, all summands in \eqref{current1} are equal to zero, except for the one corresponding to $M = 0$. When $M = 0$, Definition \ref{xvertexweights} and \eqref{wm} imply that $\overline{W_x}^{(M)} (\nu) = 1$ if $\nu$ is empty and is zero otherwise. 
	
	Thus, when $b_2 = 0$, \eqref{current1} is equal to $\mathbb{E} \big[ q^{k \mathfrak{H} (x - 1, t)} \big]$; so, Corollary \ref{heightgeneral} becomes a contour integral identity for $q$-moments of the height function of the stochastic six-vertex model with $(b_1, 0)$-Bernoulli initial data; it can be quickly verified that this identity coincides with the one given by the $m = 1$ case of Proposition 4.4 of \cite{PTAEPSSVM}. 
\end{rem}

\begin{proof}[Proof of Corollary \ref{heightgeneral}]
Definition \ref{xvertexweights}, Definition \ref{wlambda}, and Remark \ref{ww} together imply that \eqref{current1} is equal to 
\begin{flalign}
\label{wwm} 
\displaystyle\sum_{M = 0}^{\infty} \widehat{W}_{q, \vartheta, \kappa, b_1, b_2} (M) \displaystyle\sum_{\nu \in \Sign^+} \overline{\mathcal{W}}_x^{(M)} (\nu) \mathbb{E}^{(\nu)} \big[ q^{k (\mathfrak{H} (x, t) - M)} \big] = \displaystyle\sum_{\nu \in \Sign^+} \mathcal{W}_x (\nu) \mathbb{E}^{(\nu)} \big[ q^{k (\mathfrak{H} (x - 1, t) - m_1 (\nu))} \big]. 
\end{flalign} 

Now, since particles in the stochastic six-vertex model that start to the right of $x$ remain to the right of $x$ for all time, we have that $\mathbb{E}^{(\nu)} \big[ q^{k (\mathfrak{H} (x - 1, t) - m_1 (\nu))} \big] = \mathbb{E}^{(\nu (x, J))} \big[ q^{k (\mathfrak{H} (x - 1, t) - m_1 (\nu))} \big]$. Furthermore, we can also see that if a finite number $J$ of paths enter through the $x$-axis, then $\mathfrak{H} (x - 1, t) = \mathfrak{h}_{\lambda} (x) - J$. On the right side of \eqref{wwm}, the initial data is given by $\nu \in \Sign_J^+$; thus we are considering a stochastic six-vertex model with $J - m_1 (\nu)$ particles initially on the $x$-axis (since we ignore all parts of $\nu$ equal to one when using $\nu$ to define initial data). Thus, we have that $\mathfrak{H} (x - 1, t) - m_1 (\nu) = \mathfrak{h}_{\lambda} (x) - J$, if $\ell (\nu) = J$. 

In particular, the right side of \eqref{wwm} (and hence \eqref{current1}) is equal to 
\begin{flalign*}
\displaystyle\sum_{\nu \in \Sign^+} \mathcal{W}_x (\nu) \mathbb{E}^{(\nu (x, J))} \big[ q^{k (\mathfrak{H} (x - 1, t) - m_1 (\nu))} \big] = \displaystyle\sum_{\nu \in \Sign^+} \mathcal{W}_x (\nu) \mathbb{E}^{(\nu (x, J))} \big[ q^{k (\mathfrak{h}_{\lambda} (x) - J)} \big],
\end{flalign*}

\noindent which is \eqref{heightsixvertex}. Here we have used the fact that $\mathcal{W}_x (\nu) = 0$ if $\ell \big( \nu (x, J) \big) \ne J$, which is a consequence of Proposition \ref{wm2}. Now, the corollary follows from Proposition \ref{heightsixvertex}. 
\end{proof}

\section{Fredholm Determinant Identities} 

\label{DeterminantsDoubleSided}

The goal of this section is to produce Fredholm determinant identities (see Appendix \ref{Determinants1} for our conventions on Fredholm determinants) for the stochastic six-vertex model and ASEP with certain types of double-sided Bernoulli initial data. These are given by Theorem \ref{determinantw} and Theorem \ref{determinant}, respectively. 

The proofs of these two theorems will proceed as follows. Corollary \ref{heightgeneral} provides explicit identities for $q$-moments of the current of the stochastic six-vertex model with vertical $b_1$-Bernoulli initial data and horizontal Bernoulli initial data weighted by $\mathcal{W}$. Of interest to us is the case $\vartheta = 0$ and $b_2 \in (0, 1)$, which corresponds to a Bernoulli weighting. 

However, there are two issues that prevent us from immediately accessing this. The first is that Corollary \ref{heightgeneral} only applies when $\vartheta$ is a negative power of $q$; in particular, it does not apply when $\vartheta = 0$. The second is that Corollary \ref{heightgeneral} only yields the $k$-th moment for the current if $|q^{-k} \beta_2|$ is sufficiently small; this is insufficient to completely characterize the distribution of $\mathfrak{H} (x, t)$. 

To remedy these issues, we first consider the case $\vartheta \in \{ q^{-1}, q^{-2}, \ldots \}$. Then, it will be possible to analytically continue Corollary \ref{heightgeneral} so that it holds for all $k \ge 1$, which will suffice to establish Theorem \ref{determinantw} when $\vartheta$ is a negative power of $q$.  

To extend this result to $\vartheta = 0$ (or in general, to all $\vartheta \in \mathbb{C}$), we use a (modification of an) idea presented in Section 3.4 of \cite{HFSE}. In particular, we will take $b_2$ to be a formal, instead of complex, variable and view Theorem \ref{determinantw} as an identity of formal power series in $b_2$. Then, we will compare the coefficients of $b_2$ on both sides of this identity; through a polynomiality argument, we will show that if they are equal for $\vartheta \in \{ q^{-1}, q^{-2}, \ldots \}$, then they must be equal for all $\vartheta \in \mathbb{C}$. 

This section is organized as follows. In Section \ref{FormalDeterminant}, we explain and state the formal Fredholm determinant identity (as an identity of formal power series), given by Proposition \ref{determinantgeneralformal}. In Section \ref{DeterminantSpin12} we specialize this identity to the case when $b_2 \in \mathbb{C}$, thereby yielding Theorem \ref{determinantw}, which can be degenerated to Theorem \ref{determinant} (the analogous identity for the ASEP). In Section \ref{ProofFormalDeterminant} we provide the proof of Proposition \ref{determinantgeneralformal}.

\subsection{A Formal Fredholm Determinant Identity}

\label{FormalDeterminant}

 In what follows, fix $b_1, \delta_1, \delta_2 \in (0, 1)$ satisfying $\delta_1 < \delta_2$; define $q$, $\kappa$, $s$, $u$, and $\beta_1$ as in \eqref{stochasticparameters}.

Now, let $b_2$ be a formal variable, and let $\mathbb{C} \llbracket b_2 \rrbracket$ denote the ring of formal power series in $b_2$. Its topology is the $b_2$-adic topology on the inverse limit $\varprojlim \mathbb{C} [b_2] / (b_2^k)$; in particular, a sequence $\{ a_i \}_{i \ge 1} = \sum_{j = 0}^{\infty} a_{i, j} b_2^j \in \mathbb{C} \llbracket b_2 \rrbracket$ converges to $a = \sum_{j = 1}^{\infty} a_j b_2^j \in \mathbb{C} \llbracket b_2 \rrbracket$ if and only if $\{ a_{i, j} \}_{i \ge 1}$ converges to $a_i$ for each $i \ge 0$. Let $\beta_2 = b_2 (1 - b_2)^{-1}$. 

The weights $W_{\vartheta, q, b_2} (j; i)$ and $\widehat{W}_{q, \vartheta, \kappa, b_1, b_2} (m)$ from Definition \ref{xvertexweights} still exist, but are now elements of the formal power series ring $\mathbb{C} \llbracket b_2 \rrbracket$; the same holds for the weights $\mathcal{W}_x (\lambda), \overline{\mathcal{W}}_x^{(M)} (\lambda) \in \mathbb{C} \llbracket b_2 \rrbracket$ (recall Definition \ref{wlambda}), for any signature $\lambda$. 

The following result states the formal Fredholm determinant identity for the stochastic six-vertex model; its proof is postponed to Section \ref{ProofFormalDeterminant}. Observe here that $\vartheta \in \mathbb{C}$ is arbitrary. 

\begin{prop}
\label{determinantgeneralformal}

Fix $x, t \in \mathbb{Z}_{> 0}$; $0 < \delta_1 < \delta_2 < 1$; $b_1 \in (0, 1)$; and $\vartheta \in \mathbb{C}$. Let $b_2$ be a formal variable. Denote $q$, $u$, $\kappa$, $s$, and $\beta_1$ as in \eqref{stochasticparameters}, and let $\beta_2 = b_2 (1 - b_2)^{-1} \in \mathbb{C} \llbracket b_2 \rrbracket$. For any $z \in \mathbb{C}$, let 
\begin{flalign}
\label{gvformal}
g_V (z; x, t) = \displaystyle\frac{(\kappa^{-1} z + q)^x (\beta_2 q \kappa z^{-1}; q)_{\infty}}{(z + q)^t (q^{-1} \beta_1^{-1} z; q)_{\infty} (\vartheta \beta_2 q \kappa z^{-1}; q)_{\infty}}, 
\end{flalign}

\noindent which is an element of $\mathbb{C} \llbracket b_2 \rrbracket$. Let $\mathscr{C}_V$ be a positively oriented circle in the complex plane, centered at some non-positive real number, containing $-q$ and $0$, but leaving outside $- \kappa$, $-1$, and $\beta_1 q$. 

Fix a sufficiently large real number $R > 1$, a real number $\delta \in (0, 1)$ sufficiently close to $1$, and a sufficiently small real number $d \in (0, 1)$ such that 
\begin{flalign}
\label{doublesidedinequalities}
q^R  < \displaystyle\inf_{w, w' \in \mathscr{C}_V} | w^{-1} w'|; \quad q^{1 - \delta} > \displaystyle\sup_{w \in \mathscr{C}_V} \max \{ |  \kappa^{-1} w |, | qw | \} ; \quad \big| \Im q^{id} \big| < \displaystyle\sup_{\substack{w, w' \in \mathscr{C}_V \\ \Re w^{-1} w' \ge 0 \\ |w^{-1} w' | \le q^{\delta}}} \left| \Im w^{-1} w' \right|. 
\end{flalign}

\noindent For any $\zeta \in \mathbb{C} \setminus\mathbb{R}_{\ge 0}$, we have that 
\begin{flalign}
\label{determinantformalvertex}
\begin{aligned} 
\displaystyle\frac{(\kappa \beta_2 \beta_1^{-1}; q)_{\infty}}{(\vartheta \kappa \beta_2 \beta_1^{-1}; q)_{\infty}} \displaystyle\sum_{M = 0}^{\infty} \left( \displaystyle\frac{\kappa \beta_2}{\beta_1} \right)^M \displaystyle\frac{(\vartheta; q)_M }{(q; q)_M}  \displaystyle\sum_{\nu \in \Sign^+} \overline{\mathcal{W}_x}^{(M)} (\nu) \mathbb{E}^{(\nu)} & \left[ \displaystyle\frac{1}{\big( \zeta q^{\mathfrak{H} (x, t) - M	} ; q \big)_{\infty}} \right] \\
& = \det \big( \Id + K_{\zeta} \big)_{L^2 (\mathscr{C}_V)}, 
\end{aligned}
\end{flalign}

\noindent where the kernel $K_{\zeta} \in \mathbb{C} \llbracket b_2 \rrbracket$ is defined by 

\begin{flalign}
\label{formalkernel}
K_{\zeta} (w, w') = \displaystyle\frac{1}{2 \textbf{\emph{i}}} \displaystyle\int_{D_{R, d, \delta}} \displaystyle\frac{g_V (w; x, t)}{g_V (q^r w; x, t)}  \displaystyle\frac{(-\zeta)^r dr}{\sin (\pi r) (q^r w - w')}, 
\end{flalign}

\noindent where the contour $D_{R, d, \delta}$ is given by Definition \ref{contourslong}. In \eqref{determinantformalvertex}, $\overline{\mathcal{W}_x}^{(M)}$ is given by \eqref{wm} of Definition \ref{wlambda}, and $\mathbb{E}^{(\nu)}$ is defined above Proposition \ref{heightsixvertex}. 	
\end{prop}

\begin{rem}

\label{determinantformalvertexseries}

Let us briefly explain the meaning of \eqref{determinantformalvertex}. The left side of \eqref{determinantformalvertex} is an infinite sum; its convergence (with respect to the topology of $\mathbb{C} \llbracket b_2 \rrbracket$) will follow from Lemma \ref{coefficientb2} below. 

The right side of \eqref{determinantformalvertex} is dependent on the kernel $K_{\zeta}$, given by \eqref{formalkernel}. As we will show in the proof of Lemma \ref{coefficientb2}, the integrand on the right side of \eqref{formalkernel} is an element of $\mathbb{C} \llbracket b_2 \rrbracket$, which means that it is of the form $\sum_{j = 0}^{\infty} a_j (r) b_2^j$. The kernel $K_{\zeta}$ is then defined to be the sum $\sum_{j = 0}^{\infty} I_j b_2^j$, where $I_j = \displaystyle\int_{D_{R, d, \delta}} a_j (r) dr$; we will verify convergence of these integrals in the proof of Lemma \ref{coefficientb2}. 

Then, the Fredholm determinant on the right side of \eqref{determinantformalvertex} is obtained by inserting $K_{\zeta} \in \mathbb{C} \llbracket b_2 \rrbracket$ into the definition \eqref{determinantsum} of the Fredholm determinant. This yields an infinite sum of elements in $\mathbb{C} \llbracket b_2 \rrbracket$; we will again verify the convergence of this sum in the proof of Lemma \ref{coefficientb2}. Hence, the right side of \eqref{determinantformalvertex} is also an element of $\mathbb{C} \llbracket b_2 \rrbracket$. 

Proposition \ref{determinantgeneralformal} states that the left and right sides of \eqref{determinantformalvertex} are equal as elements of the formal power series ring $\mathbb{C} \llbracket b_2 \rrbracket$. 

\end{rem}

\noindent Let us establish that both sides of \eqref{determinantformalvertex} indeed converge in $\mathbb{C} \llbracket b_2 \rrbracket$; in fact, the following lemma estimates the coefficient of $\beta_2^j$ in both sides of \eqref{determinantformalvertex}. 

\begin{lem}

\label{coefficientb2} 
	
Adopt the notation of Proposition \ref{determinantgeneralformal}, and fix the parameters $\delta_1$, $\delta_2$, $b_1$, $\vartheta$, $x$, $t$, $R$, $d$, $\delta$, $\zeta$, and also the contour $\mathscr{C}_V$. Then, both sides of \eqref{determinantformalvertex} are convergent power series in the ring $\mathbb{C} \llbracket b_2 \rrbracket$. Equivalently, there exist complex numbers $E_1, E_2, \ldots $ and $H_1, H_2, \ldots $ such that the left side of \eqref{determinantformalvertex} is equal to $\sum_{j = 0}^{\infty} E_j \beta_2^j$ and the right side of \eqref{determinantformalvertex} is equal to $\sum_{j = 0}^{\infty} H_j \beta_2^j$. 

Furthermore, there exists a constant $C$ such that $|E_j|, |H_j| \le C^j$ for each $j \ge 1$. 
 
\end{lem}

\begin{proof}

Let us first consider the left side. Since $|q| < 1$ and $\zeta \notin \mathbb{R}_{\ge 0}$, we have that $|(\zeta q^{\mathfrak{H} (x, t) - M}; q)_{\infty}^{-1}|$ is bounded, independent of $M \ge 0$. Furthermore, since there only finitely many signatures $\nu$ such that $\overline{\mathcal{W}_x}^{(M)} (\nu) \ne 0$, we deduce the existence of a constant $C_1 > 0$ such that the coefficient of $\beta_2^j$ in the first sum on the left side of \eqref{determinantformalvertex} is bounded by $C_1^j$.

Furthermore, the fact that $|q| < 1$ implies that there exists a constant $C_2 > 0$ such that the coefficient of $\beta_2^j$ in both $(\kappa \beta_2 \beta_1^{-1}; q)_{\infty}$ and $(\vartheta \kappa \beta_2 \beta_1^{-1}; q)_{\infty}^{-1}$ are bounded by $C_2^j$. Hence, there exists a constant $C_3 > 0$ such that the coefficient of $\beta_2^j$ on the left side of \eqref{determinantformalvertex} is bounded by $C_3^j$. Setting $C \ge C_3$, it follows that $|E_j| \le C^j$. 

Next, we analyze the right side of \eqref{determinantformalvertex}. In particular, let us consider the kernel $K_{\zeta}$ given by \eqref{formalkernel}. The part of the integrand dependent on $\beta_2$ is the quotient $g_V (w; x, t) g_V (q^r w; x, t)^{-1}$. Let $\widetilde{g}_V (w; x, t)$ be equal to $g_V (w; x, t)$, with $\beta_2$ set to $0$. Then, we have that 
\begin{flalign*}
\displaystyle\frac{g_V (w; x, t)}{g_V (q^r w; x, t)} = \left( \displaystyle\frac{\widetilde{g}_V (w; x, t)}{\widetilde{g}_V (q^r w; x, t)} \right) \displaystyle\frac{(\beta_2 q \kappa w^{-1}; q)_{\infty} (\vartheta \beta_2 q^{1 - r} \kappa w^{-1}; q)_{\infty}}{(\beta_2 q^{1 - r} \kappa w^{-1}; q)_{\infty} (\vartheta\beta_2 q \kappa w^{-1}; q)_{\infty}}. 
\end{flalign*}

Using the definition and compactness of the contour $\mathscr{C}_V$, it is quickly verified that the quotient $\widetilde{g}_V (w; x, t) \widetilde{g}_V (q^r w; x, t)^{-1}$ remains uniformly bounded over all $r \in D_{R, d, \delta}$ and $w \in \mathscr{C}_V$. Now, since $|q| < 1$ and $\Re r \le R$, there exists a constant $C_4 > 0$ so that the coefficient of $\beta_2^j$ in $(\beta_2 q \kappa w^{-1}; q)_{\infty} (\vartheta \beta_2 q^{1 - r} \kappa w^{-1}; q)_{\infty} (\vartheta \beta_2 q^{1 - r} \kappa w^{-1}; q)_{\infty}^{-1} (\beta_2 q^{1 - r} \kappa w^{-1}; q)_{\infty}^{-1}$ is bounded by $C_4^j$. Thus, there exists $C_5 > 0$ such that the coefficient of $\beta_2^j$ in $g_V (w; x, t) g_V (q^r w; x, t)^{-1}$ is bounded by $C_5^j$.

Since $q^r w - w' \ne 0$ for all $r \in D_{R, d, \delta}$ (and also since $\mathscr{C}_V$ is compact, and since the image of $q^r$ as $r$ ranges over $D_{R, d, \delta}$ is compact), there exists a constant $C_6 > 0$, independent of $r$, so that the coefficient of $\beta_2^j$ in the integrand of the right side of \eqref{formalkernel} is at most $C_6^j / |\sin (\pi r)|$. Due to the exponential decay of $\big| \sin (\pi r) \big|^{-1}$ as $r$ ranges over $D_{R, d, \delta}$, we conclude that there exists a constant $C_7 > 0$ so that the coefficient of $\beta_2^j$ in $K_{\zeta}$ is bounded by $C_7^j$. 			

Now, let us estimate the coefficient of $\beta_2^j$ in $\det \big[ K_{\zeta} (w_i, w_h) \big]_{i, h = 1}^k$. By the above, we may express $K_{\zeta} (w_i, w_h) = \sum_{n = 0}^{\infty} a_{n; i, h} \beta_2^n$, for some complex numbers $a_{n; i, h}$ satisfying $\big| a_{n; i, h} \big| < C_7^n$. For any ordered set of non-negative integers $\textbf{n} = \{ n_1, n_2, \ldots , n_k \} $, let $\textbf{A}_{\textbf{n}}$ denote the $k \times k$ matrix whose $(i, h)$-entry is $a_{n_i; i, h}$. Expanding $\det \big[ K_{\zeta} (w_i, w_h) \big]_{i, h = 1}^k$, we deduce that the coefficient of $\beta_2^j$ in $\det \big[ K_{\zeta} (w_i, w_h) \big]_{i, h = 1}^k$ is equal to $\sum_{|\textbf{n}| = j} \det \textbf{A}_{\textbf{n}}$, where $\textbf{n} = (n_1, n_2, \ldots , n_k)$ is summed over all $k$-tuples satisfying $|\textbf{n}| = \sum_{i = 1}^k n_i = j$; there are $\binom{k + j - 1}{j}$ such $k$-tuples $\textbf{n}$. 

Furthermore, by Hadamard's inequality, we have that 
\begin{flalign*}
\big| \det \textbf{A}_{\textbf{n}} \big| & \le \displaystyle\prod_{h = 1}^k \left( \displaystyle\sum_{i = 1}^k |a_{n_i; i, h}|^2 \right)^{1 / 2} 	\le C_7^j k^{k / 2}.
\end{flalign*}

\noindent Thus, the coefficient of $\beta_2^j$ in $\det \big[ K_{\zeta} (w_i, w_h) \big]_{i, h = 1}^k$ is bounded by $C_7^j k^{k / 2} \binom{k + j}{j}$.

Inserting this into \eqref{determinantsum}, we deduce that the coefficient of $\beta_2^j$ in the right side of \eqref{determinantformalvertex} is at most equal to $\sum_{k = 0}^{\infty} C_8^j k^{k / 2} \binom{k + j}{j} k!^{-1}$, for some $C_8 > 0$ (dependent on the length of $\mathscr{C}_V$). Thus, this coefficient is bounded by 
\begin{flalign*}
C_8^j \displaystyle\sum_{k = 1}^{\infty} \left( \displaystyle\frac{C_9}{k} \right)^{k / 2} \displaystyle\binom{k + j}{j} & = C_8^j \displaystyle\sum_{k = 1}^{3j - 1} \left( \displaystyle\frac{C_9}{k} \right)^{k / 2} \displaystyle\binom{k + j}{j}  + C_8^j \displaystyle\sum_{k = 3j}^{\infty} \left( \displaystyle\frac{C_9}{k} \right)^{k / 2} \displaystyle\binom{k + j}{j} < C_{10}^j, 
\end{flalign*}

\noindent for some $C_9, C_{10} > 0$. Setting $C \ge C_{10}$ yields $|H_j| \le C^j$, which implies the second part of the lemma. 
\end{proof}

\subsection{Applications of Proposition \ref{determinantgeneralformal}} 

\label{DeterminantSpin12}

In this section we establish Theorem \ref{determinantw} and Theorem \ref{determinant}, which are Fredholm determinant identities for the stochastic six-vertex model and ASEP with double-sided Bernoulli initial data, respectively. The stochastic six-vertex model is addressed in Section \ref{DeterminantVertex} and the ASEP in Section \ref{DegenerateExclusion}. 

\subsubsection{An Identity for the Stochastic Six-Vertex Model} 

\label{DeterminantVertex}

Our goal is to establish Theorem \ref{determinantw}. First, however, we will prove an analog of Proposition \ref{determinantgeneralformal} for $b_2 \in \mathbb{C}$ sufficiently small instead of formal. 

\begin{prop}
\label{determinantgeneral}

Fix $x$, $t$, $\delta_1$, $\delta_2$, $b_1$, $q$, $u$, $\kappa$, $s$, $\beta_1$, $\vartheta$, $R$, $d$, $\delta$, $\zeta$, and $\mathscr{C}_V$ as in Proposition \ref{determinantgeneralformal}. Let $b_2 \in \mathbb{C}$ be a complex number, and denote $\beta_2 = b_2 / (1 - b_2)$. For any $z \in \mathbb{C}$, define $g_V (z; x, t)$ as in \eqref{gvformal} and $K_{\zeta} (w, w')$ as in \eqref{formalkernel}; both are now elements of $\mathbb{C}$. 

There exists a constant $c > 0$ such that, if $|b_2| < c$, then \eqref{determinantformalvertex} holds as an identity over $\mathbb{C}$.

\end{prop}

\begin{proof}

In view of Proposition \ref{determinantgeneralformal}, it suffices to show that both sides of \eqref{determinantformalvertex} converge absolutely as power series in $b_2$ (or equivalently in $\beta_2$) when $b_2$ is a complex number satisfying $|b_2| < c$. If we take $c$ to be sufficiently small so that $|b_2| < c$ implies that $|\beta_2| < C^{-1}$ (where $C$ is given in Lemma \ref{coefficientb2}), then this is a consequence of Lemma \ref{coefficientb2}. 
\end{proof}

Using Proposition \ref{determinantgeneral}, we can now establish the following theorem, which is a Fredholm determinant identity for the current of the stochastic six-vertex model. In what follows, we recall that a contour $\gamma \subset \mathbb{C}$ is called \emph{star-shaped} (with respect to the origin) if, for each real number $a \in \mathbb{R}$, there exists exactly one complex number $z_a \in \gamma$ such that $z_a / |z_a| = e^{ \textbf{i} a}$. 

\begin{thm}
\label{determinantw}

Fix $x, t \in \mathbb{Z}_{> 0}$ and $\delta_1, \delta_2, b_1, b_2 \in (0, 1)$. Denote $q = \delta_1 / \delta_2$, $\beta_1 = b_1 / (1 - b_1)$, $\beta_2 = b_2 / (1 - b_2)$, and $\kappa = (1 - \delta_1) / (1 - \delta_2)$. Assume that $\kappa \beta_2 < \beta_1$. For any $z \in \mathbb{C}$, let 
\begin{flalign*}
g_V (z; x, t) = \displaystyle\frac{(\kappa^{-1} z + q)^x (q \beta_2 \kappa z^{-1}; q)_{\infty}}{(z + q)^t (q^{-1} \beta_1^{-1} z; q)_{\infty}}. 
\end{flalign*}

\noindent Let $\Gamma_V \subset \mathbb{C}$ be a positively oriented, star-shaped contour in the complex plane containing $0$ and $q \kappa \beta_2$, but leaving outside $-q \kappa$ and $q \beta_1$. Let $\mathcal{C}_V \subset \mathbb{C}$ be a positively oriented, star-shaped contour contained inside $q^{-1} \Gamma_V$; that contains $0$, $-q$, $q \kappa \beta_2$, and $\Gamma_V$; but that leaves outside $q \beta_1$. 

Let $\zeta = - q^p < 0$ for some real number $p \in \mathbb{R}$. Then, 
\begin{flalign}
\label{determinantvertexw}
( \kappa \beta_2 \beta_1^{-1} ; q )_{\infty} \displaystyle\sum_{M = 0}^{\infty} \displaystyle\frac{\big( \kappa \beta_2 \beta_1^{-1} \big)^M}{(q; q)_M} \mathbb{E} \left[ \displaystyle\frac{1}{(\zeta q^{\mathfrak{H} (x, t) - M}; q)_{\infty}} \right] = \det \big( \Id + V_{\zeta} \big)_{L^2 (\mathcal{C}_V)}, 
\end{flalign}

\noindent where 

\begin{flalign}
\label{t}
V_{\zeta} (w, w') = \displaystyle\frac{1}{2 \textbf{\emph{i}} \log q} \displaystyle\sum_{j = - \infty}^{\infty} \displaystyle\int_{\Gamma_V} \displaystyle\frac{v^{p - 1} w^{-p}}{\sin \big( \frac{\pi}{\log q} (\log v - \log w + 2 \pi \textbf{\emph{i}}j ) \big) (w' - v)} \displaystyle\frac{g_V (w; x, t)}{g_V (v; x, t)} dv. 
\end{flalign}

\noindent The expectation in \eqref{determinantvertexw} is with respect to the stochastic six-vertex model with double-sided $(b_1, b_2)$-Bernoulli initial data, with left jump probability $\delta_1$ and right jump probability $\delta_2$. 
\end{thm}

See Figure \ref{contoursdeterminantw} for examples of the contours $\mathcal{C}_V$ and $\Gamma_V$. 

\begin{figure}

\begin{center}

\begin{tikzpicture}[
      >=stealth,
			scale = .6
			]

			\draw[<->, black] (-4.5, 0) -- (4, 0); 
			\draw[<->, black] (0, -3.5) -- (0, 3.5);

			\draw[black, fill] (3, 0) circle [radius=.04] node [black, below = 0, scale = .8] {$q \beta_1$};
			\draw[black, fill] (1.7, 0) circle [radius=.04] node [black, left = 5, below = 0, scale = .8] {$q \kappa \beta_2$};
			\draw[black, fill] (-.6, 0) circle [radius=.04] node [black, below = 0, scale = .8] {$q$};
			\draw[black, fill] (-3, 0) circle [radius=.04] node [black, below=0, scale = .8] {$q \kappa $};

			\draw[black, thick] (-.2, 0) circle [radius=2.4] node [black, left = 15, below = 22, scale = .8] {$\Gamma_V$};
			\draw[black, thick] (-.5, 0) circle [radius=3] node [black, right = 45, below = 30, scale = .8] {$\mathcal{C}_V$};
			\draw[black, dashed] (-.1667, 0) circle [radius=1];

\end{tikzpicture}

\end{center}

\caption{\label{contoursdeterminantw} Shown above (in solid) are examples of the contours $\mathcal{C}_V$ and $\Gamma_V$; the dashed curve is the image of $\mathcal{C}_V$ under multiplication by $q$. }

\end{figure}
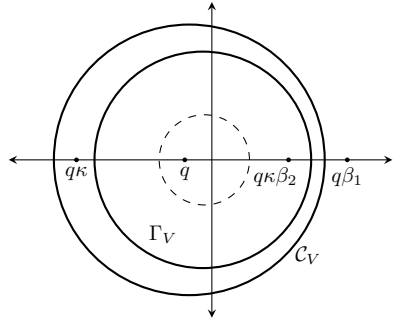

\begin{rem}
	\label{b20vertex} 
	
	If we set $b_2 = 0$ in Theorem \ref{determinantw}, then all summands on the left side of \eqref{determinantvertexw} are zero except for the one corresponding to $M = 0$. Thus, \eqref{determinantvertexw} becomes a Fredholm determinant identity for the $q$-Laplace transform of the height function of the stochastic six-vertex model with $(b_1, 0)$-Bernoulli initial data, which can be seen to match with the $m = 1$ and $K = V$ case of Proposition 5.1 in \cite{PTAEPSSVM}.
\end{rem}

\begin{proof}[Proof of Theorem \ref{determinantw}] 

For small values of $b_2$, this theorem will follow from changing variables in Proposition \ref{determinantgeneral} (and setting $\vartheta = 0$). We will then establish the theorem for larger $b_2$ through analytic continuation. In what follows, we adopt the notation of Proposition \ref{determinantgeneral}. 

Recall that $c > 0$ was defined in Corollary \ref{determinantw}; we may reduce the size of $c$ if necessary to assume that $\kappa c / (1 - c) < q^R \inf_{z \in \mathscr{C}_V} |z|$. Let $\widetilde{b}_2 \in \mathbb{C}$ be a complex number satisfying $|\widetilde{b}_2| < c$. By Proposition \ref{determinantgeneral}, \eqref{determinantformalvertex} holds with $b_2$ replaced with $\widetilde{b}_2$; in what follows, denote $\widetilde{\beta}_2 = \widetilde{b}_2 / (1 - \widetilde{b}_2)$.

Let us change variables $v = q^p w$ in the definition \eqref{formalkernel} of the kernel $K_{\zeta}$. To understand how this affects the identity \eqref{formalkernel} for $K_{\zeta}$, we must analyze the contour for $v$. 

To that end, first observe that if $d$ is sufficiently small (which we will always assume to be the case), the contour $D_{R, d, \delta}$ can be written as the union $D_{R, d, \delta} = I' \cup \bigcup_{k \in \mathbb{Z} \setminus\{ -1 \}} I_k$. Here, $I_k$ is defined to be the interval
\begin{flalign*}
I_k = \Big[ R + \textbf{i} \big(d + 2 \pi |\log q|^{-1} k \big), R + \textbf{i} \big(d + 2 \pi |\log q|^{-1} (k + 1) \big) \Big],
\end{flalign*} 

\noindent for each integer $k$, and $I'$ is defined to be the piecewise linear curve
\begin{flalign*}
I' = \Big[ R + \textbf{i} \big(d - 2 \pi |\log q|^{-1} \big), R - \textbf{i} d \Big] \cup \big[ R -  \textbf{i} d, \delta - \textbf{i} d \big] \cup \big[ \delta - \textbf{i} d, \delta + \textbf{i} d \big] \cup [\delta + \textbf{i} d, R + \textbf{i} d]. 
\end{flalign*}

Now, for each integer $k$, the map from $r$ to $q^r w$ is a bijection from $I_k$ to $-Q_{R, d, \delta; w}^{(1)}$, where the negative refers to reversal of orientation and the contour $Q_{R, d, \delta; w}^{(1)}$ is a positively oriented circle of radius $q^R$, centered at $0$; here, the negative orientation is due to the fact that $\log q < 0$. 

Furthermore, the map from $r$ to $q^r w$ is a bijection from $I'$ to the contour $- Q_{R, d, \delta; w}$, where $Q_{R, d, \delta; w}$ is defined to be the union of four curves $\mathcal{J}_1 \cup \mathcal{J}_2 \cup \mathcal{J}_3 \cup \mathcal{J}_4$, which are defined are as follows. The curve $\mathcal{J}_1$ is the major arc of the circle centered at $0$, with radius $q^R |w|$, connecting $q^{R - id} w$ to $q^{R + id } w$. The curve $\mathcal{J}_2$ is the line segment in the complex plane connecting $q^{R + id} w$ to $q^{\delta + id} w$; the curve $\mathcal{J}_3$ is the minor arc of the circle centered at $0$, with radius $q^{\delta} |w|$, connecting $q^{\delta + id} w$ to $q^{\delta - id} w$; and the curve $\mathcal{J}_4$ is the line segment in the complex plane connecting $q^{\delta - id} w$ to $q^{R - id} w$. We refer to Figure \ref{contoursgammaq} for an example of the contour $Q_{R, d, \delta; w}$. 

Thus, since $dv = v \log q dr$ and $(-\zeta)^r = v^p w^{-p}$, we obtain that
\begin{flalign*}
K_{\zeta} (w, w') & = - \displaystyle\frac{1}{2 \textbf{i} \log q} \displaystyle\sum_{j \ne 0} \displaystyle\oint_{Q_{R, d, \delta; w}^{(1)}} \displaystyle\frac{g_V (w; x, T)}{g_V (v; x, t)}  \displaystyle\frac{v^{p - 1} w^{-p} }{\sin \big( \frac{\pi}{\log q} ( \log v - \log w + 2 \pi \textbf{i} j ) \big)} \displaystyle\frac{dv}{v - w'}  \\
& \qquad - \displaystyle\frac{1}{2 \textbf{i} \log q} \displaystyle\oint_{Q_{R, d, \delta; w}} \displaystyle\frac{g_V (w; x, T)}{g_V (v; x, T)}  \displaystyle\frac{v^{p - 1} w^{-p} 	}{\sin \big( \frac{\pi}{\log q} ( \log v - \log w ) \big) } \displaystyle\frac{dv}{v - w'} , 
\end{flalign*}

\noindent where we have taken the principal branch of the logarithm. 

We can rewrite $K_{\zeta} = K_{\zeta} (w, w')$ as 
\begin{flalign}
\label{kvw}
\begin{aligned}
K_{\zeta} & = - \displaystyle\frac{1}{2 \textbf{i} \log q} \displaystyle\sum_{j = -\infty}^{\infty} \displaystyle\oint_{Q_{R, d, \delta; w}} \displaystyle\frac{g_V (w; x, T)}{g_V (v; x, t)}  \displaystyle\frac{v^{p - 1} w^{-p}}{\sin \big( \frac{\pi}{\log q} ( \log v - \log w + 2 \pi \textbf{i} j ) \big)}\displaystyle\frac{dv}{ v - w'}  \\
& \quad - \displaystyle\frac{1}{2 \textbf{i} \log q} \displaystyle\sum_{j \ne 0} \displaystyle\oint_{Q_{R, d, \delta; w}^{(2)} } \displaystyle\frac{g_V (w; x, T)}{g_V (v; x, T)}  \displaystyle\frac{v^{p - 1} w^{-p}} {\sin \big( \frac{\pi}{\log q} ( \log v - \log w + 2 \pi \textbf{i} j) \big)} \displaystyle\frac{dv}{ v - w'} , 
\end{aligned}
\end{flalign}

\noindent where the contour $Q_{R, d, \delta; w}^{(2)} = Q_{R, d, \delta; w}^{(1)} - Q_{R, d, \delta; w}$ is the union of four curves $\mathcal{C}_1 \cup \mathcal{C}_2 \cup \mathcal{C}_3 \cup \mathcal{C}_4$, which are defined as follows. The curve $\mathcal{C}_1$ is the line segment in the complex plane connecting $q^{R - id} w$ to $q^{\delta - id} w$. The curve $\mathcal{C}_2$ is the minor arc of the circle centered at $0$, with radius $q^{\delta} |w|$, connecting $q^{\delta - id} w$ to $q^{\delta + id} w$. The curve $\mathcal{C}_3$ is the line segment connecting $q^{\delta + id} w$ to $q^{R + id} w$. The curve $\mathcal{C}_4$ is the minor arc of the circle centered at $0$, with radius $q^R |w|$, connecting $q^{R + id} w$ to $q^{R - id} w$. We refer to Figure \ref{contoursgammaq} for an example of the contour $Q_{R, d, \delta; w}^{(2)}$. 

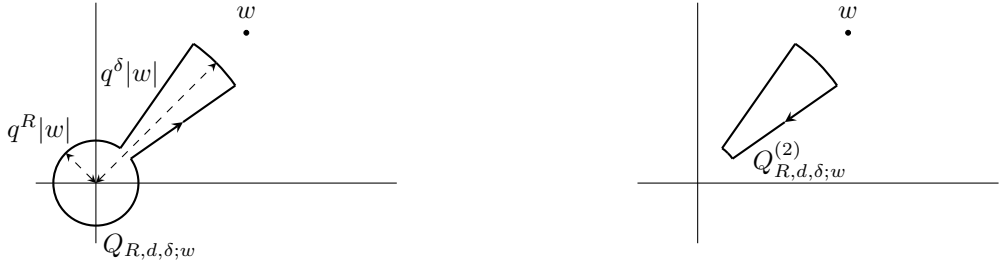
\begin{figure}

\begin{center}

\begin{tikzpicture}[
      >=stealth,
			scale=.8
			]

			\draw[-, black] (5, 0) -- (-1, 0); 
			\draw[-, black] (0, -1) -- (0, 3);

			\draw[->, black,  thick] (.57923,.40558) -- (1.44808, 1.01395);
			\draw[-, black,  thick] (1.44808, 1.01395) -- (2.31692, 1.62232);
			\draw[-, black,  thick] (.40558, .57923) -- (1.62232, 2.31692);
			
			\draw[black, thick] (2, 2) arc (45:55:2.8284);
			\draw[black, thick] (2, 2) arc (45:35:2.8284);
			\draw[black, thick] (.40558, .57923) arc (55:180:.7071);
			\draw[black, thick] (.57923, .40558) arc (35:-180:.7071) node [black, below = 24, right = 15] {$Q_{R, d, \delta; w}$};

			\draw[<->, black, dashed] (0, 0) -- (-.5, .5) node[black, left = 10, above = 0] {$q^R |w|$};
			\draw[<->, black, dashed] (0, 0) -- (2, 2) node[black, below = 4, left = 18] {$q^{\delta} |w|$};

			\draw[black, fill] (2.5, 2.5) circle [radius=.04] node [black, above = 2] {$w$};

			\draw[-, black] (15, 0) -- (9, 0); 
			\draw[-, black] (10, -1) -- (10, 3);

			\draw[-, black,  thick] (10.57923,.40558) -- (11.44808, 1.01395);
			\draw[->, black,  thick] (12.31692, 1.62232) -- (11.44808, 1.01395);
			\draw[-, black,  thick] (10.40558, .57923) -- (11.62232, 2.31692);
			
			\draw[black, thick] (12, 2) arc (45:55:2.8284);
			\draw[black, thick] (12, 2) arc (45:35:2.8284);
			\draw[black, thick] (10.5, .5) arc (45:55:.7071);
			\draw[black, thick] (10.5, .5) arc (45:35:.7071) node [black, below = 1, right = 4] {$Q_{R, d, \delta; w}^{(2)}$};

			\draw[black, fill] (12.5, 2.5) circle [radius=.04] node [black, above = 2] {$w$};

\end{tikzpicture}

\end{center}

\caption{\label{contoursgammaq} Shown to the left is the contour $Q_{R, d, \delta; w}$. Shown to the right is the contour $Q_{R, d, \delta; w}^{(2)}$. } 
\end{figure}

We claim that each summand of the second sum on the right side of \eqref{kvw} is equal to $0$. To verify this, we analyze the poles of the corresponding integrands, which are at $v \in \{0, w', -q \kappa \} \cup \bigsqcup_{k = 1}^{\infty} \{ q^k \kappa \widetilde{\beta}_2 \} $. Observe that there is no pole at $v = w$ (or at $v \in w q^{\mathbb{Z}}$) since $j$ is non-zero in each summand. 

We claim that none of these poles are contained in the interior of $Q_{R, d, \delta; w}^{(2)}$. For instance, $q^{k + 1} \kappa \widetilde{\beta}_2$ is contained in the interior of $Q_{R, d, \delta; w}^{(1)}$ (since $|q \kappa \widetilde{\beta}_2| < \kappa c / (c - 1) < q^R \inf_{z \in \mathscr{C}_V} |z|$) and thus cannot be contained inside $Q_{R, d, \delta; w}^{(2)}$. Similarly, one can check that the other poles at $0$, $w'$, and $-q \kappa$ are outside of $Q_{R, d, \delta; w}^{(2)}$. Hence, the second term on the right side of \eqref{kvw} is equal to $0$. 

Now, observe that the poles for $v$ of the integrands in the first sum on the right side of \eqref{kvw} are at $v \in \{ 0, w', -q \kappa \} \cup \{ q^j w \}_{j \in \mathbb{Z}} \cup \{ q \kappa \beta_2, q^2 \kappa \beta_2, \ldots \} $; of these, the poles at $v \in \{ 0, q w, q^2 w, \ldots \} \cup \{ q \kappa \beta_2, q^2 \kappa \beta_2, \ldots \}$ are contained inside the contour $Q_{R, d, \delta; w}$ and the others are left outside. The poles for $w$ of the same integrands are at $w \in \{ -q, 0, q \beta_1, \beta_1, q^{-1} \beta_1, \ldots \}$; of these, the poles at $w \in \{ 0, -q \}$ are contained inside the contour $\mathscr{C}_V$ and the others are left outside. 

Recall that $\mathcal{C}_V$ contains $0$, $-q$, $q \kappa \beta_2$, and $\Gamma_V$, but leaves outside $q \beta_1$; similarly, $\Gamma_V$ contains $0$, $q \kappa \beta_2$, $q \mathcal{C}_V$, but leaves outside $-q \kappa$. Thus, due to the fact that all contours are star-shaped, $\mathcal{C}_V$ and $\mathscr{C}_V$ contain and leave outside the same poles of $K_{\zeta}$; similarly, $Q_{R, d, \delta; w}$ and $\Gamma_V$ contain and leave outside the same poles of the integrand on the right side of \eqref{kvw}. Therefore, again since $\mathcal{C}_V$ and $\Gamma_V$ are star-shaped, it follows by moving along rays through the origin that $Q_{R, d, \delta; w}$ and $\mathscr{C}_V$ can be continuously and simultaneously deformed to $\Gamma_V$ and $\mathcal{C}_V$, respectively, without crossing any poles of any integrand on the right side of \eqref{kvw}. 

In particular, the contour $Q_{R, d, \delta; w}$ for $v$ can be replaced by $\Gamma_V$, and the contour $\mathscr{C}_V$ for $w$ can be replaced by $\mathcal{C}_V$; the latter statement holds by Remark \ref{determinantcontourdeform}. Thus, Proposition \ref{determinantgeneral} (after setting $b_2 = \widetilde{b}_2$ and $\vartheta = 0$) becomes  
\begin{flalign}
\label{t11w}
\big( \kappa \widetilde{\beta}_2 \beta_1^{-1}; q \big)_{\infty} \displaystyle\sum_{M = 0}^{\infty} \displaystyle\frac{(\kappa \widetilde{\beta}_2 \beta_1^{-1})^M }{(q; q)_M}  \displaystyle\sum_{\nu \in \Sign^+} \overline{\mathcal{W}_x}^{(M)} (\nu) \mathbb{E}^{(\nu)} & \left[ \displaystyle\frac{1}{\big( \zeta q^{\mathfrak{H} (x, t) - M	} ; q \big)_{\infty}} \right]  = \det \big( \Id + V_{\zeta} \big)_{L^2 (\mathcal{C}_V)}, 
\end{flalign}

\noindent for $|\widetilde{b_2}| < c$ sufficiently small. We claim that both sides of \eqref{t11w} are analytic in $\widetilde{\beta}_2$ over all $\widetilde{\beta}_2 \in \mathbb{C}$ satisfying $|\kappa \widetilde{\beta}_2| < \beta_1$ and $\Re \widetilde{\beta}_2 > 0$. 

Let us first consider the left side of \eqref{t11w}. First, the prefactor $\big( \kappa \widetilde{\beta}_2 \beta_1^{-1}; q \big)_{\infty}$ is analytic in $\widetilde{\beta}_2$ on this domain. Next, observe that $\mathbb{E}^{(\nu)} \big[ ( \zeta q^{\mathfrak{H} (x, t) - M	} ; q )_{\infty}^{-1} \big] $ is independent of $b_2$ and is bounded uniformly in $M \in \mathbb{Z}_{\ge 0}$. Furthermore, each weight $\overline{\mathcal{W}_x}^{(M)} (\nu)$ is a polynomial in $b_2$, independent of $M$ (since $\vartheta = 0$), and non-zero for only finitely many signatures $\nu$. Hence the second sum on the left side of \eqref{t11w} is a polynomial in $\widetilde{b_2}$ whose coefficients are uniformly bounded in $M \ge 0$. Therefore, due to the factor of $(\kappa \widetilde{\beta}_2 \beta_1^{-1})^M$, the first sum on the left side of \eqref{t11w} is a convergent power series (and thus analytic) in $\widetilde{\beta}_2$ on the domain $|\kappa \widetilde{\beta}_2| < \beta_1$. 

Next we consider the right side of \eqref{t11w}. Fix an arbitrary open subset $U \subset \mathbb{C}$ whose closure $\overline{U}$ lies entirely to the right of the imaginary axis and is contained in the interior of the circle centered at $0$ with radius $\kappa^{-1} \beta_1$. It suffices to show that the right side of \eqref{t11w} is analytic in $\widetilde{\beta}_2 \in U$. To that end, observe that there exist star-shaped contours $\mathcal{C}_V^{(U)}$ and $\Gamma_V^{(U)}$, both satisfying the restrictions of the proposition, such that $U$ is contained in the interior of both $\mathcal{C}_V^{(U)}$ and $\Gamma_V^{(U)}$. 

Since the Fredholm determinant $\det \big( \Id + V_{\zeta} \big)_{L^2 (\mathcal{C}_V)}$ does not depend on the choice of contours $\mathcal{C}_V$ and $\Gamma_V$ satisfying the restrictions of the proposition (due to Remark \ref{determinantcontourdeform}), we can use $\mathcal{C}_V^{(U)}$ and $\Gamma_V^{(U)}$ in place of $\mathcal{C}_V$ and $\Gamma_V$, respectively. Using these contours in \eqref{determinantvertexw} and \eqref{t}, we first find that $V_{\zeta} (w, w')$ is analytic in $\widetilde{\beta}_2 \in U$, since the integrand on the right side of \eqref{t} is analytic in $\widetilde{\beta}_2$ for all $v$ on the contour $\Gamma_V^{(U)}$. Thus, due to the absolute convergence of the sum in the definition \eqref{determinantsum} (which holds by compactness of $\mathcal{C}_V$; see the first part of Lemma \ref{determinantclosekernels}), we deduce that $\det \big( \Id + V_{\zeta} \big)_{L^2 (\mathcal{C}_V^{(U)})}$ is analytic in $\widetilde{\beta}_2 \in U$. In particular, $\det \big( \Id + V_{\zeta} \big)_{L^2 (\mathcal{C}_V)}$ is analytic in $\widetilde{\beta_2} \in U$. Since $U$ was arbitrary, this verifies that the right side of \eqref{t11w} is analytic.

Hence, \eqref{t11w} follows for all $\widetilde{\beta}_2 \in \mathbb{C}$ satisfying $|\kappa \widetilde{\beta}_2| < \beta_1$ and $\Re \widetilde{\beta}_2 > 0$, by uniqueness of analytic continuation.  

In particular, we may let $\widetilde{b}_2 = b_2 \in (0, 1)$, as in the statement of the theorem. In that case (since we already set $\vartheta = 0$), $\overline{\mathcal{W}_x}^{(M)} (\nu)$ is equal to the probability of selecting a signature $\nu$ chosen by appending each integer in the set $\{ 1, 2, \ldots , x \}$ independently with probability $b_2$, we deduce that 
\begin{flalign}
\label{expectationnu1}
\displaystyle\sum_{\nu \in \Sign^+} \overline{\mathcal{W}_x}^{(M)} (\nu) \mathbb{E}^{(\nu)} \left[ \displaystyle\frac{1}{\big( \zeta q^{\mathfrak{H} (x, t) - M	} ; q \big)_{\infty}} \right] &= \mathbb{E} \left[ \displaystyle\frac{1}{\big( \zeta q^{\mathfrak{H} (x, t) - M	} ; q \big)_{\infty}} \right], 
\end{flalign}

\noindent where $\mathbb{E}$ denotes the expectation with respect to the stochastic six-vertex model with double-sided $(b_1, b_2)$-Bernoulli initial data. Here, we have used the fact that the current $\mathfrak{H} (x, t)$ of the stochastic six-vertex model does not depend on any of the initial positions of any particles to the right of $x$. Now the theorem follows from inserting \eqref{expectationnu1} into \eqref{t11w}. 
\end{proof}

\subsubsection{Degeneration to the ASEP} 

\label{DegenerateExclusion}

As observed and used in a number of other works (see, for instance, \cite{CSSVMEP, PTAEPSSVM, SSVM, HSVMRSF, IPSVMSF}), the stochastic six-vertex model exhibits a limit degeneration to the ASEP. In this section, we apply this degeneration to Theorem \ref{determinantw} to obtain a Fredholm determinant identity for the ASEP with double-sided Bernoulli initial data.

The following proposition, which appears as Corollary 4 of \cite{CSSVMEP}, states this degeneration precisely in terms of convergence of currents. In what follows, we recall the definitions of the currents $J$ of the ASEP and $\mathfrak{H}$ of the stochastic six-vertex model from Section \ref{AsymmetricExclusions} and Section \ref{StochasticVertex}. 

\begin{prop}[{\cite[Corollary 4]{CSSVMEP}}]

\label{currentmodelprocess}
Fix real numbers $R, L \ge 0$ and $b_1, b_2 \in [0, 1]$. Let $\varepsilon > 0$ be a real number, and denote $\delta_1 = \delta_{1; \varepsilon} = \varepsilon L$ and $\delta_2 = \delta_{2; \varepsilon} = \varepsilon R$; assume that $\delta_1, \delta_2 \in (0, 1)$. Also fix $r \in \mathbb{R}$, $t \in \mathbb{R}_{> 0}$, and $x \in \mathbb{Z}$.

Let $p_{\varepsilon} (x; r) = \mathbb{P}_V \big[ \mathfrak{H} \big( x + \lfloor \varepsilon^{-1} t \rfloor, \lfloor \varepsilon^{-1} t \rfloor \big) \ge r \big]$, where the probability $\mathbb{P}_V$ is under the stochastic six-vertex model $\mathcal{P} (\delta_1, \delta_2)$ with $(b_1, b_2)$ double-sided Bernoulli initial data.\footnote{This result was actually originally stated for arbitrary initial data, but for simplicity we restrict to the case of double-sided Bernoulli initial data.} Furthermore, let $p (x; r) = \mathbb{P}_A \big[ J_t \big( x \big) \ge r \big]$, where the probability $\mathbb{P}_A$ is under the ASEP with left jump rate $L$, right jump rate $R$, and with $(b_1, b_2)$ double-sided Bernoulli initial data. 

Then, $\lim_{\varepsilon \rightarrow 0} p_{\varepsilon} (x; r) = p (x; r)$. 

\end{prop}

 Applying the degeneration explained above, we deduce the following Fredholm determinant identity for the current of the ASEP. 

\begin{thm}
\label{determinant}

Fix $R > L > 0$; $0 < b_2 < b_1 < 1$; $x \in \mathbb{Z}$; and $t \in \mathbb{R}_{> 0}$. Consider the ASEP with left jump rate $L$, right jump rate $R$, and double-sided $(b_1, b_2)$-Bernoulli initial data. Let $q = L / R < 1$, $\beta_i = b_i / (1 - b_i)$, for each $i \in \{ 1 , 2 \}$. Denote 
\begin{flalign*}
g_A (z; x, t) = (z + q)^x \exp \left( \displaystyle\frac{tq (R - L)}{z + q} \right) \displaystyle\frac{(q \beta_2 z^{-1}; q)_{\infty}}{(q^{-1} \beta_1^{-1} z; q)_{\infty}}. 
\end{flalign*}

\noindent Let $\Gamma_A \subset \mathbb{C}$ be a positively oriented, star-shaped contour in the complex plane containing $0$ and $q \beta_2$, but leaving outside $-q$ and $q \beta_1$. Let $\mathcal{C}_A \subset \mathbb{C}$ be a positively oriented, star-shaped contour contained inside $q^{-1} \Gamma_A$; that contains $0$, $-q$, $q \beta_2$, and $\Gamma_A$; but that leaves outside $q \beta_1$. 

Let $\zeta = -q^p < 0$, for some real number $p \in \mathbb{R}$. Then, 
\begin{flalign}
\label{determinantprocess}
( \beta_2 \beta_1^{-1}; q )_{\infty} \displaystyle\sum_{M = 0}^{\infty} \displaystyle\frac{\big( \beta_2 \beta_1^{-1} \big)^M}{(q; q)_M} \mathbb{E} \left[ \displaystyle\frac{1}{(\zeta q^{J_T (x) - M}; q)_{\infty}} \right] = \det \big( \Id + A_{\zeta} \big)_{L^2 (\mathcal{C}_A)}, 
\end{flalign}

\noindent where
\begin{flalign}
\label{kernelp} 
A_{\zeta} (w, w') = \displaystyle\frac{1}{2 \textbf{\emph{i}} \log q} \displaystyle\sum_{j = - \infty}^{\infty} \displaystyle\int_{\Gamma_A} \displaystyle\frac{v^{p - 1} w^{-p} dv }{\sin \big( \frac{\pi}{\log q} (\log v - \log w + 2 \pi \textbf{\emph{i}}j ) \big) (w' - v)} \displaystyle\frac{g_A (w; x, T)}{g_A (v; x, T)}. 
\end{flalign}

\end{thm}

\begin{rem}
	
	\label{b20particles} 
	
	Similar to what was mentioned in Remark \ref{b20vertex}, the $b_2 = 0$ specialization of Theorem \ref{determinant} can be seen to coincide with the $m = 1$ and $K = A$ case of Proposition 5.1 in \cite{PTAEPSSVM}. 
	
\end{rem}

\begin{proof}[Proof of Theorem \ref{determinant}]

This theorem will follow from taking the ASEP degeneration in Theorem \ref{determinantw}. In what follows, let $\varepsilon > 0$ be some positive real number, $\delta_1 = \varepsilon L$, $\delta_2 = \varepsilon R$, and $t = \lfloor \varepsilon^{-1} T \rfloor$. 

As $\varepsilon$ tends to $0$, Proposition \ref{currentmodelprocess} shows that the distribution of the current $\mathfrak{H} (x + t, t)$ of the stochastic six-vertex model $\mathcal{P} (\delta_1, \delta_2)$ converges to the distribution of the current $J_T (x)$ of the ASEP. Therefore, since $\big( \zeta q^{J_T (x)}; q \big)_{\infty}^{-1}$ is uniformly bounded, $\big( \zeta q^{J_T (x)}; q^{-1} \big)_M^{-1}$ is uniformly bounded, and $\beta_2 \beta_1^{-1} < 1$, we deduce that the left side of \eqref{determinantvertexw} converges to the left side of \eqref{determinantprocess}, as $\varepsilon$ tends to $0$. 

Moreover, it is quickly verified that 
\begin{flalign*}
\displaystyle\lim_{\varepsilon \rightarrow 0} g_V (z; x + t, t) & = e^{T (L - R)} g_A (z; x, T). 
\end{flalign*}

\noindent Hence, $V_{\zeta} (w, w')$ tends to $A_{\zeta} (w, w')$ as $\varepsilon$ tends to $0$, for all $w, w' \in \mathcal{C}_A$. 

Now, as $\varepsilon$ tends to $0$, $\kappa$ tends to $1$. Therefore, if $\varepsilon$ is sufficiently small (in comparison to $R$, $\delta$, and $d$), we may take the contour $\mathcal{C}_V$ from Theorem \ref{determinantprocess} to be the contour $\mathcal{C}_A$. 

Since $\mathcal{C}_A$ is compact, the fact that $\lim_{\varepsilon \rightarrow 0} V_{\zeta} (w, w') = A_{\zeta} (w, w')$ implies that the right side of \eqref{determinantvertexw} converges to the right side of \eqref{determinantprocess} as $\varepsilon$ tends to $0$ (for instance, by Lemma \ref{determinantlimitkernels}). 

Thus, \eqref{determinantprocess} is a limit degeneration of \eqref{determinantvertexw}, and the proof is complete. 
\end{proof}

\begin{rem}

\label{smallasymmetry}

Taking the \emph{weakly asymmetric limit} (see \cite{SEPS}) of Theorem \ref{determinant} should lead to a Fredholm determinant identity for the solution to the KPZ equation with two-sided Brownian initial data; this might yield a new proof of Theorem 2.9 of \cite{HFSE}. 

\end{rem}

\subsection{Proof of Proposition \ref{determinantgeneralformal}}

\label{ProofFormalDeterminant} 

In this section we establish the Fredholm determinant identity Proposition \ref{determinantgeneralformal}. This will mainly follow what was done in Section 3 and Section 5 of \cite{DDQA} (which was based on earlier work from Section 3.2.2 of \cite{MP}), but there are two main obstructions that prevent us from applying that work directly. First, we have a parameter $b_2$ that is formal instead of complex. Second, our $q$-moment identities given by Corollary \ref{heightgeneral} are only stated for $\vartheta = q^{-J}$ for some positive integer $J$; recall that $\vartheta \in \mathbb{C}$ could be arbitrary in Proposition \ref{determinantgeneralformal}. 

Thus, we will first establish Proposition \ref{determinantgeneralformal} when $\vartheta = q^{-J}$ for some positive integer $J$; this will essentially follow the works of \cite{MP, DDQA}, adapted to the formal setting, which we will explain in Section \ref{munegativej}. Then, in Section \ref{mucomplex}, we will show how to analytically continue this result from $\vartheta \in \{ q^{-1}, q^{-2}, \ldots \}$ to arbitrary $\vartheta \in \mathbb{C}$.

\subsubsection{The Case \texorpdfstring{$\vartheta = q^{-J}$}{}}

\label{munegativej}

The goal of this section is to establish Corollary \ref{formaldeterminantqmuj}, which is Proposition \ref{determinantgeneralformal} restricted to the case when $\vartheta = q^{-J}$, for some positive integer $J$. 

We will do this by first considering the case when $b_2$ is complex, instead of formal, but small (in a sense dependent on $J$); this is given by Proposition \ref{determinant1formal}. We will then deduce Corollary \ref{formaldeterminantqmuj} from Proposition \ref{determinant1formal} through contour integration. 

To establish Proposition \ref{determinant1formal}, we require the following lemma that rewrites (and analytically continues) the right side of \eqref{qmsequation}. 

\begin{lem}

\label{momentsumformal}

Adopt the notation of Proposition \ref{determinantgeneral}, assume that $\vartheta = q^{-J}$ for some positive integer $J$, and assume that	 $b_2 \in \mathbb{C}$ is sufficiently small so that $|\kappa \vartheta \beta_2| < \inf_{z \in \mathscr{C}_V} |z| $. Let
\begin{flalign}
\label{fdoublesided}
f (z) = \left( \displaystyle\frac{1 + z}{1 + q^{-1} z} \right)^t \left( \displaystyle\frac{1 + q^{-1} \kappa^{-1} z}{1 + \kappa^{-1} z} \right)^x \displaystyle\frac{1}{1 - q^{-1} \beta_1^{-1} z} \displaystyle\frac{1 - \vartheta \kappa \beta_2 z^{-1}}{1 - \kappa \beta_2 z^{-1}},
\end{flalign}

\noindent and let $\textbf{\emph{m}}_k$ be equal to \eqref{current1}. If $\mathscr{C}_V$ is as in Proposition \ref{determinantgeneralformal}, then 
\begin{flalign}
\label{mkcontourformal}
\textbf{\emph{m}}_k = (q; q)_k \displaystyle\sum_{|\lambda| = k} \displaystyle\frac{1}{(2 \pi \textbf{\emph{i}})^{\ell (\lambda)} \prod_{j = 1}^{\infty} m_j!} \displaystyle\oint \cdots \displaystyle\oint \det \left[ \displaystyle\frac{1}{w_i - q^{\lambda_j} w_j} \right]_{i, j = 1}^{\ell (\lambda)} \displaystyle\prod_{i = 1}^{\ell (\lambda)} d w_i \displaystyle\prod_{j = 0}^{\lambda_i - 1} f(q^j w_i),
\end{flalign} 

\noindent where each $w_i$ is integrated along the contour $\mathscr{C}_V$. Here, $\lambda = (\lambda_1, \lambda_2, \ldots )$ is summed over all partitions of size $k$, and $\ell (\lambda)$ denotes the length of $\lambda$. 

\end{lem}

\begin{proof}

In the case when $|\beta_2|$ is sufficiently small so that $|\kappa \beta_2| < q^{k + 1}$ and $|\kappa \beta_2| < q^{k + 1} \beta_1$, this lemma is a direct consequence of Corollary \ref{heightgeneral} and Lemma \ref{deformcontoursmoments}; we apply the latter lemma with choice of parameters $v^{-1} = \kappa \beta_2$; $U = ( - q^{-1}, - q^{-1}, \ldots )$; $\Xi = ( 1, 1, \ldots ) $; $S = ( q \beta_1, - \kappa, - \kappa, \ldots )$; $\mathscr{C} = \mathscr{C}_V$; and $f$ given by \eqref{fdoublesided}. 

Thus, it suffices to establish the lemma under the hypothesis $|\kappa \vartheta \beta_2| < \inf_{z \in \mathscr{C}_V} |z|$, which is independent of $k$. We will do this by showing both sides of \eqref{mkcontourformal} are analytic in $\beta_2$ on the domain in which $|\kappa \vartheta \beta_2| < \inf_{z \in \mathscr{C}_V} |z|$. 

Let us first examine the right side, which is a finite sum of integrals along the compact contour $\mathscr{C}_V$. Thus, it suffices to show that each integrand of this integral is nonsingular, for all $w_1, w_2, \ldots , w_{\ell (\lambda)} \in \mathscr{C}$ and for all $\beta_2$ satisfying $|\kappa \vartheta \beta_2| < q^{-1} \inf_{z \in \mathscr{C}_V} |z|$. Since $\mathscr{C}_V$ is a circle containing $0$, it contains $q \mathscr{C}_V$ in its interior. Hence, each factor $w_i - q^{\lambda_j} w_j$ is non-zero. 

Furthermore, since $\vartheta = q^{-J}$, we find from the definition \eqref{fdoublesided} of $f$ that the product $\prod_{j = 0}^{\lambda_i - 1} f(q^j z)$ does not have a pole at $z = q^{-r} \kappa \beta_2$ for any integer $r \ge J$. Thus, $\prod_{j = 0}^{\lambda_i - 1} f(q^j w_i)$ remains nonsingular for each $w_i \in \mathscr{C}_V$ and $\beta_2$ satisfying $|\kappa \vartheta \beta_2| < \inf_{z \in \mathscr{C}_V} |z|$. This establishes that the right side of \eqref{mkcontourformal} is analytic. 

Next, we examine the left side of \eqref{mkcontourformal}, which is \eqref{current1}. First, observe that there are only finitely may signatures $\nu$ for which $\overline{\mathcal{W}_x}^{(M)} (\nu)$ is non-zero; furthermore, since $\vartheta = q^{-J}$, there are only finitely many $M$ for which $(\vartheta; q)_M$ is non-zero. Thus, both sums in \eqref{current1} are finite. Moreover, since each summand in \eqref{current1} is analytic in $\beta_2$, and since the quotient $(\kappa \beta_2 \beta_1^{-1}; q)_{\infty} (\vartheta \kappa \beta_2 \beta_1^{-1}; q)_{\infty}^{-1}$ is analytic in $\beta_2$ in the domain in which $|\kappa \vartheta \beta_2| < \inf_{z \in \mathscr{C}_V} |z| < |\beta_1|$, we deduce that \eqref{current1}, and thus the left side of \eqref{mkcontourformal}, is analytic in $\beta_2$. 

Thus, the lemma follows from analytic continuation. 
\end{proof}

\noindent The following proposition uses Lemma \ref{momentsumformal} to produce a Fredholm determinant identity in the case $b_2 \in \mathbb{C}$ and $\vartheta \in \{ q^{-1}, q^{-2}, \ldots \}$. 

\begin{prop}

\label{determinant1formal}

Adopt the notation of Proposition \ref{determinantgeneral}, assume that $\vartheta = q^{-J}$ for some positive integer $J$, and that $b_2 \in \mathbb{C}$ satisfies $|\kappa \vartheta \beta_2| < q^R \inf_{z \in \mathscr{C}} |z|$. Then, \eqref{determinantformalvertex} holds for any $\zeta \in \mathbb{C} \setminus\mathbb{R}_{\ge 0}$. 

\end{prop}

\begin{proof}

In what follows, we adopt the notation of Lemma \ref{momentsumformal}. To establish this proposition, we would like to apply Lemma \ref{momentdeterminant} in which $f(z)$ is given by \eqref{fdoublesided} and $g_V$ satisfies $f(z) = g_V (z) / g_V (qz)$ and is explicitly given by
\begin{flalign}
\label{gvqj}
g_V (z; x, t) = \displaystyle\frac{(\kappa^{-1} z + q)^x}{(z + q)^t (q^{-1} \beta_1^{-1} z; q)_{\infty}} \displaystyle\prod_{h = 0}^{J - 1} \displaystyle\frac{1}{1 - q^{-h} \beta_2 \kappa z^{-1}}. 
\end{flalign}

\noindent To apply this lemma, we must verify its three conditions. 

First, we require the existence of the constant $B > 0$ satisfying $\big| f (q^n w) \big| < B$ and $|q^n w - w'| > B^{-1}$ for all $w, w' \in \mathscr{C}_V$ and integers $n \ge 0$. This follows from the compactness and convexity of $\mathscr{C}_V$, as well as from the fact that $q^n w$ is not a pole of $f$, for any $w \in \mathscr{C}_V$ and non-negative integer $n$ (the latter statement can be checked directly from the definitions \eqref{fdoublesided} of $f$ and the contour $\mathscr{C}_V$). 

Second, we require that $R, d, \delta$ satisfy the estimates \eqref{convergencedeterminant2}; this can be quickly deduced from the definition \eqref{gvqj} of $g_V$, the definition and compactness of the contour $\mathscr{C}_V$, and the inequalities \eqref{doublesidedinequalities}. 

Third, we require that $z = q^s w$ is not a pole of $g_V (z)$, for any $w \in \mathscr{C}_V$ and $s \in \mathbb{C}$ to the right of $D_{R, d, \delta}$. In view of \eqref{gvqj}, the poles of $g_V (z)$ are at $z \in \{ -q, \kappa \beta_2, q^{-1} \kappa \beta_2, \ldots , q^{1 - J} \kappa \beta_2 \} \cup \bigsqcup_{j = 0}^{\infty} \{ q^{1 - j} \beta_1 \}$. The fact that $q^s w \notin \{ -q \} \cup \bigsqcup_{j = 0}^{\infty} \{ q^{1 - j} \beta_1 \}$ follows quickly from the estimates \eqref{doublesidedinequalities} and the definition of the contour $\mathscr{C}_V$. The fact that $q^s w \notin  \{ -q, \kappa \beta_2, q^{-1} \kappa \beta_2, \ldots , q^{1 - J} \kappa \beta_2 \}$ follows from \eqref{doublesidedinequalities} and the additional assumption $|\vartheta \kappa \beta_2| < q^R \inf_{z \in \mathscr{C}_V} |z|$. This verifies the third condition of Lemma \ref{momentdeterminant}. 	

Thus, Lemma \ref{momentdeterminant} applies and implies that $\det \big( \Id + K_{\zeta} \big)_{L^2 (\mathscr{C}_V)}$ is analytic in $\zeta \in \mathbb{C} \setminus\mathbb{R}_{\ge 0}$. Furthermore, it yields 
\begin{flalign}
\label{determinantkzetaw}
\det \big( \Id + K_{\zeta} \big)_{L^2 (\mathscr{C}_V)} & = \displaystyle\sum_{k = 0}^{\infty} \displaystyle\frac{\zeta^k \textbf{m}_k}{(q; q)_k},  
\end{flalign}

\noindent for any $\zeta \in \mathbb{C} \setminus\mathbb{R}_{\ge 0}$ such that $|\zeta| < B^{-1} (1 - q)^{-1}$. 

Since $\textbf{m}_k$ is equal to \eqref{current1}, we deduce from \eqref{determinantkzetaw} that 
\begin{flalign}
\label{determinantkzetaw1} 
\begin{aligned}
\det \big( \Id + K_{\zeta} \big)_{L^2 (\mathscr{C}_V)} & = \displaystyle\frac{(\kappa \beta_2 \beta_1^{-1}; q)_{\infty}}{(\vartheta \kappa \beta_2 \beta_1^{-1}; q)_{\infty}}\displaystyle\sum_{M = 0}^{\infty} \left( \displaystyle\frac{\kappa \beta_2}{\beta_1} \right)^M \displaystyle\frac{(\vartheta; q)_M}{(q; q)_M} \\
& \qquad \qquad \qquad \qquad \times \displaystyle\sum_{\nu \in \Sign^+} \overline{\mathcal{W}_x}^{(M)} (\nu) \displaystyle\sum_{k = 0}^{\infty} \mathbb{E}^{(\nu)} \left[ \displaystyle\frac{(\zeta q^{\mathfrak{H} (x, t) - M})^k}{(q; q)_k}  \right].
\end{aligned}
\end{flalign}

Now, since $\vartheta = q^{-J}$, we have that $(\vartheta; q)_m = 0$ for $m > J$. Therefore, the first sum on the right side of \eqref{determinantkzetaw} can instead range over $M \in [0, J]$. Thus, if we moreover assume that $|\zeta| < q^{t + J}$, then $|\zeta q^{\mathfrak{H} (x, t) - M}| < 1$ (here, we are using the fact that $\mathfrak{H} (x, t) \ge -t$), meaning that we can commute the expectation with the sum over $k$. Then, the $q$-binomial theorem yields 
\begin{flalign}
\label{determinantkzetaw2}
\begin{aligned}
\det \big( \Id + K_{\zeta} \big)_{L^2 (\mathscr{C}_V)} & = \displaystyle\frac{(\kappa \beta_2 \beta_1^{-1}; q)_{\infty}}{(\vartheta \kappa \beta_2 \beta_1^{-1}; q)_{\infty}}\displaystyle\sum_{M = 0}^{\infty} \left( \displaystyle\frac{\kappa \beta_2}{\beta_1} \right)^M \displaystyle\frac{(\vartheta; q)_M}{(q; q)_M} \\
& \qquad \qquad \qquad \qquad \times \displaystyle\sum_{\nu \in \Sign^+} \overline{\mathcal{W}_x}^{(M)} (\nu) \mathbb{E}^{(\nu)} \left[ \displaystyle\frac{1}{(\zeta q^{\mathfrak{H} (x, t) - M}; q)_{\infty}} \right].
\end{aligned}
\end{flalign}

\noindent This implies the proposition when $|\zeta| < \min \{ q^{t + J}, B^{-1} (1 - q)^{-1} \}$. 

To establish the lemma for general $\zeta \in \mathbb{C} \setminus\mathbb{R}_{\ge 0}$, observe that both sides of \eqref{determinantkzetaw2} are analytic in $\zeta \in \mathbb{C} \setminus\mathbb{R}_{\ge 0}$. Indeed, this is true for the left side by the second statement of Lemma \ref{momentdeterminant}. This is also true for the right side since $(\zeta q^m; q)_{\infty}$ is analytic in $\zeta \in \mathbb{C} \setminus\mathbb{R}_{\ge 0}$, for any integer $m$, and since both sums of the right side of \eqref{determinantkzetaw2} are finite. Thus, the lemma follows for all $\zeta \in \mathbb{C} \setminus\mathbb{R}_{\ge 0}$ by uniqueness of analytic continuation. 
\end{proof}

\noindent Now, we can show that Proposition \ref{determinantgeneralformal} holds when $\vartheta = q^{-J}$ for some positive integer $J$. 

\begin{cor}

\label{formaldeterminantqmuj}

Adopt the notation of Proposition \ref{determinantgeneralformal}. If $\vartheta = q^{-J}$ for some positive integer $J$, then \eqref{determinantformalvertex} holds. 

\end{cor}

\begin{proof}

Adopt the notation of Lemma \ref{coefficientb2}. Then, left side of \eqref{determinantformalvertex} is equal to $\sum_{j = 0}^{\infty} E_j \beta_2^j$, and the right side is equal to $\sum_{j = 0}^{\infty} H_j \beta_2^j$, where $|E_j|, |H_j| < C^j$, for each integer $j \ge 1$. We would like to show that $E_k = H_k$, for each non-negative integer $k$, under the assumption that $\vartheta = q^{-J}$. 

Now, let $\overline{b}_2 \in \mathbb{C}$ be a complex number, and let $\overline{\beta}_2 = \overline{b}_2 / (1 - \overline{b}_2)$. Assume that $\overline{b}_2$ is sufficiently small so that $|\overline{\beta}_2| < \min \{ C^{-1}, \kappa^{-1} \vartheta^{-1} \beta_1, \kappa^{-1} \vartheta^{-1} q^{R + 1}\}$. 

Then, Proposition \ref{determinant1formal} applies, and \eqref{determinantformalvertex} holds with $b_2$ and $\beta_2$ replaced with $\overline{b}_2$ and $\overline{\beta}_2$, respectively. By Lemma \ref{coefficientb2}, both sides of \eqref{determinantformalvertex} admit a power series expansion in $\overline{\beta}_2$; since $|\overline{\beta}_2| < C^{-1}$, this power series expansion coincides with the power series expansion given by Lemma \ref{coefficientb2}. Thus, we deduce that $\sum_{k = 0}^{\infty} E_k \overline{\beta}_2^k = \sum_{k = 0}^{\infty} H_k \overline{\beta}_2^k$. 

Since this holds for all complex numbers $\overline{b}_2$ satisfying $|\overline{\beta}_2| < \min \{ C^{-1}, \kappa^{-1} \vartheta^{-1} \beta_1, \kappa^{-1} \vartheta^{-1} q^{R + 1} \}$, we can integrate $\overline{\beta}_2$ along a small circle centered at the origin. Specifically, we find that 
\begin{flalign}
\label{ekhkmuj}
E_k = \displaystyle\frac{1}{2 \pi \textbf{i}} \displaystyle\oint \overline{\beta}_2^{-k - 1} \left( \displaystyle\sum_{j = 0}^{\infty} E_j \overline{\beta}_2^j \right) d \overline{\beta}_2 = \displaystyle\frac{1}{2 \pi \textbf{i}} \displaystyle\oint \overline{\beta}_2^{-k - 1} \left( \displaystyle\sum_{j = 0}^{\infty} F_j \overline{\beta}_2^j \right) d \overline{\beta}_2 = H_k, 
\end{flalign}

\noindent where the $\overline{\beta}_2$ is integrated along a positively oriented circle centered at the origin with radius less than $\min \{ C^{-1}, \kappa^{-1} \vartheta^{-1} \beta_1, \kappa^{-1} \vartheta^{-1} q^{R + 1} \}$. The corollary now follows from \eqref{ekhkmuj}. 
\end{proof}

\subsubsection{Analytic Continuation in \texorpdfstring{$\vartheta$}{}} 

\label{mucomplex}

In the previous section, we verified Proposition \ref{determinantgeneralformal} when $\vartheta$ is a negative integer power of $q$. Our goal in this section is to establish Proposition \ref{determinantgeneralformal} for all $\vartheta \in \mathbb{C}$. To do this, it suffices to show that $E_k = H_k$, for each integer $k \ge 0$, where $E_k$ and $H_k$ are defined by Lemma \ref{coefficientb2}. 

Recall that $E_k$ and $H_k$ are dependent on the parameters $\vartheta$, $\delta_1$, $\delta_2$, $b_1$, $\zeta$, $t$, and $x$. We will emphasize their dependence on $\vartheta$ by denoting $E_k = E_k (\vartheta)$ and $H_k = H_k (\vartheta)$. 

By Corollary \ref{formaldeterminantqmuj}, we have that $E_k (\vartheta) = H_k (\vartheta)$ when $\vartheta$ is a negative integer power of $q$. We would like to analytically extend this result to all $\vartheta \in \mathbb{C}$. To accomplish this, we will show that $E_k (\vartheta)$ and $H_k (\vartheta)$ are polynomials in $\vartheta$ (fixing all other parameters $\zeta$, $\delta_1$, $\delta_2$, $b_1$, $t$, and $x$) through the following lemma. 

\begin{lem}

\label{coefficientbetapolynomial}

Adopt the notation of Proposition \ref{determinantgeneralformal} and Corollary \ref{coefficientb2}. Fix the positive integers $x$ and $t$, positive real numbers $\delta_1, \delta_2, b_1 \in (0, 1)$, and complex number $\zeta \in \mathbb{C} \setminus\mathbb{R}_{\ge 0}$. Then, $E_k (\vartheta)$ and $H_k (\vartheta)$ are polynomials in $\vartheta$.  
\end{lem}

\begin{proof}
Let us begin with $E_k (\vartheta)$, so consider the left side of \eqref{determinantformalvertex}. Let 
\begin{flalign*}
S_2^{(M)} = \displaystyle\frac{\kappa^M \beta_1^{-M} }{(q; q)_M} \displaystyle\sum_{\nu \in \Sign^+} \overline{\mathcal{W}_x}^{(M)} (\nu) \mathbb{E} \left[ \displaystyle\frac{1}{\big( \zeta q^{\mathfrak{H} (x, t) - M}; q \big)_{\infty}} \right]; \qquad S_1 = \displaystyle\sum_{M = 0}^{\infty} \beta_2^M (\vartheta; q)_M S_2^{(M)}. 
\end{flalign*}

\noindent In particular, the left side of \eqref{determinantformalvertex} is equal to $(\kappa \beta_2 \beta_1^{-1}; q)_{\infty} (\vartheta \kappa \beta_2 \beta_1^{-1}; q)_{\infty}^{-1} S_1$. 

Now, the expectation $\mathbb{E}^{(\nu)} \big[ (\zeta q^{\mathfrak{H} (x, t) - M}; q)_{\infty}^{-1} \big]$ does not depend on $\vartheta$. Furthermore, the weight $\overline{\mathcal{W}_x}^{(M)} (\nu)$ is a polynomial in $\vartheta$ and non-zero for only finitely many signatures $\nu$. Thus, $S_2^{(M)}$ is a polynomial in $\vartheta$, for each non-negative integer $M$. Since $(\vartheta; q)_M$ is a polynomial in $\vartheta$, it follows that the coefficient of $\beta_2^M$ in $S_1$ is a polynomial in $\vartheta$. 

Furthermore, expanding $(\kappa \beta_2 \beta_1^{-1}; q)_{\infty} (\vartheta \kappa \beta_2 \beta_1^{-1}; q)_{\infty}^{-1}$ as a power series in $\beta_2$, we deduce that its coefficient of $\beta_2^j$ is a polynomial in $\vartheta$, for each $j \ge 0$. Therefore, the coefficient of $\beta_2^j$ in $(\kappa \beta_2 \beta_1^{-1}; q)_{\infty} (\vartheta \kappa \beta_2 \beta_1^{-1}; q)_{\infty}^{-1} S_1$ is also a polynomial in $\vartheta$, meaning that $E_j (\vartheta)$ is a polynomial in $\vartheta$. 

Let us verify the same statement for $H_j (\vartheta)$, so consider the right side of \eqref{determinantformalvertex}. We will first show that the coefficient of $\beta_2^j$ in $K_{\zeta}$ is a polynomial in $\vartheta$, for each $j \ge 0$. The integral defining $K_{\zeta}$ only exhibits dependence on $\beta_2$ through the quotient $g_V (w; x, t) g_V (q^r w; x, t)^{-1}$, which is equal to 
\begin{flalign*}
\displaystyle\frac{g_V (w; x, t)}{g_V (q^r w; x, t)} = \left( \displaystyle\frac{\widetilde{g}_V (w; x, t)}{\widetilde{g}_V (q^r w; x, t)} \right) \displaystyle\frac{(\beta_2 q \kappa w^{-1}; q)_{\infty} (\vartheta \beta_2 q^{1 - r} \kappa w^{-1}; q)_{\infty}}{(\beta_2 q^{1 - r} \kappa w^{-1}; q)_{\infty} (\vartheta\beta_2 q \kappa w^{-1}; q)_{\infty}},
\end{flalign*}

\noindent where $\widetilde{g_V} (z; x, t)$ is $g_V (z; x, t)$ with $\beta_2$ set to $0$. 

In particular, $g_V (w; x, t) g_V (q^r w; x, t)^{-1}$ is dependent on $\beta_2$ only through the term 
\begin{flalign}
\label{qproduct}
\displaystyle\frac{( q^{1 - r}\kappa \beta_2 w^{-1}; q)_{\infty} ( q \vartheta \kappa \beta_2 w^{-1}; q)_{\infty}}{( q \kappa \beta_2 w^{-1}; q)_{\infty}( q^{1 - r}\vartheta \kappa \beta_2 w^{-1}; q)_{\infty}}. 
\end{flalign}

Expanding \eqref{qproduct} as a power series in $\beta_2$, we find that its coefficient of $\beta_2^j$ is a polynomial in $\vartheta$ with coefficients in $\mathbb{C} [\kappa, q, q^{-r}, w^{-1}]$. Since $\mathscr{C}_V$ is compact and also since $\widetilde{g_V} (w; x, t) \widetilde{g_V} (q^r w; x, t)^{-1}$ is nonsingular for $w \in \mathscr{C}_V$ and $r \in D_{R, d, \delta}$, the coefficient $\beta_2^j$ in $g_V (w; x, t) g_V (q^r w; x, t)^{-1}$ is a polynomial in $\vartheta$ with uniformly bounded coefficients over $w \in \mathscr{C}_V$ and $r \in D_{R, d, \delta}$. Furthermore, since setting $\beta_2 = 0$ removes the dependence of the integrand on $\vartheta$, we also observe that this polynomial is constant (independent of $\vartheta$) when $j = 0$. 

Due to the exponential decay of $\sin (\pi r)$ as $r$ ranges across the contour $D_{R, d, \delta}$, this integrand can be integrated, so the coefficient of $\beta_2^j$ in $K_{\zeta} (w, w')$ is a polynomial in $\vartheta$ (and again constant when $j = 0$). The degree of this polynomial is independent of $w, w' \in \mathscr{C}_V$. Therefore, for any positive integer $k$ and complex numbers $w_1, w_2, \ldots , w_k \in \mathscr{C}_V$, the coefficient of $\beta_2^j$ in $\det \big[ K (w_h, w_i)\big]_{h, i = 1}^k$ is a polynomial in $\vartheta$, whose degree is independent of the $w_i$ and sufficiently large $k$. 

Hence, from the definition \eqref{determinantsum} that expresses a Fredholm determinant as an infinite sum (which converges absolutely by Lemma \ref{coefficientb2}), we deduce that the coefficient of $\beta_2^k$ in $\det \big( \Id + K_{\zeta} \big)_{L^2 (\mathscr{C}_V)}$ is a polynomial in $\vartheta$. Hence, $H_k (\vartheta)$ is a polynomial in $\vartheta$. 
\end{proof}

\noindent Using Corollary \ref{formaldeterminantqmuj} and Lemma \ref{coefficientbetapolynomial}, we can establish Proposition \ref{determinantgeneralformal}. 

\begin{proof}[Proof of Proposition \ref{determinantgeneralformal}]

Adopt the notation of Lemma \ref{coefficientb2}, so the left and right sides of \eqref{determinantformalvertex} are equal to $\sum_{j = 0}^{\infty} E_j (\vartheta) \beta^j$ and $\sum_{j = 0}^{\infty} H_j (\vartheta) \beta^j$, respectively. By Lemma \ref{coefficientbetapolynomial}, $E_j (\vartheta)$ and $H_j (\vartheta)$ are polynomials in $\vartheta$. Since Corollary \ref{formaldeterminantqmuj} shows that $E_j (\vartheta) = H_j (\vartheta)$ whenever $\vartheta$ is in the infinite set $\{ q^{-1}, q^{-2}, \ldots \}$), we deduce that $E_j (\vartheta) = H_j (\vartheta)$ for all $\vartheta \in \mathbb{C}$; this implies the proposition. 
\end{proof}

\section{Preliminaries on Asymptotic Analysis} 

\label{StationaryReduction}

Theorem \ref{determinant} and Theorem \ref{determinantw} provide Fredholm determinant identities for the ASEP and stochastic six-vertex model with double-sided $(b_1, b_2)$-Bernoulli initial data, respectively. Our goal for the remainder of this paper will be to use these identities to produce asymptotics. 

Observe that the set of $b_1$ and $b_2$ for which these theorems apply is restricted; for the ASEP it holds when $b_1 > b_2$, and for the stochastic six-vertex model it holds when $\beta_1 > \kappa \beta_2$. The stationary (or translation-invariant) cases of interest to us are when $b_1 = b_2$ (for the ASEP) or when $\beta_1 = \kappa \beta_2$ (for the stochastic six-vertex model). This is outside the domain of applicability of \eqref{determinantprocess} and \eqref{determinantvertexw}. 

Thus, we instead analyze these identities near the boundary of their domains. For instance, for the ASEP, we take $b_1 > b_2$ but $b_1 - b_2$ very small, and then study the asymptotics for both sides of \eqref{determinantprocess}.  

In Section \ref{LeftStationary} we establish the asymptotics for the left sides of \eqref{determinantprocess} and \eqref{determinantvertexw} in these regimes. In Section \ref{ProofCurrent} we state the asymptotics (to be proven later) for the right sides of these identities. In Section \ref{NearInitial} we explain how these results can be used to establish Theorem \ref{currentfluctuations} and Theorem \ref{processfluctuations}.

\subsection{Asymptotics of the Left Side} 

\label{LeftStationary}

 Observe that the left sides of \eqref{determinantprocess} and \eqref{determinantvertexw} are sums of the similar form, 
\begin{flalign}
\label{leftsum1}
(\omega; q)_{\infty} \displaystyle\sum_{M = 0}^{\infty} \displaystyle\frac{\omega^M}{(q; q)_M} \mathbb{E} \left[ \displaystyle\frac{1}{ ( \zeta q^{J - M}; q )_{\infty}} \right]. 
\end{flalign} 

\noindent For the ASEP, we set $\omega = \beta_2 \beta_1^{-1}$ and $J = J_T (x)$ in \eqref{leftsum1}; similarly, for the stochastic six-vertex model, we set $\omega = \kappa \beta_2 \beta_1^{-1}$ and $J = \mathfrak{H} (x, t)$. Let us moreover set $\zeta = \zeta_{m, s; T} = -q^{ s T^{1 / 3} - mT}$; here, $m$ will be related to the hydrodynamical limit for the current. 

It will be useful to consider a normalization of \eqref{leftsum1}. Thus, we define 
\begin{flalign}
\label{homegas}
H_{\omega} (s) = H_{\omega; T} (m, s; J) = \displaystyle\frac{(q \omega ; q)_{\infty} }{T^{1 / 3}} \displaystyle\sum_{M = 0}^{\infty} \displaystyle\frac{\omega^M}{(q; q)_M } \left( \displaystyle\frac{1}{( - q^{s T^{1 / 3} - m T + J - M}; q )_{\infty}} \right), 
\end{flalign}

\noindent and we denote $H (s) =  H_1 (s)$. 

The following proposition states that the derivative $H' (s)$ should approximate $\textbf{1}_{J > m T - s T^{1 / 3}}$ as $T$ tends to $\infty$.

\begin{prop}
\label{linearsum}

Let $J \in \mathbb{Z}$; $T \in \mathbb{Z}_{> 1}$; $\omega = 1 - T^{-10}$; and $m, s \in \mathbb{R}$. We have that $\big| H_{\omega} (s) \big| < T^{10}$.

Furthermore, let $\varepsilon \in (0, 1)$, and denote $X_T = J - m T + s T^{1 / 3}$. If $|X_T| < T^5$, then  
\begin{flalign}
\label{homegalarget}
\left| \displaystyle\frac{H_{\omega} (s + \varepsilon) - H_{\omega} (s) }{\varepsilon} - \textbf{\emph{1}}_{X_T \ge 0} - \left( 1 + \displaystyle\frac{X_T}{\varepsilon T^{1 / 3}} \right) \textbf{\emph{1}}_{X_T \in (- \varepsilon T^{1 / 3}, 0)} \right| \le \displaystyle\frac{C}{\varepsilon T^{1 / 3}}, 
\end{flalign}

\noindent for some constant $C$ independent of $J$, $\varepsilon$, and $T$. 

\end{prop}

\begin{proof}

The definition \eqref{homegas} of $H_{\omega} (s)$ and the fact that $(-q^{s T^{1 / 3} - m T + J - M}; q)_{\infty}^{-1} < 1$ together yield
\begin{flalign}
\label{homegasbounded}
\big| H_{\omega} (s) \big| & \le \displaystyle\frac{(q \omega ; q)_{\infty}}{T^{1 / 3}} \displaystyle\sum_{M = 0}^{\infty} \displaystyle\frac{\omega^M}{(q; q)_M} = \displaystyle\frac{1}{T^{1 / 3} (1 - \omega)} < T^{10}, 
\end{flalign}

\noindent where we applied the $q$-binomial identity to deduce the equality. This establishes the first statement of the proposition.

Next we consider the second statement \eqref{homegalarget} of the proposition. Again applying \eqref{homegas} yields  
\begin{flalign}
\label{homegasvarepsilons}
\displaystyle\frac{H_{\omega} (s + \varepsilon) - H_{\omega} (s) }{\varepsilon} = S_1 + S_2 - S_3,
\end{flalign}

\noindent where 
\begin{flalign}
\label{s1s2s3} 
\begin{aligned}
 S_1 & = \displaystyle\frac{(q \omega ; q)_{\infty}}{\varepsilon T^{1 / 3}} \displaystyle\sum_{M = 0}^{\lfloor \varepsilon T^{1 / 3} \rfloor - 1 } \displaystyle\frac{\omega^M}{(q; q)_M \big( -q^{X_T + \varepsilon T^{1 / 3} - M}; q \big)_{\infty}}; \\
S_2 & = \displaystyle\frac{(q \omega ; q)_{\infty}}{\varepsilon T^{1 / 3}} \displaystyle\sum_{M = 0}^{\infty} \displaystyle\frac{\omega^{M + \lfloor \varepsilon T^{1 / 3} \rfloor } }{(q; q)_{M + \lfloor \varepsilon T^{1 / 3} \rfloor} \big( -q^{X_T + \varepsilon T^{1 / 3} - M - \lfloor \varepsilon T^{1 / 3} \rfloor }; q \big)_{\infty}};    \\
S_3 & = \displaystyle\frac{(q \omega ; q)_{\infty}}{\varepsilon T^{1 / 3}} \displaystyle\sum_{M = 0}^{\infty} \displaystyle\frac{\omega^M }{(q; q)_M \big( -q^{X_T - M}; q \big)_{\infty}}. 
\end{aligned}
\end{flalign}

We first analyze $S_1$. Since $1 - \omega = T^{-10}$, we have that $\big| (q \omega; q)_{\infty} - (q; q)_{\infty} \big| = \mathcal{O} \big( T^{-10} \big)$  and $\big| \omega^M - 1 \big| = \mathcal{O} \big( T^{- 29 / 3} \big)$, for each $M \in [0, \varepsilon T^{1 / 3} ]$. Inserting these estimates into the definition \eqref{s1s2s3} of $S_1$ and using the fact that $\big( -q^{X_T + \varepsilon T^{1 / 3} - M}; q \big)_{\infty}^{-1} \in (0, 1)$, one quickly verifies that 
\begin{flalign}
\label{firstsumomega12}
\Bigg| S_1  - \widetilde{S_1} \Bigg| = \mathcal{O} \left( \displaystyle\frac{1}{\varepsilon T^{10}} \right), \qquad \text{where} \quad \widetilde{S_1} =  \displaystyle\frac{( q; q)_{\infty}}{\varepsilon T^{1 / 3}} \displaystyle\sum_{M = 0}^{\lfloor \varepsilon T^{1 / 3} \rfloor - 1 } \displaystyle\frac{1}{(q; q)_M \big( -q^{X_T + \varepsilon T^{1 / 3} - M}; q \big)_{\infty}},
\end{flalign}

\noindent is formed by replacing each $\omega$ with $1$ in the definition \eqref{s1s2s3} of $S_1$. To analyze $\widetilde{S_1}$, observe that
\begin{flalign*}
\left| \displaystyle\frac{(q; q)_{\infty}}{(q; q)_M} \right| = \mathcal{O} (q^M); \qquad \left| \displaystyle\frac{1}{\big( -q^{X_T + \varepsilon T^{1 / 3} - M}; q \big)_{\infty}} - \textbf{1}_{X_T + \varepsilon T^{1 / 3} \ge M} \right| = \mathcal{O} \big( q^{|X_T + \varepsilon T^{1 / 3} - M|} \big), 
\end{flalign*}

\noindent for any $M \ge 0$. Inserting these estimates into the definition \eqref{firstsumomega12} of $\widetilde{S_1}$ yields 
\begin{flalign}
\label{firstsumomega13}	
\begin{aligned}
\left| \widetilde{S}_1 - \displaystyle\frac{1}{\varepsilon T^{1 / 3}} \displaystyle\sum_{M = 0}^{\lfloor \varepsilon T^{1 / 3} \rfloor - 1} \textbf{1}_{X_T + \varepsilon T^{1 / 3} \ge M} \right| & = \displaystyle\frac{1}{\varepsilon T^{1 / 3}} \displaystyle\sum_{M = 0}^{\lfloor \varepsilon T^{1 / 3} \rfloor - 1} \Big( \mathcal{O} (q^M) + \mathcal{O} \big( q^{|X_T + \varepsilon T^{1 / 3} - M|} \big) \Big) \\
& =  \mathcal{O} \left( \displaystyle\frac{1}{\varepsilon T^{1 / 3}} \right). 
\end{aligned}
\end{flalign}

\noindent Inserting \eqref{firstsumomega13} into \eqref{firstsumomega12} yields
\begin{flalign}
\label{firstsumomega16}
\left| S_1 - \textbf{1}_{X_T \ge 0} - \left( 1 + \displaystyle\frac{X_T}{\varepsilon T^{1 / 3}} \right) \textbf{1}_{X_T \in (-\varepsilon T^{1 / 3}, 0)} \right|  = \mathcal{O} \left( \displaystyle\frac{1}{\varepsilon T^{1 / 3}} \right). 
\end{flalign}

\noindent This estimates the first sum $S_1$ on the right side of \eqref{homegasvarepsilons}. Next, we would like to bound $S_2 - S_3$. To that end, observe that 
\begin{flalign}
\label{s2s3}
|S_2 - S_3| \le \displaystyle\frac{(q \omega ; q)_{\infty}}{\varepsilon T^{1 / 3}} \displaystyle\sum_{M = 0}^{\infty} D_M, 
\end{flalign}

\noindent where for each $M \ge 0$,  
\begin{flalign}
\label{termd}
D_M = \left| \displaystyle\frac{\omega^{M + \lfloor \varepsilon T^{1 / 3} \rfloor } }{(q; q)_{M + \lfloor \varepsilon T^{1 / 3} \rfloor} \big( -q^{X_T + \varepsilon T^{1 / 3} - M - \lfloor \varepsilon T^{1 / 3} \rfloor }; q \big)_{\infty}} - \displaystyle\frac{\omega^M }{(q; q)_M \big( -q^{X_T - M}; q \big)_{\infty}} \right|. 
\end{flalign}

\noindent Since $\omega = 1 - T^{-10}$, we have that $\big| \omega^{M + \lfloor \varepsilon T^{1 / 3} \rfloor } - \omega^M	\big| = \mathcal{O} \big( T^{-29 / 3} \big)$. Moreover, we have that $\big| (q; q)_{M + \lfloor \varepsilon T^{1 / 3} \rfloor}^{-1} - (q; q)_M^{-1} \big| = \mathcal{O} (q^M)$. 

Furthermore, due to the fact that $(-q^j; q)_{\infty}^{-1}$ decays exponentially in $j$ as $-j$ tends to $\infty$ and converges to $1$ exponentially in $j$ as $j$ tends to $\infty$, we deduce that 
\begin{flalign}
\label{qomega2small}
\begin{aligned}
& \left| \displaystyle\frac{1}{\big( -q^{X_T + \varepsilon T^{1 / 3} - M - \lfloor \varepsilon T^{1 / 3} \rfloor }; q \big)_{\infty}} - \displaystyle\frac{1}{\big( -q^{X_T - M}; q \big)_{\infty}} \right| = \mathcal{O} \big( q^{| X_T - M| } \big); \\
& \displaystyle\frac{1}{\big( -q^{X_T + \varepsilon T^{1 / 3} - M - \lfloor \varepsilon T^{1 / 3} \rfloor }; q \big)_{\infty}} + \displaystyle\frac{1}{\big( -q^{X_T - M}; q \big)_{\infty}} = \mathcal{O} \big( q^{M - X_T} \big), 
\end{aligned}
\end{flalign}
	
\noindent Inserting the estimates \eqref{qomega2small} into \eqref{termd} yields 
\begin{flalign}
\label{qomega3small} 
D_M = \mathcal{O} \big( T^{-9} \big) + \mathcal{O} (q^M) + \mathcal{O} \big( q^{| X_T - M| } \big) \qquad \text{and} \qquad D_M = \mathcal{O} \big( q^{M - X_T} \big), 
\end{flalign}

\noindent for any $M \ge 0$. Inserting the first estimate in \eqref{qomega3small} into \eqref{s2s3} when $M \le T^5$, inserting the second estimate in \eqref{qomega3small} into \eqref{s2s3} when $M > T^5$, and using the fact that $|X_T| < T^5$ yields  
\begin{flalign}
\label{qomega4small} 
|S_2 - S_3| = \mathcal{O} \left( \displaystyle\frac{1}{\varepsilon T^5} \right) + \mathcal{O} \left( \displaystyle\frac{1}{\varepsilon T^{1 / 3}} \right) = \mathcal{O} \left( \displaystyle\frac{1}{\varepsilon T^{1 / 3}} \right). 
\end{flalign}

\noindent Now \eqref{homegalarget} follows from \eqref{homegasvarepsilons}, \eqref{firstsumomega16}, and \eqref{qomega4small}. This verifies the second part of the proposition. 
\end{proof}

\subsection{Application to the Probability Distribution of \texorpdfstring{$J_T (pT)$}{}}

\label{ProofCurrent}

We would now like to show how Theorem \ref{currentfluctuations} and Theorem \ref{processfluctuations} follow from Theorem \ref{determinant} and Theorem \ref{determinantw}, respectively. To that end, we must state the asymptotics for the right sides of \eqref{determinantprocess} and \eqref{determinantvertexw}. 

The first proposition below does this for the ASEP, and the second proposition below does this for the stochastic six-vertex model. The proofs of these propositions will be the topics of Section \ref{DeterminantNew} and Section \ref{RightStationary}. 

\begin{prop} 
\label{convergencedeterminant}

Adopt the notation of Theorem \ref{determinant}, and assume that $R = L + 1$. Fix $b \in (0, 1)$, and denote $\chi = b (1 - b)$ and $\beta = b / (1 - b)$. Further denote 
\begin{flalign*}
x = x (T) = \lfloor (1 - 2 b) T + 2 c \chi^{1 / 3} T^{2 / 3} \rfloor; \qquad \zeta = \zeta_T = -q^{2b c \chi^{1 / 3} T^{2 / 3} + \chi^{2 / 3} s T^{1 / 3} - b^2 T},  
\end{flalign*}

Let $\{ b_1 = b_{1, T} \}_{T \ge 0} \subset (0, 1)$ and $\{ b_2 = b_{2, T} \} \subset (0, 1)$ be sequences that converge to $b$ as $T$ tends to $\infty$, in such a way that $\beta = (\beta_1 + \beta_2) / 2 + \mathcal{O} \big( T^{-10} \big)$ and $\omega = \omega_T = \beta_2 \beta_1^{-1} = 1 - T^{-10}$, for all $T > 10$; here, $\beta_i = b_i / (1 - b_i)$ for each $i \in \{ 1, 2 \}$. 

Then, recalling $K_{\Ai, c^2}$ from Definition \ref{stationarykernel1} and $g(c, s)$ from Definition \ref{stationarydistribution}, we have that
\begin{flalign*}
\displaystyle\lim_{T \rightarrow \infty} \displaystyle\frac{\det \big( \Id + A_{\zeta} \big)_{L^2 (\mathcal{C}_A)}}{\chi^{2 / 3} T^{1 / 3} (1 - \omega)}  = g(c, s) \det \big( \Id - K_{\Ai, c^2} \big)_{L^2 (s, \infty)}. 
\end{flalign*}

\end{prop}

\begin{prop}
\label{convergencedeterminantmodel}

Adopt the notation of Theorem \ref{determinantw}, and fix $b \in (0, 1)$; denote $\Lambda = b + \kappa (1 - b)$, $\chi = b (1 - b)$, and $\beta = b / (1 - b)$. Further denote  
\begin{flalign*}
\theta = \kappa^{-1} \Lambda^2; \qquad \rho = \big( 1 - \kappa^{-1} \big)^{2 / 3} \kappa^{-1 / 3} \chi^{1 / 3} \Lambda^{1 / 3}; \qquad f = \Lambda^{-1 / 3} (\kappa - 1)^{1 / 3} \chi^{2 / 3}, 
\end{flalign*}

\noindent and set 
\begin{flalign*}
\quad x = x (T) = \lfloor \theta T + 2 \rho \Lambda c T^{2 / 3} \rfloor ; \qquad \zeta = \zeta_T = -q^{2 \rho b c T^{2 / 3} + f s T^{1 / 3} - b^2 (1 - \kappa^{-1}) T}. 
\end{flalign*}

Let $\{ b_1 = b_{1, T} \}_{T \ge 0} \subset (0, 1)$ and $\{ b_2 = b_{2, T} \} \subset (0, 1)$ be sequences that converge to $b$ and $\Lambda^{-1} b$, respectively, as $T$ tends to $\infty$, in such a way that $\beta = (\beta_1 + \kappa \beta_2) / 2 + \mathcal{O} \big( T^{-10} \big)$ and $\omega = \omega_T = \kappa \beta_2 \beta_1^{-1} = 1 - T^{-10}$, for all $T > 10$; here, $\beta_i = b_i / (1 - b_i)$ for each $i \in \{ 1, 2 \}$. 

Then, recalling $K_{\Ai, c^2}$ from Definition \ref{stationarykernel1} and $g(c, s)$ from Definition \ref{stationarydistribution}, we have that
\begin{flalign*}
\displaystyle\lim_{T \rightarrow \infty} \displaystyle\frac{\det \big( \Id + V_{\zeta} \big)_{L^2 (\mathcal{C}_V)}}{f T^{1 / 3} (1 - \omega)} = g(c, s) \det \big( \Id - K_{\Ai, c^2} \big)_{L^2 (s, \infty)}. 
\end{flalign*}

\end{prop} 

\begin{rem}

Both Proposition \ref{convergencedeterminant} and Proposition \ref{convergencedeterminantmodel} should be valid as long as $1 - \omega_T$ decays sufficiently quickly. However, our proof no longer seems to apply if $1 - \omega_T$ decays too quickly, for instance if $1 - \omega_T = \mathcal{O} (e^{-T})$. 

\end{rem}

We would like to establish Theorem \ref{currentfluctuations} (or Theorem \ref{processfluctuations}) using Theorem \ref{determinant} (or Theorem \ref{determinantw}), Proposition \ref{linearsum}, and Proposition \ref{convergencedeterminant} (or Proposition \ref{convergencedeterminantmodel}). However, there are two obstructions that prevent us from immediately doing this. 

The first is that all of our results only hold when $1 - \omega = T^{-10}$, while we instead are interested in the case $\omega = 1$. The second is that Proposition \ref{linearsum} is only stated for $|X_T| < T^5$; for sufficiently large $T$, this holds deterministically for the stochastic six-vertex model, but not for the ASEP. 

The first issue will be addressed in the next section. To remedy the second, we have the following lemma that states that the probability that the current is very large decays exponentially.

\begin{lem}
\label{jlarge}

Fix $R > L \ge 0$ and $x \in \mathbb{R}$. Consider the ASEP with left jump rate $L$, right jump rate $R$, and arbitrary initial data. There exists a constant $C > 0$ such that $\mathbb{P} \big[ |J_T (xT)| > y T^2 \big] < C e^{-y T}$, for each $y, T \in \mathbb{R}_{> 1}$. 
\end{lem}

\begin{proof}

In this proof, we recall the notation associated with tagged particles introduced in Section \ref{AsymmetricExclusions}. Now observe that, if the current is abnormally large, then particle $-1$ must have moved very far to the left or particle $0$ must have moved very far to the right. Specifically, one quickly verifies that $\textbf{1}_{J_T (xT) > y T^2} \le \textbf{1}_{X_{-1} (T) \ge y T^2 + xT}$ and $\textbf{1}_{J_T (x T) < - y T^2} \le \textbf{1}_{X_0 (T) \le xT - y T^2}$. Thus, it suffices to estimate $\mathbb{P} \big[ X_{-1} (T) \ge y T^2 + xT \big]$ and $\mathbb{P} \big[ X_0 (T) \le xT - y T^2 \big]$. 

To that end, observe that $X_{-1} (T)$ is stochastically dominated by a simple, asymmetric random walk that jumps one space to the right according to the ringing times of an exponential clock of rate $R$. The probability that such a random walk attempts to jump right at least $y T^2 + x T$ times in the time interval $[0, T]$ is bounded by $C' e^{- y T}$, for some constant $C' > 0$. Hence, $\mathbb{P} \big[ X_{-1} (T) \ge y T^2 + xT \big] < C' e^{-y T}$. Through similar reasoning (after enlarging $C'$ if necessary), we find that $\mathbb{P} \big[ X_0 (T) \le xT - y T^2 \big] < C' e^{-y T}$.

Thus, $\mathbb{P} \big[ \big| J_T (xT) \big| > y T^2 \big] \le \mathbb{P} \big[ X_{-1} (T) \ge y T^2 + xT \big] + \mathbb{P} \big[ X_0 (T) \le xT - y T^2 \big] \le C e^{- y T}$, where $C \ge 2 C'$; this establishes the lemma. 
\end{proof}

Now we can establish the following two corollaries, which concern the asymptotics for the ASEP and stochastic six-vertex model with near-stationary (or near-translation-invariant) initial data.

\begin{cor}
\label{currentfluctuationsnear}

Adopt the notation of Proposition \ref{convergencedeterminant}, and consider the ASEP with left jump rate $L$, right jump rate $R$, and double-sided $(b_{1, T}, b_{2, T})$-Bernoulli initial data. Then, for any real numbers $c, s \in \mathbb{R}$, we have that 
\begin{flalign*}
\displaystyle\lim_{T \rightarrow \infty} \mathbb{P} \left[ J_{T} \big( (1 - 2b) T  + 2 c \chi^{1 / 3} T^{2 / 3} \big) \ge b^2 T - 2bc \chi^{1 / 3} T^{2 / 3} - \chi^{2 / 3} s T^{1 / 3} \right] = F_{\BR; c} (s). 
\end{flalign*}

\end{cor}

\begin{cor}
\label{processfluctuationsnear}

Adopt the notation of Proposition \ref{convergencedeterminantmodel}, and consider the stochastic six-vertex model with left jump rate $\delta_1$, right jump rate $\delta_2$, and double-sided $(b_{1, T}, b_{2, T})$-Bernoulli initial data. Recall the definitions of $x$, $y$, $\varsigma$, and $\mathcal{F}$ from Theorem \ref{processfluctuations}. 

Then, for any real numbers $c, s \in \mathbb{R}$, we have that 
\begin{flalign*}
\displaystyle\lim_{T \rightarrow \infty} \mathbb{P} \Big[ \mathfrak{H} \big( x (T + \varsigma c T^{2 / 3}), y T \big) \ge b_1 b_2 (\delta_2 - \delta_1) T - b_1 (1 - \delta_2) \varsigma c T^{2 / 3} - \mathcal{F} s T^{1 / 3} \Big] = F_{\BR; c} (s). 
\end{flalign*}
\end{cor}

\begin{proof}[Proof of Corollary \ref{currentfluctuationsnear} Assuming Proposition \ref{convergencedeterminant}]

Denote $m = b^2 - 2 bc \chi^{1 /3 } T^{-1 / 3}$, and apply Theorem \ref{determinant} with $x = x(T) = \lfloor (1 - 2b) T + 2c \chi^{1 / 3} T^{2 / 3} \rfloor$ and $\zeta = -q^{\chi^{2 / 3}s T^{1/ 3} - m T}$. We obtain that
\begin{flalign}
\label{homegaasymmetric1}
T^{1 / 3} (1 - \omega) \mathbb{E} \Big[ H_{\omega; T} \big( m, \chi^{2 / 3} s; J_T (x) \big) \Big] = \det \big( \Id + A_{\zeta; \chi^{2 / 3} s} \big)_{L^2 (\mathcal{C}_A)}, 
\end{flalign}

\noindent where we have recalled the definition of $H_{\omega; T} (m, s; J)$ from \eqref{homegas}, and we have denoted $A_{\zeta; \chi^{2 / 3} s} = A_{\zeta}$ to emphasize the dependence on $s$. 

Let $\varepsilon > 0$. Apply Theorem \ref{determinant} again, now with $\zeta = -q^{\chi^{2 / 3} (s + \varepsilon) T^{1 / 3} - m T}$. Subtracting from the result \eqref{homegaasymmetric1}, and dividing by $\chi^{2 / 3} \varepsilon T^{1 / 3} (1 - \omega)$, yields 
\begin{flalign}
\label{homegaasymmetric2}
\begin{aligned}
\displaystyle\frac{1}{\chi^{2 / 3} \varepsilon} \mathbb{E} & \Big[ H_{\omega; T} \big( m, \chi^{2 / 3} ( s + \varepsilon); J_T (x) \big) - H_{\omega; T} \big(  m, \chi^{2 / 3} s; J_T (x) \big)  \Big]  \\
& = \displaystyle\frac{1}{\chi^{2 / 3} \varepsilon T^{1 / 3} (1 - \omega) } \Big( \det \big( \Id + A_{\zeta; \chi^{2 / 3} (s + \varepsilon) } \big)_{L^2 (\mathcal{C}_A)} - \det \big( \Id + A_{\zeta; \chi^{2 / 3} s} \big)_{L^2 (\mathcal{C}_A)}\Big). 
\end{aligned} 
\end{flalign}	

We would like to apply Lemma \ref{jlarge} to deduce that the left side of \eqref{homegaasymmetric2} is small when restricting to the event that $|J_T (x)| > T^4$, and then apply Proposition \ref{linearsum} to more precisely estimate the left side of \eqref{homegaasymmetric2} when $|J_T (x)| \le T^4$. 

To that end, denote 
\begin{flalign*}
\mathcal{H} = \mathcal{H} (\varepsilon, s) = \mathcal{H} \big( \varepsilon, s; J_T (x) \big) = H_{\omega; T} \big( m, \chi^{2 / 3} ( s + \varepsilon); J_T (x) \big) - H_{\omega; T} \big(  m, \chi^{2 / 3} s; J_T (x) \big), 
\end{flalign*} 

\noindent and observe that
\begin{flalign}
\label{homegaasymmetric3}
& \displaystyle\frac{1}{\chi^{2 / 3} \varepsilon} \mathbb{E} [ \mathcal{H}  ] = \displaystyle\frac{1}{\chi^{2 / 3} \varepsilon} \mathbb{E} \big[  \textbf{1}_{|J_T (x)| > T^4} \mathcal{H} \big] + \displaystyle\frac{1}{\chi^{2 / 3} \varepsilon} \mathbb{E} \big[ \textbf{1}_{|J_T (x)| \le T^4} \mathcal{H} \big].
\end{flalign}

\noindent By the Lemma \ref{jlarge} and the first statement of Proposition \ref{linearsum}, 
\begin{flalign}
\label{homegaasymmetric4}
\displaystyle\frac{1}{\chi^{2 / 3} \varepsilon} \mathbb{E} \big[  \textbf{1}_{|J_T (x)| > T^4} | \mathcal{H}  |\big] \le \displaystyle\frac{T^{10} \mathbb{P} \big[ \big| J_T (x) \big| \ge T^4 \big]}{\chi^{2 / 3} \varepsilon} = \mathcal{O} \left( \displaystyle\frac{T^{10}}{\varepsilon e^T}  \right). 
\end{flalign}

\noindent By the second statement of Proposition \ref{linearsum} and Lemma \ref{jlarge} (and also the fact that $|J_T (x)| < T^4$ implies that $|J_T (x) - mT + \chi^{2 / 3} s T^{1 / 3}| < T^5$, for sufficiently large $T$), we estimate
\begin{flalign}
\label{homegaasymmetric5}
\displaystyle\frac{1}{\chi^{2 / 3} \varepsilon} \mathbb{E} & \big[ \textbf{1}_{|J_T (x)| < T^4} \mathcal{H} \big] = \mathbb{E} \big[ \mathcal{Z} (s, \varepsilon) \big] + \mathcal{O} \left( \displaystyle\frac{1}{\varepsilon T^{1 / 3} } \right) ,
\end{flalign}

\noindent where 
\begin{flalign}
\mathcal{Z} (s, \varepsilon) =  \textbf{1}_{X_{T, s} \ge 0} + \left( 1 + \displaystyle\frac{X_{T, s}}{\varepsilon \chi^{2 / 3} T^{1 / 3}}\right) \textbf{1}_{X_{T, s} \in (-\varepsilon \chi^{2 / 3} T^{- 1 / 3}, 0)}, 
\end{flalign}

\noindent where $X_{T, s} = J_T (x) - mT + \chi^{2 / 3} s T^{1 / 3}$. Combining \eqref{homegaasymmetric3}, \eqref{homegaasymmetric4}, and \eqref{homegaasymmetric5}, and inserting into \eqref{homegaasymmetric2}, we obtain that 
\begin{flalign}
\label{homegaasymmetric6}
& \displaystyle\frac{ \det \big( \Id + A_{\zeta; \chi^{2 / 3} (s + \varepsilon)} \big)_{L^2 (\mathcal{C}_A)} - \det \big( \Id + A_{\zeta; \chi^{2 / 3} s} \big)_{L^2 (\mathcal{C}_A)}}{\chi^{2 / 3} \varepsilon (1 - \omega) T^{1 / 3}} = \mathbb{E} \big[ \mathcal{Z} (s, \varepsilon) \big] + \mathcal{O} \left( \displaystyle\frac{1}{\varepsilon T^{1 / 3} } \right). 
\end{flalign}	

\noindent Now, we apply Proposition \ref{convergencedeterminant} to the left side of \eqref{homegaasymmetric6} to deduce that 
\begin{flalign}
\label{homegaasymmetric7}
\begin{aligned}
& \displaystyle\frac{1}{\varepsilon} \bigg( g(c, s + \varepsilon) \det \big( \Id - K_{\Ai, c^2} \big)_{L^2 (s + \varepsilon, \infty)} -	g(c, s) \det \big( \Id - K_{\Ai, c^2} \big)_{L^2 (s, \infty)} \bigg) \\ 
& \quad = \mathbb{E} \big[ \mathcal{Z} (s, \varepsilon) \big] + \mathcal{O} \left( \displaystyle\frac{1}{\varepsilon T^{1 / 3} } \right) + \mathcal{O} \big( \varepsilon^{-1} h(T) \big), 
\end{aligned}
\end{flalign}

\noindent where $h(T)$ is some function determined by Proposition \ref{convergencedeterminant} that tends to $0$ as $T$ tends to $\infty$.\footnote{In fact, it is possible to take $h(T) = T^{-1 / 3}$, although we will not verify this.} Applying the definition of $F_{\BR; c} (s)$ (see Definition \ref{stationarydistribution}), we deduce from \eqref{homegaasymmetric7} that 
\begin{flalign}
\label{homegaasymmetric8}
 \mathbb{E} \big[ \mathcal{Z} (s, \varepsilon) \big] = \displaystyle\frac{1}{\varepsilon} \displaystyle\int_s^{s + \varepsilon} F_{\BR; c} (r) dr + \mathcal{O} \big( \varepsilon^{-1} T^{-1 / 3} \big) + \mathcal{O} \big( \varepsilon^{-1} h(T) \big), 
\end{flalign}

\noindent for any real number $s \in \mathbb{R}$. Applying \eqref{homegaasymmetric8} again with $s$ replaced with $s - \varepsilon$ and observing that $\mathbb{E} \big[ \mathcal{Z} (s - \varepsilon, \varepsilon) \big] \le \mathbb{E} \big[ \textbf{1}_{X_{T, s} \ge 0} \big] \le \mathbb{E} \big[ \mathcal{Z} (s, \varepsilon) \big]$, we deduce that there exists a constant $C > 0$ such that 
\begin{flalign}
\label{homegaasymmetric9}
\begin{aligned}
\displaystyle\frac{1}{\varepsilon} \displaystyle\int_{s - \varepsilon}^s F_{\BR; c} (r) dr & \le \mathbb{E} \big[ \textbf{1}_{X_{T, s} \ge 0} \big] + \displaystyle\frac{C}{\varepsilon T^{1 / 3}} + \displaystyle\frac{C h(T)}{\varepsilon} \\
& \le \displaystyle\frac{1}{\varepsilon} \displaystyle\int_{s}^{s + \varepsilon} F_{\BR; c} (r) dr + \displaystyle\frac{2 C}{\varepsilon T^{1 / 3}} + \displaystyle\frac{2 C h(T)}{\varepsilon}. 
\end{aligned}
\end{flalign}

 Now, set $\varepsilon = \varepsilon_T$ to depend on $T$ in such a way that $\varepsilon_T$, $\varepsilon_T^{-1} h(T)$, and $\varepsilon_T^{-1} T^{-1 / 3}$ all tend to $0$ as $T$ tends to $\infty$. Then, \eqref{homegaasymmetric9} and the continuity of $F_{\BR; c} (s)$ imply that
\begin{flalign}
\label{homegaasymmetric10}
\displaystyle\lim_{T \rightarrow \infty} \mathbb{P} [J_T (x) \ge mT - \chi^{2 / 3} s T] = \displaystyle\lim_{T \rightarrow \infty} \mathbb{P} [X_{T, s} \ge 0] = \displaystyle\lim_{T \rightarrow \infty} \mathbb{E} [\textbf{1}_{X_{T, s} \ge 0}] = F_{\BR; c} (s), 
\end{flalign}

\noindent from which we deduce the theorem. 
\end{proof}

\begin{proof}[Proof of Corollary \ref{processfluctuationsnear} Assuming Proposition \ref{convergencedeterminantmodel}]

Let $\overline{T}$ be a large positive integer, define $\Lambda$, $\rho$, and $f$ as in Proposition \ref{convergencedeterminantmodel}, set $m = 2 \rho b c \overline{T}^{2 / 3} - b^2 (1 - \kappa^{-1}) \overline{T}$, and set $x = x (\overline{T}) = \lfloor \theta \overline{T} + 2 \rho \Lambda c \overline{T}^{2 / 3} \rfloor$. 

Following the proof of Corollary \ref{currentfluctuations}, and using the fact that we now deterministically have that $\mathfrak{H} (x, T) \le x + T < T^4$ deterministically (for all sufficiently large $T$), we obtain that 
\begin{flalign}
\label{modelprobability1}
\displaystyle\lim_{\overline{T} \rightarrow \infty} \mathbb{P} \big[ \mathfrak{H} (\theta \overline{T} + 2 \rho \Lambda_1 c \overline{T}^{2 / 3}, \overline{T}) \ge b^2 (1 - \kappa^{-1}) \overline{T} - 2 \rho b c \overline{T}^{2 / 3} - s f \overline{T}^{1 / 3} \big] = F_{\BR; c} (s), 
\end{flalign}

\noindent which is the analog of \eqref{homegaasymmetric10}. 

Now set $\overline{T} = y T = \Lambda_2 (1 - \delta_1) T$. Then, \eqref{modelprobability1} becomes 
\begin{flalign*}
\displaystyle\lim_{T \rightarrow \infty} \mathbb{P} & \big[ \mathfrak{H}(x T + 2 \rho \Lambda_1 y^{2 / 3} c T^{2 / 3}, y T) \ge b_1^2 \Lambda_2 (\delta_2 - \delta_1) T - 2 \rho b_1 y^{2 / 3} T^{2 / 3} - s f y^{1 / 3} T^{1 / 3} \big] = F_{\BR; c} (s). 
\end{flalign*}

\noindent Now the corollary follows from the above identity and the facts that $b_1 \Lambda_2 = b_2$, $2 \rho y^{2 / 3} = (1 - \delta_2) \varsigma$, and $f y^{1 / 3} = \mathcal{F}$. 
\end{proof}

\subsection{Proofs of Theorem \ref{currentfluctuations} and Theorem \ref{processfluctuations}}

\label{NearInitial} 

Corollary \ref{currentfluctuationsnear} and Corollary \ref{processfluctuationsnear} provide fluctuations for the current of the ASEP and stochastic six-vertex model under near-stationary (or near-translation-invariant) initial data. We would like to modify these results so that they hold under exactly stationary (or exactly translation-invariant) initial data. As explained below, we will do this through probabilistic coupling arguments. 

This will be slightly quicker for the stochastic six-vertex model, so we begin with the proof of Theorem \ref{processfluctuations}. 	

\begin{proof}[Proof of Theorem \ref{processfluctuations}]

Let $\{ b_{1, T} \}_{T \in \mathbb{Z}_{> 0}} \in (0, 1)$ be a sequence that converges to $b_1$ such that the following holds. If we denote $\beta_{1, T} = b_{1, T} / (1 - b_{1, T})$, then $\omega_T = \kappa \beta_2 \beta_{1, T}^{-1}$ satisfies $\omega_T = 1 - T^{-10}$, if $T > 10$ (which we assume throughout). 

Now, consider two stochastic six-vertex models, one called $\textbf{S}$ with double-sided $(b_1, b_2)$-Bernoulli initial data, and one called $\textbf{S}_T$ with double-sided $(b_{1, T}, b_2)$-Bernoulli initial data. For each $i \in \mathbb{Z}_{> 0}$, let $\varphi^{(x)} (i)$ denote the indicator that $(i, 1)$ is an entrance site for a path in $\textbf{S}$, and let $\varphi^{(y)} (i)$ denote the indicator that $(1, i)$ is an entrance site for a path in $\textbf{S}$. Define $\varphi_T^{(x)} (i)$ and $\varphi_T^{(y)} (i)$ similarly, but for the model $\textbf{S}_T$. 

Couple $\varphi$ and $\varphi_{T}$ in such a way that $\varphi^{(x)} (i) = \varphi_T^{(x)} (i)$ and $\mathbb{P} \big[ \varphi_{i; T}^{(y)} \ne \varphi_i^{(y)} \big] = b_{1, T} - b_1 = \mathcal{O} \big( T^{-10} \big)$, for each $i \in \mathbb{Z}_{> 0}$. Moreover, couple the models $\textbf{S}$ and $\textbf{S}_T$ so that they coincide in the strip $\mathbb{Z}_{> 0} \times [0, k]$, where $k \in \mathbb{Z}_{\ge 0}$ is the smallest integer such that $\varphi^{(y)} (k + 1) \ne \varphi_T^{(y)} (k + 1)$, and evolve independently outside this strip. 

For each $Y \in \mathbb{R}_{> 0}$, let $E_Y$ denote the event that $k \le Y$. Since $\mathbb{P} \big[ \varphi_{i; T}^{(y)} \ne \varphi_i^{(y)} \big] = \mathcal{O} \big( T^{-10} \big)$, a union estimate implies that $\mathbb{P} \big[ E_T \big] = \mathcal{O} \big( Y T^{-10} \big)$. Now, let $\mathbb{P}_{\textbf{S}}$ denote the probability measure with respect to the stochastic six-vertex model $\textbf{S}$, and let $\mathbb{P}_{\textbf{S}_T}$ denote the probability measure with respect to the model $\textbf{S}_T$. Let $X, Z \in \mathbb{R}_{> 0}$. 

Due to our coupling of $\textbf{S}$ and $\textbf{S}_T$, we have that 
\begin{flalign}
\label{nearstationarystochasticmodel}
\big| \mathbb{P}_{\textbf{S}} \big[ \mathfrak{H} (X, Y) \ge Z \big] - \mathbb{P}_{\textbf{S}_Y } \big[ \mathfrak{H} (X, Y) \ge Z \big] \big| \le \mathbb{P} \big[ E_Y \big] = \mathcal{O} \big( Y T^{-10} \big). 
\end{flalign}

Now, Theorem \ref{processfluctuations} follows from Corollary \ref{processfluctuationsnear} and \eqref{nearstationarystochasticmodel}, after setting $X = x (T + \varsigma c T^{2 / 3})$, $Y = y T$, and $Z = b_1 b_2 (\delta_2 - \delta_1) T - b_1 (1 - \delta_2) \varsigma c T^{2 / 3} - \mathcal{F} T^{1 / 3}$, and letting $T$ tend to $\infty$. 
\end{proof}

\noindent Now we turn to the proof of Theorem \ref{currentfluctuations}, whose proof will be similar to that of Theorem \ref{processfluctuations}. 

\begin{proof}[Proof of Theorem \ref{currentfluctuations}]

By rescaling time, we can assume that $\delta = R - L = 1$. 

Let $\{ b_T \}_{T \in \mathbb{R}_{> 0}} \in (0, 1)$ be a sequence that converges to $b$ such that the following holds. If we denote $\beta = b / (1 - b)$ and $\beta_T = b_T / (1 - b_T)$, then $\omega_T = \beta \beta_T^{-1}$ satisfies $\omega_T = 1 - T^{-10}$, if $T > 1$ (which we assume throughout). Observe that $b_T > b$ for all $T > 1$.

Let us consider two ASEPs, $\textbf{A}$ under $b$-stationary initial data, and $\textbf{A}_T$ under double-sided $(b_T, b)$-Bernoulli initial data. For each $i \in \mathbb{Z}$ and $t \in \mathbb{R}_{\ge 0}$, let $\eta_t^{(\textbf{A})} (i)$ denote the indicator that a particle is at position $i$ at time $t$ in $\textbf{A}$. Define $\eta_t^{(\textbf{A}_T)} (i)$ similarly. 

Couple $\eta_0^{(\textbf{A})}$ and $\eta_0^{(\textbf{A}_T)}$ in such a way that $\eta_0^{(\textbf{A})} (i) = \eta_0^{(\textbf{A}_T)} (i)$ for each $i \in \mathbb{Z}_{> 0}$; $\eta_0^{(\textbf{A}_T)} (i) \ge \eta_0^{(\textbf{A})} (i)$ for each $i \in \mathbb{Z}_{\le 0}$; and $\mathbb{P} \big[ \eta_0^{(\textbf{A})} (i) \ne \eta_0^{(\textbf{A}_T)} (i) \big] = b_T - b = \mathcal{O} \big( T^{-10} \big)$, for each $i \in \mathbb{Z}_{\le 0}$. 

We also couple dynamics of the processes $\textbf{A}$ and $\textbf{A}_T$ under the \emph{basic coupling} \cite{CEP}. For a more precise definition of, and some properties associated with, the basic coupling we refer to Section 1, Part 3 of \cite{IS}; Section 2 of \cite{OCVDA}; or Section 2 of \cite{CFA}. Let us list two of these properties (a more thorough explanation can be found in the above references). 

The first property is often referred to \emph{attractivity} (or \emph{monotonicity}), which states that $\eta_t^{\textbf{A}_T} (i) \ge \eta_t^{\textbf{A}} (i)$, for each $t > 0$ and $i \in \mathbb{Z}$. Thus, we can consider the \emph{defect process} $\textbf{D}_T = \{ \eta_t^{\textbf{D}_T} (i) \}_{i \in \mathbb{Z}, T \in \mathbb{R}_{\ge 0}}$, defined by setting $\eta_t^{\textbf{D}_T} (i) = \eta_t^{\textbf{A}_T} (i) - \eta_t^{\textbf{A}} (i)$. This process corresponds to the particles in $\textbf{A}_T$ but not in $\textbf{A}$, which we refer to as \emph{defect particles}; we refer to the particles of $\textbf{A} = \textbf{A}_T \setminus \textbf{D}_T$ as \emph{original particles}.  

The second property is that defect particles behave as \emph{second class particles}, meaning the following. If an original particle attempts to jump to a location occupied by a defect particle, then the two particles switch positions; for all other types of jumps, the exclusion property is upheld. 

Now, for any $X, Z \in \mathbb{R}$, let us estimate the difference $\big| \mathbb{P}_{\textbf{A}} \big[ J_T (X) > Z \big] - \mathbb{P}_{\textbf{A}_T} \big[ J_T (X) > Z \big] \big|$. Since the initial data for $\textbf{A}$ and $\textbf{A}_T$ coincide to the right of $0$, the current $J_T (X)$ of $\textbf{A}$ and the current $J_T (X)$ of $\textbf{A}_T$ are unequal only on the event $E_T (X)$ that there exists a defect particle to right of $X$ at time $T$. Thus, 
\begin{flalign}
\label{nearstationaryprocess}
\big| \mathbb{P}_{\textbf{A}} \big[ J_T (X) > Z \big] - \mathbb{P}_{\textbf{A}_T} \big[ J_T (X) > Z \big] \big| \le \mathbb{P} \big[ E_T (X) \big]. 
\end{flalign}

To estimate $\mathbb{P} \big[ E_T (X) \big]$, observe that $E_T (X) \subseteq E_{T; 1} (X) \cup E_{T; 2} (X)$, where $E_{T; 1} (X)$ is the event that there exists a defect particle initially in the interval $[- T^5, 0]$, and $E_{T; 2} (X)$ is the event that there exists a defect particle initially to the left of $-T^5$ that is to the right of $X$ by time $T$. 

Since $\mathbb{P} \big[ \eta_0^{(\textbf{A})} (i) \ne \eta_0^{(\textbf{A}_T)} (i) \big] = \mathcal{O} \big( T^{-10} \big)$, a union estimate yields $\mathbb{P} \big[ E_{T; 1} \big]  = \mathcal{O} \big( T^{-5} \big)$. 

To bound $\mathbb{P} \big[ E_{T; 2} \big]$, observe that the trajectory of any second-class particle is stochastically dominated by the trajectory of an asymmetric random walk that jumps one space to the right according to the ringing times of an exponential clock of rate $R + L$. This, and a large deviations estimate for such a random walk, quickly imply that  $\mathbb{P} \big[ E_{2; T} (X) \big] = \mathcal{O} \big( e^{-T} \big)$ for any $X \ge - T^2$. 

Hence, if $X \ge -T^2$, we deduce that $\mathbb{P} \big[ E_T (X) \big] \le \mathbb{P} \big[ E_{T, 1} (X) \big] + \mathbb{P} \big[ E_{T, 2} (X) \big] = \mathcal{O} \big( T^{-5} \big)$. 

Now, Theorem \ref{currentfluctuations} follows from Corollary \ref{currentfluctuationsnear} and \eqref{nearstationaryprocess}, after setting $X = (1 - 2b) T + 2 c \chi^{1 / 3} T^{2 / 3}$ and $Z = b^2 T - 2 b c \chi^{1 / 3} T^{2 / 3} - \chi^{2 / 3} s T^{1 / 3}$. 
\end{proof}

\section{Pre-processing the Determinant}

\label{DeterminantNew}

We would now like to establish Proposition \ref{convergencedeterminant} and Proposition \ref{convergencedeterminantmodel}, which provide asymptotics for the Fredholm determinants given in Theorem \ref{determinant} and Theorem \ref{determinantw}, respectively. 

Asymptotic analyses in similar settings have been implemented in several previous works \cite{PTAEPSSVM, MP, SSVM, CDERF}, but the following concern prevents one from directly proceeding with those methods. In Theorem \ref{determinantw} (for example), the contours $\mathcal{C}_V$ and $\Gamma_V$ must contain $q \kappa \beta_2$, but avoid $q \beta_1$. However, in the setting of Proposition \ref{convergencedeterminant}, $q \kappa \beta_1$ and $q \beta_2$ are very close to $q \beta$, meaning that $\mathcal{C}_V$ and $\Gamma_V$ must nearly touch at $q \beta$. Thus, the integrand on the right side of \eqref{kernelp} is singular near $q \beta$, which poses issues for asymptotic analysis. 

One way to circumvent this issue would be to deform $\mathcal{C}_V$ (and $\Gamma_V$) through the poles $q \beta_1$ and $q \kappa \beta_2$, and account for all residues accumulated in the process. Unfortunately, this would involve having to track all residues accumulated under the $\mathcal{C}_V$ deformation by each summand on the right side of \eqref{determinantsum}, which is a troublesome task. 

To avoid this, we follow Section 6.5 of \cite{HFSE}. Namely, our goal in this section is to establish Lemma \ref{lclexponential}, which (asymptotically) reformulates the Fredholm determinant $\det \big( \Id + V_{\zeta} \big)_{L^2 (\mathcal{C}_V)}$ as one over $L^2 (\mathbb{R}_{> 0})$; after this, addressing the poles $q \beta_1$ and $q \kappa \beta_2$ will be simpler, and a more direct asymptotic analysis will be possible in Section \ref{RightStationary}.

Unfortunately, as to be explained in Remark \ref{cc1ggamma1stationarybd} below, this reformulation is only possible if $\Re v < \Re w$ for all $v \in \Gamma_V$ and $w \in \mathcal{C}_V$. As can be seen from the definitions of these contours in Theorem \ref{determinantw} (and also from Figure \ref{contoursdeterminantw}), no choice of such contours is possible.  

Thus, unlike in \cite{HFSE}, we must first perform some ``pre-processing'' of the Fredholm determinant, which will proceed in three parts. First, in Section \ref{Contours}, we explain one choice of contours $\mathcal{C}_A$, $\Gamma_A$, $\mathcal{C}_V$, and $\Gamma_V$ that will be particularly amenable to asymptotic analysis. Although useful, these contours do not satisfy $\Re v < \Re w$ for all $v \in \Gamma$ and $w \in \mathcal{C}$. So, in Section \ref{DeterminantDifferentContour}, we explain how to ``truncate'' these contours so that the parts of these contours that remain do satisfy this property; this will result in some small error, which we will show can be ignored in the asymptotics. Next, in Section \ref{ckr} we rewrite this ``truncated'' Fredholm determinant as one over $L^2 (\mathbb{R}_{> 0})$; this will be similar to what was done in Lemma 6.11 of \cite{HFSE}.

\subsection{Contours}

\label{Contours}

The purpose of this section is to exhibit one choice of contours $\mathcal{C}_A$, $\mathcal{C}_V$, $\Gamma_A$, and $\Gamma_V$ that will later lead to proofs of the first parts of Proposition \ref{convergencedeterminant} and Proposition \ref{convergencedeterminantmodel}; we will do this in Section \ref{Contoursckgammak}. Our definitions are based on a certain ``exponential form'' for the kernels $A_{\zeta}$ and $V_{\zeta}$, to be explained in Section \ref{ExponentialKernelStationary}. 

\subsubsection{An Exponential Form for the Kernel} 

\label{ExponentialKernelStationary} 

The analysis of the stochastic six-vertex model will be similar to that of the ASEP, so we perform them in parallel. To that end, we introduce the notation $K \in \{ V, A \}$, which refers to the stochastic six-vertex model when $K = V$ and to the ASEP when $K = A$. 

Define the parameters $\mathcal{F}_K$, $m_K$, and $\eta_K$ through 
\begin{flalign}
\label{fetamk}
\begin{aligned}
\mathcal{F}_A & = \chi^{2 / 3}; \qquad \qquad \qquad \qquad \eta_A = 1 - 2b + \displaystyle\frac{2 c \chi^{1 / 3}}{T^{1 / 3}}; \quad  m_A = b^2 - \displaystyle\frac{2 b \chi^{1 / 3} c}{T^{1 / 3}}; \\
\mathcal{F}_V &= \Lambda^{-1 / 3} (\kappa - 1)^{1 / 3} \chi^{2 / 3}; \quad \eta_V = \theta + \displaystyle\frac{2 \rho \Lambda c}{T^{1 / 3}}; \qquad \qquad m_V = b^2 \big( 1 - \kappa^{-1} \big) - \displaystyle\frac{2 \rho b c}{T^{1 / 3}}, 
\end{aligned}
\end{flalign}

\noindent where we recall the notation of Theorem \ref{currentfluctuations}, Theorem \ref{processfluctuations}, Theorem \ref{determinant}, and Theorem \ref{determinantw}. 

Set $\zeta = - q^{s \mathcal{F}_K T^{1 / 3} - m_K T}$ and $x = x(T) = \eta_K T$ in Theorem \ref{determinant} (when $K = A$) and Theorem \ref{determinantw} (when $K = V$). The replacement of $x(T) = \lfloor \eta_K T \rfloor$ (which was what was originally stated Proposition \ref{convergencedeterminant} and Proposition \ref{convergencedeterminantmodel}) with $x(T) = \eta_K T$ is for notational convenience and will not affect the asymptotics. 

Using the notation \eqref{fetamk}, the kernels $K = K_{\zeta}$ from Proposition \ref{convergencedeterminant} and Proposition \ref{convergencedeterminantmodel} can be rewritten as
\begin{flalign}
\label{kernelk}
\begin{aligned}
K (w, w')  = \displaystyle\frac{1}{2 \textbf{i} \log q} \displaystyle\sum_{j = -\infty}^{\infty} & \displaystyle\int_{\Gamma_K} \left( \displaystyle\frac{v}{w} \right)^{\mathcal{F}_K s T^{1 / 3}} \displaystyle\frac{1}{\sin \big( \pi (\log q)^{-1} (\log v - \log w + 2 \pi \textbf{i}j) \big) v (w' - v)} \\
& \times \exp \Big( \big( G_K (w) - G_K (v) \big) T \Big) \displaystyle\frac{(\varkappa \beta_2 q w^{-1}; q)_{\infty} (\beta_1^{-1} q^{-1} v; q)_{\infty}}{(\varkappa \beta_2 q v^{-1}; q)_{\infty} (\beta_1^{-1} q^{-1} w; q)_{\infty}} dv. 
\end{aligned}
\end{flalign}

\noindent In \eqref{kernelk}, the constant $\varkappa = \varkappa_K$ and the function $G_K (z)$ depend on whether $K = A$ or $K = V$. Explicitly, $\varkappa = 1$ when $K = A$, and $\varkappa = \kappa$ when $K = V$. Furthermore, 
\begin{flalign}
\label{gastationary}
G_A (z) =  \Omega_A (z) + 2 c \chi^{1 / 3} T^{-1 / 3} P_A (z); \qquad G_V (z) = \Omega_V (z) + 2 c \rho T^{-1 / 3} P_V (z), 
\end{flalign}

\noindent where 
\begin{flalign}
\label{omegap}
\begin{aligned} 
 P_A (z) & = \log (z + q) -  b  \log z; \qquad \qquad \Omega_A (z) = \displaystyle\frac{q}{q + z} + \big( 1 - 2 b \big) \log (z + q) + b^2 \log z ; \\
P_V (z) & = \Lambda \log (\kappa^{-1} z + q) - b \log z; \qquad \Omega_V (z) = \theta \log (\kappa^{-1} z + q) - \log (z + q) + b^2 (1 - \kappa^{-1}) \log z.
\end{aligned}
\end{flalign}

To perform a saddle point analysis of the kernel $K$, we require the critical points of $G_K$. Inserting the definitions \eqref{fetamk} of $\eta_K$ and $m_K$ into \eqref{gastationary} and \eqref{omegap} yields 
\begin{flalign*}
G_A' (z) & = \displaystyle\frac{(1 - b)^2 (z - q \beta)^2}{z (z + q)^2} + \displaystyle\frac{2 (1 - b) \chi^{1 / 3} (z - q \beta) c T^{-1 / 3}}{z (z + q)}; \\
G_V' (z) & = \displaystyle\frac{(1 - b)^2 (\kappa - 1) (z - q \beta)^2}{z (z + q \kappa) (z + q)} + \displaystyle\frac{\kappa \rho (1 - b) (z - q \beta) c T^{-1 / 3}}{z (z + q \kappa)}. 
\end{flalign*}

\noindent Thus, for either $K \in \{ V, A \}$, we find that 
\begin{flalign}
\label{gkderivatives}
G_K' (q \beta ) = 0; \qquad (q \beta)^2 G_K'' (q \beta) = 2 \mathcal{F}_K^2 c T^{-1 / 3}; \qquad (q \beta)^3 G_K''' (q \beta) = 2 \mathcal{F}_K^3 + \mathcal{O} \big( T^{-1 / 3} \big). 
\end{flalign}

\noindent Therefore, if we define $R_K$ through the identity
\begin{flalign}
\label{rstationary}
R_K \left( \displaystyle\frac{\mathcal{F}_K (z - q \beta)}{q \beta} \right) = G_K (z) - G_K (q \beta) - \displaystyle\frac{1}{3} \left( \displaystyle\frac{\mathcal{F}_K (z - q \beta)}{q \beta} \right)^3 - \left( \displaystyle\frac{\mathcal{F}_K (z - q \beta)}{q \beta} \right)^2 c T^{-1 / 3}, 
\end{flalign}

\noindent then we deduce from \eqref{gkderivatives} and a Taylor expansion that 
\begin{flalign}
\label{gzpsistationary}
R_K \left( \displaystyle\frac{\mathcal{F}_K (z - q \beta)}{q \beta} \right) = \mathcal{O} \big( (z - q \beta)^4 \big) + \mathcal{O} \big( (z - q \beta)^3 T^{-1 / 3} \big). 
\end{flalign}

\subsubsection{Choice of Contours \texorpdfstring{$\mathcal{C}_K$}{} and \texorpdfstring{$\Gamma_K$}{}} 

\label{Contoursckgammak}

We would now like to select contours $\mathcal{C}_K$ and $\Gamma_K$ such that $\Re \big( G_K (w) - G_K (v) \big) < 0$, for all $w \in \mathcal{C}_K$ and $v \in \Gamma_K$ away from $q \beta$, subject to the restrictions stated in Theorem \ref{determinant} and Theorem \ref{determinantw}. If we are able to do this, then the integrand on the right side of \eqref{kernelk} will decay exponentially away from $q \beta$. This will allow us to localize that integral around the critical point $q \beta$ and use the estimate \eqref{gzpsistationary} to simplify the asymptotics. For the remainder of this paper, we omit the subscript $K$ to simplify notation. 

Our choice for $\mathcal{C}$ and $\Gamma$ will be similar to the choice presented in Section 5 of \cite{SSVM} and Section 6 of \cite{PTAEPSSVM}. In those works, the authors take $\mathcal{C}$ and $\Gamma$ to follow level lines of the equation $\Re \Omega (z) = \Omega (\psi)$ (recall the definition of $\Omega (z)$ from \eqref{omegap}), which are depicted in Figure \ref{l1l2l3} below in the case $K = V$. 

These level lines were studied at length in Section 5.1 of \cite{SSVM} and Proposition 6.7 of \cite{PTAEPSSVM}. The following proposition summarizes those results.\footnote{The papers \cite{SSVM} and \cite{PTAEPSSVM} actually studied level lines of generalizations of $\Omega (z)$; Proposition \ref{linesgv} is the result of specializing Proposition 6.1 of \cite{PTAEPSSVM} to the case $\eta = \theta$ and Proposition 6.7 of \cite{PTAEPSSVM} to the case $\eta = 1 - 2b$. }

\begin{prop}[{\cite[Section 5.1]{SSVM}, \cite[Proposition 6.1]{PTAEPSSVM}, \cite[Proposition 6.7]{PTAEPSSVM}}]

\label{linesgv}

There exist three simple, closed curves, $\mathcal{L}_1 = \mathcal{L}_{1; K} $, $\mathcal{L}_2 = \mathcal{L}_{2; K} $, and $\mathcal{L}_3 = \mathcal{L}_{3; K}$, that all pass through $q \beta$ and satisfy the following properties. 

\begin{enumerate}

\item{ \label{vzpsij} Any $z \in \mathbb{C}$ satisfying $\Re \Omega (z) = \Omega (q \beta)$ lies on $\mathcal{L}_j$ for some $j \in \{1, 2, 3 \}$. }

\item{ \label{vzpsiall} Any complex number $z \in \mathcal{L}_1 \cup \mathcal{L}_2 \cup \mathcal{L}_3$, except $z = -q$ if $K = A$, satisfies $\Re \Omega (z) = \Omega (q \beta)$. }

\item{ \label{anglev} The level lines $\mathcal{L}_1$, $\mathcal{L}_2$, and $\mathcal{L}_3$ are all star-shaped. }

\item{ \label{containmentv} We have that $\mathcal{L}_1 \cap \mathcal{L}_2 = \mathcal{L}_2 \cap \mathcal{L}_3 = \mathcal{L}_1 \cap \mathcal{L}_3 = \{ q \beta \}$. Furthermore, $\mathcal{L}_1 \setminus \{ q \beta \}$ is contained in the interior of $\mathcal{L}_2$, and $\mathcal{L}_2 \setminus \{ q \beta \}$ is contained in the interior of $\mathcal{L}_3$. }

\item{ \label{qkappaq0v} The interior of $\mathcal{L}_1$ contains $0$, but not $-q$; the interior of $\mathcal{L}_2$ contains $0$, but not $- q \kappa$; and the interior of $\mathcal{L}_3$ contains $0$, $-q$, and $-q \kappa$. Moreover, if $K = V$, then the interior of $\mathcal{L}_2$ contains $-q$; if $K = A$, then $-q$ lies on the curve $\mathcal{L}_2$.  }

\item{ \label{psianglesv} The level line $\mathcal{L}_1$ meets the positive real axis (at $q \beta$) at angles $5 \pi / 6$ and $- 5 \pi /6$; the level line $\mathcal{L}_2$ meets the positive real axis (at $q \beta$) at angles $\pi / 2$ and $- \pi / 2$; and the level line $\mathcal{L}_3$ meets the positive real axis (at $q \beta$) at angles $\pi / 6$ and $- \pi /6$. }

\item{ \label{positiverealv} For all $z$ in the interior of $\mathcal{L}_2$ but strictly outside of $\mathcal{L}_1$, we have that $\Re \big( \Omega (z) - \Omega (q \beta) \big) > 0$. }

\item{ \label{negativerealv} For all $z$ in the interior of $\mathcal{L}_3$ but strictly outside of $\mathcal{L}_2$, we have that $\Re \big( \Omega (z) - \Omega (q \beta) \big) < 0$. }

\end{enumerate} 

\end{prop}

\begin{figure}

\begin{minipage}{0.45\linewidth}

\centering

\begin{tikzpicture}[
      >=stealth,
      auto,
      style={
        scale = 1
      }
			]

			\draw[<->, black	] (0, -3.5) -- (0, 3.5) node[black, above = 0] {$\Im z$};
			\draw[<->, black] (-4, 0) -- (2, 0) node[black, right = 0] {$\Re z$};
			
			\draw[->,black, thick] (1.18, 0) -- (1.35, .34641);
			\draw[-,black, thick] (1.18, 0) -- (1.35, -.34641);
			\draw[black, thick] (1.35, .34641) arc (15.4:180:1.34536) node [black, above = 42, right = 10 	] {$\mathcal{C}_V$}; 
			\draw[black, thick] (1.35, -.34641) arc (-15.4:-180:1.34536);
			
			\draw[->, black, thick] (1.05, 0) -- (.85, .34641); 
			\draw[-, black, thick] (1.05, 0) -- (.85, -.34641); 
			\draw[black, thick] (.85, .34641) arc (21.67:180:.9)  node [black, above = 12, right = 7 	] {$\Gamma_V$}; 
			\draw[black, thick] (.85, -.34641) arc (-21.67:-180:.9);

			\path[draw, dashed] (1.1, 0) -- (1.08, .02) -- (1.06, .03) -- (1.03, .04) -- (1, .05) -- (.98, .06) -- (.95, .08) -- (.93, .09) -- (.89, .1) -- (.86, .12) -- (.83, .13) -- (.8, .14) -- (.75, .16) -- (.71, .17) -- (.65, .19) -- (.61, .2) -- (.56, .21) -- (.5, .22) -- (.47, .23) -- (.43, .24) -- (.38, .24) -- (.35, .24) -- (.3, .24) -- (.27, .24) -- (.22, .24) -- (.18, .23) -- (.14, .22) -- (.08, .21) -- (.06, .19) -- (.04, .18) -- (0, .17) -- (-.02, .15) -- (-.04, .13) -- (-.06, .12) -- (-.07, .1) -- (-.08, .09) -- (-.08, .08) -- (-.09, .07) -- (-.1, .04) -- (-.11, .03) -- (-.11, 0);
			
			\path[draw, dashed] (1.1, 0) -- (1.08, -.02) -- (1.06, -.03) -- (1.03, -.04) -- (1, -.05) -- (.98, -.06) -- (.95, -.08) -- (.93, -.09) -- (.89, -.1) -- (.86, -.12) -- (.83, -.13) -- (.8, -.14) -- (.75, -.16) -- (.71, -.17) -- (.65, -.19) -- (.61, -.2) -- (.56, -.21) -- (.5, -.22) -- (.47, -.23) -- (.43, -.24) -- (.38, -.24) -- (.35, -.24) -- (.3, -.24) -- (.27, -.24) -- (.22, -.24) -- (.18, -.23) -- (.14, -.22) -- (.08, -.21) -- (.06, -.19) -- (.04, -.18) -- (0, -.17) -- (-.02, -.15) -- (-.04, -.13) -- (-.06, -.12) -- (-.07, -.1) -- (-.08, -.09) -- (-.08, -.08) -- (-.09, -.07) -- (-.1, -.04) -- (-.11, -.03) -- (-.11, 0);
			
			\draw[black, dashed] (0, 0) circle [radius=1.1];
			
			\path[draw, dashed] (1.1, 0) -- (1.23, .1) -- (1.4, .3) -- (1.57, .71) -- (1.6, 1) -- (1.57, 1.3) -- (1.4, 1.76) -- (1.23, 2.02) -- (1.07, 2.22) -- (.9, 2.37) -- (.74, 2.49) -- (.57, 2.58) -- (.4, 2.68) -- (.23, 2.75) -- (.07, 2.8) -- (-.1, 2.85) -- (-.27, 2.87) -- (-.43, 2.88) -- (-.6, 2.88) -- (-.77, 2.87) -- (-.94, 2.86) -- (-1.1, 2.83) -- (-1.27, 2.78) -- (-1.44, 2.74) -- (-1.6, 2.67) -- (-1.77, 2.59) -- (-1.94, 2.51) -- (-2.1, 2.39) -- (-2.27, 2.26) -- (-2.44, 2.12) -- (-2.6, 1.94) -- (-2.77, 1.7) -- (-2.94, 1.47) -- (-3.1, 1.13) -- (-3.27, .57) -- (-3.32, 0); 
			
			\path[draw, dashed] (1.1, 0) -- (1.23, -.1) -- (1.4, -.3) -- (1.57, -.71) -- (1.6, -1) -- (1.57, -1.3) -- (1.4, -1.76) -- (1.23, -2.02) -- (1.07, -2.22) -- (.9, -2.37) -- (.74, -2.49) -- (.57, -2.58) -- (.4, -2.68) -- (.23, -2.75) -- (.07, -2.8) -- (-.1, -2.85) -- (-.27, -2.87) -- (-.43, -2.88) -- (-.6, -2.88) -- (-.77, -2.87) -- (-.94, -2.86) -- (-1.1, -2.83) -- (-1.27, -2.78) -- (-1.44, -2.74) -- (-1.6, -2.67) -- (-1.77, -2.59) -- (-1.94, -2.51) -- (-2.1, -2.39) -- (-2.27, -2.26) -- (-2.44, -2.12) -- (-2.6, -1.94) -- (-2.77, -1.7) -- (-2.94, -1.47) -- (-3.1, -1.13) -- (-3.27, -.57) -- (-3.32, 0); 
			
			\filldraw[fill=black, draw=black] (-.5, 0) circle [radius=.03] node [black, below = 0] {$-q$};
			
			\filldraw[fill=black, draw=black] (-2, 0) circle [radius=.03] node [black, below = 0] {$-q \kappa$};

\end{tikzpicture}

\end{minipage}
\qquad
\begin{minipage}{0.45\linewidth}

\centering

\begin{tikzpicture}[
      >=stealth,
			scale=.4
			]
			
			\draw[<->] (-4, 5) -- (4, 5);
			
			\draw[->,black,very thick] (3, 1.535) -- (2, 3.268);
			\draw[-,black,very thick] (2, 3.268) -- (1, 5);
			\draw[-,black,very thick]  (2, 6.732) -- (3, 8.465) node [black, above = 0, scale = .8] {$\mathcal{C}_V^{(1)}$};
			\draw[->,black,very thick] (1, 5) -- (2, 6.732);
			
			\draw[->,black,very thick] (-3, 1.535) -- (-2, 3.268);
			\draw[-,black,very thick] (-1, 5) -- (-2, 3.268);
			\draw[-,black,very thick] (-2, 6.732) -- (-3, 8.465) node [black, above = 0, scale = .8] {$\Gamma_V^{(1)}$};
			\draw[->,black,very thick] (-1, 5) -- (-2, 6.732);
			
			\filldraw[fill=black, draw=black] (1, 5) circle [radius=.09] node [black, left = 2, below = 4, scale = .6] {$q \alpha$};
			
			\filldraw[fill=black, draw=black] (0, 5) circle [radius=.09] node [black, right = 0, above = 0, scale = .6] {$q \psi$};
			
			\filldraw[fill=black, draw=black] (-1, 5) circle [radius=.09] node [black, right = 2, below = 4, scale = .6] {$q \gamma$};
			
			\filldraw[fill=black, draw=black] (2, 5) circle [radius=.09] node [black, right = 0, above = 0, scale = .6] {$q \beta_1$};
			
			\filldraw[fill=black, draw=black] (-2, 5) circle [radius=.09] node [black, right = 0, above = 0, scale = .6] {$q \varkappa \beta_2$};
\end{tikzpicture}

\end{minipage}

\caption{\label{l1l2l3} To the left, the three level lines $\mathcal{L}_1$, $\mathcal{L}_2$, and $\mathcal{L}_3$ are depicted as dashed curves; the contours $\Gamma_V$ and $\mathcal{C}_V$ are depicted as solid curves and are labeled. To the right are two contours $\mathcal{C}^{(1)}$ and $\Gamma^{(1)}$, ``zoomed in'' around the critical point $q \beta$.  }
\end{figure}
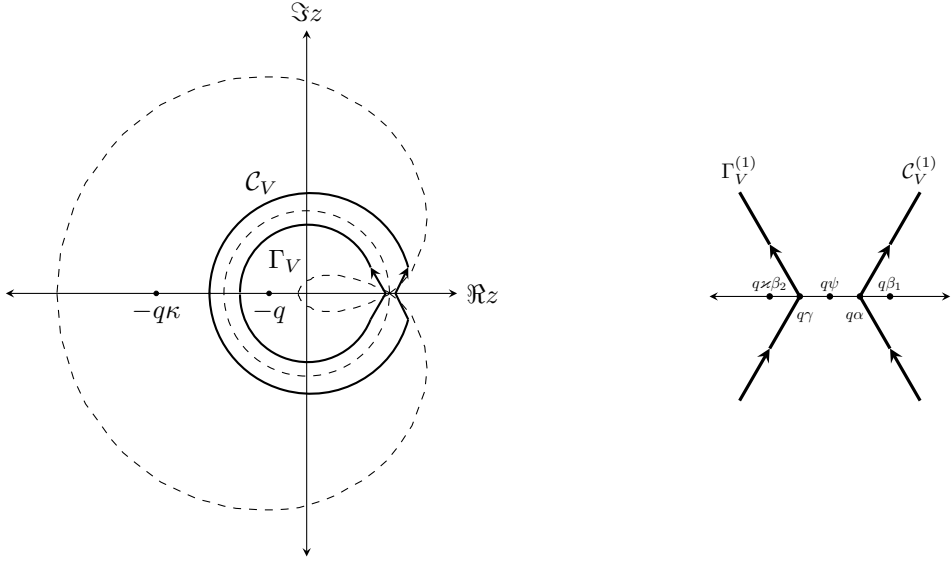

Now, let us explain how to select the contours $\mathcal{C}$ and $\Gamma$. They will each be the union of two contours, a ``small piecewise linear part'' near $q \beta$, and a ``large curved part'' that closely follows the level line $\mathcal{L}_2$. The contours given by the following definition will be the linear parts of $\mathcal{C}$ and $\Gamma$. Examples of these contours are given in Figure \ref{l1l2l3}.

\begin{definition}

\label{crgammar}

For a real number $r \in \mathbb{R}$ and a positive real number $\varepsilon > 0$ (possibly infinite), let $\mathfrak{W}_{r, \varepsilon}$ denote the piecewise linear curve in the complex plane that connects $r + \varepsilon e^{- \pi \textbf{i} / 3} $ to $r$ to $r + \varepsilon e^{\pi \textbf{i} / 3}$. Similarly, let $\mathfrak{V}_{r, \varepsilon}$ denote the piecewise linear curve in the complex plane that connects $r + \varepsilon e^{- 2 \pi \textbf{i} / 3}$ to $r$ to $r + \varepsilon e^{2 \pi \textbf{i} / 3}$. 

\end{definition} 

Now, let us make the following definitions; Definition \ref{linearvertexrightstationary} and Definition \ref{curvedvertexrightstationary} define the piecewise linear and curved parts of the contours $\mathcal{C}$ and $\Gamma$, respectively. Then, Definition \ref{vertexrightcontoursstationary} defines the contours $\mathcal{C}$ and $\Gamma$. 

In what follows, we denote
\begin{flalign}
\label{beta}
&  \alpha_K = \displaystyle\frac{3 \beta_1 + \varkappa \beta_2}{4}; \qquad \psi_K = \displaystyle\frac{\beta_1 + \varkappa \beta_2}{2}; \qquad \gamma_K = \displaystyle\frac{\beta_1 + 3 \varkappa \beta_2}{4}. 
\end{flalign} 

\noindent Observe that $|\beta - \alpha_K|, |\beta - \psi_K|, |\beta - \gamma_K| = \mathcal{O} \big( T^{-10} \big)$, due to the hypothesis $\beta = (\beta_1 + \varkappa \beta_2) / 2 + \mathcal{O} \big( T^{-10} \big)$ stated in Proposition \ref{convergencedeterminant} and Proposition \ref{convergencedeterminantmodel}. 

\begin{definition}

\label{linearvertexrightstationary} 

For each $K \in \{ V, A \}$, let $\mathcal{C}_K^{(1)} = \mathfrak{W}_{q \alpha_K, \varepsilon_K}$ and $\Gamma_K^{(1)} = \mathfrak{V}_{q \gamma_K, \varepsilon_K}$, where $\varepsilon_K$ is chosen to be sufficiently small (independent of $T$) so that the following four properties hold. 

	\begin{itemize}
\item{For sufficiently large $T$, the two conjugate endpoints of $\mathcal{C}^{(1)}$ lie between $\mathcal{L}_2$ and $\mathcal{L}_3$, so that their distances from $\mathcal{L}_2$ and $\mathcal{L}_3$ are bounded away from $0$, independent of $T$.}

\item{For all sufficiently large $T$, the two conjugate endpoints of $\Gamma^{(1)}$ lie between $\mathcal{L}_1$ and $\mathcal{L}_2$, so that their distance from $\mathcal{L}_1$ and $\mathcal{L}_2$ is bounded away from $0$, independent of $T$.}

\item{For all $z \in \mathbb{C}$ satisfying $|z - q \beta| < 2 \varepsilon_K$, we have that 
\begin{flalign}
\label{rksmall}
\left| R_K \left( \displaystyle\frac{\mathcal{F}_K (z - q \beta)}{q \beta} \right) \right|  < \mathcal{F}_K^3 \left| \displaystyle\frac{(z - q \beta)}{2} \right|^3,
\end{flalign}

\noindent where we recall the definition of $R$ from \eqref{rstationary}.}

\item{We have that $|v / w| \in (q^{1 / 2}, 1)$ for all $v, w \in \mathbb{C}$ satisfying $|v - q \beta|, |w - q \beta| < 2 \varepsilon_K$.}

\end{itemize}

\noindent Such a positive real number $\varepsilon$ is guaranteed to exist by part \ref{psianglesv} of Proposition \ref{linesgv} and the estimate \eqref{gzpsistationary}. 

\end{definition}

\begin{definition}

\label{curvedvertexrightstationary} 

Let $\mathcal{C}_K^{(2)}$ denote a positively oriented contour from the top endpoint $q \alpha_K + \varepsilon e^{\pi \textbf{i}/ 3}$ of $\mathcal{C}_K^{(1)}$ to the bottom endpoint $q \alpha_K + \varepsilon e^{- \pi \textbf{i}/ 3}$ of $\mathcal{C}^{(1)}$, and let $\Gamma_K^{(2)}$ denote a positively oriented contour from the top endpoint $q \gamma_K + \varepsilon e^{2 \pi \textbf{i}/ 3}$ of $\Gamma_K^{(1)}$ to the bottom endpoint $q \gamma_K + \varepsilon e^{-2 \pi \textbf{i}/ 3}$ of $\Gamma_K^{(1)}$, satisfying the following five properties. 

\begin{itemize} 

\item{The contour $\mathcal{C}^{(2)}$ remains between the level lines $\mathcal{L}_2$ and $\mathcal{L}_3$, so that the distance from $\mathcal{C}^{(2)}$ to $\mathcal{L}_2$ and $\mathcal{L}_3$ remains bounded away from $0$, independent of $T$.}

\item{The contour $\Gamma^{(2)}$ remains between the level lines $\mathcal{L}_1$ and $\mathcal{L}_2$, so that the distance from $\mathcal{C}^{(2)}$ to $\mathcal{L}_1$ and $\mathcal{L}_2$ remains bounded away from $0$, independent of $T$.}

\item{ The contour $\mathcal{C}_K^{(1)} \cup \mathcal{C}^{(2)}$ is star-shaped.}

\item{The contour $\Gamma_K^{(1)} \cup \Gamma^{(2)}$ is star-shaped and, if $K = V$, it does not contain $-q \kappa$.}

\item{The contours $\mathcal{C}_K^{(2)}$ and $\Gamma_K^{(2)}$ are both sufficiently close to $\mathcal{L}_{2; K}$ so that the interior of $\Gamma_K^{(1)} \cup \Gamma_K^{(2)}$ contains the image of $\mathcal{C}_K^{(1)} \cup \mathcal{C}_K^{(2)}$ under multiplication by $q^{1 / 2}$. }

\end{itemize}

Such contours $\mathcal{C}_K^{(2)}$ and $\Gamma_K^{(2)}$ are guaranteed to exist by part \ref{anglev} and part \ref{qkappaq0v} of Proposition \ref{linesgv}. 

\end{definition} 

\begin{definition}

\label{vertexrightcontoursstationary}

Set $\mathcal{C}_K = \mathcal{C}_K^{(1)} \cup \mathcal{C}_K^{(2)}$ and $\Gamma_K = \Gamma_K^{(1)} \cup \Gamma_K^{(2)}$. 

\end{definition}

The fact that these contours satisfy the properties specified in Theorem \ref{determinant} and Theorem \ref{determinantw} is a consequence of their definitions; hence, we can use them to establish Proposition \ref{convergencedeterminant} and Proposition \ref{convergencedeterminantmodel}. 

We furthermore have the following estimate, which is what we aimed for in the beginning of this section. 

\begin{lem}

\label{rightcv2gammav2exponential}

For sufficiently large $T$, there exists a positive real number $c_1 > 0$ (independent of $T$), such that 
\begin{flalign}
\label{gvwsmall}
\max \left\{ \sup_{\substack{w \in \mathcal{C} \\ v \in \Gamma^{(2)}}} \Re \big( G(w) - G(v) \big), \sup_{\substack{w \in \mathcal{C}^{(2)} \\ v \in \Gamma}} \Re \big( G(w) - G(v) \big) \right\} < - c_1.
\end{flalign}
\end{lem}

\begin{proof}

From part \ref{positiverealv} and part \ref{negativerealv} of Proposition \ref{linesgv}, we find that $\Re \big( \Omega (w) - \Omega (v) \big) < 0$ for all $w \in \mathcal{C}$ and $v \in \Gamma^{(2)}$, and that $\Re \big( \Omega (w) - \Omega (v) \big) < 0$ for all $w \in \mathcal{C}^{(2)}$ and $v \in \Gamma$ (uniformly in $T$). 

Thus, by compactness of $\mathcal{C}$ and $\Gamma$, there exists a constant $c_1' > 0$ (independent of $T$) such that  
\begin{flalign}
\label{gvwsmall2}
\max \left\{ \sup_{\substack{w \in \mathcal{C} \\ v \in \Gamma^{(2)}}} \Re \big( \Omega (w) - \Omega (v) \big), \sup_{\substack{w \in \mathcal{C}^{(2)} \\ v \in \Gamma}} \Re \big( \Omega (w) - \Omega (v) \big) \right\} < - c_1'.
\end{flalign}

\noindent Now the existence of $c_1$ as in \eqref{gvwsmall} follows from \eqref{gvwsmall2} and the fact that $\big| G (z) - \Omega (z) \big| = \mathcal{O} \big( T^{-1 / 3} \big)$ (recall \eqref{gastationary}) for all $z \in \mathcal{C} \cup \Gamma$. 
\end{proof}

\subsection{Contour Truncation} 

\label{DeterminantDifferentContour}

As explained previously, our next goal will be to ``truncate'' $\mathcal{C}$ to $\mathcal{C}^{(1)}$ and $\Gamma$ to $\Gamma^{(1)}$ without producing too much error in the Fredholm determinant $\det \big( \Id + K \big)_{L^2 (\mathcal{C})}$; once this is done, we will have $\Re v < \Re w$ for all $w \in \mathcal{C}^{(1)}$ and $v \in \Gamma^{(1)}$, which will allow us to rewrite $\det \big( \Id + K \big)_{\mathcal{C}^{(1)}}$ as a Fredholm determinant over $L^2 (\mathbb{R}_{> 0})$ in the next section.

Thus, define 
\begin{flalign}
\label{tildekzetastationary}
\begin{aligned}
\widetilde{K} (w, w')  = \displaystyle\frac{1}{2i \log q} \displaystyle\sum_{j = -\infty}^{\infty} & \displaystyle\int_{\Gamma^{(1)}} \left( \displaystyle\frac{v}{w} \right)^{\mathcal{F} s T^{1 / 3}} \displaystyle\frac{1}{\sin \big( \pi (\log q)^{-1} (\log v - \log w + 2 \pi \textbf{i}j) \big) v (w' - v)} \\
& \times \exp \Big( \big( G (w) - G (v) \big) T \Big) \displaystyle\frac{(\varkappa \beta_2 q w^{-1}; q)_{\infty} (\beta_1^{-1} q^{-1} v; q)_{\infty}}{(\varkappa \beta_2 q v^{-1}; q)_{\infty} (\beta_1^{-1} q^{-1} w; q)_{\infty}} dv, 
\end{aligned}
\end{flalign}

\noindent for all $w, w' \in \Gamma$ and each $K \in \{ V, A \}$. The difference between the integrals on the right sides of \eqref{tildekzetastationary} and \eqref{kernelk} is that the former is integrated along $\Gamma^{(1)}$ and the latter is integrated along $\Gamma$. 

The following lemma permits us to replace $\det \big( \Id + K \big)_{L^2 (\mathcal{C})}$ by $\det \big( \Id + \widetilde{K} \big)_{L^2 (\mathcal{C}^{(1)})}$, up to an error that decays exponentially in $T$.

\begin{lem}

\label{smallww1stationary}

There exist constants $c_2, C_1 > 0$, both independent of $T > 10$, such that 
\begin{flalign*}
\Big| \det \big( \Id + K \big)_{L^2 (\mathcal{C})}  - \det \big( \Id + \widetilde{K} \big)_{L^2 (\mathcal{C}^{(1)})}  \Big| < C_1 \exp \big( - c_2 T \big). 
\end{flalign*}
\end{lem}

This lemma will follow from Lemma \ref{determinantclosekernels}, Lemma \ref{determinantsmallcontour}, and estimates on $\big| K (w, w') - \widetilde{K} (w, w') \big|$ and $ \big| \widetilde{K} (w, w') \big|$. To that end, we have the following two lemmas.

\begin{lem}

\label{ktildekstationarynear}

There exist constants $c_3, C_2 > 0$ such that, for all $w, w' \in \mathcal{C}$, we have that 
\begin{flalign}
\label{smallkwkwc}
\big| K (w, w') - \widetilde{K} (w, w') \big| \le C_2 e^{-c_3 T}; \qquad \displaystyle\sup_{w \in \mathcal{C}^{(2)}} \big| \widetilde{K} (w, w') \big| \le C_2 e^{-c_3 T}.  
\end{flalign}

\end{lem}

\begin{lem}

\label{uniformstationaryk}

There exists a constant $C_3 > 0$ such that, for all $w, w' \in \mathcal{C}$, we have that 
\begin{flalign}
\label{uniformstationarykww}
\big| K (w, w') \big|, \big| \widetilde{K} (w, w') \big| \le \displaystyle\frac{C_3 \log T}{|w - q \beta| + |w' - q \beta|}.
\end{flalign}
 
\end{lem}

\begin{proof}[Proof of Lemma \ref{ktildekstationarynear}]

As we will see, this lemma will essentially be a consequence of Lemma \ref{rightcv2gammav2exponential}; recall the definition of $c_1 > 0$ from that lemma.

We first address the second estimate in \eqref{smallkwkwc}. Applying Lemma \ref{rightcv2gammav2exponential} in the integral defining $K (w, w')$, we deduce that 
\begin{flalign}
\label{kc2}
\begin{aligned}
\big|  \textbf{1}_{w \in \mathcal{C}^{(2)}} \widetilde{K} (w, w') \big| & \le \displaystyle\frac{e^{-c_1 T}}{2 |\log q|} \displaystyle\int_{\Gamma} \bigg| \displaystyle\frac{1}{v (w' - v) } \left( \displaystyle\frac{v}{w} \right)^{\mathcal{F} s T^{1 / 3}}  \displaystyle\frac{(\varkappa \beta_2 q w^{-1}; q)_{\infty} (\beta_1^{-1} q^{-1} v; q)_{\infty}}{(\beta_1^{-1} q^{-1} w; q)_{\infty} (\varkappa \beta_2 q v^{-1}; q)_{\infty} } \bigg|  \\
& \qquad \qquad \qquad  \times  \displaystyle\sum_{j = -\infty}^{\infty} \left| \displaystyle\frac{1}{\sin \big( \pi (\log q)^{-1} (\log v - \log w + 2 \pi \textbf{i}j) \big)} \right| dv. 
\end{aligned}
\end{flalign}

\noindent Now, observe that there is some sufficiently large constant $C_4 > 0$ such that the seven inequalities  
\begin{flalign}
\label{vwinequalities1}
\begin{aligned}
& \left| \displaystyle\frac{v}{w} \right|^{\mathcal{F} s T^{1 / 3}} \le e^{C_4 T^{1 / 3}}; \quad \left| \displaystyle\frac{1}{v} \right| < C_4;  \quad \left| \displaystyle\frac{(q^{-1} \beta_1^{-1} v; q)_{\infty} }{(q \varkappa \beta_2 v^{-1}; q)_{\infty}} \right| < C_4; \quad \left| \displaystyle\frac{(q \varkappa \beta_2 w^{-1}; q)_{\infty} }{(q^{-1} \beta_1^{-1} w; q)_{\infty}} \right| < C_4;\\
& \qquad \quad \left| \displaystyle\frac{1}{w - v} \right|, \left| \displaystyle\frac{1}{w' - v} \right| \le C_4 T^{10}; \qquad \left| \displaystyle\frac{1}{\sin \big( \pi (\log q)^{-1} (\log v - \log w) \big)} \right| \le \displaystyle\frac{C_4}{|v - w|};   \\
& \qquad \qquad \qquad \qquad \qquad \displaystyle\sum_{j \ne 0} \left| \displaystyle\frac{1}{\sin \big( \pi (\log q)^{-1} (\log v - \log w + 2 \pi \textbf{i}j) \big)} \right| < C_4, 
\end{aligned}
\end{flalign}

\noindent all hold, for any $v \in \Gamma$ and $w, w' \in \mathcal{C}$. Indeed, the first and second inequalities hold since $\Gamma$ and $\mathcal{C}$ are compact do not pass through $0$. The third inequality holds since $\sup_{w \in \mathcal{C}} |w - q^m \varkappa \beta_2| |w - q^m \beta_1|^{-1}$ and $\sup_{w \in \mathcal{C}} |w - q^m \varkappa \beta_2|^{-1} |w - q^m \beta_1|$ are uniformly bounded (and converge to $1$ exponentially quickly in $m$), for all $m \in \mathbb{Z}_{> 0}$ and over all $w \in \mathcal{C}$. The fourth inequality holds similar to the third. The fifth inequality holds since $|w - v|, |w' - v| \ge \alpha - \gamma = q (\beta_1 - \varkappa \beta_2) / 2 = q \beta_1 (1 - \omega) / 2 = q \beta_1 T^{-10} / 2$, for all $w' \in \mathcal{C}$ and $v \in \Gamma$. 

The sixth inequality holds by a Taylor expansion and the fact that $v / w$ is bounded away from any integral power of $q$, except for $1$. The seventh inequality holds since $\big| \sin \big( \pi (\log q)^{-1} (\log v - \log w + 2 \pi \textbf{i}j) \big) \big|$ is also bounded away from $0$ (since $j \ne 0$) and increases exponentially in $|j|$. 

Inserting the seven estimates \eqref{vwinequalities1} into \eqref{kc2}, and also using the compactness of the contour $\Gamma$, yields the second estimate in \eqref{smallkwkwc} for sufficiently large $C_1$ and sufficiently small $c_2 > 0$. 

The first estimate (on $\big| K (w, w') - \widetilde{K} (w, w') \big|$) in \eqref{smallkwkwc} is entirely analogous. Indeed, using Lemma \ref{rightcv2gammav2exponential} (to bound $\exp \big( T (G(w) - G(v)) \big)$ when $w \in \mathcal{C}$ and $v \in \Gamma^{(2)}$); the seven estimates \eqref{vwinequalities1}; and the compactness of $\Gamma$, we deduce that 
\begin{flalign*}
\big| & K (w, w')  - \widetilde{K} (w, w') \big| \\
& \quad \le \displaystyle\frac{1}{2 |\log q|} \displaystyle\int_{\Gamma^{(2)}} \displaystyle\sum_{j = -\infty}^{\infty}  \bigg| \left( \displaystyle\frac{v}{w} \right)^{\mathcal{F} s T^{1 / 3}} \displaystyle\frac{1}{\sin \big( \pi (\log q)^{-1} (\log v - \log w + 2 \pi \textbf{i}j) \big) v (w' - v)} \\
& \quad \qquad \qquad \qquad  \times \exp \Big( \big( G(w) - G(v) \big) T \Big) \displaystyle\frac{(\varkappa \beta_2 q w^{-1}; q)_{\infty} (\beta_1^{-1} q^{-1} v; q)_{\infty}}{(\varkappa \beta_2 q v^{-1}; q)_{\infty} (\beta_1^{-1} q^{-1} w; q)_{\infty}} \bigg| dv \le C_6 e^{-c_4 T}
\end{flalign*}

\noindent for some constants $c_4, C_6 > 0$. Taking $c_2 < c_4$ and $C_1 > C_6$ yields the second estimate in \eqref{smallkwkwc}. 
\end{proof}

\begin{proof}[Proof of Lemma \ref{uniformstationaryk}]

Due to the first estimate in \eqref{smallkwkwc}, it suffices to establish \eqref{uniformstationarykww} only for $\big| \widetilde{K} (w, w') \big|$. Furthermore, due to the second estimate in \eqref{smallkwkwc}, we can assume $w \in \mathcal{C}^{(1)}$. 

Now, the seven inequalities \eqref{vwinequalities1} estimate all factors of the integrand \eqref{tildekzetastationary} defining the kernel $\widetilde{K}$, except for the exponential term. To address this term, we claim that there exists constants $c_4, C_4 > 0$ (independent of $T$) such that 
\begin{flalign}
\label{exponentialfst13}
\left| \left( \displaystyle\frac{v}{w} \right)^{\mathcal{F} s T^{1 / 3}} \exp \Big( T \big( G (w) - G(v) \big) \Big) \right|  < C_4 \exp \big( -c_4 T (|w - q \beta|^3 + |v - q \beta|^3) \big),
\end{flalign}

\noindent for all $w \in \mathcal{C}^{(1)}$ and $v \in \Gamma^{(1)}$. Indeed, from \eqref{rstationary}, we have that 
\begin{flalign}
\label{tgwv}
\begin{aligned}
T \big( G (w) - G (v) \big) & = \displaystyle\frac{T \mathcal{F}^3}{3 (q \beta)^3 } \big( (w - q \beta)^3 - (v - q \beta)^3 \big) + \displaystyle\frac{c T^{2 / 3} \mathcal{F}^2}{(q \beta)^2} \big( (w - q \beta)^2 - (v - q \beta)^2 \big) \\
& \qquad + T R \left( \displaystyle\frac{\mathcal{F}_K (w - q \beta)}{q \beta} \right) - T R \left( \displaystyle\frac{\mathcal{F}_K (v - q \beta)}{q \beta}  \right).
\end{aligned}
\end{flalign}

\noindent Now, the third requirement of Definition \ref{linearvertexrightstationary} implies that \eqref{rksmall} holds for all $z \in \mathcal{C}^{(1)} \cup \Gamma^{(1)}$. Furthermore, it quickly follows from the definition of the contours $\mathcal{C}^{(1)}$ and $\Gamma^{(1)}$ (see Definition \ref{linearvertexrightstationary}), and the fact that $| \alpha - \beta|, |\gamma - \beta| = \mathcal{O} \big( T^{-10} \big)$, that there exists $C_5 > 0$ such that 
\begin{flalign}
\label{wvsmall}
\Re (w - q \beta)^3 < C_5 - \displaystyle\frac{|w - q \beta|^3}{2} ; \qquad \Re (v - q \beta)^3 > \displaystyle\frac{|v - q \beta|^3}{2} - C_5,
\end{flalign}

\noindent  for all $w \in \mathcal{C}^{(1)}$ and $v \in \Gamma^{(1)}$.

Combining \eqref{rksmall}, \eqref{tgwv}, and \eqref{wvsmall}, we deduce the existence of a constant $c_5, C_6 > 0$ such that 
\begin{flalign}
\label{tgwv1}
\Re T \big( G (w) - G (v) \big) < C_6 - c_5 T \big( |w - q \beta|^3 + |v - q \beta|^3 \big) , 
\end{flalign}

\noindent for all $w \in \mathcal{C}^{(1)}$ and $v \in \Gamma^{(1)}$. In particular, \eqref{tgwv1} implies that there exist $c_4, C_4 > 0$ such that \eqref{exponentialfst13} holds. 

Now, inserting \eqref{exponentialfst13} and the last six inequalities from \eqref{vwinequalities1} into \eqref{tildekzetastationary}, we deduce the existence of a constant $C_7, C_8, C_9 > 0$ such that 
\begin{flalign}
\label{stationarykuniform1}
\begin{aligned}
\big| \widetilde{K} (w, w') \big| & \le \displaystyle\frac{1}{2 |\log q|} \displaystyle\int_{\Gamma^{(1)}} \displaystyle\sum_{j = -\infty}^{\infty}  \bigg| \displaystyle\frac{1}{\sin \big( \pi (\log q)^{-1} (\log v - \log w + 2 \pi \textbf{i}j) \big) v (w' - v)} \\
& \qquad \quad  \times \left( \displaystyle\frac{v}{w} \right)^{\mathcal{F} s T^{1 / 3}} \exp \Big( \big( G(w) - G(v) \big) T \Big) \displaystyle\frac{(\varkappa \beta_2 q w^{-1}; q)_{\infty} (\beta_1^{-1} q^{-1} v; q)_{\infty}}{(\varkappa \beta_2 q v^{-1}; q)_{\infty} (\beta_1^{-1} q^{-1} w; q)_{\infty}} \bigg| dv  \\
& < C_7 \displaystyle\oint_{\Gamma} \Bigg( 1 + \left| \displaystyle\frac{1}{(v - w) (w' - v)} \right| \Bigg) dv \le C_9 + \displaystyle\frac{C_8  \log T }{|w - q \beta| + |w' - q \beta|}. 
\end{aligned} 
\end{flalign}

\noindent The last estimate in \eqref{stationarykuniform1} is due to compactness of the contour $\Gamma$, and also due to the fact that $|v - w|$ and $|w' - v|$ are both uniformly bounded away from $0$, except when $v$ and $w$ (or $v$ and $w'$) are both in a neighborhood of $q \beta$, in which case it is bounded below by some multiple of $T^{-10}$. 

Now, \eqref{uniformstationarykww} follows from \eqref{stationarykuniform1} and the compactness of the contour $\mathcal{C}$. 
\end{proof}

\noindent Now, we can establish Lemma \ref{smallww1stationary}. 

\begin{proof}[Proof of Lemma \ref{smallww1stationary}]

It suffices to show that 
\begin{flalign}
\label{twosidedkernelc2small1}
\Big| \det \big( \Id + K \big)_{L^2 (\mathcal{C})} - \det \big( \Id + \widetilde{K} \big)_{L^2 (\mathcal{C})}  \Big| & < \displaystyle\frac{C_1}{2 e^{c_1 T}}; \\
\label{twosidedkernelc2small2}
 \Big| \det \big( \Id + \widetilde{K} \big)_{L^2 (\mathcal{C})} - \det \big( \Id + \widetilde{K} \big)_{L^2 (\mathcal{C}^{(1)})}  \Big| & < \displaystyle\frac{C_1}{2 e^{c_1 T}}, 
\end{flalign}

\noindent for some constants $c_1, C_1 > 0$. 

The first estimate \eqref{twosidedkernelc2small1} will be a consequence of Lemma \ref{determinantclosekernels}, the first estimate \eqref{smallkwkwc} on $|K - \widetilde{K}|$, and the uniform estimate \eqref{uniformstationarykww}. The second estimate \eqref{twosidedkernelc2small2} will be a consequence of Lemma \ref{determinantsmallcontour}, the second estimate in \eqref{smallkwkwc} on $\big| K(w, w') \big|$ when $w \in \mathcal{C}^{(2)}$, and the uniform estimate \eqref{uniformstationarykww} on $\widetilde{K}$. 

Given this, the derivations of both estimates \eqref{twosidedkernelc2small1} and \eqref{twosidedkernelc2small2} are entirely analogous, so we only go through \eqref{twosidedkernelc2small1}. To that end, we apply Lemma \ref{determinantclosekernels} to deduce that that 
\begin{flalign}
\label{twosidedkernelc2small3}
\Big| \det \big( \Id + K \big)_{L^2 (\mathcal{C})} - \det \big( \Id + \widetilde{K} \big)_{L^2 (\mathcal{C})}  \Big| < C, 
\end{flalign}

\noindent where 
\begin{flalign}
\label{c}
\begin{aligned}
C =  \displaystyle\sum_{k = 1}^{\infty} & \displaystyle\frac{2^k k^{k / 2}}{(k - 1)!} \displaystyle\oint_{\mathcal{C}} \cdots \displaystyle\oint_{\mathcal{C}}  \left| \displaystyle\frac{1}{k} \displaystyle\sum_{j = 1}^k \big| K(w_i, w_j) - \widetilde{K} (w_i, w_j)\big|^2 \right|^{1 / 2} d w_1  \\
& \times \displaystyle\prod_{i = 2}^k \left| \displaystyle\frac{1}{k} \displaystyle\sum_{j = 1}^k \Big( \big| K(w_i, w_j) \big|^2 + \big| K(w_i, w_j) - \widetilde{K} (w_i, w_j)  \big|^2 \Big) \right|^{1 / 2} d w_i. 
\end{aligned}
\end{flalign}

 Now, recall the definitions of $c_3$, $C_2$, and $C_3$ from Lemma \ref{ktildekstationarynear} and Lemma \ref{uniformstationaryk}, and also assume that $T$ is sufficiently large so that $C_2 e^{-c_2 T} < 1$ (this assumption amounts to enlarging the constant $C_1$ in \eqref{twosidedkernelc2small1}). 

Inserting the first estimate in \eqref{smallkwkwc}, the estimate on $K$ in \eqref{uniformstationarykww}, and the estimate $|w_i - q \beta|^2 + |w_j - q \beta|^2  \ge |w_i - q \beta|^2$ into the definition of $C$, we deduce that 
\begin{flalign}
\label{twosidedkernelc2small4}
\begin{aligned}
C & \le C_2 e^{-c_3 T}  \displaystyle\sum_{k = 1}^{\infty} \displaystyle\frac{2^k k^{k / 2}}{(k - 1)!} \displaystyle\oint_{\mathcal{C}} \cdots \displaystyle\oint_{\mathcal{C}} \displaystyle\prod_{i = 1}^k \left| 1 + \displaystyle\frac{C_3^2 (\log T)^2 }{|w_i - q \beta|^2}  \right|^{1 / 2} d w_i \\
& \le C_2 e^{-c_3 T}  \displaystyle\sum_{k = 1}^{\infty} \displaystyle\frac{2^k k^{k / 2} \big( \log T \big)^k }{(k - 1)!} \Bigg| \displaystyle\oint_{\mathcal{C}} \bigg| 1 + \displaystyle\frac{C_3^2 }{|w - q \beta|^2}  \bigg|^{1 / 2} d w \Bigg|^k  \\
& \le e^{-c_3 T} \displaystyle\sum_{k = 1}^{\infty} \displaystyle\frac{C_4^k k^{k / 2} (\log T)^{2k} }{(k - 1)!} \le \exp \big( C_5 (\log T)^{C_5} - c_3 T \big), 
\end{aligned}
\end{flalign}

\noindent for some constants $C_4, C_5 > 0$. The third estimate above holds since $\mathcal{C}$ is compact and since $|w - q \beta|$ is bounded below by some multiple of $T^{-10}$, for all $w \in \mathcal{C}$. Inserting \eqref{twosidedkernelc2small4} into \eqref{twosidedkernelc2small3}, we deduce the lemma. 
\end{proof}

\noindent From Lemma \ref{smallww1stationary}, we obtain the following corollary.

\begin{cor}
\label{limitgammac}

We have that 
\begin{flalign*}
\displaystyle\lim_{T \rightarrow \infty} \displaystyle\frac{1}{1 - \omega} \Big( \det \big( \Id + K \big)_{\mathcal{C}} - \det \big( \Id + \widetilde{K} \big)_{\mathcal{C}^{(1)}} \Big) = 0. 
\end{flalign*} 
\end{cor}

\begin{proof}
From Lemma \ref{smallww1stationary}, we have that $\det \big( \Id + K \big)_{\mathcal{C}} - \det \big( \Id + \widetilde{K} \big)_{\mathcal{C}^{(1)}} = \mathcal{O} \big( e^{-c_2 T} \big)$, for some constant $c_2 > 0$. Thus, the corollary follows from the fact that $(1 - \omega)^{-1} = T^{10}$. 
\end{proof}

In view of Corollary \ref{limitgammac}, we can now direct our attention towards the ``truncated'' Fredholm determinant $(1 - \omega)^{-1} \det \big( \Id + \widetilde{K} \big)_{\mathcal{C}^{(1)}} $.

\subsection{Reformulating the Fredholm Determinant}

\label{ckr}

In this section, we rewrite the ``truncated'' Fredholm determinant $(1 - \omega)^{-1} \det \big( \Id + \widetilde{K} \big)_{\mathcal{C}^{(1)}} $ as a Fredholm determinant over $L^2 (\mathbb{R}_{> 0})$. There is a reasonably well-known way to do this that is based on a commutation identity for the Fredholm determinant; we will explain this in Section \ref{RealDeterminant}. 

Unfortunately, if one directly attempts to asymptotically analyze the kernel (over $L^2 (\mathbb{R}_{> 0})$) from Section \ref{RealDeterminant}, the result will be singular in the large time limit. So, in Section \ref{TruncationKernel}, we conjugate and truncate this kernel to ensure it remains regular as $T$ tends to $\infty$.

\subsubsection{Rewriting the Determinant Over \texorpdfstring{$L^2 (\mathbb{R}_{> 0})$}{}}

\label{RealDeterminant}

Having truncated the contours $\mathcal{C}$ and $\Gamma$ to $\mathcal{C}^{(1)}$ and $\Gamma^{(1)}$, respectively, we now rescale around $q \beta$. To that end, define $\sigma = \sigma_K = \mathcal{F}_K^{-1} T^{-1 / 3}$, and change variables  
\begin{flalign}
\label{stationaryvvhatwwhat}
v = q \psi + q \psi \sigma \widehat{v}; \quad w = q \psi + q \psi \sigma \widehat{w}; \quad w' = q \psi + q \psi \sigma \widehat{w'}; \quad \widehat{K} (\widehat{w}, \widehat{w'}) = q \psi \sigma \widetilde{K} (w, w'),
\end{flalign} 

\noindent where we recall the definition of $\psi = \psi_K$ from \eqref{beta}. Furthermore, for any contour $\mathcal{D} \subset \mathbb{C}$, define $\widehat{\mathcal{D}} = q^{-1} \psi^{-1} \sigma^{-1} \big( \mathcal{D} - q \psi \big)$, which is the set consisting of all numbers of the form $q^{-1} \psi^{-1} \sigma^{-1} (z - q \psi)$, for some $z \in \mathcal{D}$. 

The change of variables \eqref{stationaryvvhatwwhat} implies that 
\begin{flalign}
\label{stationarydeterminants1}
\det \big( \Id + \widetilde{K} \big)_{L^2 (\mathcal{C}^{(1)})} = \det \big( \Id + \widehat{K} \big)_{L^2 (\widehat{\mathcal{C}}^{(1)})}. 
\end{flalign}

\noindent Implementing the change of variables \eqref{stationaryvvhatwwhat} in the definition \eqref{tildekzetastationary} of $\widetilde{K}$, we deduce that $\widehat{K}$ can be explicitly written as 
\begin{flalign}
\label{kernelkhat}
\begin{aligned}
\widehat{K} (\widehat{w}, \widehat{w'}) & = \displaystyle\frac{\sigma}{2 \textbf{i} \log q} \displaystyle\sum_{j = -\infty}^{\infty} \displaystyle\int_{\widehat{\Gamma}^{(1)}} \left( \displaystyle\frac{1 + \sigma \widehat{v}}{1 + \sigma \widehat{w}} \right)^{\mathcal{F} s T^{1 / 3}}\displaystyle\frac{ \big( \varkappa \beta_2 (\psi + \psi \sigma \widehat{w})^{-1}; q \big)_{\infty} \big( \beta_1^{-1} (\psi + \psi \sigma \widehat{v}); q \big)_{\infty} }{ \big( \beta_1^{-1} (\psi + \psi \sigma \widehat{w}); q \big)_{\infty} \big( \varkappa \beta_2 \big( \psi + \psi \sigma \widehat{v})^{-1}; q \big)_{\infty}} \\
& \qquad  \times \displaystyle\frac{\exp \Big( \big( G(q \psi + q \psi \sigma \widehat{w}) - G( q \psi + q \psi \sigma \widehat{v}) \big) T \Big) d \widehat{v} }{\big( 1 + \sigma \widehat{v} \big) (\widehat{v} - \widehat{w'}) \sin \big( \pi (\log q)^{-1} (\log (1 + \sigma \widehat{v}) - \log (1 + \sigma \widehat{w}) + 2 \pi \textbf{i}j) \big)}. 
\end{aligned}
\end{flalign}

\noindent Furthermore, observe from Definition \ref{linearvertexrightstationary} that if we denote
\begin{flalign*}
\Delta = \Delta_K = \displaystyle\frac{\beta_1 - \varkappa \beta_2}{2 \psi} = \displaystyle\frac{\beta_1 - \psi}{\psi} = \displaystyle\frac{\psi - \varkappa \beta_2}{\psi} = \mathcal{O} \big( T^{-10} \big), 
\end{flalign*}

\noindent then $\widehat{\Gamma}^{(1)} = \mathfrak{V}_{- \Delta / 2 \sigma, \varepsilon / q \psi \sigma }$ and $\widehat{\mathcal{C}}^{(1)} = \mathfrak{V}_{\Delta / 2 \sigma, \varepsilon / q \psi \sigma }$. 

Using the identity \eqref{kernelkhat}, we will rewrite the Fredholm determinant $\det \big( \Id + \widehat{K} \big)_{\widehat{\mathcal{C}_K}^{(1)}}$ as Fredholm determinant over $L^2 (\mathbb{R}_{> 0})$. To that end, define (for each $K \in \{ V, A \}$) the kernel $L (y, y') = L_K (y, y')$ through the identity 
\begin{flalign}
\label{lkdefinition}
\begin{aligned}
L (y, y') = \displaystyle\frac{\sigma \pi}{ (2 \pi \textbf{i})^2 \log q} & \displaystyle\sum_{j = -\infty}^{\infty} \displaystyle\int_{\widehat{\mathcal{C}}^{(1)}} \displaystyle\int_{\widehat{\Gamma}^{(1)}} \displaystyle\frac{\exp \Big( \big( G(q \psi + q \psi \sigma w) - G( q \psi + q \psi \sigma v) \big) T + vy - w y' \Big)}{\big( 1 + \sigma v \big) \sin \big( \pi (\log q)^{-1} (\log (1 + \sigma w) - \log (1 + \sigma v) + 2 \pi \textbf{i}j) \big)} \\
& \times \left( \displaystyle\frac{1 + \sigma v}{1 + \sigma w} \right)^{\mathcal{F} s T^{1 / 3}} \displaystyle\frac{ \big( \varkappa \beta_2 (\psi + \psi \sigma w)^{-1}; q \big)_{\infty} \big( \beta_1^{-1} (\psi + \psi \sigma v); q \big)_{\infty} }{\big( \varkappa \beta_2 \big( \psi + \psi \sigma v)^{-1}; q \big)_{\infty} \big( \beta_1^{-1} (\psi + \psi \sigma w); q \big)_{\infty}} dv dw, 
\end{aligned}
\end{flalign}

\noindent for all positive real numbers $y, y' > 0$. 

\begin{rem}

\label{cc1ggamma1stationarybd}

Observe that $\Re v \le - \Delta / 2 \sigma < 0 < \Delta / 2 \sigma \le \Re w$ for all $w \in \widehat{\mathcal{C}}^{(1)}$ and $v \in \widehat{\Gamma}^{(1)}$. Thus, since $\mathcal{C}^{(1)}$ and $\Gamma^{(1)}$ are compact, we deduce from \eqref{lkdefinition} that there exists a constant $C$ (dependent on $T$) such that 
\begin{flalign}
\label{lexponential}
\big| L (y, y') \big| < C \exp \left( - \displaystyle\frac{\Delta}{2 \sigma} (y + y') \right).
\end{flalign}

\noindent In particular, by the first statement of Lemma \ref{determinantclosekernels}, it follows that the Fredholm determinant $\det \big( \Id - L_K \big)_{L^2 (\mathbb{R}_{> 0})}$ converges as a series of the form \eqref{determinantsum}. 

Important here was the fact that $\Re v < 0 < \Re w$ for all $w \in \widehat{\mathcal{C}}^{(1)}$ and $v \in \widehat{\Gamma}^{(1)}$. This is not true over the more general domain $\widehat{v} \in \widehat{\Gamma}$ and $\widehat{w} \in \widehat{\mathcal{C}}$. Hence, if we did not perform the truncation in Section \ref{DeterminantDifferentContour}, we would have not obtained an estimate of the form \eqref{lexponential} on $L$ and so the Fredholm determinant  $\det \big( \Id - L_K \big)_{L^2 (\mathbb{R}_{> 0})}$ would have not been guaranteed to converge. 

\end{rem}

In view of Remark \ref{cc1ggamma1stationarybd}, the Fredholm determinant $\det \big( \Id - L_K \big)_{L^2 (\mathbb{R}_{> 0})}$ exists. The following proposition shows why it is relevant. 

\begin{lem}
\label{determinantreals}

For each $K \in \{ V, A \}$ and fixed $T > 10$ we have that 
\begin{flalign}
\label{determinantckdeterminantreals}
\det \big( \Id + \widehat{K} \big)_{L^2 (\widehat{\mathcal{C}_K}^{(1)})} = \det \big( \Id - L_K \big)_{L^2 (\mathbb{R}_{> 0})}. 
\end{flalign}
\end{lem} 

\begin{proof}

For any $r \in \mathbb{R}_{> 0}$ and $z \in \widehat{\mathcal{C}}^{(1)}$, define $B(r, z) = e^{-rz}$ and 
\begin{flalign}
\label{kerneld}
\begin{aligned}
D(z, r) = \displaystyle\frac{- \sigma}{2 \textbf{i} \log q} \displaystyle\sum_{j = -\infty}^{\infty} & \displaystyle\int_{\widehat{\Gamma}^{(1)}} \displaystyle\frac{\exp \Big( \big( G(q \psi + q \psi \sigma z) - G( q \psi + q \psi \sigma v) \big) T + v r \Big) }{\big( 1 + \sigma v \big) \sin \big( \pi (\log q)^{-1} (\log (1 + \sigma v) - \log (1 + \sigma z) + 2 \pi \textbf{i}j) \big)}   \\
& \times \left( \displaystyle\frac{1 + \sigma v}{1 + \sigma z} \right)^{\mathcal{F} s T^{1 / 3}} \displaystyle\frac{\big( \varkappa \beta_2 (\psi + \psi \sigma z)^{-1}; q \big)_{\infty} \big( \beta_1^{-1} (\psi + \psi \sigma v); q \big)_{\infty}}{\big( \varkappa \beta_2 \big( \psi + \psi \sigma v)^{-1}; q \big)_{\infty} \big( \beta_1^{-1} (\psi + \psi \sigma z); q \big)_{\infty}} dv. 
\end{aligned}
\end{flalign}

\noindent Then, we find that 
\begin{flalign*}
\widehat{K} (\widehat{w}, \widehat{w'}) = - \displaystyle\int_0^{\infty} D(\widehat{w}, r) B(r, \widehat{w'}) dr; \qquad L (y, y') = \displaystyle\frac{1}{2 \pi \textbf{i}} \displaystyle\int_{\mathcal{C}^{(1)}} B (w, y) D(y', w) dw, 
\end{flalign*} 

\noindent from which we deduce that $\widehat{K} = - D B$ and $L = BD$. 

Thus, the lemma follows from the fact that
\begin{flalign*}
\det \big( \Id + \widehat{K} \big)_{L^2 (\widehat{\mathcal{C}}^{(1)}} = \det \big( \Id - D B \big)_{L^2 (\widehat{\mathcal{C}}^{(1)})} = \det \big( \Id - B D \big)_{L^2 (\mathbb{R}_{> 0})} = \det \big( \Id - L \big)_{L^2 (\mathbb{R}_{> 0})}.
\end{flalign*}
\end{proof}

\subsubsection{Asymptotic Regularization}

\label{TruncationKernel}

As defined in \eqref{lkdefinition}, the kernel $L (y, y')$ will be singular in the infinite $T$ limit. To regularize this asymptotic singularity, we must conjugate $L$, by replacing it with $e^{c (y - y')} L (y, y')$. However, the kernel $e^{c (y - y')} L (y, y')$ happens to be singular for finite $T$ (see, for instance, Remark \ref{gsmallyy}). Thus, we will perform another truncation, by setting the kernel to zero if either $y$ or $y'$ is sufficiently large. 

Specifically, for any real numbers $a, \Upsilon_T \in \mathbb{R}$, we define 
\begin{flalign}
\label{barldefinition}
\overline{L}^{(a)} (y, y') = \overline{L}_K^{(a; \Upsilon_T)} (y, y') = e^{a (y - y')} \textbf{1}_{|y| < \Upsilon_T} \textbf{1}_{|y'| < \Upsilon_T} L (y, y').   
\end{flalign}

The following lemma shows that, if we take the family $\{ \Upsilon_T \}$ to increase sufficiently quickly with $T$, then we can replace $L$ with $\overline{L}^{(c)}$ in the large $T$ limit. 

\begin{lem}

\label{lclexponential}

Recall the definition of $c$ from Proposition \ref{convergencedeterminant} (in the case $K = A$) and Proposition \ref{convergencedeterminantmodel} (in the case $K = V$). There exists a sufficiently quickly growing family $\{ \Upsilon_T \} \subset \mathbb{R}_{> 0}$ of positive real numbers such that $\Upsilon_T > e^T$ for each $T > 1$ and
\begin{flalign}
\label{largedeterminantc}
\displaystyle\lim_{T \rightarrow \infty} \displaystyle\frac{1}{1 - \omega_T} \Big| \det \big( \Id - L \big)_{L^2 (\mathbb{R}_{> 0})}  - \det \big( \Id - \overline{L}^{(c; \Upsilon_T)} \big)_{L^2 (\mathbb{R}_{> 0})}  \Big| = 0. 
\end{flalign}

\end{lem}

\begin{proof}

First, observe from the definition \eqref{determinantsum} of a Fredholm determinant that 
\begin{flalign}
\label{lcl0}
\det \big( \Id - \overline{L}^{(c; \Upsilon_T)} \big)_{L^2 (\mathbb{R}_{> 0})} = \det \big( \Id - \overline{L}^{(a; \Upsilon_T)} \big)_{L^2 (\mathbb{R}_{> 0})}, 
\end{flalign}

\noindent for any $a, \Upsilon_T \in \mathbb{R}$. 

Thus, it suffices to establish the lemma in the case $c = 0$. To that end, we claim that 
\begin{flalign}
\label{largedeterminant}
\displaystyle\lim_{\Upsilon \rightarrow \infty} \Big( \det \big( \Id - \overline{L}^{(0; \Upsilon)} \big)_{L^2 (\mathbb{R}_{> 0})} - \det \big( \Id - \overline{L}^{(0; \infty)} \big)_{L^2 (\mathbb{R}_{> 0})} \Big) = 0,
\end{flalign}

\noindent for each fixed $T > 0$. 

To see this, apply Corollary \ref{determinantlimitkernels}\footnote{The parameter $T$ in that corollary is replaced by $\Upsilon$ in this context.} to the sequence of contours $\mathcal{C}_{\Upsilon} = \mathbb{R}_{> 0}$, and the sequence of kernels $\{ \overline{L}^{(0; \Upsilon)}_{\Upsilon > 0} \}$ converging to $L$. The estimate \eqref{lexponential} from Remark \ref{cc1ggamma1stationarybd} implies that the kernels $\overline{L}^{(0; \Upsilon)}$ are dominated by a function $\textbf{K}: \mathbb{C} \rightarrow \mathbb{R}_{\ge 0}$ of the form $\textbf{K} (y) = C e^{-\Delta y / 2 \sigma}$. The three conditions of Lemma \ref{determinantlimitkernels} are then quickly verified, from which we deduce \eqref{largedeterminant}. 

Now, \eqref{largedeterminant} implies the existence of the family $\{ \Upsilon_T \}$ satisfying $\Upsilon_T > e^T$ and 
\begin{flalign}
\label{largedeterminant0}
\displaystyle\lim_{T \rightarrow \infty} \displaystyle\frac{1}{1 - \omega_T} \Big| \det \big( \Id - L \big)_{L^2 (\mathbb{R}_{> 0})}  - \det \big( \Id - \overline{L}^{(0; \Upsilon_T)} \big)_{L^2 (\mathbb{R}_{> 0})}  \Big| = 0. 
\end{flalign}

\noindent Thus, \eqref{largedeterminantc} follows from \eqref{lcl0} (applied with $a = c$) and \eqref{largedeterminant0}.
\end{proof}

\noindent Due to Lemma \ref{lclexponential}, we can now focus on the Fredholm determinant $\det \big( \Id - \overline{L}^{(c)} \big)_{L^2 (\mathbb{R}_{> 0})}$ instead of $\det \big( \Id - L \big)_{L^2 (\mathbb{R}_{> 0})}$.

\section{Analysis of the Determinant}

\label{RightStationary}

In this section we establish Proposition \ref{convergencedeterminant} and Proposition \ref{convergencedeterminantmodel} through analysis of the Fredholm determinant $\det \big( \Id - \overline{L} \big)_{L^2 (\mathbb{R}_{> 0})}$, where we abbreviate $\overline{L} = \overline{L}^{(c)}$. The obstruction that prevents us from doing this immediately is that the contours $\widehat{\mathcal{C}}^{(1)}$ and $\widehat{\Gamma}^{(1)}$ still nearly touch at $0$, which poses an issue due to the pole at $v = w$ in the integral on the right side of \eqref{lkdefinition} defining the kernel $\overline{L}^{(c)}$.

To remedy this, we follow what was done in Section 7 of \cite{HFSE}. Specifically, we deform the contours $\widehat{\mathcal{C}}^{(1)}$ and $\widehat{\Gamma}^{(1)}$ to obtain contours $\widetilde{\mathcal{C}}$ and $\widetilde{\Gamma}$ that remain detached in the infinite $T$ limit. This deformation produces residues in the integral defining $\overline{L}^{(c)}$, which must be tracked. This will be done in Section \ref{DeformationResidues}. The resulting expression for $\overline{L}^{(c)}$ becomes a finite-rank perturbation for a more manageable kernel, which allows for the determinant $\det \big( \Id - \overline{L}^{(c)} \big)$ to be rewritten more explicitly. This will be done in Section \ref{DeterminantFiniteRankPerturbation}. Then, the resulting expression for this Fredholm determinant will then become amenable to a saddle-point asymptotic analysis, which will lead to the proofs of Proposition \ref{convergencedeterminant} and Proposition \ref{convergencedeterminantmodel} in Section \ref{TLarge}.

\subsection{Deforming the Contours  \texorpdfstring{$\widehat{\mathcal{C}}^{(1)}$}{} and \texorpdfstring{$\widehat{\Gamma}^{(1)}$}{}}

\label{DeformationResidues} 

As mentioned previously, our first goal is to deform the contours $\widehat{\mathcal{C}}^{(1)}$ and $\widehat{\Gamma}^{(1)}$ so that they remain detached in the infinite $T$ limit. To that end, define the following contours, which are depicted in Figure \ref{cgammahatcgammatilde}. 

\begin{definition}

\label{ctildegammatilde}

Recalling the definition of $c$ from Proposition \ref{convergencedeterminant} and Proposition \ref{convergencedeterminantmodel}, fix a positive real number $E > 2 |c| + 1$. Define the contours $\widetilde{\mathcal{C}} = \mathfrak{W}_{E; \varepsilon / q \psi \sigma}$ and $\widetilde{\Gamma} = \mathfrak{V}_{-E; \varepsilon / q \psi \sigma}$, where we recall that $\varepsilon$ was given in Definition \ref{linearvertexrightstationary}. Stated alternatively, $\widetilde{\mathcal{C}}$ is formed by translating $\widehat{\mathcal{C}}^{(1)}$ to the right by $E - \Delta / 2 \sigma$, and $\widetilde{\Gamma}$ is formed by translating $\widehat{\Gamma}^{(1)}$ to the left by $E - \Delta / 2 \sigma$. 

\end{definition}

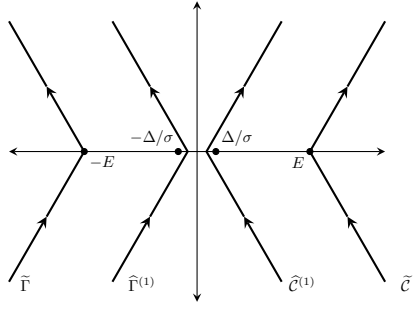
\begin{figure}

\begin{center}

\begin{tikzpicture}[
      >=stealth,
			scale = .5
			]
			
			\draw[<->] (-5, 0) -- (5, 0);
			\draw[<->] (0, -4) -- (0, 4);
			
			\draw[->,black, thick]  (5, -3.465) -- (4, -1.732) node[black, right = 22, below=20, scale = .6] {$\widetilde{\mathcal{C}}$};
			\draw[-,black, thick] (4, -1.732) -- (3, 0);
			\draw[-,black, thick] (5, 3.465) -- (4, 1.732);
			\draw[->,black, thick] (3, 0) -- (4, 1.732);
			
			\draw[->,black, thick] (-5, -3.465) -- (-4, -1.732) node[black, left = 8, below=20, scale = .6] {$\widetilde{\Gamma}$};
			\draw[-,black, thick] (-3, 0) -- (-4, -1.732);
			\draw[-,black, thick] (-5, 3.465) -- (-4, 1.732);
			\draw[->,black, thick] (-3, 0) -- (-4, 1.732);
			
			\draw[->,black, thick] (2.25, -3.465) -- (1.25, -1.732) node[black, right = 22, below=20, scale = .6] {$\widehat{\mathcal{C}}^{(1)}$};
			\draw[-,black, thick] (1.25, -1.732) -- (.25, 0);
			\draw[-,black, thick] (2.25, 3.465) -- (1.25, 1.732);
			\draw[->,black, thick] (.25, 0) -- (1.25, 1.732);
				
			\draw[->,black, thick] (-2.25, -3.465) -- (-1.25, -1.732) node[black, left = 3, below=20, scale = .6] {$\widehat{\Gamma}^{(1)}$};
			\draw[-,black, thick] (-.25, 0) -- (-1.25, -1.732);
			\draw[-,black, thick] (-2.25, 3.465) -- (-1.25, 1.732);
			\draw[->,black, thick] (-.25, 0) -- (-1.25, 1.732);

			\draw[black, fill] (.5, 0) circle [radius=.08] node [black, right = 8, above=0, scale = .6] {$\Delta / \sigma$};
			\draw[black, fill] (-.5, 0) circle [radius=.08] node [black, left = 10, above=0, scale = .6] {$ - \Delta / \sigma$};
			\draw[black, fill] (-3, 0) circle [radius=.08] node [black, below=4, right = 0, scale = .6] {$-E$};
			\draw[black, fill] (3, 0) circle [radius=.08] node [black, below=4, left = 0, scale = .6] {$E$};
\end{tikzpicture}

\end{center}

\noindent \caption{\label{cgammahatcgammatilde} Depicted above, ordered from left to right, are the contours $\widetilde{\Gamma}$, $\widehat{\Gamma}^{(1)}$, $\widehat{\mathcal{C}}^{(1)}$, and $\widetilde{\mathcal{C}}$. }	
\end{figure}

Let us examine how deforming $\widehat{\mathcal{C}}^{(1)}$ to $\widetilde{\mathcal{C}}$ and $\widehat{\Gamma}^{(1)}$ to $\widetilde{\Gamma}$ affects the kernel $\overline{L}$. To that end, observe that $\overline{L}$ can be rewritten as 
\begin{flalign}
\label{ldefinition}
\overline{L} (y, y') = \displaystyle\frac{1}{ (2 \pi \textbf{i})^2 } \displaystyle\int_{\widehat{\mathcal{C}}^{(1)}} \displaystyle\int_{\widehat{\Gamma}^{(1)}} I (y, y'; w, v) dw dv,
\end{flalign}

\noindent where
\begin{flalign}
\label{stationaryintegrandlc}
\begin{aligned}
I (y, y'; w, v) & = \textbf{1}_{|y| < \Upsilon_T} \textbf{1}_{|y'| < \Upsilon_T}  \exp \Big( \big( G(q \psi + q \psi \sigma w) - G( q \psi + q \psi \sigma v) \big) T + (v + c) y - (w + c) y' \Big) \\
& \quad \times \displaystyle\frac{1}{1 + \sigma v } \left( \displaystyle\frac{1 + \sigma v}{1 + \sigma w} \right)^{\mathcal{F} s T^{1 / 3}} \displaystyle\frac{ \big( \varkappa \beta_2 (\psi + \psi \sigma w)^{-1}; q \big)_{\infty} \big( \beta_1^{-1} (\psi + \psi \sigma v); q \big)_{\infty} }{\big( \varkappa \beta_2 \big( \psi + \psi \sigma v)^{-1}; q \big)_{\infty} \big( \beta_1^{-1} (\psi + \psi \sigma w); q \big)_{\infty}}  \\
& \quad \times \displaystyle\frac{\sigma \pi}{\log q} \displaystyle\sum_{j = -\infty}^{\infty} \displaystyle\frac{1}{\sin \big( \pi (\log q)^{-1} (\log (1 + \sigma w) - \log (1 + \sigma v) + 2 \pi \textbf{i}j) \big)}. 
\end{aligned}
\end{flalign}

Now, deforming $\widehat{\mathcal{C}}^{(1)}$ to $\widetilde{\mathcal{C}}$ crosses the pole at $w = (\beta_1 - \varkappa \beta) / \sigma \psi = \Delta / \sigma$ of the integrand on the right side of \eqref{stationaryintegrandlc}. Similarly, deforming $\widehat{\Gamma}^{(1)}$ to $\widetilde{\Gamma}$ crosses the pole at $v = (\varkappa \beta_2 - \beta) / \sigma \psi = - \Delta / \delta$ of this integrand. If $T$ is sufficiently large (which we will always assume to be the case) then these are the only poles of the integrand crossed by the deformation. 

Thus, deforming the contours in the identity \eqref{ldefinition} involves summing the residues corresponding to the poles at $w = \sigma^{-1} \Delta $ or $v = - \sigma^{-1} \Delta$ (or both). In particular, we find that 
\begin{flalign}
\label{lksum}
\overline{L} (y, y') = \widetilde{L} (y, y') + \mathscr{V} (y, y') + \mathscr{W} (y, y') + \mathscr{B} (y, y'), 
\end{flalign}

\noindent where 
\begin{flalign}
\label{ltildedefinition}
\widetilde{L} (y, y') & = \displaystyle\frac{1}{( 2 \pi \textbf{i})^2}  \displaystyle\int_{\widetilde{\mathcal{C}}} \displaystyle\int_{\widetilde{\Gamma}} I (y, y'; w, v) dw dv,
\end{flalign}

\noindent coincides with the definition of $\overline{L}$ in \eqref{ldefinition}, but where the integration is now over the contours $\widetilde{\mathcal{C}}$ and $\widetilde{\Gamma}$. Furthermore, in \eqref{lksum}, $\mathscr{V}$ is the residue corresponding to the pole $v = - \sigma^{-1} \Delta $; $\mathscr{W}$ is the negative of the residue corresponding to the pole $w = - \sigma^{-1} \Delta $; where $\mathscr{B}$ is the negative of the residue corresponding to both of these poles. Here, we have taken the negative of all residues corresponding to $w = \sigma^{-1} \Delta $ due to the orientation of the contour $\mathfrak{W}$. 	

To express $\mathscr{V}$, $\mathscr{W}$, and $\mathscr{B}$ explicitly, we introduce several preliminary functions. Define 
\begin{flalign}
\begin{aligned}
\label{stationaryintegrandq}
H (y'; w) & = \textbf{1}_{|y'| < \Upsilon_T} \exp \Big( \big( G(q \psi + q \psi \sigma w) - G( q \varkappa \beta_2 ) \big) T - (w + c) y' \Big)   \\
& \qquad \times \displaystyle\frac{\varkappa \beta_2 }{\beta} \left( \displaystyle\frac{\varkappa \beta_2 }{\psi (1 + \sigma w)} \right)^{\mathcal{F} s T^{1 / 3}} \displaystyle\frac{\big( \varkappa \beta_2 (\psi + \psi \sigma w)^{-1}; q \big)_{\infty}}{\big( \beta_1^{-1} (\psi + \psi \sigma w); q \big)_{\infty}}  \\
& \qquad \times \displaystyle\frac{\sigma \pi}{\log q} \displaystyle\sum_{j = -\infty}^{\infty} \displaystyle\frac{1}{\sin \big( \pi (\log q)^{-1} (\log (\psi + \psi \sigma w) - \log (\varkappa \beta_2) + 2 \pi \textbf{i}j) \big)},
\end{aligned}
\end{flalign}

\noindent and 
\begin{flalign}
\label{stationaryintegrandqbar}
\begin{aligned}
\overline{H} (y; v) & = \textbf{1}_{|y| < \Upsilon_T} \exp \Big( \big( G(\beta_1 q) - G( q \psi + q \psi \sigma v) \big) T + (v + c) y \Big)  \\
& \qquad \times \displaystyle\frac{\beta_1}{\psi \big( 1 + \sigma v \big)} \left( \displaystyle\frac{\psi (1 + \sigma v)}{\beta_1} \right)^{\mathcal{F} s T^{1 / 3}} \displaystyle\frac{ \big( \beta_1^{-1} (\psi + \psi \sigma v); q \big)_{\infty}}{\big( \varkappa \beta_2 \big( \psi + \psi \sigma v)^{-1}; q \big)_{\infty} }  \\
& \qquad \times \displaystyle\frac{\sigma \pi}{\log q} \displaystyle\sum_{j = -\infty}^{\infty} \displaystyle\frac{1}{\sin \big( \pi (\log q)^{-1} (\log (\beta_1) - \log (\psi + \psi \sigma v) + 2 \pi \textbf{i}j) \big)}. 
\end{aligned}
\end{flalign}

\noindent Set 
\begin{flalign}
\label{definitionqbarq}
Q(y') = \displaystyle\frac{1}{2 \pi \textbf{i}} \displaystyle\int_{\widetilde{\mathcal{C}}} H (y'; w) dw; \qquad Q (y) = \displaystyle\frac{1}{2 \pi \textbf{i}} \displaystyle\int_{\widetilde{\Gamma}} \overline{H} (y, v) dv. 
\end{flalign}

\noindent Now denote 
\begin{flalign*}
& f_1 (y') = \displaystyle\frac{\big( \omega; q \big)_{\infty} Q(y') }{ \sigma \big( q; q \big)_{\infty}} ; \qquad \qquad \qquad \qquad \qquad \qquad \qquad g_1 (y) = \exp \big( cy - \sigma^{-1} \Delta y \big) \textbf{1}_{|y| < \Upsilon_T} ; \\
& f_2 (y') = \displaystyle\frac{(\omega; q)_{\infty} \exp \big( - cy' - \sigma^{-1} \Delta y' \big) \textbf{1}_{|y'| < \Upsilon_T}}{\sigma (q; q)_{\infty}}; \qquad \qquad g_2 (y) = Q(y); 
\end{flalign*}

\noindent and 
\begin{flalign}
\label{g3}
\begin{aligned}
& f_3 (y') = \displaystyle\frac{(\omega; q)_{\infty} \exp \big( - cy' - \sigma^{-1} \Delta y' \big) \textbf{1}_{|y'| < \Upsilon_T}}{\sigma (q; q)_{\infty}};  \\
& g_3 (y) = \displaystyle\frac{ \varkappa \beta_1 \beta_2 \omega^{\mathcal{F} s T^{1 / 3}}}{\psi^2} \left( \displaystyle\frac{(\omega; q)_{\infty}}{(q; q)_{\infty}} \right) \exp \Big( \big( G (\beta_1 q) - G (\varkappa \beta_2 q) \big) T + cy - \sigma^{-1} \Delta y \Big)  \textbf{1}_{|y| < \Upsilon_T}   \\
& \qquad \qquad \times \displaystyle\frac{\pi}{\log q} \displaystyle\sum_{j = - \infty}^{\infty} \displaystyle\frac{1}{\sin \big( \pi (\log q))^{-1} (2 \pi \textbf{i}j - \log \omega) \big) }. 
\end{aligned}
\end{flalign}

\noindent Then, we have that 
\begin{flalign}
\label{vwb}
\mathscr{V} (y, y') = f_1 (y') g_1 (y); \qquad \mathscr{W} (y, y') = f_2 (y') g_2 (y) ; \qquad \mathscr{B} (y, y') = f_3 (y') g_3 (y). 
\end{flalign}

\noindent It can be quickly verified that $\mathscr{V}$ is equal to the residue of $I$ at $v = - \sigma^{-1} \Delta$; that $\mathscr{W}$ is equal to negative of the residue of $I$ at $w = \sigma^{-1} \Delta$; and that $\mathscr{B}$ is equal to the negative of the residue at both $v = - \sigma^{-1} \Delta$ and $w = \sigma^{-1} \Delta$. 

In view of \eqref{vwb}, the identity \eqref{lksum} can be rewritten as 
\begin{flalign}
\label{llfiniterankperturbation}
\overline{L} = \widetilde{L} + f_1 \otimes g_1 + f_2 \otimes g_2 + f_3 \otimes g_3. 
\end{flalign}

\noindent In particular, $\overline{L}$ is a finite-rank perturbation of $\widetilde{L}$, meaning that we can express $\det \big( \Id - \overline{L} \big)_{L^2 (\mathbb{R}_{> 0})}$ in terms of $\det \big( \Id - \widetilde{L} \big)_{L^2 (\mathbb{R}_{> 0})}$. We will do this in the next section. 

\begin{rem}

\label{gsmallyy}

Suppose that $c \ne 0$. If we did not have the truncation term $\textbf{1}_{|y| < \Upsilon_T} \textbf{1}_{|y'| < \Upsilon_T}$ in the definition \eqref{barldefinition} of $\overline{L}$, then either $f_3 (y')$ or $g_3 (y)$ would not be in $L^2 (\mathbb{R}_{> 0})$. As a result, \eqref{llfiniterankperturbation} would not be a valid identity of operators as operators on $L^2 (\mathbb{R}_{> 0})$. Our introduction of these truncation factors rectifies this issue. 	 
\end{rem}

\subsection{Evaluation of the Determinant}

\label{DeterminantFiniteRankPerturbation}

The purpose of this section is to analyze the Fredholm determinant $\det \big( \Id - \overline{L} \big)_{L^2 (\mathbb{R}_{> 0})}$ using the identity \eqref{llfiniterankperturbation}. To that end, the following two estimates will be useful. The first is a uniform bound on $\widetilde{L}$, $Q$, and $\overline{Q}$; the second states that the \emph{resolvent} $R = \big( \Id - \widetilde{K} \big)^{-1}$ is a bounded operator on $L^2 \big( \mathbb{R}_{> 0} \big)$. 

Both of these estimates will be established later, in Section \ref{InequalityAsymptoticsStationary}. Specifically, the first will be established directly after the proof of Lemma \ref{exponentialhhbarintegrand}, and the second will be established through Corollary \ref{limitrstationary} (see Remark \ref{resolventbounded}). 

\begin{lem}

\label{lkqkbounded}

There exists a constant $C > 0$ (independent of $T$) such that the three estimates
\begin{flalign}
\label{qlqestimate}
\big| \widetilde{L} (y, y') \big| \le C e^{ (-|c| - 1) (y + y')}; \qquad \big| Q (y') \big| < C e^{(-|c| - 1) y' }; \qquad \big| \overline{Q} (y) \big| < C e^{(-|c| - 1) y},
\end{flalign}

\noindent hold for all positive real numbers $y, y' \in \mathbb{R}_{> 0}$. 
\end{lem}

\begin{lem}

\label{invertiblel}

For sufficiently large $T$, the resolvent $R = (\Id - \widetilde{L}_K)^{-1}$ is a bounded linear operator $L^2 (\mathbb{R}_{> 0 })$. 
\end{lem}

Throughout this section, we assume that $T$ is sufficiently large, so that \eqref{llfiniterankperturbation} and Lemma \ref{invertiblel} both hold. In view of \eqref{llfiniterankperturbation}, we have that 
\begin{flalign}
\label{determinantlkproduct3}
\det \big( \Id - L \big)_{L^2 (\mathbb{R}_{> 0})} = \det \big( \Id - \widetilde{L} \big)_{L^2 (\mathbb{R}_{> 0})} \det \Big( \Id - \big\langle R f_i, g_j \big\rangle_{1 \le i, j \le 3} \Big). 
\end{flalign}

\noindent Hence, we are interested in the determinant of the $3 \times 3$ matrix $\textbf{M}$,  defined by
\begin{flalign}
\label{matrixm} 
\textbf{M} = \left[ \begin{array}{ccc} 1 - \big\langle R f_1, g_1 \big\rangle & - \big\langle R f_2, g_1 \big\rangle & - \big\langle R f_3, g_1 \big\rangle\\ - \big\langle R f_2, g_1 \big\rangle & 1 - \big\langle R f_2, g_2 \big\rangle & - \big\langle R f_2, g_3 \big\rangle\\ - \big\langle R f_1, g_3 \big\rangle & - \big\langle R f_2, g_3 \big\rangle & 1 - \big\langle R f_3, g_3 \big\rangle \end{array}\right], 
\end{flalign}

\noindent where $\langle f, g \rangle = \displaystyle\int_0^{\infty} f (x) g (x) dx$ denotes the inner product on $L^2 (\mathbb{R}_{> 0})$. 

To that end, we will individually analyze the entries $\{ m_{i, j} \}_{1 \le i, j \le 3}$ of $\textbf{M}$. We first address the $m_{i, j}$ with $i, j \in \{ 1, 2 \}$. 

Let us begin with $m_{1, 1} = 1 - \big\langle R f_1, g_1 \big\rangle$. Since $f_1 (y') = (\omega; q) Q (y') / \sigma (q; q)_{\infty}$, and since $Q$ is bounded and decays exponentially by Lemma \ref{lkqkbounded}, we have that 
\begin{flalign}
\label{qf1}
f_1 (y') = \displaystyle\frac{(1 - \omega) Q (y')}{\sigma} + \mathcal{O} \left( \displaystyle\frac{(1 - \omega)^2}{\sigma e^{(|c| + 1) y'}} \right). 
\end{flalign}

\noindent Hence,  
\begin{flalign}
\label{m11perturbation}
\begin{aligned}
m_{1, 1} & = 1 - \big\langle R f_1, g_1 \big\rangle  \\
& = 1 + \displaystyle\frac{\omega - 1}{\sigma} \Big\langle R Q (y) , \exp \left( cy - \displaystyle\frac{\Delta y}{\sigma} \right) \textbf{1}_{|y| < \Upsilon_T} \Big\rangle + \Bigg\langle \mathcal{O} \left( \displaystyle\frac{(1 - \omega)^2}{\sigma e^{(|c| + 1) y}} \right), e^{(c - \sigma^{-1} \Delta) y}  \textbf{1}_{|y| < \Upsilon_T} \Bigg\rangle  \\
& = 1 + \displaystyle\frac{\omega - 1}{\sigma} \Big\langle R Q (y) , \exp \left( cy - \displaystyle\frac{\Delta y}{\sigma} \right) \textbf{1}_{|y| < \Upsilon_T} \Big\rangle + \mathcal{O} \big( T^{1 / 3} (1 - \omega)^2 \big), 
\end{aligned}
\end{flalign}

\noindent where we deduced the second estimate from \eqref{qlqestimate}, \eqref{qf1}, and the boundedness of $R$ (see Lemma \ref{invertiblel}). 

For a similar reason, we have that  
\begin{flalign}
\label{m12perturbation}
m_{1, 2} & = - \big\langle R f_1, g_2 \big\rangle = \displaystyle\frac{\omega - 1}{\sigma} \big\langle R Q (y), \overline{Q} (y) \big\rangle + \mathcal{O} \big( T^{1 / 3} (1 - \omega)^2 \big). 
\end{flalign}

\noindent Recalling from its definition that $f_2 (y') = \mathcal{O} \big( \sigma^{-1} (1 - \omega) e^{-cy'} \big)$, from Lemma \ref{lkqkbounded} that $g_2 (y) = \overline{Q} (y) = \mathcal{O} \big( e^{(- |c| - 1) y} \big)$, and from Lemma \ref{invertiblel} that $R$ is bounded, we deduce that 
\begin{flalign}
\label{m22}
m_{2, 2} & = 1 - \big\langle R f_2, g_2 \big\rangle = 1 + \mathcal{O} \big( \sigma^{-1} (1 - \omega) \big). 
\end{flalign}

\noindent Moreover, we have that 
\begin{flalign}
\label{m21}
\begin{aligned}
m_{2, 1} & = - \big\langle R f_2, g_1 \big\rangle = - \big\langle f_2, g_1 \rangle - \big\langle \widetilde{L} R f_2, g_1 \big\rangle \\
&= \displaystyle\frac{(\omega; q)_{\infty}}{\sigma (q; q)_{\infty}} \left( \Big\langle R \widetilde{L} (e^{(-c - \sigma^{-1} \Delta) y} \textbf{1}_{|y| < \Upsilon_T}),  ( e^{(c - \sigma^{-1} \Delta) y} \textbf{1}_{|y| < \Upsilon_T} )  \Big\rangle - \displaystyle\int_0^{\Upsilon_T} \exp \big( - 2 \sigma^{-1} \Delta y \big) dy \right) \\
& = \displaystyle\frac{\omega - 1}{2 \Delta} + \mathcal{O} \big( \sigma^{-1} (1 - \omega) \big) = - 1 - \mathcal{O} \big( \sigma^{-1} (1 - \omega) \big). 
\end{aligned}
\end{flalign}

The second identity in \eqref{m21} follows from the fact that $R = \big( \Id - \widetilde{L} \big)^{-1}$. The third identity follows from the explicit forms of $f_2$ and $g_1$ given in Section \ref{DeformationResidues}. The fourth approximation follows from evaluating the integral (and recalling that $\Upsilon_T^{-1} < e^{-T} = o (1 - \omega)$); the estimate $\widetilde{L} (y, y') = \mathcal{O} \big( e^{(-|c| - 1) (y + y')} \big)$ (from Lemma \ref{lkqkbounded}); and the boundedness of the resolvent $R$ (from Lemma \ref{invertiblel}). The fifth approximation follows from the fact that $2 \Delta = 1 - \omega + \mathcal{O} \big( (1 - \omega)^2 \big)$. 

Thus, \eqref{m11perturbation}, \eqref{m12perturbation}, \eqref{m22}, and \eqref{m21} estimate $m_{i, j}$ when $i, j \in \{ 1, 2 \}$. Now, in order to estimate $m_{i, j}$ when $i$ or $j$ is equal to $3$, we will use the following refined approximations of $f_3 (y')$ and $g_3 (y)$; the full precision of these estimates will only be used for $m_{3, 3}$, although lower order approximations will also be used for $m_{1, 3}$, $m_{2, 3}$, $m_{3, 1}$, and $m_{3, 2}$. 

\begin{lem}

\label{fg3approximation} 

For any positive real numbers $y$ and $y'$, we have that  
\begin{flalign}
\label{g3approximation}
\begin{aligned}
g_3 (y) & = \Bigg( 1 + (1 - \omega) \bigg( \displaystyle\sum_{j = 1}^{\infty} \displaystyle\frac{q^j}{1 - q^j} - \mathcal{F} s T^{1 / 3} + \displaystyle\frac{\pi}{\log q} \displaystyle\sum_{j \ne 0} \displaystyle\frac{1}{\sin \big( 2 \pi^2 (\log q)^{-1} \textbf{\emph{i}} j \big)} \bigg) + \mathcal{O} \big( T^{2 / 3} (1 - \omega)^2 \big) \Bigg) \\ 
& \qquad \times \exp \big( cy - \sigma^{-1} \Delta y \big) \textbf{\emph{1}}_{|y| < \Upsilon_T},
\end{aligned}
\end{flalign}

\noindent and 	
\begin{flalign}
\label{f3approximation}
f_3 (y') = \displaystyle\frac{(1 - \omega) }{\sigma} \left( 1 + (1 - \omega) \displaystyle\sum_{j = 1}^{\infty} \displaystyle\frac{q^j}{1 - q^j} + \mathcal{O} \big( (1 - \omega)^2 \big) \right) \exp \big( -cy' - \sigma^{-1} \Delta y') \textbf{\emph{1}}_{|y'| < \Upsilon_T} , 
\end{flalign}

\noindent where the implicit constants in the errors are independent of $y, y' > 0$. 

\end{lem}

\begin{proof}

This lemma will follow from the definitions of $g_3 (y)$ and $f_3 (y')$ given in \eqref{g3} and the five estimates 
\begin{flalign}
\label{omegainequality1}
\begin{aligned}
& \displaystyle\frac{\varkappa \beta_2 \beta_1}{\psi^2} = 1 + \mathcal{O} \big( (1 - \omega)^2 \big); \qquad  \displaystyle\frac{(\omega; q)_{\infty}}{(q; q)_{\infty}} = 1 - \omega + (1 - \omega)^2 \displaystyle\sum_{j = 1}^{\infty} \displaystyle\frac{q^j}{1 - q^j} + \mathcal{O} \big( (1 - \omega)^3 \big) ; \\
& G (q \beta_1) - G (q \varkappa \beta_2) = \mathcal{O} \big( (1 - \omega)^2 \big); \quad \omega^{\mathcal{F} s T^{1 / 3}} = 1 + \mathcal{F} s T^{1 / 3} (1 - \omega) + \mathcal{O} \big( T^{2 / 3} (1 - \omega)^2 \big); \\
& \left( \displaystyle\frac{\pi}{\log q} \right) \displaystyle\frac{1 - \omega}{\sin \big( \pi (\log q)^{-1} (2 \pi \textbf{i}j - \log \omega )\big)} = \textbf{1}_{j = 0} + \textbf{1}_{j \ne 0} \left( \displaystyle\frac{\pi}{\log q} \right) \displaystyle\frac{1 - \omega + \mathcal{O} \big( (1 - \omega)^2 \big)}{\sin \big( 2 \pi^2 (\log q)^{-1} \textbf{i} j \big)}. 
\end{aligned}
\end{flalign}

The first estimate above follows from the facts that $\beta_1 - \varkappa \beta_2 = \mathcal{O} (1 - \omega)$ and that $\beta_1 + \varkappa \beta_2 = 2 \psi$. The second estimate follows from expanding $(q \omega; q)_{\infty}$ as a power series in $1 - \omega$. The third estimate follows from a Taylor expansion, and the fact that $G' (q \beta) = 0$. The fourth estimate follows from a Taylor expansion in $1 - \omega$, and the fact that $1 - \omega = o \big(T^{-1 / 3} \big)$. The fifth estimate follows from Taylor expanding the denominator in $1 - \omega$.

Now, \eqref{g3approximation} follows from inserting the five estimates \eqref{omegainequality1} into the definition \eqref{g3} of $g_3$, and using the fact that $\big| \sin \big( 2 \pi^2 (\log q)^{-1} \textbf{i} j \big) \big|$ increases exponentially in $|j|$. Similarly, \eqref{f3approximation} follows from inserting the second estimate in \eqref{omegainequality1} into the definition \eqref{g3} of $f_3$. 
\end{proof}

\noindent Now, let us estimate the remaining entries $m_{1, 3}$, $m_{2, 3}$, $m_{3, 1}$, $m_{3, 2}$, and $m_{3, 3}$. Applying \eqref{qf1} and \eqref{g3approximation}, we deduce (similar to \eqref{m11perturbation} and \eqref{m12perturbation}) that 
\begin{flalign}
\label{m13perturbation}
m_{1, 3} & = - \big\langle R f_1, g_3 \big\rangle = \displaystyle\frac{\omega - 1}{\sigma} \Big\langle R Q (y) , \exp \big( cy - \sigma^{-1} \Delta y \big) \textbf{1}_{|y| < \Upsilon_T} \Big\rangle + \mathcal{O} \big( \sigma^{-3} (1 - \omega)^2 \big), 
\end{flalign}

\noindent and 
\begin{flalign}
\label{m32}
m_{3, 2} & = - \big\langle R f_3, g_2 \big\rangle = \displaystyle\frac{\omega - 1}{\sigma} \bigg\langle R \Big( \textbf{1}_{|y| < \Upsilon_T} \exp \big( -cy - \sigma^{-1} \Delta y \big) \Big) , \overline{Q} (y) \Bigg\rangle + \mathcal{O} \big( \sigma^{-1} (1 - \omega)^2 \big). 
\end{flalign} 

\noindent Analogous to \eqref{m21}, we deduce from the definition of $f_2$, the estimate \eqref{g3approximation} for $g_3$, the estimate $\widetilde{L} (y, y') = \mathcal{O} \big( e^{(-|c| - 1) (y + y')} \big)$ (see Lemma \ref{lkqkbounded}), and the boundedness of $R$ (see Lemma \ref{invertiblel}) that 
\begin{flalign}
\label{m23}
m_{2, 3} & = - \big\langle R f_2, g_3 \big\rangle = - 1 - \mathcal{O} \big( T (1 - \omega) \big); \qquad m_{3, 1} = - \big\langle R f_3, g_1 \big\rangle = - 1 - \mathcal{O} \big( T (1 - \omega) \big).  
\end{flalign} 

\noindent Now we approximate $m_{3, 3}$. From \eqref{g3approximation} and \eqref{f3approximation}, we find that
\begin{flalign}
\label{m33}
\begin{aligned}
m_{3, 3} &= 1 - \big\langle R f_3, g_3 \big\rangle = 1 - \big\langle f_3, g_3 \rangle - \big\langle \widetilde{L} R f_3, g_3 \big\rangle \\
&= 1 - \displaystyle\frac{(1 - \omega)}{\sigma} \left( 1 + (1 - \omega) \displaystyle\sum_{k = 1}^{\infty} \displaystyle\frac{q^k}{1 - q^k} + \mathcal{O} \big( (1 - \omega)^2 \big) \right)  \\
& \quad \times \Bigg( 1 + (1 - \omega) \bigg( \displaystyle\sum_{k = 1}^{\infty} \displaystyle\frac{q^k}{1 - q^k} - \mathcal{F} s T^{1 / 3} + \displaystyle\frac{\pi }{\log q} \displaystyle\sum_{j \ne 0} \displaystyle\frac{1}{\sin \big( 2 \pi^2 (\log q)^{-1} i j \big)} \bigg) + \mathcal{O} \left( \displaystyle\frac{(1 - \omega)^2}{\sigma^2}  \right) \Bigg)   \\ 
& \quad \times \displaystyle\int_0^{\Upsilon_T} \exp (-2 \sigma^{-1} \Delta y) dy - \big\langle R \widetilde{L} f_3, g_3 \big\rangle  \\
&= (\omega - 1) \left( 2 \displaystyle\sum_{j = 1}^{\infty} \displaystyle\frac{q^j}{1 - q^j} - \displaystyle\frac{s}{\sigma} + \displaystyle\frac{\pi}{\log q} \displaystyle\sum_{j \ne 0} \displaystyle\frac{1}{\sin \big( 2 \pi^2 (\log q)^{-1} i j \big)} \right)  \\ 
& \quad + \displaystyle\frac{\omega - 1	}{\sigma} \Big\langle \widetilde{L} R \exp \big( - cy - \sigma^{-1} \Delta y \big) \textbf{1}_{|y| < \Upsilon_T}, \exp \big( cy - \sigma^{-1} \Delta y \big) \textbf{1}_{|y| < \Upsilon_T} \Big\rangle + \mathcal{O} \big( \sigma^{-2} (1 - \omega)^2 \big), 
\end{aligned}
\end{flalign} 

\noindent where in the fourth approximation, we used the facts that $\Upsilon_T > e^T$ and $\sigma = \mathcal{F}^{-1} T^{-1 / 3}$. We also applied Lemma \ref{fg3approximation} and, as in \eqref{m21} and \eqref{m23}, the facts that $\widetilde{L} (y, y) = \mathcal{O} \big( e^{(-|c| - 1) (y + y')} \big)$ and that $R$ is bounded to approximate $\big\langle R \widetilde{L} f_3, g_3 \big\rangle $. 

Combining our approximations for the entries of $\textbf{M}$, we deduce the following proposition. 

\begin{prop}
\label{determinantlimitbeta} 

We have that 
\begin{flalign}
\label{determinantlimitbetaequation}
\begin{aligned}
\det \big( \Id - \overline{L} \big)_{L^2 (\mathbb{R}_{> 0 }) } & = \displaystyle\frac{1 - \omega}{\sigma} \det \big( \Id - \widetilde{L} \big)_{L^2 (\mathbb{R}_{> 0})} \Bigg( s - \Big\langle R Q (y) , \exp (cy -\sigma^{-1} \Delta y ) \textbf{\emph{1}}_{|y| < \Upsilon_T} \Big\rangle \\ 
& \quad - \Big\langle R \big( \exp ( -cy - \sigma^{-1} \Delta y ) \textbf{\emph{1}}_{|y| < \Upsilon_T} \big) , \overline{Q} (y) \Big\rangle  - \big\langle R Q (y), \overline{Q} (y) \big\rangle \\ 
& \quad - \Big\langle \widetilde{L} R \big( \exp ( -cy - \sigma^{-1} \Delta y ) \textbf{\emph{1}}_{|y| < \Upsilon_T} \big) , \exp ( cy - \sigma^{-1} \Delta y ) \textbf{\emph{1}}_{|y| < \Upsilon_T} \Big\rangle    \\ 
& \quad - \displaystyle\frac{\pi \sigma}{\log q} \displaystyle\sum_{j \ne 0} \displaystyle\frac{1}{\sin \big( 2 \pi^2 (\log q)^{-1} \emph{i} j \big)} - 2 \sigma \displaystyle\sum_{k = 1}^{\infty} \displaystyle\frac{q^k}{1 - q^k} + \mathcal{O} \big( \sigma^{-6} (1 - \omega) \big)  \Bigg) . 
\end{aligned}
\end{flalign}
\end{prop}

\begin{proof}
In view of \eqref{determinantlkproduct3}, we have that $\det \big( \Id - \overline{L} \big)_{L^2 (\mathbb{R}_{> 0 }) } = \det \big( \Id - \widetilde{L} \big)_{L^2 (\mathbb{R}_{> 0})} \det \textbf{M}$, where $\textbf{M}$ is defined by \eqref{matrixm}. Now, the proposition follows from \eqref{m11perturbation}, \eqref{m12perturbation}, \eqref{m22}, \eqref{m21}, \eqref{m13perturbation}, \eqref{m32}, \eqref{m23}, \eqref{m33}, and the general fact that 
\begin{flalign*}
\det \left[ \begin{array}{ccc} z_{1, 1} + 1 & z_{1, 2} & z_{1, 3} \\ z_{2, 1} - 1 & z_{2, 2} + 1 & z_{2, 3} \\ z_{3, 1} - 1 & z_{3, 2} & z_{3, 3} \end{array} \right] = z_{1, 2} + z_{1, 3} + z_{3, 2} + z_{3, 3} + \mathcal{O} \big( \sigma^{-6} (1 - \omega)^2 \big),
\end{flalign*}

\noindent whenever $z_{i, j}$ are complex numbers satisfying $z_{i, j} = \mathcal{O} \big( \sigma^{-3} (1 - \omega) \big) $. 
\end{proof}

\subsection{Asymptotic Analysis}

\label{TLarge}

The right side of \eqref{determinantlimitbetaequation} is amenable to a direct saddle-point asymptotic analysis; this will be done in Section \ref{InequalityAsymptoticsStationary}. This result will then be reconciled with the statements of Proposition \ref{convergencedeterminant} and Proposition \ref{convergencedeterminantmodel} in Section \ref{ProofAsymptoticsStationary}.

\subsubsection{Estimates and Asymptotics}

\label{InequalityAsymptoticsStationary}

The goal of this section is to provide estimates on and asymptotics of the kernel $\widetilde{L}$, the resolvent $R$, and the functions $Q$ and $\overline{Q}$. Recall that these are all in terms of the functions $I$, $H$, and $\overline{H}$, defined in \eqref{stationaryintegrandlc}, \eqref{stationaryintegrandq}, and \eqref{stationaryintegrandqbar}, respectively. Therefore, we will first estimate and deduce asympotitcs for these three functions. In proofs below, we will freely interchange $\psi$ with $\beta$ without mention; in all cases, this can be quickly justified through a Taylor expansion and the estimate $\beta - \psi = \mathcal{O} \big( T^{-10} \big)$

The following lemma provides an exponential bound on $I$, $H$, and $\overline{H}$. 

\begin{lem}

\label{exponentialhhbarintegrand}

There exist constants $c_1, C_1 > 0$ such that the three estimates 
\begin{flalign}
\label{exponentialstationaryintegrandh}
 \big| H (y'; w) \big| \le C_1 \exp \big( - c_1 |w|^3 - (|c| + 1) y' \big); \quad \big| \overline{H} (y; v) \big| \le C_1 \exp \big( - c_1 |v|^3 - (|c| + 1) y  \big); 
\end{flalign}

\noindent and 
\begin{flalign}
\label{exponentialstationaryintegrand1}
\big| I (y, y'; w, v) \big| \le C_1 \exp \big(  - c_1 (|w|^3 + |v|^3) - (|c| + 1) (y + y') \big); 
\end{flalign}

\noindent hold for all $y, y' \in \mathbb{R}_{> 0}$, $w \in \widetilde{\mathcal{C}}$, and $v \in \widetilde{\Gamma}$. 
\end{lem}

\begin{proof}

Let us begin with the estimate on $I$. Observe that there exists a constant $C_1 > 0$ (independent of $T$) such that the eight inequalities 
\begin{flalign}
\label{vwinequality2} 
\begin{aligned}
& \qquad \left| \displaystyle\frac{1}{1 + \sigma v} \right| < C_1; \qquad \Big| \exp \big( - (w + c) y' \big) \big| < e^{-y' (|c| + 1)}; \qquad \Big| \exp \big( (v + c) y \big) \big| < e^{-y (|c| + 1)}; \\
& \qquad \qquad \qquad \qquad \displaystyle\frac{\big( \varkappa \beta_2 (\psi + \psi \sigma w)^{-1}; q \big)_{\infty} }{\big( \beta_1^{-1} (\psi + \psi \sigma w); q \big)_{\infty} } < C_1; \qquad \displaystyle\frac{\big( \beta_1^{-1} (\psi + \psi \sigma v); q \big)_{\infty}}{\big( \varkappa \beta_2 \big( \psi + \psi \sigma v)^{-1}; q \big)_{\infty} } < C_1;  \\
& \left| \displaystyle\frac{1 + \sigma v}{1 + \sigma w} \right|^{\mathcal{F} s T^{1 / 3}} < \exp \Big( C_1 \big(  |v| + |w| \big) \Big);  \qquad \left| \displaystyle\frac{\sigma}{\sin \big( \pi (\log q)^{-1} \big( \log (1 + \sigma w) - \log (1 + \sigma v)) \big)} \right| < C_1  \\
& \qquad \qquad \textbf{1}_{j \ne 0} \left| \displaystyle\frac{\sigma}{\sin \big( \pi (\log q)^{-1} \big( \log (1 + \sigma w) - \log (1 + \sigma v) + 2 \pi \textbf{i}j ) \big)} \right| < C_1 T^{-1 / 3} e^{-|j| / C_1}, 
\end{aligned}
\end{flalign}

\noindent hold, for each $y, y' \in \mathbb{R}_{> 0}$, $w \in \widetilde{\mathcal{C}}$, and $v \in \widetilde{\Gamma}$. The second inequality holds since $\Re w \ge E \ge 2 |c| + 1$ for all $w \in \widetilde{\mathcal{C}}$, and the third inequality holds since $\Re v \le - E \le - 2 |c| - 1$ for all $v \in \widetilde{\Gamma}$ (recall Definition \ref{ctildegammatilde}). The remaining can be derived in a very similar way to those already presented in \eqref{vwinequalities1}, so we omit their proofs. 

The eight inequalities \eqref{vwinequality2} address all factors in the definition of $I$ except for the exponential $\exp \big( (G(q \psi + q \psi \sigma w) - G(q \psi + q \psi \sigma v)) T \big)$, so let us estimate this term. 

From \eqref{rstationary}, we have that 
\begin{flalign*}
T \big( G(q \psi + q \psi \sigma w) - G ( q \beta ) \big) = \displaystyle\frac{w^3 }{3} + c w^2 + T R \big( T^{- 1 / 3} w \big) + \mathcal{O} \big( T^{-10} \big).  
\end{flalign*}

\noindent Now, the third part of Definition \ref{linearvertexrightstationary} implies that $\big| R(z) \big| < z^3 / 8$ for any $|z| < 2 \varepsilon / q \psi \sigma$. Thus, 
\begin{flalign*}
\Big| T \big( G(q \psi + q \psi \sigma w) - G ( q \beta ) \big) - \displaystyle\frac{w^3 }{3} - c w^2 \Big| \le \displaystyle\frac{|w|^3}{8}, 
\end{flalign*}

\noindent for sufficiently large $T$. In particular, since $w^3$ decreases as $-|w|^3$ (to leading order), as $w \in \widetilde{\mathcal{C}}$ tends to $\infty$, we deduce the existence of a constant $C_3 > 0$ (independent of $T$) such that 
\begin{flalign}
\label{smallgwstationary}
\big( G(q \psi + q \psi \sigma w) - G( q \beta) \big) T < C_3 -  \displaystyle\frac{|w|^3}{5}, 
\end{flalign}

\noindent for all $w \in \widetilde{\mathcal{C}}$. Similarly, after increasing $C_3$ if necessary, we obtain that
\begin{flalign}
\label{smallgvstationary}
\big( G(q \beta) - G( q \psi + q \psi \sigma v) \big) T < C_3 -  \displaystyle\frac{|v|^3}{5}, 
\end{flalign}

\noindent for all $v \in \widetilde{\Gamma}$. From \eqref{smallgwstationary} and \eqref{smallgvstationary}, we deduce the existence of a constant $C_4 > 0$ (independent of $T$) such that 
\begin{flalign}
\label{smallgwvstationary}
\exp \Big( \big( G(q \psi + q \psi \sigma w) - G( q \psi + q \psi \sigma v) \big) T \Big) \le C_4 \exp \left( - \displaystyle\frac{1}{5} \big( |w|^3 + |v|^3 \big) \right), 
\end{flalign}

\noindent for all $w \in \widetilde{\mathcal{C}}$ and $v \in \widetilde{\Gamma}$. Thus, the estimate \eqref{exponentialstationaryintegrand1} follows from inserting the eight inequalities \eqref{vwinequality2} and the exponential estimate \eqref{smallgwvstationary} into the definition \eqref{stationaryintegrandlc} of $I (y, y'; w, v)$.  

The estimates on $H$ and $\overline{H}$ are entirely analogous and thus omitted.
\end{proof}

\noindent Using Lemma \ref{exponentialhhbarintegrand}, we can establish the estimate given by Lemma \ref{lkqkbounded} in Section \ref{DeterminantFiniteRankPerturbation}. 

\begin{proof}[Proof of Lemma \ref{lkqkbounded}]
The estimate on $\widetilde{L}$ follows from the definition \eqref{ltildedefinition} of $\widetilde{L}$ in terms of $I$, and from integrating the estimate \eqref{exponentialstationaryintegrand1} over $w \in \widetilde{\mathcal{C}}$ and $v \in \widetilde{\Gamma}$. Similarly, the estimates on $Q$ and $\overline{Q}$ follow from the definitions \eqref{definitionqbarq} of $Q$ and $\overline{Q}$ in terms of $H$ and $\overline{H}$, respectively, and from integrating the estimate \eqref{exponentialstationaryintegrandh}. 
\end{proof}

\noindent The following lemma evaluates the pointwise large $T$ limits of the functions $I$, $H$, and $\overline{H}$.

\begin{lem}

\label{integrandstationarylimitshbarh}

For any fixed positive real numbers $y$ and $y'$, and any fixed complex numbers $w \in \mathfrak{W}_{E, \infty}$ and $v \in \mathfrak{V}_{-E, \infty}$, we have that 
\begin{flalign}
\label{integrandstationarylimithbarh}
\begin{aligned}
& \displaystyle\lim_{T \rightarrow \infty} H (y'; w) = - \displaystyle\frac{1}{w} \exp \left( \displaystyle\frac{w^3}{3} + c w^2 - s w - y' (w + c) \right);  \\
&  \displaystyle\lim_{T \rightarrow \infty} \overline{H} (y; v) = \displaystyle\frac{1}{v} \exp \left( sv - \displaystyle\frac{v^3}{3} - c v^2 + y (v + c) \right); 
\end{aligned}
\end{flalign}

\noindent and 
\begin{flalign}
\label{integrandstationarylimit2}
\displaystyle\lim_{T \rightarrow \infty} I (y, y'; w, v) = \left( \displaystyle\frac{1}{w - v} \right) \displaystyle\frac{ \exp \big( w^3 / 3 + cw^2 - sw - y' (w + c) \big)}{\exp \big( v^3 / 3 + cv^2 - sv - y (v + c) \big)} .
\end{flalign}

\end{lem}

\begin{proof}

Let us begin with the limit \eqref{integrandstationarylimit2} of $I$. To that end, observe that 
\begin{flalign}
\label{integrandllimitstationary}
\begin{aligned}
& \displaystyle\lim_{T \rightarrow \infty} \textbf{1}_{|y| < \Upsilon_T} \textbf{1}_{|y'| < \Upsilon_T} = 1; \quad \displaystyle\lim_{T \rightarrow \infty} \displaystyle\frac{1}{1 + \sigma v} = 1; \quad \displaystyle\lim_{T \rightarrow \infty} \left( \displaystyle\frac{1 + \sigma v}{1 + \sigma w} \right)^{\mathcal{F} s T^{1 / 3}} = \exp \big( s (v - w) \big);  \\
& \qquad \qquad \displaystyle\lim_{T \rightarrow \infty} \displaystyle\frac{ \big( \varkappa \beta_2 (\psi + \psi \sigma w)^{-1}; q \big)_{\infty} }{\big( \beta_1^{-1} (\psi + \psi \sigma w); q \big)_{\infty}} = - 1; \qquad \displaystyle\lim_{T \rightarrow \infty} \displaystyle\frac{ \big( \beta_1^{-1} (\psi + \psi \sigma v); q \big)_{\infty} }{\big( \varkappa \beta_2 \big( \psi + \psi \sigma v)^{-1}; q \big)_{\infty} } = - 1;   \\
& \qquad \qquad \displaystyle\lim_{T \rightarrow \infty} \displaystyle\frac{\pi \sigma}{\log q} \displaystyle\sum_{j \ne 0} \displaystyle\frac{1}{\big( 1 + \sigma v \big) \sin \big( \pi (\log q)^{-1} (\log (1 + \sigma w) - \log (1 + \sigma v) + 2 \pi \textbf{i}j) \big)} = 0;  \\
& \qquad \qquad \displaystyle\lim_{T \rightarrow \infty} \left( \displaystyle\frac{\pi \sigma}{\log q} \right) \displaystyle\frac{1}{\big( 1 + \sigma v \big) \sin \big( \pi (\log q)^{-1} (\log (1 + \sigma w) - \log (1 + \sigma v) ) \big)}= \displaystyle\frac{1}{w - v}, 
\end{aligned}
\end{flalign}

\noindent for any $w \in \mathfrak{W}_{E, \infty}$ and $v \in \mathfrak{V}_{-E, \infty}$. 

The first identity holds since $\Upsilon_T > e^T$ tends to $\infty$ as $T$ tends to $\infty$. The second follows from the fact that $\sigma$ tends to $0$ as $T$ tends to $\infty$. The third is due to the fact that $\mathcal{F} T^{1 / 3} = \sigma^{-1}$. For the fourth, observe that since $\big| \beta_1 - \varkappa \beta_2 \big| = \mathcal{O} \big( T^{-10} \big)$ and $\sigma^{-1} = \mathcal{O} \big( T^{1 / 3} \big)$, we have that 
\begin{flalign*}
\displaystyle\lim_{T \rightarrow \infty} \displaystyle\frac{ \big( \varkappa \beta_2 (\psi + \psi \sigma w)^{-1}; q \big)_{\infty} }{\big( \beta_1^{-1} (\psi + \psi \sigma w); q \big)_{\infty}} &= \displaystyle\lim_{T \rightarrow \infty} \displaystyle\frac{1 - \varkappa \beta_2 (\psi + \psi \sigma w)^{-1}}{1 - \beta_1^{-1} (\psi + \psi \sigma w)} \displaystyle\lim_{T \rightarrow \infty} \displaystyle\frac{ \big( q \varkappa \beta_2 (\psi + \psi \sigma w)^{-1}; q \big)_{\infty} }{\big( q \beta_1^{-1} (\psi + \psi \sigma w); q \big)_{\infty}} \\
&= \displaystyle\lim_{T \rightarrow \infty} \displaystyle\frac{1 - \varkappa \beta_2 \beta^{-1} (1 + \sigma w)^{-1}}{1 - \psi \beta_1^{-1} (1 + \sigma w)}  = -1. 
\end{flalign*}

\noindent This establishes the fourth identity in \eqref{integrandllimitstationary}. The fifth follows similarly to the fourth. The sixth follows from summing the eighth estimate in \eqref{vwinequality2} over $j \in \mathbb{Z} \setminus \{ 0 \}$. The seventh follows from a Taylor expansion in $\sigma$. 

The seven limits \eqref{integrandllimitstationary} address all terms appearing in the definition \eqref{stationaryintegrandlc} of $I$, except for the term $\exp \big( ( G(q \psi + q \psi \sigma w) - G( q \psi + q \psi \sigma v) ) T \big)$. To estimate this term, we recall from \eqref{rstationary} that 
\begin{flalign}
\label{gwgvstationary}
\begin{aligned}
\displaystyle\lim_{T \rightarrow \infty} \big( G(q \psi + q \psi \sigma w) - G( q \beta) \big) T = \displaystyle\lim_{T \rightarrow \infty} \left( \displaystyle\frac{w^3}{3} + c w^2 + T R \big( T^{-1 / 3} w \big) \right) =  \displaystyle\frac{w^3}{3} + c w^2; \\
\displaystyle\lim_{T \rightarrow \infty} \big( G(q \beta) - G( q \psi + q \psi \sigma v) \big) T = - \displaystyle\lim_{T \rightarrow \infty} \left( \displaystyle\frac{v^3}{3} + c v^2 + T R \big( T^{-1 / 3} v \big) \right) = - \displaystyle\frac{v^3}{3} - c v^2, 
\end{aligned}
\end{flalign}

\noindent for each fixed $v, w \in \mathbb{C}$; here we have used the fact that $R (z) = \mathcal{O} (z^4)$ (which follows from \eqref{gzpsistationary}). 

Subtracting the two identities in \eqref{gwgvstationary}, exponentiating the result, and applying the seven identities \eqref{integrandllimitstationary} into \eqref{stationaryintegrandlc}, we deduce the limit statement \eqref{integrandstationarylimit2} for $I$. 

The limiting statements for $H$ and $\overline{H}$ are entirely analogous and thus omitted. 
\end{proof}

\noindent Using the limit \eqref{exponentialstationaryintegrand1} of $I$, the estimate \eqref{integrandstationarylimit2} on $I$, and the identity \eqref{ltildedefinition} that expresses $\widetilde{L}$ in terms of $I$, we deduce the following asymptotic results on $\widetilde{K}$.

\begin{prop}
\label{limitkernelresolvent}

Recalling the shifted Airy kernel $K_{\Ai; a}$ from Definition \ref{stationarykernel}, we have that
\begin{flalign}
\label{tildellimit}
& \displaystyle\lim_{T \rightarrow \infty} \widetilde{L} (y, y') = K_{\Ai; c^2 + s} (y, y').
\end{flalign}

\noindent Furthermore, 
\begin{flalign}
\label{tildeldeterminantlimit}
\displaystyle\lim_{T \rightarrow \infty} \det \big( \Id - \widetilde{L} \big)_{L^2 (\mathbb{R}_{> 0})} = \det \big( \Id - K_{\Ai; c^2 + s} \big)_{L^2 (\mathbb{R}_{> 0})}. 
\end{flalign}

\end{prop}

\begin{proof}

Let us begin with the first identity \eqref{tildellimit}. By the definition \eqref{ltildedefinition} of $\widetilde{L}$ in terms of $I$, the limiting statement \eqref{integrandstationarylimit2}, the uniform estimate \eqref{exponentialstationaryintegrand1}, and the dominated convergence theorem, we deduce that  
\begin{flalign}
\label{tildellimit1}
\displaystyle\lim_{T \rightarrow \infty} \widetilde{L} (y, y') = \displaystyle\lim_{T \rightarrow \infty} \displaystyle\frac{1}{(2 \pi \textbf{i})^2} \displaystyle\int_{\widetilde{\mathcal{C}}} \displaystyle\int_{\widetilde{\Gamma}} \displaystyle\frac{ \exp \big( w^3 / 3 + cw^2 - sw - y' (w + c) \big)}{\exp \big( v^3 / 3 + cv^2 - sv - y (v + c) \big)} \left( \displaystyle\frac{1}{w - v} \right)  dv dw .
\end{flalign}

\noindent Now, since the integrand on the right side of \eqref{tildellimit1} decays exponentially in $|w|$ and $|v|$, since $\widetilde{\mathcal{C}} = \mathfrak{W}_{E, \varepsilon / q \psi \sigma}$, since $\widetilde{\Gamma} = \mathfrak{V}_{-E, \varepsilon / q \psi \sigma}$, and since $\sigma^{-1}$ tends to $\infty$ at $T$ tends to $\infty$, we deduce from \eqref{tildellimit1} that 
\begin{flalign*}
\displaystyle\lim_{T \rightarrow \infty} \widetilde{L} (y, y') = \displaystyle\frac{1}{(2 \pi \textbf{i})^2} \displaystyle\int_{\mathfrak{W}_{E, \infty}} \displaystyle\int_{\mathfrak{V}_{-E, \infty}} \displaystyle\frac{ \exp \big( w^3 / 3 + cw^2 - sw - y' (w + c) \big)}{\exp \big( v^3 / 3 + cv^2 - sv - y (v + c) \big)} \left( \displaystyle\frac{1}{w - v} \right)  dv dw .
\end{flalign*}

\noindent Changing variables $w' = w + c$ and $v' = v + c$, we find that 
\begin{flalign}
\label{tildellimit2}
\begin{aligned}
\displaystyle\lim_{T \rightarrow \infty} \widetilde{L} (y, y') & = \displaystyle\frac{1}{(2 \pi \textbf{i})^2} \displaystyle\int_{\mathfrak{W}_{E + c, \infty}} \displaystyle\int_{\mathfrak{V}_{c - E, \infty}} \displaystyle\frac{ \exp \big( w'^3 / 3 - (s + c^2 + y') w' \big)}{\exp \big( v'^3 / 3 + (s + c^2 + y) v' \big)} \left( \displaystyle\frac{1}{w' - v'} \right)  dv' dw' \\
& = \displaystyle\frac{1}{(2 \pi \textbf{i})^2} \displaystyle\int_{\mathfrak{W}_{E + c, \infty}} \displaystyle\int_{\mathfrak{V}_{c - E, \infty}} \displaystyle\int_0^{\infty} \displaystyle\frac{ \exp \big( w'^3 / 3 - (\lambda + s + c^2 + y') w' \big)}{\exp \big( v'^3 / 3 + (\lambda + s + c^2 + y) v' \big)} d \lambda d v' dw' \\
& = \displaystyle\int_0^{\infty} \Ai (\lambda + s + c^2 + y') \Ai (\lambda + s + c^2 + y) d \lambda = K_{\Ai; c^2 + s} (y, y').
\end{aligned}
\end{flalign}

\noindent Now \eqref{tildellimit} follows from \eqref{tildellimit2}. 

To establish \eqref{tildeldeterminantlimit}, we apply Lemma \ref{determinantlimitkernels} to the kernels $\{ \widetilde{L} (y, y') \}$, the contour $\mathbb{R}_{> 0}$, and the kernel $K_{\Ai; c^2} (y + s, y' + s)$. Let us verify the conditions of that proposition. First, we require a dominating function $\textbf{K} (y)$ such that $\sup_{y' \in \mathbb{R}_{> 0}} \big| \widetilde{L} (y, y') \big| < \textbf{K} (y)$, for all $y \in \mathbb{R}_{> 0}$, and such that second and third properties listed in the statement of Lemma \ref{determinantlimitkernels} both hold. This is guaranteed by Lemma \ref{lkqkbounded}, which implies that we may in fact take $\textbf{K} (y) = C e^{-y / C}$, for some sufficiently large constant $C$. Second, we require that $\lim_{T \rightarrow \infty} \widetilde{L} (y, y') = K_{\Ai; c^2 + s} (y, y')$; this is guaranteed by \eqref{tildellimit}. 

Thus, Lemma \ref{determinantlimitkernels} applies, from which we deduce \eqref{tildeldeterminantlimit}. 
\end{proof}

\begin{cor}

\label{limitrstationary}

We have that 
\begin{flalign*}
\displaystyle\lim_{T \rightarrow \infty} \big\| \big( \Id - \widetilde{L} \big)^{-1} - \big( \Id - K_{\Ai; c^2 + s } \big)^{-1} \big\| = 0.
\end{flalign*}

\end{cor}

\begin{proof}

The proof of this corollary is similar to that of Lemma 7.6 in \cite{HFSE}. 

By \eqref{tildellimit}, we have that $\lim_{T \rightarrow \infty} \widetilde{L} (y, y') = K_{\Ai; c^2 + s} (y, y')$. Thus, for sufficiently large $T$, we have that $\big\| \widetilde{L} - K_{\Ai; c^2 + s} \big\| < \big\| ( \Id - K_{\Ai; c^2 + s} )^{-1} \big\|$. For such $T$, we have that  
\begin{flalign}
\label{resolventerror}
\big\| \big( \Id - \widetilde{L} \big)^{-1} - \big( \Id - K_{\Ai, c^2 + s} \big)^{-1} \big\| \le \displaystyle\sum_{j = 1}^{\infty} \| \widetilde{L} - K_{\Ai; c^2 + s} \big\|^j \Big\| \big( \Id - K_{\Ai; c^2 + s} \big)^{-1} \Big\|^{j + 1}.
\end{flalign}

\noindent By \eqref{tildellimit} and the dominated convergence theorem, the right side of \eqref{resolventerror} tends to $0$ as $T$ tends to $\infty$. This implies the lemma. 
\end{proof}

\begin{rem}

\label{resolventbounded}

Lemma \ref{invertiblel} immediately follows from Proposition \ref{limitrstationary}. 

\end{rem}

\begin{prop}
\label{limitkernel}

We have that 
\begin{flalign}
\label{qlimit}
\displaystyle\lim_{T \rightarrow \infty} Q(y') & = - \exp \left( \displaystyle\frac{2c^3}{3} + c s \right) \displaystyle\int_0^{\infty} \Ai (c^2 + s + y' + \lambda) e^{c \lambda} d \lambda; \\
\label{qbarlimit}
 \displaystyle\lim_{T \rightarrow \infty} \overline{Q} (y) & = - \exp \left( \displaystyle\frac{- 2c^3}{3} - c s \right) \displaystyle\int_0^{\infty} \Ai (c^2 + s + y + \lambda) e^{- c \lambda} d \lambda. 
\end{flalign}
\end{prop} 

The proof of this proposition is very similar to that of Proposition \ref{limitkernelresolvent} and is thus omitted.

\subsubsection{Proof of Proposition \ref{convergencedeterminant} and Proposition \ref{convergencedeterminantmodel}} 

\label{ProofAsymptoticsStationary}

In this section, we use Proposition \ref{determinantlimitbeta} and the asymptotics given in Section \ref{InequalityAsymptoticsStationary} to establish Proposition \ref{convergencedeterminant} and Proposition \ref{convergencedeterminantmodel}. This will be similar to what was done in Section 8 of \cite{HFSE}. 

Following the notation in that work, we define the operator $B_s (x, y) = \Ai (x + y + s)$ on $L^2 (\mathbb{R}_{> 0})$; specifically, for any function $f \in L^2 (\mathbb{R}_{> 0})$, we have that $(B_s f) (x) = \displaystyle\int_0^{\infty} \Ai (x + y + s) f(y) dy$. Furthermore, for any $a \in \mathbb{R}$, we define the function $e_a (x) = e^{ax}$. 

Although $e_a (x)$ is not always in $L^2 (\mathbb{R}_{> 0})$, the operator $B_s$ acts on all functions of the form $e_a (x)$ due to the fact that the Airy function $\Ai (x)$ has decay of type $e^{- x^{3 / 2} / C}$, for some constant $C > 0$, as $x$ tends to $\infty$. 

In view of the above definitions, we have by \eqref{qlimit} and \eqref{qbarlimit} that 
\begin{flalign}
\label{qlimit2}
\displaystyle\lim_{T \rightarrow \infty} Q (y') = - \exp \left( cs - \displaystyle\frac{c^3}{3} \right) (B_s e_c) (y' + c^2); \quad \displaystyle\lim_{T \rightarrow \infty} \overline{Q} (y) = - \exp \left( \displaystyle\frac{c^3}{3} - cs \right) (B_s e_{-c}) (y + c^2), 
\end{flalign}

\noindent for all positive real numbers $y, y' > 0$. Moreover, \eqref{tildellimit} and Corollary \ref{limitrstationary} imply that 
\begin{flalign}
\label{limitkernel1}
\displaystyle\lim_{T \rightarrow \infty} \widetilde{L} (y, y') = B_s^2 (y + c^2, y' + c^2); \quad \displaystyle\lim_{T \rightarrow \infty} \big( \Id - \widetilde{L} \big)^{-1} (y, y') = \big( \Id - B_s^2 \big)^{-1} (y + c^2, y' + c^2). 
\end{flalign}

\noindent Now, combining Corollary \ref{limitgammac}, \eqref{stationarydeterminants1}, Lemma \ref{determinantreals}, and Lemma \ref{lclexponential}, we obtain that 
\begin{flalign*}
\displaystyle\lim_{T \rightarrow \infty} \displaystyle\frac{\sigma}{1 - \omega} \det \big( \Id + K \big)_{L^2 (\mathcal{C})} = \displaystyle\lim_{T \rightarrow \infty} \displaystyle\frac{\sigma}{1 - \omega} \det \big( \Id - \overline{L} \big)_{L^2 (\mathbb{R}_{> 0})}. 
\end{flalign*}

\noindent Thus, from the facts that $\sigma = \mathcal{F}^{-1} T^{-1 / 3}$ and $\Delta = \mathcal{O} (T^{-10})$, from \eqref{tildeldeterminantlimit}, and from Proposition \ref{determinantlimitbeta}, we deduce that 
\begin{flalign}
\label{limitdeterminantomega1sigma}
\begin{aligned}
\displaystyle\lim_{T \rightarrow \infty} & \displaystyle\frac{\det \big( \Id + K \big)_{L^2 (\mathcal{C})}}{\mathcal{F} T^{1 / 3} (1 - \omega)} = \mathcal{G} \det \big( \Id - K_{\Ai; c^2 + s} \big)_{L^2 (\mathbb{R}_{> 0})},
\end{aligned}
\end{flalign}

\noindent where  
\begin{flalign}
\label{g}
\begin{aligned}
\mathcal{G} & = \displaystyle\lim_{T \rightarrow \infty} \Bigg( s - \Big\langle R Q (y) , \exp (cy -\sigma^{-1} \Delta y ) \textbf{1}_{|y| < \Upsilon_T} \Big\rangle  \\ 
& \qquad - \Big\langle R \big( \exp ( -cy - \sigma^{-1} \Delta y ) \textbf{1}_{|y| < \Upsilon_T} \big) , \overline{Q} (y) \Big\rangle  - \big\langle R Q (y), \overline{Q} (y) \big\rangle  \\ 
& \qquad - \Big\langle \widetilde{L} R \big( \exp ( -cy - \sigma^{-1} \Delta y ) \textbf{1}_{|y| < \Upsilon_T} \big) , \exp ( cy - \sigma^{-1} \Delta y ) \textbf{1}_{|y| < \Upsilon_T} \Big\rangle \Bigg) .
\end{aligned}
\end{flalign}

\noindent Now we use the asymptotic results \eqref{qlimit2}, and \eqref{limitkernel1} to evaluate the limits of the scalar products appearing on the right side of \eqref{limitdeterminantomega1sigma}. Using these identities, the uniform estimates given by Lemma \ref{lkqkbounded}, and the fact that $\lim_{T \rightarrow \infty} \sigma^{-1} \Delta = 0$, we deduce from the dominated convergence theorem that 
\begin{flalign}
\label{rqe}
\begin{aligned}
& \displaystyle\lim_{T \rightarrow \infty} \Big\langle R Q (y) , \exp (cy -\sigma^{-1} \Delta y ) \textbf{1}_{|y| < \Upsilon_T} \Big\rangle = \displaystyle\lim_{T \rightarrow \infty} \Big\langle R Q (y) , e_c (y) \Big\rangle  \\
& \quad = - \exp \left( cs - \displaystyle\frac{c^3}{3} \right) \displaystyle\int_0^{\infty} \displaystyle\int_0^{\infty} ( \Id - B_s^2 )^{-1} \big(x + c^2, y + c^2 \big) ( B_s e_c ) \big( y + c^2 \big) e_c \big( x + c^2 \big) dx dy; 
\end{aligned}
\end{flalign}

\begin{flalign}
\label{req}
\begin{aligned}
& \displaystyle\lim_{T \rightarrow \infty}  \Big\langle R \big( \exp ( -cy - \sigma^{-1} \Delta y ) \textbf{1}_{|y| < \Upsilon_T} \big) , \overline{Q} (y) \Big\rangle = \displaystyle\lim_{T \rightarrow \infty} \Big\langle R \big( e_{-c} (y) \big) , \overline{Q} (y) \Big\rangle  \\
& \quad = - \exp \left( \displaystyle\frac{c^3}{3} - cs \right) \displaystyle\int_0^{\infty} \displaystyle\int_0^{\infty} ( \Id - B_s^2 )^{-1} \big( x + c^2, y + c^2 \big) e_{-c} \big( y + c^2 \big) ( B_s e_{-c} ) \big( x + c^2 \big)  dx dy; 
\end{aligned}
\end{flalign}

\begin{flalign}
\label{qrq}
\displaystyle\lim_{T \rightarrow \infty} & \big\langle R Q (y), \overline{Q} (y) \big\rangle = \displaystyle\int_0^{\infty} \displaystyle\int_0^{\infty} ( \Id - B_s^2 )^{-1} \big( x + c^2, y + c^2 \big) ( B_s e_c ) \big(y + c^2 \big) ( B_s e_{-c} ) \big( x + c^2 \big) dx dy; 
\end{flalign}

\begin{flalign}
\label{elre}
\begin{aligned}
\displaystyle\lim_{T \rightarrow \infty} & \Big\langle \widetilde{L} R \big( \exp ( -cy - \sigma^{-1} \Delta y ) \textbf{1}_{|y| < \Upsilon_T} \big) , \exp ( cy - \sigma^{-1} \Delta y ) \textbf{1}_{|y| < \Upsilon_T} \Big\rangle = \displaystyle\lim_{T \rightarrow \infty} \Big\langle R \widetilde{L} \big( e_{-c} (y) \big) , e_c (y) \Big\rangle  \\
& = \displaystyle\int_0^{\infty} \displaystyle\int_0^{\infty} \displaystyle\int_0^{\infty} ( \Id - B_s^2 )^{-1} \big( x + c^2, y + c^2\big)  B_s^2 \big( y + c^2, z + c^2 \big) e_{-c} \big( z \big) e_c \big( x \big)  dx dy dz. 
\end{aligned}
\end{flalign}

To establish Proposition \ref{convergencedeterminant} and Proposition \ref{convergencedeterminantmodel}, it suffices to insert \eqref{rqe}, \eqref{req}, \eqref{qrq}, and \eqref{elre} into \eqref{limitdeterminantomega1sigma}, and verify that the result is equal to $g(c, s) \det \big( \Id - K_{\Ai; c^2 + s} \big)_{L^2 (\mathbb{R}_{> 0})}$, given by Definition \ref{stationarydistribution}. Recall that the function $g(c, s)$ is given in terms of the functions $\mathcal{R}$, $\Phi$, and $\Psi$, given by Definition \ref{stationarykernel}. 

As observed in Section 8 of \cite{HFSE}, these functions can also be written in terms of the operator $B_s$ and the function $e_c$ and $e_{-c}$ through the identities 
\begin{flalign}
\label{functionsg}
\mathcal{R} & = f + \exp \left( \displaystyle\frac{c^3}{3} - cs \right) \displaystyle\int_0^{\infty} (B_s e_{-c}) \big( y + c^2 \big) e_{-c} \big( y + c^2 \big) dy; \\
\Phi (x + s) & = \exp \left( \displaystyle\frac{c^3}{3} - cs \right) \displaystyle\int_0^{\infty} B_s^2 (x + c^2, y + c^2) e_{-c} (y + c^2) dy - (B_s e_c) \big( x + c^2 \big) ; \\
\Psi (y + s) & = \exp \left( cs - \displaystyle\frac{c^3}{3}  \right) e_c (y + c^2)	 - (B_s e_{-c}) (y + c^2), 
\end{flalign}

\noindent for any positive real numbers $x$ and $y$. 

Expressing the scalar products appearing in \eqref{gstationary} as integrals, one quickly verifies from \eqref{g}, \eqref{rqe}, \eqref{req}, \eqref{qrq}, \eqref{elre}, and \eqref{functionsg} 	
\begin{flalign}
\label{scalar4}
\begin{aligned}
\displaystyle\lim_{T \rightarrow \infty} \mathcal{G} = \mathcal{R} - \big\langle (\Id - K_{\Ai; c^2 + s})^{-1} \textbf{P}_s \Phi, \textbf{P}_s \Psi \big\rangle = g (c, s). 
\end{aligned}
\end{flalign}

\noindent Combining \eqref{limitdeterminantomega1sigma} with \eqref{scalar4} yields  
\begin{flalign*}
\displaystyle\frac{\det \big( \Id + K \big)_{L^2 (\mathcal{C})}}{\mathcal{F} T^{1 / 3} (1 - \omega)} = g(c, s) \det \big( \Id - K_{\Ai; c^2 + s} \big) ,   
\end{flalign*}

\noindent from which we deduce Proposition \ref{convergencedeterminant} and Proposition \ref{convergencedeterminantmodel}.

\appendix

\section{Mapping to the Six-Vertex Model}

\label{SixVertex}

As mentioned in Section \ref{Introduction}, there is a mapping from the stochastic six-vertex model to Gibbs measures for the more well-studied symmetric six-vertex model \cite{ESMSM, RESI, SVMFBC, TEIOCSRAA, ISVM, TT}. Our goal in this section is to explain this mapping in more detail. Specifically, in Section \ref{PropertyMapping} we define Gibbs measures for the six-vertex model and explain how Gibbs measures for any ferroelectric six-vertex model arise from a stochastic six-vertex model. In Section \ref{Translation}, we explain why the measures $\mathcal{P} (\delta_1, \delta_2; b_1, b_2)$ (introduced in Section \ref{StationaryModel}) are translation-invariant.

\subsection{Gibbs Measures for the Six-Vertex Model}

\label{PropertyMapping}

A \emph{full-plane six-vertex configuration} is a family of up-right directed paths on the infinite lattice $\mathbb{Z}^2$, such that each path is infinite in both (down-left and up-right) directions, and such that no two paths share an edge. These are different from the six-vertex directed-path ensembles defined in Section \ref{StochasticVertex}, in that paths are no longer restricted to the positive quadrant. Let $\Omega$ denote the set of all full-plane six-vertex configurations. 

As in Section \ref{StochasticVertex}, each vertex in $\mathbb{Z}^2$ has one of six possible arrow configurations, listed in the top row of Figure \ref{sixvertexfigure}. Recall the notation of Section \ref{PathEnsembles} that associates an arrow configuration with a quadruple of non-negative integers $(i_1, j_1; i_2, j_2)$; here, $i_1$, $j_1$, $i_2$, and $j_2$ are the numbers of incoming and outgoing vertical and horizontal arrows through the vertex. 

To define Gibbs measures, we require some additional notation. For any subset $\Lambda \subset \mathbb{Z}$, let $\partial \Lambda$ denote the \emph{boundary} of $\Lambda$, which consists of all vertices in $\mathbb{Z}^2 \setminus \Lambda$ that are adjacent to some vertex in $\Lambda$. Furthermore, for any $\omega \in \Omega$, let $\omega |_{ \Lambda}$ denote the restriction of $\omega$ to $ \Lambda$. If $\Lambda$ is finite, let $N_{\omega; \Lambda} (i_1, j_1; i_2, j_2)$ denote the number of vertices $(x, y) \in \Lambda$ with arrow configuration $(i_1, j_1; i_2, j_2)$ for each $i_1, j_1, i_2, j_2$. We abbreviate $N_1 = N_{\omega; \Lambda} (0, 0; 0, 0)$; $N_2 = N_{\omega; \Lambda} (1, 1; 1, 1)$; $N_3 = N_{\omega; \Lambda} (1, 0; 1, 0)$; $N_4 = N_{\omega; \Lambda} (0, 1; 0, 1)$; $N_5 = N_{\omega; \Lambda} (1, 0; 0, 1)$, and $N_6 = N_{\omega; \Lambda} (0, 1, 1, 0)$. 

The following defines Gibbs measures for the six-vertex model. 

\begin{definition}

\label{sixvertexmeasure}

Fix $a_1, a_2, b_1, b_2, c_1, c_2 > 0$. A probability measure $\mu$ on $\Omega$ is said to have the \emph{Gibbs property} (for the six-vertex model with weights $(a_1, a_2, b_1, b_2, c_1, c_2)$) if the following holds. For any finite subset $\Lambda \subset \mathbb{Z}^2$, the probability $\mu_{\Lambda} (\omega)$ of selecting $\omega \in \Omega$, conditioned on $\omega |_{\mathbb{Z}^2 \setminus \Lambda}$, is proportional to $a_1^{N_1} a_2^{N_2} b_1^{N_3} b_2^{N_4} c_1^{N_5} c_2^{N_6}$. 

\end{definition}

Observe in particular that the stochastic six-vertex model $\mathcal{P} (\delta_1, \delta_2)$ defined in Section \ref{StochasticVertex} satisfies the Gibbs property, with $(a_1, a_2, b_1, b_2, c_1, c_2) = (1, 1, \delta_1, \delta_2, 1 - \delta_1, 1 - \delta_2)$. The fact that $\delta_1 \ne \delta_2$ makes this a Gibbs measure for an \emph{asymmetric} six-vertex model. 

The more commonly studied case, however, is the \emph{symmetric} six-vertex model, in which $a_1 = a = a_2$, $b_1 = b = b_2$, and $c_1 = c = c_2$. Many properties of this model (see, for instance, Chapters 8 and 9 of \cite{ESMSM}) can be derived from the \textit{anisotropy parameter} $\Delta = (a^2 + b^2 - c^2) / 2ab$. In particular, the six-vertex model exhibits \emph{phases} depending on whether $\Delta > 1$ (\emph{ferroelectric}), $-1 < \Delta < 1$ (\emph{disordered}), and $\Delta < -1$ (\emph{antiferroelectric}). Here, we will explain a connection between the stochastic six-vertex model and the symmetric six-vertex model in the ferroelectric phase. 

Specifically, let us fix an arbitrary triple $(a, b, c) \in \mathbb{R}_{> 0}$ satisfying $\Delta > 1$ and explain how the stochastic six-vertex model $\mathcal{P} (\delta_1, \delta_2)$ (with suitably selected $\delta_1$ and $\delta_2)$ satisfies the Gibbs property for the ferroelectric, symmetric six-vertex model with parameters $(a, b, c)$. Observe that, by rotating the model clockwise 90 degrees and by interchanging empty horizontal bonds with filled horizontal bonds, we switch the pairs of configurations $\big( (0, 0; 0, 0), (0, 1; 0, 1) \big)$ and $\big( (1, 1; 1, 1), (1, 0; 1, 0) \big)$. This allows us to switch the parameters $a$ and $b$ if necessary to assume that $a \ge b$. In fact, since we will only consider the case $\Delta > 1$ (which implies $a \ne b$) throughout the remainder of this appendix, we assume further that $a > b$. 

Now, under the assumptions $a > b$ and $\Delta = (a^2 + b^2 - c^2) / 2ab > 1$, we will show that we can \emph{conjugate} the hextuple $(a, a, b, b, c, c)$ into the form $(1, 1, \delta_1, \delta_2, 1 - \delta_1, 1 - \delta_2)$ without changing the Gibbs property. This will rely on certain \emph{conservation laws} satisfied by the quantities $N_{\omega; \Lambda} (i_1, j_1; i_2, j_2)$. The first four laws listed below are quickly verified (and are also explained, in the similar case of the six-vertex model on a torus, in Section 3 of \cite{ILMCSVM}), so their derivations are omitted. The fifth law is a consequence of the second, third, and fourth laws. 

\begin{enumerate}

\item{ \label{1} The quantity $N_1 + N_2 + N_3 + N_4 + N_5 + N_6 = |\Lambda|$ is constant.}

\item{ \label{2} Conditioned on $\omega|_{\partial \Lambda}$, the quantity $N_2 + N_4 + N_5$ is constant. }

\item{ \label{3} Conditioned on $\omega|_{\partial \Lambda}$, the quantity $N_2 + N_3 + N_6$ is constant. }

\item{ \label{4} Conditioned on $\omega|_{\partial \Lambda}$, the quantity $N_5 - N_6$ is constant. }

\item{ \label{5} Conditioned on $\omega|_{\partial \Lambda}$, the quantity $N_3 - N_4$ is constant. }

\end{enumerate}

In particular, in view of conservation laws \ref{1}, \ref{4}, and \ref{5}, the quantity $a_1^{N_1} a_2^{N_2} b_1^{N_3} b_2^{N_4} c_1^{N_5} c_2^{N_6}$ changes only by a proportionality constant if we replace the weights $(a_1, a_2, b_1, b_2, c_1, c_2)$ with $(w a_1, w a_2, w t^{-1} b_1, w t b_2, w \xi c_1, w \xi^{-1} c_2)$, for any $w, t, \xi \in \mathbb{R}$. This implies that the Gibbs property for the six-vertex model is not altered under this conjugation of weights. 

Now, let us set $(a_1, a_2, b_1, b_2, c_1, c_2) = (a, a, b, b, c, c)$ with $\Delta > 1$ and $a > b$, and also set 
\begin{flalign}
\label{transformparameters}
\begin{aligned}
& \delta_1 = \displaystyle\frac{b}{a} \big( \Delta - \sqrt{\Delta^2 - 1} \big); \quad \delta_2 = \displaystyle\frac{b}{a} \big( \Delta + \sqrt{\Delta^2 - 1} \big); \\
& w = a^{-1}; \quad t= \Delta + \sqrt{\Delta^2 - 1}; \quad \xi = \displaystyle\frac{a (1 - \delta_1)}{c}. 
\end{aligned} 
\end{flalign}

It is quickly verified under the assumptions $a > b$ and $\Delta > 1$ that $\delta_1, \delta_1 \in (0, 1)$. Furthermore, we find that $(w a, w a, w t^{-1} b, w t b, w \xi c, w \xi^{-1} c) = (1, 1, \delta_1, \delta_2, 1 - \delta_1, 1 - \delta_2)$. This implies that the stochastic six-vertex model $\mathcal{P} (\delta_1, \delta_2)$ (with $\delta_1$ and $\delta_2$ defined as in \eqref{transformparameters}) satisfies the Gibbs property for the ferroelectric, symmetric six-vertex model with parameters $(a, a, b, b, c, c)$.

\subsection{Translation Invariance} 

\label{Translation}

A widely studied subclass of Gibbs measures for the six-vertex model consist of those that are \emph{translation-invariant}. Although we do not know of a complete classification of translation-invariant Gibbs measures for the six-vertex model, it is generally believed\footnote{Such a statement has been established for dimer and square ice models \cite{RS}, but we are not certain if those proofs apply to the general six-vertex model.} \cite{SVMFBC, ISVM} that they form a two-parameter family parametrized by pairs $(h, v) \in \mathbb{R}_{> 0}^2$, called \emph{slopes}. The \emph{slope} $(h, v)$ of a translation-invariant Gibbs measure corresponds to the rates of growth of the height function of the six-vertex model in the horizontal and vertical directions. Explicitly, 
\begin{flalign}
\label{slopeheight}
h = \displaystyle\lim_{Z \rightarrow \infty} \displaystyle\frac{\mathfrak{H} (X + Z, Y) - \mathfrak{H} (X, Y)}{Z}; \qquad v = \displaystyle\lim_{Z \rightarrow \infty} \displaystyle\frac{\mathfrak{H} (X, Y) - \mathfrak{H} (X, Y + Z)}{Z},
\end{flalign}

\noindent where $\mathfrak{H}$ denotes the height function (also called the current) of the six-vertex model, which was at first in Section \ref{StochasticVertex} only defined on the positive quadrant but can be quickly extended to all of $\mathbb{Z}^2$.

Let us explain how the stochastic six-vertex model measures $\mathcal{P} (\delta_1, \delta_2; b_1, b_2)$ (defined in Section \ref{StationaryModel}) form a one-parameter subfamily of translation-invariant Gibbs measures for the ferroelectric, symmetric six-vertex model with weights $(a, a, b, b, c, c)$; here, the parameters $a$, $b$, and $c$ are related to the parameters $\delta_1$ and $\delta_2$ through the identity \eqref{transformparameters}. The fact that these measures satisfy the Gibbs property for these six-vertex models was discussed in the previous Section \ref{PropertyMapping}; thus, it suffices to verify the translation-invariance.

To that end, we have the lemma below. In what follows, for any $(x, y) \in \mathbb{Z}^2$, $\varphi^{(v)} (x, y)$ denotes the indicator for the event that an arrow vertically enters through $(x, y)$; that is, $\varphi^{(v)} (x, y)$ denotes the indicator for the event that an arrow points from $(x, y - 1)$ to $(x, y)$. Similarly, $\varphi^{(h)} (x, y)$ denotes the event that an arrow horizontally enters through $(x, y)$. 

\begin{lem}

\label{stationarystochastic}

Fix $\delta_1, \delta_2, b_1, b_2 \in (0, 1)$, let $\kappa = (1 - \delta_1) / (1 - \delta_2)$, and let $\beta_i = b_i / (1 - b_i)$ for each $i \in \{1 , 2 \}$. Assume that $\beta_1 = \kappa \beta_2$, and consider the stochastic six-vertex model $\mathcal{P} (\delta_1, \delta_2; b_1, b_2)$. Let $(x, y) \in \mathbb{Z}_{> 0}^2$. 

Then, $\big\{ \varphi^{(h)} (x, y), \varphi^{(h)} (x, y + 1), \ldots \big\} \cup \big\{ \varphi^{(v)} (x, y), \varphi^{(v)} (x + 1, y), \ldots \big\}$ are mutually independent. Furthermore, $\big\{ \varphi^{(h)} (x, y), \varphi^{(h)} (x, y + 1), \ldots \big\}$ are $0-1$ Bernoulli random variables with means $b_1$, and $\big\{ \varphi^{(v)} (x, y), \varphi^{(v)} (x + 1, y), \ldots \big\}$ are $0-1$ Bernoulli random variables with means $b_2$. 

\end{lem}

\begin{figure}[t]

\begin{center}

\begin{tikzpicture}[
      >=stealth,
			scale = .8
			]

			\draw[-, black] (-7.5, -.8) -- (7.5, -.8);
			\draw[-, black] (-7.5, 0) -- (7.5, 0);
			\draw[-, black] (-7.5, 2) -- (7.5, 2);
			\draw[-, black] (-7.5, -.8) -- (-7.5, 2);
			\draw[-, black] (7.5, -.8) -- (7.5, 2);
			\draw[-, black] (-5, -.8) -- (-5, 2);
			\draw[-, black] (5, -.8) -- (5, 2);
			\draw[-, black] (-2.5, -.8) -- (-2.5, 2);
			\draw[-, black] (2.5, -.8) -- (2.5, 2);
			\draw[-, black] (0, -.8) -- (0, 2);

			\draw[->, black,  thick] (3.85, 1) -- (4.65, 1);
			\draw[->, black,  thick] (3.75, .1) -- (3.75, .9);

			\draw[->, black,  thick] (-1.25, .1) -- (-1.25, .9);
			\draw[->, black,  thick] (-1.25, 1.1) -- (-1.25, 1.9);

			\draw[->, black,  thick] (1.35, 1) -- (2.15, 1);
			\draw[->, black,  thick] (.35, 1) -- (1.15, 1);
			
			\draw[->, black,  thick] (6.25, 1.1) -- (6.25, 1.9);
			\draw[->, black,  thick] (5.35, 1) -- (6.15, 1);
			
			\draw[->, black,  thick] (-3.75, 1.1) -- (-3.75, 1.9);
			\draw[->, black,  thick] (-3.75, .1) -- (-3.75, .9);
			\draw[->, black,  thick] (-3.65, 1) -- (-2.85, 1);
			\draw[->, black,  thick] (-4.65, 1) -- (-3.85, 1);
				
			\filldraw[fill=gray!50!white, draw=black] (-6.25, 1) circle [radius=.1] node [black,below=26, scale = .7] {$(1 - b_1)(1 - b_2)$};
			\filldraw[fill=gray!50!white, draw=black] (-3.75, 1) circle [radius=.1] node [black,below=26, scale = .7] {$b_1 b_2 $};
			\filldraw[fill=gray!50!white, draw=black] (-1.25, 1) circle [radius=.1] node [black,below=26, scale = .7] {$(1 - b_1) b_2 \delta_1$};
			\filldraw[fill=gray!50!white, draw=black] (1.25, 1) circle [radius=.1] node [black,below=26, scale = .7] {$b_1 (1 - b_2) \delta_2$};
			\filldraw[fill=gray!50!white, draw=black] (3.75, 1) circle [radius=.1] node [black,below=26, scale = .7] {$(1 - b_1) b_2 (1 - \delta_1)$};
			\filldraw[fill=gray!50!white, draw=black] (6.25, 1) circle [radius=.1] node [black,below=26, scale = .7] {$b_1 (1 - b_2) (1 - \delta_2)$};

\end{tikzpicture}

\end{center}

\caption{\label{sixvertexfigurecritical} The top row in the chart depicts the six possible arrow configurations for the six-vertex model; the bottom row shows the corresponding probabilities given by the Gibbs measure $\mathcal{P} (\delta_1, \delta_2; b_1, b_2)$. }
\end{figure}

\begin{proof}

We establish this lemma by induction on $x + y$. It holds in the base case $(x, y) = (1, 1)$ due to the definition of double-sided $(b_1, b_2)$-Bernoulli initial data. 

Thus, assume that it holds for $(x, y) = (X, Y) \in \mathbb{Z}_{> 0}^2$. It suffices to show that the lemma also holds when $(x, y) = (X, Y + 1)$ and $(x, y) = (X + 1, Y)$. The two cases are entirely analogous, so we only address the first. By shifting the lattice to the left $X - 1$ spaces and down $Y - 1$ spaces, we may assume that $(X, Y) = (1, 1)$; we would like to establish the lemma for $(x, y) = (1, 2)$. 

To that end, recall that there are six possible arrow configurations at the vertex $(1, 1)$; they are listed in the top row of Figure \ref{sixvertexfigurecritical}. Their corresponding probabilities are listed in the second row of that Figure; the factors of $b_1$, $b_2$, $1 - b_1$, and $1 - b_2$ are due to the double-sided Bernoulli boundary data prescribing the entrance of arrows through $(1, 1)$. 

Observe that, under the assumption $\beta_1 = \kappa \beta_2$, the sum of the second, third, and sixth probabilities in Figure \ref{sixvertexfigurecritical} is $b_1 b_2 + (1 - b_1) b_2 \delta_1 + b_1 (1 - b_2) (1 - \delta_2) = b_2$. This implies that $\varphi^{(v)} (1, 2)$ is a $0-1$ Bernoulli random variable with mean $b_2$. Similarly, the sum of the second, fourth, and fifth probabilities in Figure \ref{sixvertexfigurecritical} is equal to $b_1$, which implies that $\varphi^{(h)} (2, 1)$ is a $0-1$ Bernoulli random variable with mean $b_1$. Furthermore, the covariance of these two random variables is $0$, since the the probability that both a vertical arrow enters through $(1, 2)$ and a horizontal arrow enters through $(2, 1)$ is $b_1 b_2$ (the second probability in Figure \ref{sixvertexfigurecritical}). Thus, $\varphi^{(v)} (1, 2)$ and $\varphi^{(h)} (2, 1)$ are independent, since they are $0-1$ Bernoulli random variables with zero covariance. 

Now we can repeat this procedure to deduce that $\varphi^{(v)} (2, 2)$ and $\varphi^{(h)} (3, 1)$ are independent $0-1$ Bernoulli random variables with means $b_2$ and $b_1$, respectively. Furthermore, both of these random variables only depend on $\varphi^{(h)} (2, 1)$ and $\varphi^{(v)} (2, 1)$, which are both independent of $\varphi^{(v)} (1, 2)$. This implies that $\varphi^{(v)} (2, 2)$ and $\varphi^{(h)} (3, 1)$ are independent from $\varphi^{(v)} (1, 2)$. Repeating, we find that $\big\{ \varphi^{(v)} (1, 2), \varphi^{(v)} (2, 2), \ldots \big\}$ are mutually independent $0-1$ Bernoulli random variables with means $b_2$. Furthermore, these random variables are all also independent from $\big\{ \varphi^{(h)} (1, 2), \varphi^{(h)} (1, 3), \ldots \big\}$, which are mutually independent $0-1$ Bernoulli random variables with means $b_1$, due to the double-sided Bernoulli initial data. 

This establishes the lemma for $(x, y) = (1, 2)$, thereby completing the proof. 
\end{proof}

In view of Lemma \ref{stationarystochastic}, the stochastic six-vertex measure $\mathcal{P} (\delta_1, \delta_2; b_1, b_2)$ can be extended to all of $\mathbb{Z}^2$ as follows. For each non-negative integer $N$, define $\mathcal{P} (\delta_1, \delta_2; b_1, b_2; N)$ to be the measure on $\mathbb{Z}_{> -N}^2$ formed by shifting $\mathcal{P} (\delta_1, \delta_2; b_1, b_2)$ $N$ spaces down and $N$ spaces to the left. By Lemma \ref{stationarystochastic}, the measures $\big\{ \mathcal{P} (\delta_1, \delta_2; b_1, b_2; N) \big\}_{N \in \mathbb{Z}_{\ge 0}}$ are compatible in the sense that $\mathcal{P} (\delta_1, \delta_2; b_1, b_2; M)$ is the restriction of $\mathcal{P} (\delta_1, \delta_2; b_1, b_2; N)$ to $\mathbb{Z}_{> -M}^2$, for any $N \ge M$. 

Therefore, we can define the limit $\mathcal{P} (\delta_1, \delta_2; b_1, b_2; \infty) = \lim_{N \rightarrow \infty} \mathcal{P} (\delta_1, \delta_2; b_1, b_2; N)$. This limit is quickly seen to be a translation-invariant measure satisfying the second statement of Lemma \ref{stationarystochastic} for any $(x, y) \in \mathbb{Z}^2$; Lemma \ref{stationarystochastic} implies that the slope of this measure is $(b_1, b_2)$. Combining this statement with the discussion of Section \ref{PropertyMapping}, we deduce that the stochastic six-vertex measures $\mathcal{P} (\delta_1, \delta_2, b_1, b_2; \infty)$ form a class of translation-invariant Gibbs measures for the ferroelectric, symmetric six-vertex model (with corresponding weights $a, b, c$) that characterize a one-parameter family of slopes. 

We will next evaluate the free energy of these Gibbs measures, which we recall is defined as follows. For any positive integer $M$, let $\Lambda = \Lambda_M \subset \mathbb{Z}^2$ denote the $M \times M$ grid consisting of vertices of the form $(x, y) \in \mathbb{Z}^2$, with $0 \le x < M$ and $0 \le y < M$. For any full plane six-vertex configuration $\omega \in \Omega$, let $E_{\Lambda} (\omega) = a^{N_1 + N_2} b^{N_3 + N_4} c^{N_5 + N_6}$ denote the weight of this configuration under the six-vertex model, restricted to $\Lambda$; here, $N_1$, $N_2$, $N_3$, $N_4$, $N_5$, and $N_6$ were defined directly above Definition \ref{sixvertexmeasure}. For any six-vertex configuration $\widetilde{\omega}$ on $\mathbb{Z}^2 \backslash \Lambda$, define the \textit{partition function on $\Lambda$ with $\widetilde{\omega}$ boundary conditions} to be $Z_{\Lambda} (\widetilde{\omega}) = \sum_{\omega} E_{\Lambda} (\omega)$, where $\omega \in \Omega$ is summed over all six-vertex states on $\mathbb{Z}^2$ such that $\omega |_{\mathbb{Z}^2 \backslash \Lambda} = \widetilde{\omega}$. Then for any six-vertex Gibbs measure $\mu$, define $Z_{\Lambda} (\mu) = \mathbb{E} \big[ Z_{\Lambda} (\widetilde{\omega}) \big]$, where $\widetilde{\omega}$ is averaged with respect to $\mu$. The \textit{free energy per site} of this Gibbs measure then be $\lim_{M \rightarrow \infty} M^{-2} \log Z_{\Lambda} (\mu)$, if it exists.

Proposition \ref{energymeasure} explicitly provides the free energy per site of the measures $\mathcal{P} (\delta_1, \delta_2;  b_1, b_2; \infty)$; as expected, the result below matches with the result of Bukman-Shore (stated directly below equation (3.38) of \cite{TCPFSVM}) that evaluates the free energy of the ferroelectric six-vertex model at the conical singularity. The fact that $\mathcal{P} (\delta_1, \delta_2; b_1, b_2; \infty)$ is a product measure on horizontal and vertical lines makes this partition function significantly simpler to evaluate than in the case of dimer models \cite{VPDT} and the six-vertex model with domain-wall boundary data \cite{RMSSVM}. 

\begin{prop}

\label{energymeasure}

Fix $a, b, c \in \mathbb{R}_{> 0}$ satisfying $\Delta = (a^2 + b^2 - c^2) / 2ab > 1$ and $a > b$. Define $\delta_1$ and $\delta_2$ as in \eqref{transformparameters}, and denote $\kappa = (1 - \delta_1) / (1 - \delta_2)$. Let $b_1, b_2 \in (0, 1)$ satisfy $\beta_1 = \kappa \beta_2$, where $\beta_i = b_i / (1 - b_i)$ for each $i \in \{1 , 2 \}$. 

Then, the translation-invariant measure $\mathcal{P} (\delta_1, \delta_2; b_1, b_2; \infty)$ prescribes a translation-invariant Gibbs measure for the ferroelectric, symmetric six-vertex model with weights $(a, a, b, b, c, c)$, with slope equal to $(h, v) = (b_2, b_1)$. The free energy per site is equal to
\begin{flalign}
\label{energy}
F (a, b, c) = (h - v) \log \big( \Delta - \sqrt{\Delta^2 - 1} \big) - \log a. 
\end{flalign}

\end{prop}

\begin{proof}[Proof (Heuristic)]

The first statement of the proposition was established above, so it suffices to verify the free energy identity \eqref{energy}. We only provide a heuristic for this derivation; it can be quickly turned into a complete proof through a few additional estimates. 

Let $M$ be some large positive integer, and define $\Lambda = \Lambda_M \subset \mathbb{Z}^2$ as in the definition of free energy above. Denoting $t$ and $\xi$ as in \eqref{transformparameters}, we find that
\begin{flalign}
\label{energydelta1delta2}
E_{\Lambda} (\omega) = a^{N_1 + N_2 + N_3 + N_4 + N_5 + N_6} t^{N_3 - N_4} \xi^{N_6 - N_5} \delta_1^{N_3} \delta_2^{N_4} (1 - \delta_1)^{N_5} (1 - \delta_2)^{N_6}.
\end{flalign}

Now, $N_1 + N_2 + N_3 + N_4 + N_5 + N_6 = |\Lambda| = M^2$, deterministically. Furthermore, $N_2 + N_3 + N_6$ is equal to the number of vertical arrows in $\omega|_{\Lambda}$, and $N_2 + N_4 + N_5$ is equal to the number of horizontal arrows. Thus, $N_3 - N_4 + N_6 - N_5$ is equal to the number of horizontal arrows in $\omega|_{\Lambda}$ subtracted from the number vertical arrows in $\omega |_{\Lambda}$; although this quantity is random (dependent on the random boundary data $\omega |_{\partial \Lambda}$), Lemma \ref{stationarystochastic} implies that it should equal $M^2 \big( b_2 - b_1 + o(1) \big) = M^2 \big( h - v + o(1) \big)$, with high probability $1 - o(1)$. Moreover, since it is quickly verified that $\big| N_6 - N_5 \big| < 2M$ deterministically, \eqref{energydelta1delta2} implies that
\begin{flalign}
\label{energy1delta1delta2}
E_{\Lambda} (\omega) = a^{M^2} t^{M^2 (h - v + o(1))} \widetilde{E} (\omega),
\end{flalign} 

\noindent with high probability, where $\widetilde{E} (\omega) = \delta_1^{N_3} \delta_2^{N_4} (1 - \delta_1)^{N_5} (1 - \delta_2)^{N_6}$ is the weight of $\omega$ (restricted to $\Lambda$) under the stochastic six-vertex model. The stochasticity of this model implies that the sum of $\widetilde{E}_{\omega} (\omega)$ is equal to $1$ so, by \eqref{energy1delta1delta2}, the free energy per site of our Gibbs measure $\mathcal{P} (\delta_1, \delta_2; b_1, b_2; \infty)$ for the symmetric six-vertex model is equal to
\begin{flalign*}
- \displaystyle\lim_{M \rightarrow \infty} \displaystyle\frac{1}{M^2} \log \big( a^{M^2} t^{M^2 (h - v + o(1))} \big) = F(a, b, c), 
\end{flalign*}

\noindent where $F(a, b, c)$ is defined in \eqref{energy}. 
\end{proof}

\section{Fredholm Determinants}

\label{Determinants1} 

In this section, we provide the definition and some properties of Fredholm determinants in the form that is convenient for us. 

\begin{definition}

\label{definitiondeterminant}

Fix a contour $\mathcal{C} \subset \mathbb{C}$ in the complex plane, and let $K: \mathcal{C} \times \mathcal{C} \rightarrow \mathbb{C}$ be a meromorphic function with no poles on $\mathcal{C} \times \mathcal{C}$. We define the \emph{Fredholm determinant}
\begin{flalign}
\label{determinantsum}
\det \big( \Id + K \big)_{L^2 (\mathcal{C})} = 1 + \displaystyle\sum_{k = 1}^{\infty} \displaystyle\frac{1}{ (2 \pi \textbf{i} )^k k!} \displaystyle\int_{\mathcal{C}} \cdots \displaystyle\int_{\mathcal{C}} \det \big[ K(x_i, x_j) \big]_{i, j = 1}^k \displaystyle\prod_{j = 1}^k d x_j. 
\end{flalign}

\end{definition}

\begin{rem}

\label{sumconvergesdeterminant}

Generally, the definition of the Fredholm determinant requires that $K$ give rise to a trace-class operator on $L^2 (\mathcal{C})$. We will not use that convention here; instead, we only require that the sum on the right side \eqref{determinantsum} converges absolutely. 

\end{rem}

\begin{rem}

\label{determinantcontourdeform}

Suppose that the contour $\mathcal{C}$ can be continuously deformed to another contour $\mathcal{C}' \subset \mathbb{C}$ without crossing any poles of the kernel $K$. Then, $\det \big( \Id + K \big)_{L^2 (\mathcal{C})} = \det \big( \Id + K \big)_{L^2 (\mathcal{C}')}$, assuming that the right side of \eqref{determinantsum} remains uniformly convergent throughout the deformation process. Indeed, this can be seen directly from the definition \eqref{determinantsum}, since each summand on the right side of this identity remains the same after deforming $\mathcal{C}$ to $\mathcal{C}'$. Alternatively, see Proposition 1 of \cite{AASIC}. 

\end{rem}

The Fredholm determinant satisfies several stability properties that are useful for asymptotic analysis. We record three of these properties here; the first and third appeared as Lemma A.4 and Corollary A.5 of \cite{PTAEPSSVM}, respectively.

\begin{lem}[{\cite[Lemma A.4]{PTAEPSSVM}}]

\label{determinantclosekernels}

Adopt the notation of Definition \ref{definitiondeterminant}, and let $K_1 (z, z'): \mathcal{C} \times \mathcal{C} \rightarrow \mathbb{C}$ be another meromorphic function with no poles on $\mathcal{C} \times \mathcal{C}$. Suppose that there exists a function $\textbf{\emph{K}} : \mathcal{C} \rightarrow \mathbb{R}_{\ge 0}$ such that $\sup_{z' \in \mathcal{C}} \big| K (z, z') \big| \le \textbf{\emph{K}} (z)$ and $\sup_{z' \in \mathcal{C}} \big| K_1 (z, z') \big| \le \textbf{\emph{K}} (z)$ for all $z \in \mathcal{C}$, and also such that $\displaystyle\int_{\mathcal{C}} \big| \textbf{\emph{K}} (z) \big| dz < \infty$. Then, the Fredholm determinants $ \det \big( \Id + K \big)_{L^2 (\mathcal{C})} $ and $ \det \big( \Id + K_1 \big)_{L^2 (\mathcal{C})}$ converge absolutely (as series given by the right side of \eqref{determinantsum}). 

Furthermore, denote $K' (z, z') = K_1 (z, z') - K (z, z')$, and define the constant 
\begin{flalign}
\begin{aligned}
\label{cdeterminant2}
C = \displaystyle\sum_{k = 1}^{\infty} \displaystyle\frac{2^k k^{k / 2}}{(k - 1)!} \displaystyle\int_{\mathcal{C}} \cdots & \displaystyle\int_{\mathcal{C}} \displaystyle\prod_{i = 2}^k  \left| \displaystyle\frac{1}{k} \displaystyle\sum_{j = 1}^k \Big( \big| K (x_i, x_j) \big|^2 + \big| K' (x_i, x_j) \big|^2  \Big)  \right|^{1 / 2}  \\
& \qquad \times \left| \displaystyle\frac{1}{k} \displaystyle\sum_{j = 1}^k \big| K' (x_1, x_j) \big|^2 \right|^{1 / 2}  \displaystyle\prod_{i = 1}^k dx_i.
\end{aligned}
\end{flalign}

\noindent Then, $\Big| \det \big( \Id + K_1 \big)_{L^2 (\mathcal{C})} - \det \big( \Id + K \big)_{L^2 (\mathcal{C})}  \Big| < C$. 

\end{lem}

\begin{lem}

\label{determinantsmallcontour}

Adopt the notation of Definition \ref{definitiondeterminant}, and fix some $\varepsilon > 0$. Assume that $\mathcal{C}$ is compact and can be decomposed into the union of two contours $\mathcal{C} = \mathcal{C}^{(1)} \cup \mathcal{C}^{(2)}$, whose interiors are disjoint, such that $\sup_{z \in \mathcal{C}^{(2)}, z' \in \mathcal{C}} \big| K (z, z') \big| < \varepsilon$. 

Then, we have that 
\begin{flalign*}
\Big| \det \big( \Id + K \big)_{L^2 (\mathcal{C})} - \det \big( \Id + K \big)_{L^2 (\mathcal{C}^{(1)})}  \Big| < C \varepsilon, 
\end{flalign*}

\noindent where 
\begin{flalign*} 
C = \displaystyle\sum_{k = 1}^{\infty} \displaystyle\frac{k^{k / 2}}{\pi^k k!} \displaystyle\int_{\mathcal{C}} \cdots \displaystyle\int_{\mathcal{C}} \displaystyle\prod_{i = 2}^k \left( \displaystyle\frac{1}{k} \displaystyle\sum_{j = 1}^k \big| K (x_i, x_j) \big|^2 \right)^{1 / 2} \displaystyle\prod_{i = 1}^k d x_i. 
\end{flalign*}

\end{lem}

\begin{proof}

Since $\mathcal{C}$ is compact and $K$ is meromorphic with no poles on $\mathcal{C} \times \mathcal{C}$, $\big| K (w, w') \big|$ is uniformly bounded over all $w, w' \in \mathcal{C}$. Thus, from the first part of Lemma \ref{determinantclosekernels}, the series \eqref{determinantsum} for the Fredholm determinants $\det \big( \Id + K \big)_{L^2 (\mathcal{C})}$ and $\det \big( \Id + K \big)_{L^2 (\mathcal{C}^{(1)})}$ are both absolutely convergent. 

Subtracting these two series yields 
\begin{flalign}
\label{determinantvertexrightc2}
\begin{aligned}
\Big| \det \big( \Id + K \big)_{L^2 (\mathcal{C})} & - \det \big( \Id + K \big)_{L^2 (\mathcal{C}^{(1)})} \Big|  \\ 
& \le  \displaystyle\sum_{k = 1}^{\infty} \displaystyle\sum_{\textbf{n} } \displaystyle\frac{1}{(2 \pi)^k k!} \displaystyle\int_{\mathcal{C}^{(n_1)}} \cdots \displaystyle\int_{\mathcal{C}^{(n_k)}} \left| \det \big[ K (x_i, x_j) \big]_{i, j = 1}^k \right| \displaystyle\prod_{j = 1}^k d x_j, 
\end{aligned}
\end{flalign}

\noindent where $\textbf{n} = (n_1, n_2, \ldots , n_k) \in \{ 1, 2 \}^k$ is summed over all ordered $k$-tuples of elements in $\{ 1, 2 \}$, except for $\{ 1, 1, \ldots , 1 \}$. Equivalently, each $x_j$ is either integrated along $\mathcal{C}^{(1)}$ or $\mathcal{C}^{(2)}$, but not all are integrated along $\mathcal{C}^{(1)}$. 

Now, let us estimate an integral on the right side of \eqref{determinantvertexrightc2} in which some $x_r$ is integrated along $\mathcal{C}^{(2)}$; by symmetry, we can assume $r = 1$. Applying Hadamard's inequality and the fact that $|K (x_1, x_j)| < \varepsilon$ for each $j$, we find that this integral is bounded by 
\begin{flalign}
\label{ksmallc1}
\begin{aligned}
\displaystyle\int_{\mathcal{C}^{(1)}} \displaystyle\int_{\mathcal{C}} \cdots & \displaystyle\int_{\mathcal{C}} \left| \det \big[ K (x_i, x_j) \big]_{i, j = 1}^k \right| \displaystyle\prod_{j = 1}^k d x_j \\
& \le \displaystyle\int_{\mathcal{C}^{(1)}} \displaystyle\int_{\mathcal{C}} \cdots \displaystyle\int_{\mathcal{C}} k^{k / 2} \displaystyle\prod_{i = 1}^k \left( \displaystyle\frac{1}{k} \displaystyle\sum_{j = 1}^k \big| K (x_i, x_j) \big|^2 \right)^{1 / 2} d x_i \\
& \le \varepsilon k^{k / 2} \displaystyle\int_{\mathcal{C}^{(1)}} \displaystyle\int_{\mathcal{C}} \cdots \displaystyle\oint_{\mathcal{C}} \displaystyle\prod_{i = 2}^k \left( \displaystyle\frac{1}{k} \displaystyle\sum_{j = 1}^k \big| K (x_i, x_j) \big|^2 \right)^{1 / 2} d x_1 \displaystyle\prod_{i = 2}^k d x_i. 
\end{aligned}
\end{flalign}

\noindent Since there are $2^k - 1$ choices for $\textbf{n}$, the lemma follows from inserting \eqref{ksmallc1} into \eqref{determinantvertexrightc2}. 
\end{proof}

\begin{lem}[{\cite[Corollary A.5]{PTAEPSSVM}}]

\label{determinantlimitkernels}

Adopt the notation of Definition \ref{definitiondeterminant}. Let $\{ \mathcal{C}_T^{(2)} \}_{T \in \mathbb{R}_{> 0}}$ be a family of contours in the complex plane, and denote $\mathcal{C}_T = \mathcal{C} \cup \mathcal{C}_T^{(2)}$ for each positive real number $T$. For each $T > 0$, let $K_T: \mathcal{C}_T \times \mathcal{C}_T \rightarrow \mathbb{C}$ be a meromorphic function with no poles on $\mathcal{C}_T \times \mathcal{C}_T$.  

Assume that there exists a function $\textbf{\emph{K}} : \mathbb{C} \rightarrow \mathbb{R}_{\ge 0}$ such that the following three properties hold. 

\begin{itemize}

\item{For each $T > 0$ and $z \in \mathcal{C}_T$, we have that $\sup_{z' \in \mathcal{C}_T} \big| K_T (z, z') \big| < \textbf{\emph{K}} (z)$.}

\item{ There exists a constant $B > 0$ such that $\displaystyle\int_{\mathcal{C}_T} \big| \textbf{\emph{K}} (z) \big| dz < B$, for each $T > 0$.}

\item{ For each positive real number $\varepsilon > 0$, there exists some real number $M = M_{\varepsilon}$ such that $\displaystyle\int_{\mathcal{C}_T^{(2)}} \big| \textbf{\emph{K}} (z) \big| dz < \varepsilon$, for all $T > M$. }

\end{itemize} 

\noindent If $\lim_{T \rightarrow \infty} \textbf{\emph{1}}_{z, z' \in \mathcal{C}_T} K_T (z, z') = \textbf{\emph{1}}_{z, z' \in \mathcal{C}} K (z, z')$ for each fixed $z, z' \in \mathbb{C}$, then 
\begin{flalign*}
\displaystyle\lim_{T \rightarrow \infty} \Big( \det \big( \Id + K \big)_{L^2 (\mathcal{C})} - \det \big( \Id + K_T \big)_{L^2 (\mathcal{C}_T)} \Big) = 0.
\end{flalign*} 
\end{lem}

\section{Determinantal Generating Series} 

\label{SeriesDeterminant}

The goal of this section is to explain how to produce Fredholm determinant identities from $q$-moment identities of the type given in Proposition \ref{qms}; this will mainly follow Section 3 and Section 5 of \cite{DDQA}. Throughout this section, we fix $q \in (0, 1)$. 

An immediate inconvenience to right side of the equation \eqref{qmsequation} is that the contours for the $y_i$ are all different. The following lemma remedies this issue. Its proof is omitted, since it is very similar to that of Proposition 5.2 of \cite{DDQA} (see also Proposition 3.2.1 of \cite{MP} or Proposition 7.2 of \cite{DRPCI}).

\begin{lem}

\label{deformcontoursmoments}

Fix $k, n \in \mathbb{Z}_{> 0}$; $v \in \mathbb{C} \setminus \{ 0 \}$; $U = (u_1, u_2, \ldots , u_n) \subset \mathbb{C}$; $S = (s_1, s_2, \ldots ) \subset \mathbb{C}$; and $\Xi = (\xi_1, \xi_2, \ldots ) \subset \mathbb{C}$. Assume that $(v^{-1}, U; \Xi, S)$ is suitably spaced in the sense of Definition \ref{spacedparameters}. 

Let $f$ be a meromorphic function, whose only poles are possibly at $v^{-1}$, $\{ u_j^{-1} \}_{j \ge 1}$ and $\{ s_j \xi_j \}_{j \ge 1}$. Also let $\mathscr{C} \subset \mathbb{C}$ be a closed contour satisfying the following two properties. First, $\mathscr{C}$ contains its image under multiplication by $q$ (in particular, it contains $0$). Second, $\mathscr{C}$ contains $\bigsqcup_{j = 1}^k \{ q^{1 - j} v^{-1} \} $ and each $u_j^{-1}$, but leaves outside each $q^{-1} u_j^{-1}$ and each $s_j \xi_j$. 

Then, 
\begin{flalign}
\label{contoursumdeform} 
\begin{aligned}
& \displaystyle\frac{q^{\binom{k}{2}}}{(2 \pi \textbf{\emph{i}})^k} \displaystyle\oint \cdots \displaystyle\oint \displaystyle\prod_{1 \le i < j \le k} \displaystyle\frac{y_i - y_j}{y_i - q y_j} \displaystyle\prod_{i = 1}^k f(y_i) y_i^{-1} d y_i \\
& \quad = (q; q)_k \displaystyle\sum_{|\lambda| = k} \displaystyle\frac{1}{(2 \pi \textbf{\emph{i}})^{\ell (\lambda)} \prod_{j = 1}^{\infty} m_j!} \displaystyle\oint \cdots \displaystyle\oint \det \left[ \displaystyle\frac{1}{w_i - q^{\lambda_j} w_j} \right]_{i, j = 1}^{\ell (\lambda)} \displaystyle\prod_{i = 1}^{\ell (\lambda)} d w_i \displaystyle\prod_{j = 0}^{\lambda_i - 1} f(q^j w_i). 
\end{aligned}
\end{flalign}

\noindent On the left side of \eqref{contoursumdeform}, each $y_i$ is integrated along the contour $\gamma_i (v^{-1}, U; \Xi, S)$ from Definition \ref{firstcontours}. On the right side of \eqref{contoursumdeform}, each $w_i$ is integrated along $\mathscr{C}$. 

\end{lem}

Next, to produce Fredholm determinants, we require the following types of contours; they appeared previously in \cite{DDQA} as Definition 3.5 (in the case $\delta = 1 / 2$). 

\begin{definition}

\label{contourslong}

Let $R, d, \delta > 0$ be positive real numbers with $d, \delta < 1$ and $R > 1$. Let $D_{R, d, \delta} \subset \mathbb{C}$ denote the contour in the complex plane, with nondecreasing imaginary part, formed by taking the union of intervals 
\begin{flalign*}
(R - \textbf{i} \infty, R - \textbf{i} d ) \cup [R - \textbf{i} d , \delta - \textbf{i} d ] \cup [\delta - \textbf{i} d, \delta + \textbf{i} d] \cup [\delta + \textbf{i} d, R + \textbf{i} d ] \cup (R + \textbf{i} d, R + \textbf{i} \infty). 
\end{flalign*}

\noindent Furthermore, let $k > R$ be a positive integer. Let $z \in \mathbb{C}$ be the complex number satisfying $\Re z = R$, $\Im z > 0$, and $|z - \delta| = k$. Let $I$ denote the minor arc of the circle in the complex half-plane, centered at $\delta$, with radius $k$, connecting $z$ to its conjugate $\overline{z}$. Then let $D_{R, d, \delta; k} \subset \mathbb{C}$ denote the negatively oriented contour in the complex plane, formed by the union 
\begin{flalign*}
(\overline{z}, R - \textbf{i} d ) \cup [R - \textbf{i} d , \delta - \textbf{i} d ] \cup [\delta - \textbf{i} d, \delta + \textbf{i} d] \cup [\delta + \textbf{i} d, R + \textbf{i} d] \cup (R + \textbf{i} d, z) \cup I. 
\end{flalign*}
\end{definition}

Observe that $D_{R, d, \delta; k}$ approximates $D_{R, d, \delta}$ as $k$ tends to $\infty$. Examples of these contours are given in Figure \ref{contourslongfigure}. 

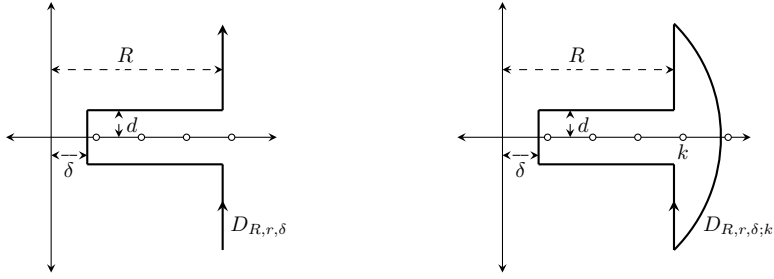
\begin{figure}

\begin{center}

\begin{tikzpicture}[
      >=stealth,
			scale = .6
			]
			
			\draw[<->, black] (-7, 0) -- (-1, 0); 
			\draw[<->, black] (-6, -3) -- (-6, 3); 
			\draw[<->, black] (9.5, 0) -- (3, 0); 
			\draw[<->, black] (4, -3) -- (4, 3); 
			
			\draw[->, black, thick] (-2.2, -2.5) -- (-2.2, -1.4) node[black, below = 10, right = 0, scale = .8] {$D_{R, r, \delta}$};
			\draw[-, black, thick] (-2.2, -1.45) -- (-2.2, -.6);
			\draw[-, black,  thick] (-5.2, -.6) -- (-2.2, -.6);
			\draw[-, black,  thick] (-5.2, -.6) -- (-5.2, .6);
			\draw [<->, black, dashed] (-4.5, 0) -- (-4.5, .6) node[black, right=5, below = 0, scale = .8] {$d$}; 
			\draw [->, black, dashed] (-5.6, -.4) -- (-6, -.4); 
			\draw [->, black, dashed] (-5.6, -.4) -- (-5.2, -.4); 
			\draw [-, white] (-5.595, -.44) -- (-5.605, -.4) node[black, below = 0, scale = .8] {$\delta$}; 
			\draw[-, black,  thick] (-5.2, .6) -- (-2.2, .6);
			\draw[->, black,  thick] (-2.2, .6) -- (-2.2, 2.5);
			\draw[->, black,  dashed] (-4.4, 1.5) -- (-6, 1.5);
			\draw[->, black,  dashed] (-4.4, 1.5) -- (-2.2, 1.5); 
			\draw[->, white] (-4.37, 1.5) -- (-4.36, 1.5) node[black, above = 0, scale = .8] {$R$};

			\draw[->, black,  thick] (7.8, -2.5) -- (7.8, -1.4)  node[black, below = 10, right = 8, scale = .8] {$D_{R, r, \delta; k}$};
			\draw[-, black,  thick] (7.8, -1.45) -- (7.8, -.6);
			\draw[-, black,  thick] (4.8, -.6) -- (7.8, -.6);
			\draw[->, black,  dashed] (5.6, 1.5) -- (4, 1.5);
			\draw[->, black,  dashed] (5.6, 1.5) -- (7.8, 1.5);
			\draw[->, white] (5.63, 1.5) -- (5.64, 1.5) node[black, above = 0, scale = .8] {$R$};
			\draw[-, black,  thick] (4.8, -.6) -- (4.8, .6);
			\draw[-, black,  thick] (4.8, .6) -- (7.8, .6);
			\draw[-, black,  thick] (7.8, 2.5) -- (7.8, .6);
			\draw [<->, black, dashed] (5.5, 0) -- (5.5, .6) node[black, right=5, below = 0, scale = .8] {$d$}; 
			\draw [->, black, dashed] (4.4, -.4) -- (4, -.4); 
			\draw [->, black, dashed] (4.4, -.4) -- (4.8, -.4); 
			\draw [-, white] (4.395, -.4) -- (4.405, -.4) node[black, below = 0, scale = .8] {$\delta$};

			\draw[black, thick] (7.8, -2.5) arc (-45:45:3.55);
			
			\filldraw[fill=white, draw=black] (-5, 0) circle [radius=.07];
			\filldraw[fill=white, draw=black] (-4, 0) circle [radius=.07];
			\filldraw[fill=white, draw=black] (-3, 0) circle [radius=.07];
			\filldraw[fill=white, draw=black] (-2, 0) circle [radius=.07];

			\filldraw[fill=white, draw=black] (5, 0) circle [radius=.07];
			\filldraw[fill=white, draw=black] (6, 0) circle [radius=.07];
			\filldraw[fill=white, draw=black] (7, 0) circle [radius=.07];
			\filldraw[fill=white, draw=black] (8, 0) circle [radius=.07] node[black, below = 0, scale = .8] {$k$};
			\filldraw[fill=white, draw=black] (9, 0) circle [radius=.07];

\end{tikzpicture}

\end{center}

\caption{\label{contourslongfigure} An example of $D_{R, r, \delta}$ is shown to the left, and an example of $D_{R, r, \delta; k}$ is shown to the right; the circles depict the locations of positive integers. }

\end{figure}

The following proposition, which is similar to Proposition 3.6 of \cite{DDQA}, states that one has a Fredholm determinant identity for a certain generating series of a sequence given by terms on the right side of \eqref{contoursumdeform}. Its proof is again omitted, since it is very similar to that of Proposition 3.6 of \cite{DDQA} (see also Proposition 3.3 and Lemma 3.4 of \cite{DDQA}, as well as Proposition 3.2.8 and Theorem 3.2.11 of \cite{MP}). 

\begin{lem}

\label{momentdeterminant}

Let $g$ be a meromorphic function, and denote $f(z) = g(z) / g (qz)$. Let $P \subset \mathbb{C}$ denote a finite set of poles of $f$, and let $\mathscr{C} \subset \mathbb{C}$ denote a closed contour in the complex plane containing $q \mathscr{C}$, $0$, and $P$, but no other poles of $f(z) / z$. Denote the right side of \eqref{contoursumdeform} (in which all $w_i$ are integrated along $\mathscr{C}$) by $\textbf{\emph{m}}_k$, for each $k \ge 1$, and denote $\textbf{\emph{m}}_0 = 1$. 

Assume the following three conditions. 

\begin{itemize}

\item{There exists some constant $B > 0$ such that $f(q^n w) < B$ and $|q^n w - w'| > B^{-1}$ for all $w, w' \in C$ and integers $n \ge 1$.}

\item{There exist positive real numbers $d, \delta \in (0, 1)$ and $R > 1$ such that 
\begin{flalign}
\label{convergencedeterminant2} 	
\displaystyle\inf_{\substack{w, w' \in C \\ k \in (2R, \infty) \cap \mathbb{Z} \\ s \in D_{R, d, \delta; k}}} |q^s w - w'| > 0; \qquad \displaystyle\sup_{\substack{ w \in C \\ k \in (2R, \infty) \cap \mathbb{Z} \\ s \in D_{R, d, \delta; k}}} \displaystyle\frac{g(w)}{g(q^s w)} < \infty. 
\end{flalign}}

\item{There does not exist any $w \in \mathscr{C}$ and $s \in \mathbb{C}$ to the right of the contour $D_{R, d, \delta}$ such that $q^s w$ is a pole of $g$.}

\end{itemize}

\noindent Then, for any complex number $\zeta \in \mathbb{C} \setminus\mathbb{R}_{\ge 0}$ satisfying $|\zeta| < B^{-1} (1 - q)^{-1}$, we have that 
\begin{flalign}
\label{generaldeterminant} 
\displaystyle\sum_{k = 0}^{\infty} \displaystyle\frac{\textbf{\emph{m}}_k \zeta^k}{(q; q)_k} = \det \big( \Id + K_{\zeta} \big)_{L^2 (C)}, 
\end{flalign}

\noindent where the kernel $K$ is defined through 
\begin{flalign*}
K_{\zeta} (w, w') = \displaystyle\frac{1}{2 \textbf{\emph{i}}} \displaystyle\int_{D_{R, d, \delta}} \displaystyle\frac{g(w)}{g(q^r w)} \displaystyle\frac{(-\zeta)^r dr}{\sin (\pi r) \big( q^r w - w' \big)}. 
\end{flalign*}

\noindent Furthermore, the right side of \eqref{generaldeterminant} is analytic in $\zeta \in \mathbb{C} \setminus\mathbb{R}_{\ge 0}$. 
\end{lem}

\end{document}